\documentclass{memo-l}

\usepackage{amsmath,amsfonts,amsthm,verbatim,mathrsfs,amssymb,verbatim,cite,fouridx}
\usepackage[usenames]{color}
\usepackage{tikz,graphicx,calligra}
\usetikzlibrary{calc}
\usepackage{esint}
\usepackage{todonotes}
\usepackage{hyperref}

\allowdisplaybreaks[3]

\newtheorem{theorem}{Theorem}[chapter]
\newtheorem{lemma}[theorem]{Lemma}
\newtheorem{proposition}[theorem]{Proposition}
\newtheorem{corollary}[theorem]{Corollary}

\newtheorem{problem}{Open problem}

\theoremstyle{definition}

\theoremstyle{remark}
\newtheorem{remark}[theorem]{Remark}

\numberwithin{section}{chapter}
\numberwithin{equation}{chapter}

\numberwithin{equation}{section}

\newcommand{\red}[1]{{\color{red}  #1}}
\newcommand{\blue}[1]{{\color{blue}  #1}}
\newcommand{\green}[1]{{\color{gray}  #1}}

\newcommand{\II}{\mathord{I\!I}}

\newcommand{\longroot}{\textrm{long}}
\newcommand{\shortroot}{\textrm{short}}
\newcommand{\fwc}{\omega} 
\newcommand{\fwa}{\varpi} 
\renewcommand{\skew}[2]{#1/#2}
\renewcommand{\subseteq}{\subset}
\renewcommand{\supseteq}{\supset}

\newcommand{\charge}{\textrm{c}}
\DeclareMathOperator{\Trace}{Tr}
\DeclareMathOperator{\weight}{weight}
\DeclareMathOperator*{\Pf}{Pf}
\DeclareMathOperator{\Hom}{Hom}

\DeclareMathOperator{\shape}{shape}

\DeclareMathOperator{\Tab}{SSYT}
\DeclareMathOperator{\sgn}{sgn}

\DeclareMathOperator{\eup}{e}

\DeclareMathOperator{\symp}{sp}

\DeclareMathOperator{\so}{so}
\newcommand{\iup}{\hspace{0.5pt}\mathrm{i}\hspace{0.5pt}}
\DeclareMathOperator{\dup}{d\hspace{-1.5pt}}

\newcommand{\tees}{t_0,t_1,t_2,t_3}
\newcommand{\ceil}[1]{\lceil#1\rceil}
\newcommand{\floor}[1]{\lfloor#1\rfloor}
\newcommand{\Ceil}[1]{\big\lceil#1\big\rceil}
\newcommand{\Floor}[1]{\big\lfloor#1\big\rfloor}
\newcommand{\qhypc}[2]{\fourIdx{}{#1}{}{#2}\phi}
\newcommand{\qhyp}[3]{\fourIdx{}{#1}{#3}{#2}\Phi}
\newcommand{\Whyp}[3]{\fourIdx{}{#1}{#3}{\hspace{-1pt}#2}W}
\renewcommand{\sc}{\scriptstyle}

\newcommand{\Rat}{\mathbb Q}
\newcommand{\Real}{\mathbb R}
\newcommand{\Nat}{\mathbb Z_{\geqslant 0}}
\newcommand{\Z}{\mathbb Z}
\newcommand{\F}{\mathbb F}
\newcommand{\Complex}{\mathbb C}
\newcommand{\Symm}{\mathfrak{S}}

\newcommand{\abs}[1]{\lvert#1\rvert}
\newcommand{\bigabs}[1]{\big\lvert#1\big\rvert}
\newcommand{\la}{\lambda}
\newcommand{\ip}[2]{\langle#1,#2\rangle}
\newcommand{\ipbig}[2]{\big\langle#1,#2\big\rangle}

\newcommand{\qbin}[2]{\genfrac{[}{]}{0pt}{}{#1}{#2}}

\newcommand{\gfrak}{\mathfrak{g}}

\DeclareMathOperator{\ch}{ch}
\DeclareMathOperator{\mult}{mult}
\DeclareMathOperator{\op}{odd}
\DeclareMathOperator{\ep}{even}
\newcommand{\halfm}[1]{(\hspace{-0.7pt}\frac{m}{2}\hspace{-0.7pt})^{#1}}

\makeindex

\begin{document}

\frontmatter

\title[Bounded Littlewood identities]{Bounded Littlewood identities}

\author{Eric M. Rains}
\address{Department of Mathematics, 
California Institute of Technology,
Pasadena, CA 91125, USA}
\email{rains@caltech.edu}

\author{S. Ole Warnaar}
\address{School of Mathematics and Physics,
The University of Queensland, Brisbane, QLD 4072,
Australia}
\email{o.warnaar@maths.uq.edu.au}
\thanks{Work supported by the National Science Foundation 
(grant number DMS-1001645) and the Australian Research Council}




\begin{abstract}
We describe a method, based on the theory of Macdonald--Koorn\-win\-der 
polynomials, for proving bounded Littlewood identities.
Our approach provides an alternative to Macdonald's partial fraction
technique and results in the first examples of bounded Littlewood identities 
for Macdonald polynomials.
These identities, which take the form of decomposition formulas for 
Macdonald polynomials of type $(R,S)$ in terms of ordinary Macdonald 
polynomials, are $q,t$-analogues of known branching formulas for 
characters of the symplectic, orthogonal and special orthogonal groups.
In the classical limit, our method implies that
MacMahon's famous ex-conjecture for the generating function of symmetric 
plane partitions in a box follows from the identification of 
$\big(\mathrm{GL}(n,\mathbb{R}),\mathrm{O}(n)\big)$ as a Gelfand pair.
As further applications, we obtain combinatorial formulas for 
characters of affine Lie algebras; 
Rogers--Ramanujan identities for affine Lie algebras, 
complementing recent results of Griffin et al.; and
quadratic transformation formulas for Kaneko--Macdonald-type basic 
hypergeometric series.
\end{abstract}

\maketitle

\tableofcontents

\chapter*{Acknowledgements}
We are grateful to Fokko van de Bult, Hjalmar Rosengren, Michael Schlosser, 
and Jasper Stokman for helpful discussions on hypergeometric functions, 
Macdonald identities and Macdonald--Koornwinder polynomials. 
We thank Richard Stanley for pointing out the paper \cite{Schur18} 
by Schur, and Chul-hee Lee for his many small corrections to an earlier
version of this paper. Finally, we wish to thank the anonymous referee 
for his or her careful reading of the paper and for noticing a connection with 
the recent work of Cohl et al.\ on basic hypergeometric 
functions~\cite{CCSHW14}.
This has led to the inclusion of the material on very-well poised multiple 
basic hypergeometric series in Section~\ref{Sec_KM}.

\mainmatter


\chapter{Introduction}

\section{Littlewood identities}
In his 1950 text on group characters \cite{Littlewood50}, 
D.~E.~Littlewood presented three identities for Schur functions which can 
be viewed as reciprocals of the Weyl denominator formulas for the classical 
groups $\mathrm{B}_n,\mathrm{C}_n$ and $\mathrm{D}_n$. 
The $\mathrm{B}_n$ case---which earlier appeared in 
an exercise by Schur \cite{Schur18}---is given by
\cite[Eq.~(11.9;\,6)]{Littlewood50}
\begin{equation}\label{Eq_L1196}
\sum_{\la} s_{\la}(x)=
\prod_{i=1}^n \frac{1}{1-x_i}
\prod_{1\leqslant i<j\leqslant n} \frac{1}{1-x_ix_j},
\end{equation}
where $s_{\la}(x)=s_{\la}(x_1,\dots,x_n)$ is a Schur function
indexed by the partition $\la$.
Almost 30 years later, Macdonald \cite{Macdonald79} proved the
following bounded analogue of \eqref{Eq_L1196}:
\begin{equation}\label{Eq_MacB}
\sum_{\substack{\la \\[1pt] \la_1\leqslant m}} s_{\la}(x)=
\frac{\det_{1\leqslant i,j\leqslant n} (x_i^{m+2n-j}-x_i^{j-1})}
{\prod_{i=1}^n (x_i-1)\prod_{1\leqslant i<j\leqslant n}(x_i-x_j)(x_ix_j-1)},
\end{equation} 
for $m$ a nonnegative integer, and observed that it implied MacMahon's 
famous conjecture \cite{MacMahon98} for the generating function of 
symmetric plane partitions in a box.
By reading off the `sequence of diagonal slices'---an idea 
more recent than \cite{Macdonald79}, see e.g., \cite{OR03}---it
immediately follows that the generating function
\begin{equation}\label{Eq_GF-pp}
\sum_{\substack{\pi\subseteq\mathrm{B}(n,n,m) \\[1pt] \pi \text{ symmetric}}} 
q^{\abs{\pi}}
\end{equation}
for symmetric plane partitions $\pi$ contained in a box $\mathrm{B}(n,n,m)$ 
of size $n\times n\times m$ is given by
\begin{equation}\label{Eq_GF-pp-Schur}
\sum_{\substack{\la\\[1pt]\la_1\leqslant m}}s_{\la}(q,q^3,\dots,q^{2n-1}).
\end{equation}
Hence MacMahon's formula \cite{MacMahon98} 
\begin{equation}\label{Eq_MacMahon}
\sum_{\substack{\pi\subseteq\mathrm{B}(n,n,m) \\[1pt] \pi \text{ symmetric}}} 
q^{\abs{\pi}}=\prod_{i=1}^n \frac{1-q^{m+2i-1}}{1-q^{2i-1}}
\prod_{1\leqslant i<j\leqslant n} \frac{1-q^{2(m+i+j-1)}}{1-q^{2(i+j-1)}}
\end{equation}
should follow from the evaluation of the determinant on the
right of \eqref{Eq_MacB} in which the variables $x_i$ are specialised 
as $x_i=q^{2i-1}$ for $1\leqslant i\leqslant n$. 
Since the unspecialised determinant is essentially a 
character of the irreducible $\mathrm{SO}(2n+1,\Complex)$-module 
of highest weight $m\fwc_n$, the required determinant evaluation 
corresponds to the $q$-dimension of this module, and follows from 
the Weyl character formula \cite{Humphreys72}.
To prove \eqref{Eq_MacB}---a $\mathrm{SO}(2n+1,\Complex)$ to 
$\mathrm{GL}(n,\Complex)$ branching formula---Macdonald developed
a partial fraction method, resulting in a more general $t$-analogue for 
Hall--Littlewood polynomials.

Since the work of Littlewood and Macdonald, many additional Littlewood 
identities have been discovered and applied to problems in combinatorics,
$q$-series and representation theory. 
Examples include the enumeration of plane partitions and related
combinatorial objects such as tableaux, tilings, longest
increasing subsequences and alternating-sign matrices 
\cite{BK72,BR01,BW14,BWZ14,Bressoud00,Bressoud99,Desarmenien86,
Desarmenien89,Goulden92,Krattenthaler93,Panova12,Panova14,
Proctor90,Stembridge90,WZ16}, the computation of characters and
branching rules for classical groups and affine Lie algebras
\cite{BW13,IW99,IW06,Kawanaka91,Kawanaka99,Krattenthaler98,Okada98},
and applications to Schubert calculus \cite{LLT89}
Rogers--Ramanujan identities \cite{GOW14,IJZ06,JZ05,Stembridge90,W06}
and multiple elliptic hypergeometric series \cite{LSW09,Rains12}.
Surprisingly, despite the interest in Littlewood identities,
$q,t$-analogues of \eqref{Eq_MacB} 
and other bounded Littlewood identities for Schur and Hall--Littlewood 
polynomials have remained elusive.
In this paper we present an approach to Littlewood 
identities based on the theory of Macdonald--Koornwinder polynomials.
As a result we obtain the missing $q,t$-analogues,
including the following generalisation of Macdonald's 
determinantal formula \eqref{Eq_MacB}.\footnote{A second 
generalisation in terms of $P^{(\mathrm{B}_n,\mathrm{C}_n)}_{\halfm{n}}$
is given in Theorem~\ref{THM_BOUNDED5}.}
\begin{theorem}\label{Thm_qtL}
For $x=(x_1,\dots,x_n)$ and $m$ a nonnegative integer,
\begin{multline}\label{Eq_qt_Littlewood}
\sum_{\substack{\la \\[1pt] \la_1\leqslant m}} 
P_{\la}(x;q,t) 
\prod_{\substack{s\in\la \\[1pt] l'(s) \textup{ even}}} 
\frac{1-q^{m-a'(s)}t^{l'(s)}}{1-q^{m-a'(s)-1}t^{l'(s)+1}}
\prod_{\substack{s\in\la \\[1pt] l(s) \textup{ even}}}
\frac{1-q^{a(s)}t^{l(s)+1}}{1-q^{a(s)+1}t^{l(s)}} \\
=(x_1\cdots x_n)^{\frac{m}{2}}\, 
P^{(\mathrm{B}_n,\mathrm{B}_n)}_{\halfm{n}}(x;q,t,t).
\end{multline}
\end{theorem}

On the left, $P_{\la}(x;q,t)$ is a Macdonald polynomial and 
$a(s),l(s),a'(s),l'(s)$ are the arm-length, leg-length, arm-colength
and leg-colength of the square $s\in\la$.
The (Laurent) polynomial
$P^{(\mathrm{B}_n,\mathrm{B}_n)}_{\halfm{n}}(x;q,t,t)$ on the right
is a Macdonald polynomial attached to the pair of root systems
$(\mathrm{B}_n,\mathrm{B}_n)$, indexed by the rectangular partition 
or `half-partition' $(\frac{m}{2},\dots,\frac{m}{2})$ of length $n$.

Our method also leads to alternative proofs, as well as further 
examples, of Littlewood identities for the characters of irreducible 
highest-weight modules of affine Lie algebras, first discovered 
in \cite{BW13}. 
In particular we find that such characters arise by taking 
suitable limits of Hall--Littlewood polynomials of type $R$.
For example, for the twisted affine Lie algebra $\mathrm{A}_{2n}^{(2)}$
we obtain the following formula for the character of the
highest-weight module $V(m\fwa_0)$ in terms of
modified Hall--Littlewood polynomials $P'_{\la}(x;t)$ and
the large-$r$ limit of the Hall--Littlewood polynomial 
$P_{\halfm{n}}^{(\mathrm{B}_r)}(x;t,t_2)$.
\begin{theorem}\label{Thm_twistedA}
Let $\alpha_0,\dots,\alpha_n$ and $\fwa_0,\dots,\fwa_n$
be the simple roots and fundamental weights
of $\mathrm{A}_{2n}^{(2)}$, and 
$\delta=2\alpha_0+\cdots+2\alpha_{n-1}+\alpha_n$ the null root.
Set
\[
t=\eup^{-\delta}
\quad\text{and}\quad
x_i=\eup^{-\alpha_i-\cdots-\alpha_{n-1}-\alpha_n/2},
\]
and let $\ch V(\Lambda)$ denote the character of the integrable
highest-weight module $V(\Lambda)$ of highest weight $\Lambda$. Then,
for $m$ a positive integer,
\begin{align}\label{Eq_char-A2n2}
\eup^{-m\fwa_0} \ch V(m\fwa_0)
&=\lim_{N\to\infty} t^{\frac{1}{2}mnN^2} 
P^{(\mathrm{B}_{2nN})}_{\halfm{2nN}}(t^{1/2}X;t,0) \\[1mm]
&=\sum_{\substack{\la \\[1pt] \la_1\leqslant m}}
t^{\abs{\la}/2} P'_{\la}(x_1^{\pm},\dots,x_n^{\pm};t), \notag
\end{align}
where $(\dots,ax_i^{\pm},\dots):=(\dots,ax_i,ax_i^{-1},\dots)$ and
$X=X_N(x;t)$ is the alphabet
\begin{equation}\label{Eq_X}
X=(x_1^{\pm},tx_1^{\pm},\dots,t^{N-1}x_1^{\pm},\dots\dots,
x_n^{\pm},tx_n^{\pm},\dots,t^{N-1}x_n^{\pm})
\end{equation}
of cardinality $2nN$.
\end{theorem}
As shown by Griffin et al.~\cite{GOW14}, 
character formulas such as \eqref{Eq_char-A2n2}
imply Rogers--Ramanujan identities through specialisation.
Following their approach, we obtain several new examples of
Rogers--Ramanujan identities labelled by affine Lie algebras.
For example, from \eqref{Eq_char-A2n2} we obtain the following
new identity, where $P_{\la}(x;t)$ is a Hall--Littlewood polynomial,
$\theta(x;q)$ a modified theta function and $(q;q)_{\infty}$ a $q$-shifted 
factorial.
\begin{theorem}[$\mathrm{A}_{2n}^{(2)}$ Rogers--Ramanujan identity]
For $m,n$ positive integers, let $\kappa=m+2n+1$. Then
\begin{multline*}
\sum_{\substack{\la \\[1pt] \la_1\leqslant m}}
q^{\abs{\la}/2} P_{\la}(1,q,q^2,\dots;q^{2n}) \\[-1mm]
=\frac{(q^{\kappa};q^{\kappa})_{\infty}^{n-1}
(q^{\kappa/2};q^{\kappa/2})_{\infty}}
{(q;q)_{\infty}^{n-1}(q^{1/2};q^{1/2})_{\infty}}
\prod_{i=1}^n \theta(q^i;q^{\kappa/2})
\prod_{1\leqslant i<j\leqslant n} \theta(q^{j-i};q^{\kappa})
\theta(q^{i+j};q^{\kappa}).
\end{multline*}
\end{theorem}

\section{Outline}
The remainder of this paper is organised as follows.
In the next chapter we review some standard material from
Macdonald--Koornwinder theory.
This includes a discussion of ordinary (or type $\mathrm{A}$)
Macdonald polynomials, Koornwinder polynomials and their
lifted and virtual analogues,
generalised Macdonald polynomials of type $(R,S)$, 
Hall--Littlewood polynomials and Rogers--Szeg\H{o} polynomials.

In Chapter~\ref{Ch_virtual} we consider two important
functionals on the ring of symmetric functions known as virtual 
Koornwinder integrals.
We review a number of earlier results for virtual Koornwinder
integrals and prove several new integral evaluations.

In Chapter~\ref{Ch_Bounded} we present our approach to bounded 
Littlewood identities. The key idea is to show that each bounded Littlewood 
identity is equivalent to the closed-form evaluation of a 
virtual Koornwinder integral. 
We apply our method to prove several new bounded Littlewood identities, 
including Theorem~\ref{Thm_qtL} and a $q,t$-analogue of the well-known
D\'esarm\'enien--Proctor--Stembridge determinant,
see Theorem~\ref{THM_BOUNDED1}.

Chapter~\ref{Ch_Apps} contains applications of the main results of 
Chapter~\ref{Ch_Bounded}. 
First, in Section~\ref{Sec_Plane-partitions}, we show that in the 
classical limit our approach to bounded Littlewood identities 
implies that MacMahon's formula \eqref{Eq_MacMahon} is a consequence
of the well-known fact that 
$\big(\mathrm{GL}(n,\mathbb{R}),\mathrm{O}(n)\big)$ forms a Gelfand pair.
Then, in Section~\ref{Sec_char}, we show how our bounded Littlewood 
identities for Hall--Littlewood polynomials 
give rise to combinatorial formulas for
characters of affine Lie algebras, such as Theorem~\ref{Thm_twistedA}.
This provides an alternative to the recent approach of
Bartlett and the second author based on the 
$\mathrm{C}_n$ Bailey lemma \cite{BW13}. 
In Section~\ref{Sec_RR} we follow recent ideas of Griffin et al.~\cite{GOW14}
and apply the combinatorial character identities to prove new 
Rogers--Ramanujan identities for the affine Lie algebras
$\mathrm{B}_n^{(1)},\mathrm{A}_{2n-1}^{(2)},\mathrm{A}_{2n}^{(2)}$ 
and $\mathrm{D}_{n+1}^{(2)}$.
The Rogers--Ramanujan identities for $\mathrm{B}_n^{(1)}$,
given in Theorem~\ref{Thm_RR-B} and Remark~\ref{Remark_Bressoud},
generalise Bressoud's well-known Rogers--Ramanujan identities 
for even moduli \cite{Bressoud80a,Bressoud80b} to arbitrary rank $n$.
As a final application, in Section~\ref{Sec_KM} we show that after
principal specialisation our bounded $q,t$-Littlewood identities give
rise to new transformation formulas for basic hypergeometric 
series of Kaneko--Macdonald-type \cite{Kaneko96,Macdonald13}.

Finally, in Chapter~\ref{Ch_Open}, we discuss a number of open problems 
arising from our work, including several conjectures.

We conclude our paper with two appendices.
Appendix~\ref{App_A} contains some technical lemmas related to the Weyl--Kac 
character formula, needed in Sections~\ref{Sec_char} and \ref{Sec_RR}.
In Appendix~\ref{App_B} we review the known connection
between elliptic Selberg integrals and multiple basic hypergeometric series,
and use this to prove a number of nonterminating quadratic transformation 
formulas for such series stated in Section~\ref{Sec_KM}.


\chapter{Macdonald--Koornwinder theory}
\section{Partitions}
A partition $\la=(\la_1,\la_2,\dots)$ is a weakly decreasing sequence of 
nonnegative integers such that only finitely many $\la_i$ are positive.
The positive $\la_i$ are called the parts of $\la$, and the number of
parts, denoted $l(\la)$, is called the length of $\la$. 
As is customary, we often ignore the tail of zeros of a partition.
If $\abs{\la}:=\sum_i \la_i=n$ we say that $\la$ is a partition of $n$.
The unique partition of $0$ is denoted by $0$.
As usual, we identify a partition with its Young diagram---a collection 
of left-aligned rows of squares such that the $i$th row contains $\la_i$ 
squares. We use the English convention for drawing Young diagrams, with rows
labelled from top to bottom and columns from left to right.
For example, the partition $(5,3,3,1)$ corresponds to 
\medskip

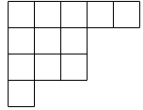
\begin{center}
\begin{tikzpicture}[scale=0.35,line width=0.3pt]
\draw (0,0)--(5,0);
\draw (0,-1)--(5,-1);
\draw (0,-2)--(3,-2);
\draw (0,-3)--(3,-3);
\draw (0,-4)--(1,-4);
\draw (0,0)--(0,-4);
\draw (1,0)--(1,-4);
\draw (2,0)--(2,-3);
\draw (3,0)--(3,-3);
\draw (4,0)--(4,-1);
\draw (5,0)--(5,-1);
\end{tikzpicture}
\end{center}
and its top right-most square has coordinates $(1,5)$.
The conjugate partition $\la'$ is obtained from $\la$ by reflection in the
main diagonal, so that the parts of $\la'$ correspond to the columns of $\la$.
For example, the conjugate of $(5,3,3,1)$ is $(4,3,3,1,1)$.
Given a partition $\la$, the multiplicity $m_i(\la)=\la'_i-\la'_{i+1}$ 
counts the number of parts of size $i$. Clearly, 
$\sum_{i\geqslant 1} m_i(\la)=l(\la)$.
If $\la$ is a rectangular partition consisting of $m$ rows and $n$ columns 
we write $\la=m^n$.
If $\la$ is a partition of length at most $n$ we also write $\la+m^n$ for
$(\la_1+m,\dots,\la_n+m)$. 
The number of even and odd parts of $\lambda$ will be denoted by 
$\ep(\la)$ and $\op(\la)$ respectively. If $\op(\la)=0$ we say that $\la$
is even. Given a partitions $\la=(\la_1,\la_2,\dots)$, we write 
$2\la$ for the even partition $(2\la_1,2\la_2,\dots)$. Similarly, if
$\la=(\la_1,\la_2,\dots)$ is even, we write $\la/2$ 
for the partition $(\la_1/2,\la_2/2,\dots)$. 

For two partitions $\la,\mu$ we write $\mu\subseteq\la$ if $\mu$ is contained
in $\la$, i.e., if $\mu_i\leqslant\la_i$ for all $i\geqslant 1$.
For $\mu\subseteq\la$, the set-theoretic difference between $\la$ and $\mu$ 
is called a skew shape. To avoid a notational clash with partition 
complementation to be defined shortly, we write this difference as 
$\skew{\la}{\mu}$ instead of the more common $\la-\mu$. 
For example, the skew shape $\skew{(5,3,3,1)}{(3,3,1)}$ is given by

\medskip
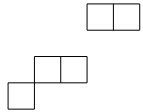
\begin{center}
\begin{tikzpicture}[scale=0.35,line width=0.3pt]
\draw (3,0)--(5,0);
\draw (3,-1)--(5,-1);
\draw (1,-2)--(3,-2);
\draw (0,-3)--(3,-3);
\draw (0,-4)--(1,-4);
\draw (0,-3)--(0,-4);
\draw (1,-2)--(1,-4);
\draw (2,-2)--(2,-3);
\draw (3,0)--(3,-1);
\draw (3,-2)--(3,-3);
\draw (4,0)--(4,-1);
\draw (5,0)--(5,-1);
\end{tikzpicture}
\end{center}
As usual, we identify $\skew{\la}{0}$ and $\la$.
A skew shape $\skew{\la}{\mu}$ containing at most one square in each column,
as in the above example, is referred to as a horizontal strip. 
Analogously, a vertical strip is a skew diagram with at most one square 
in each row. If $\skew{\la}{\mu}$ is a horizontal strip then the
partitions $\la$ and $\mu$ are said to be interlacing, which we 
denote by $\la\succ\mu$. Note that $\la\succ\mu$ if and only if
\[
\la_1\geqslant\mu_1\geqslant\la_2\geqslant\mu_2\geqslant\cdots.
\]
If $\la\subseteq m^n$ we write the complement
of $\la$ with respect to $m^n$ as $m^n-\la$, that is,
$m^n-\la=(m-\la_n,\dots,m-\la_2,m-\la_1)$. For example,
the complement of $(3,2)$ with respect to $4^3$ is $(4,2,1)$. 

The dominance order on the set of partitions is defined as follows: 
$\la\geqslant\mu$ if $\la_1+\cdots+\la_k\geqslant\mu_1+\cdots+\mu_k$ for all $k\geqslant 1$.
Note that unlike \cite{Macdonald95}, we do not assume that
$\abs{\la}=\abs{\mu}$.
If $\la\geqslant\mu$ and $\la\neq\mu$ we write $\la>\mu$. 

The arm-length, arm-colength, leg-length and leg-colength of 
the square $s=(i,j)\in\la$ are given by
\begin{align*}
a(s)=a_{\la}(s)&:=\la_i-j,  & a'(s)=a'_{\la}(s)&:=j-1, \\
l(s)=l_{\la}(s)&:=\la'_j-i, & l'(s)=l'_{\la}(s)&:=i-1,
\end{align*}
and correspond to the number of squares in the same row or column 
of $s$ immediately to the east, west, south and north of $s$ respectively.
For the square $s=(3,3)\in(8,7,7,6,4,3,1)$ we have 
$(a(s),l(s),a'(s),l'(s))=(4,3,2,2)$, as shown in the following diagram: 

\medskip

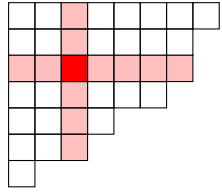
\begin{center}
\begin{tikzpicture}[scale=0.35,baseline=0cm,line width=1pt]
\draw[thin] (0,0) rectangle (1,1);
\draw[thin] (0,1) rectangle (1,2);
\draw[thin] (0,2) rectangle (1,3);
\draw[thin] (0,3) rectangle (1,4);
\draw[thin] (0,5) rectangle (1,6);
\draw[thin] (0,6) rectangle (1,7);
\draw[thin] (1,1) rectangle (2,2);
\draw[thin] (1,2) rectangle (2,3);
\draw[thin] (1,3) rectangle (2,4);
\draw[thin] (1,5) rectangle (2,6);
\draw[thin] (1,6) rectangle (2,7);
\draw[thin]  (3,2) rectangle (4,3);
\draw[thin] (3,3) rectangle (4,4);
\draw[thin] (3,5) rectangle (4,6);
\draw[thin] (3,6) rectangle (4,7);
\draw[thin] (4,3) rectangle (5,4);
\draw[thin] (4,5) rectangle (5,6);
\draw[thin] (4,6) rectangle (5,7);
\draw[thin] (5,3) rectangle (6,4);
\draw[thin] (5,5) rectangle (6,6);
\draw[thin] (5,6) rectangle (6,7);
\draw[thin] (6,5) rectangle (7,6);
\draw[thin] (6,6) rectangle (7,7);
\draw[thin] (7,6) rectangle (8,7);
\fill[red](2,4) rectangle (3,5);
\begin{scope}[color=pink]
\fill(0,4) rectangle (1,5);
\fill(1,4) rectangle (2,5);
\fill(3,4) rectangle (4,5);
\fill(4,4) rectangle (5,5);
\fill(5,4) rectangle (6,5);
\fill(6,4) rectangle (7,5);
\fill(2,1) rectangle (3,2);
\fill(2,2) rectangle (3,3);
\fill(2,3) rectangle (3,4);
\fill(2,5) rectangle (3,6);
\fill(2,6) rectangle (3,7);
\end{scope}
\draw[thin](2,4) rectangle (3,5);
\draw[thin](0,4) rectangle (1,5);
\draw[thin](1,4) rectangle (2,5);
\draw[thin](3,4) rectangle (4,5);
\draw[thin](4,4) rectangle (5,5);
\draw[thin](5,4) rectangle (6,5);
\draw[thin](6,4) rectangle (7,5);
\draw[thin](2,1) rectangle (3,2);
\draw[thin](2,2) rectangle (3,3);
\draw[thin](2,3) rectangle (3,4);
\draw[thin](2,5) rectangle (3,6);
\draw[thin](2,6) rectangle (3,7);
\end{tikzpicture}
\end{center}
We also define the closely related `$\mathrm{C}_n$-type' analogues
\[
\hat{a}(s)=\hat{a}_{\la}(s):=\la_i+j-1,\qquad\qquad  
\hat{l}(s)=\hat{l}_{\la}(s):=\la'_j+i-1.
\]
Diagrammatically, $\hat{a}(s)$ corresponds to the arm-length
of the mirror image, say $\hat{s}$, of $s\in\la$ 
upon reflection in the left boundary of $\la$.
In the previous example, $\hat{s}=(3,-2)$ with arm-length
$a(\hat{s})=\hat{a}(s)=9$. 

\medskip

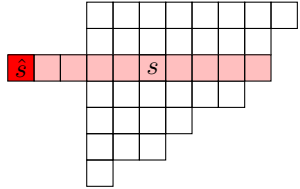
\begin{center}
\begin{tikzpicture}[scale=0.35,baseline=0cm,line width=1pt]
\draw[thin] (0,0) rectangle (1,1);
\draw[thin] (0,1) rectangle (1,2);
\draw[thin] (0,2) rectangle (1,3);
\draw[thin] (0,3) rectangle (1,4);
\draw[thin] (0,5) rectangle (1,6);
\draw[thin] (0,6) rectangle (1,7);
\draw[thin] (1,1) rectangle (2,2);
\draw[thin] (1,2) rectangle (2,3);
\draw[thin] (1,3) rectangle (2,4);
\draw[thin] (1,5) rectangle (2,6);
\draw[thin] (1,6) rectangle (2,7);
\draw[thin]  (3,2) rectangle (4,3);
\draw[thin] (3,3) rectangle (4,4);
\draw[thin] (3,5) rectangle (4,6);
\draw[thin] (3,6) rectangle (4,7);
\draw[thin] (4,3) rectangle (5,4);
\draw[thin] (4,5) rectangle (5,6);
\draw[thin] (4,6) rectangle (5,7);
\draw[thin] (5,3) rectangle (6,4);
\draw[thin] (5,5) rectangle (6,6);
\draw[thin] (5,6) rectangle (6,7);
\draw[thin] (6,5) rectangle (7,6);
\draw[thin] (6,6) rectangle (7,7);
\draw[thin] (7,6) rectangle (8,7);
\fill[red] (0,4) rectangle (-3,5);
\begin{scope}[color=pink]
\fill(-2,4) rectangle (1,5);
\fill(-1,4) rectangle (1,5);
\fill(0,4) rectangle (1,5);
\fill(1,4) rectangle (2,5);
\fill(2,4) rectangle (3,5);
\fill(3,4) rectangle (4,5);
\fill(4,4) rectangle (5,5);
\fill(5,4) rectangle (6,5);
\fill(6,4) rectangle (7,5);
\end{scope}
\draw (2.5,4.5) node {$s$};
\draw (-2.5,4.5) node {$\hat{s}$};
\draw[thin](2,4) rectangle (3,5);
\draw[thin](0,4) rectangle (1,5);
\draw[thin](1,4) rectangle (2,5);
\draw[thin](3,4) rectangle (4,5);
\draw[thin](4,4) rectangle (5,5);
\draw[thin](5,4) rectangle (6,5);
\draw[thin](6,4) rectangle (7,5);
\draw[thin](2,1) rectangle (3,2);
\draw[thin](2,2) rectangle (3,3);
\draw[thin](2,3) rectangle (3,4);
\draw[thin](2,5) rectangle (3,6);
\draw[thin](2,6) rectangle (3,7);
\draw[thin] (-3,4) rectangle (-2,5);
\draw[thin] (-2,4) rectangle (-1,5);
\draw[thin] (-1,4) rectangle (0,5);
\end{tikzpicture}
\end{center}
In much the same way, $\hat{l}(s)$ is the leg-length of the reflection of 
$s$ in the upper boundary of $\la$.
Note that for all $(A,L)\in\{(a,l),(a',l'),(\hat{a},\hat{l})\}$,
$A_{\la}(i,j)=L_{\la'}(j,i)$.
The usual hook-length of $s\in\la$ is given by 
$h(s)=a(s)+l(s)+1$.
The statistic $n(\skew{\la}{\mu})$, which will be used repeatedly, 
is defined as \cite[page 6]{Lascoux05}
\[
n(\skew{\la}{\mu}):=\sum_{s\in\skew{\la}{\mu}} l(s)=
\sum_{i\geqslant 1} \binom{\la'_i-\mu'_i}{2}.
\]
For $\mu=0$ this may also be expressed as the more familiar
\cite[page 3]{Macdonald95}
\[
n(\la)=\sum_{s\in\la} l'(s)=\sum_{i\geqslant 1} (i-1)\la_i.
\]

\medskip

Finally, we say that $\la=(\la_1,\dots,\la_n)$ is a half-partition if 
$\la_1\geqslant\la_2\geqslant\cdots\geqslant\la_n>0$ and all $\la_i$ are half-integers.
We sometimes write this as $\la=\mu+(1/2)^n$ with $\mu$ a partition of 
length at most $n$. Conversely, if $\la=\mu+(1/2)^n$, we also write 
$\mu=\la-(1/2)^n$. The length of a half-partition 
$\la=(\la_1,\dots,\la_n)$ is by definition $n$, and 
$m_i(\la)$ for $i$ a positive half-integer is the multiplicity 
of parts of size $i$.
We use half-partitions to generalise our earlier notion of complementation 
so that $m^n-\la$ makes sense for $m$ an integer or half-integer and
$\la\subseteq m^n$ a partition (of length at most $n$) or half-partition 
(of length $n$). For example, $4^3-(5/2,3/2,3/2)=(5/2,5/2,3/2)$ and
$(7/2)^3-(5/2,3/2,3/2)=(2,2,1)$. 
We extend the dominance order to the set of half-partitions in the 
obvious manner. However, partitions and half-partitions are by definition 
incomparable.

\section{Generalised $q$-shifted factorials}
Let 
\[
(z;q)_{\infty}:=\prod_{i\geqslant 0} (1-zq^i)
\quad\text{and}\quad
(z;q)_n:=\frac{(z;q)_{\infty}}{(zq^n;q)_{\infty}}
\]
be the standard $q$-shifted factorials \cite{GR04}. In this paper
we mostly view $q$-series as formal power series, but
occasionally we require $q$ to be a complex variable such that
$\abs{q}<1$.
The modified theta function is defined as
\begin{equation}\label{Eq_JTPI}
\theta(z;q):=(z;q)_{\infty}(q/z;q)_{\infty}=\frac{1}{(q;q)_{\infty}}
\sum_{k\in\Z} (-z)^k q^{\binom{k}{2}} \quad\text{for $z\neq 0$},
\end{equation}
where the equality between the product and the sum is known as
the Jacobi triple product identity \cite[Equation (II.28)]{GR04}.
We also need more general $q$-shifted factorials indexed by partitions:
\begin{subequations}\label{Eq_qshift}
\begin{align}
(z;q,t)_{\la}&:=\prod_{s\in\la}\big(1-zq^{a'(s)}t^{-l'(s)}\big)
=\prod_{i=1}^n (zt^{1-i};q)_{\la_i}, \\
C^{-}_{\la}(z;q,t)&:=\prod_{s\in\la}\big(1-zq^{a(s)}t^{l(s)}\big) 
\\[-2mm] & \qquad \qquad \quad
=\prod_{i=1}^n (zt^{n-i};q)_{\la_i} 
\prod_{1\leqslant i<j\leqslant n} \frac{(zt^{j-i-1};q)_{\la_i-\la_j}}
{(zt^{j-i};q)_{\la_i-\la_j}},  \notag \\
C^{+}_{\la}(z;q,t)&:=
\prod_{s\in\la}\big(1-zq^{\hat{a}(s)}t^{1-\hat{l}(s)}\big) \\[-2mm] 
& \qquad \qquad \quad
=\prod_{i=1}^n \frac{(zt^{2-2i};q)_{2\la_i}}{(zt^{2-i-n};q)_{\la_i}} 
\prod_{1\leqslant i<j\leqslant n} \frac{(zt^{2-i-j};q)_{\la_i+\la_j}}
{(zt^{3-i-j};q)_{\la_i+\la_j}}.  \notag
\end{align}
\end{subequations}
In all three cases, the choice of $n$ on the right is irrelevant 
as long as $n\geqslant l(\la)$.
We note that $(a;q,t)_{\la}$ is sometimes denoted as $C^0_{\la}(a;q,t)$, 
see e.g., \cite{RV07}, and that $C^{-}_{\la}(t;q,t)=c_{\la}(q,t)$ and 
$C^{-}_{\la}(q;q,t)=c'_{\la}(q,t)$, with $c_{\la}$ and $c'_{\la}$ the 
hook-length polynomials of Macdonald \cite[page 352]{Macdonald95}. 
In particular, $C^{-}_{\la}(q;q,q)=c_{\la}(q,q)=c'_{\la}(q,q)=H_{\la}(q)$
with 
\begin{equation}\label{Eq_Hhook}
H_{\la}(q):=\prod_{s\in\la}\big(1-q^{h(s)}\big)
\end{equation} 
the classical hook-length polynomial.
For $s\in\la$, let
\begin{equation}\label{Eq_blas}
b_{\la}(s;q,t):=\frac{1-q^{a_{\la}(s)} t^{l_{\la}(s)+1}}
{1-q^{a_{\la}(s)+1} t^{l_{\la}(s)}}.
\end{equation}
Then
\begin{equation}\label{Eq_b-def}
b_{\la}(q,t):=\frac{c_{\la}(q,t)}{c'_{\la}(q,t)}=
\prod_{s\in\la} b_{\la}(s;q,t).
\end{equation}
For both $q$-shifted factorials and theta functions, we use condensed
notation such as
\[
(z_1,\dots,z_k;q,t)_{\la}=(z_1;q,t)_{\la}\cdots (z_k;q,t)_{\la}.
\]

\medskip

It is an elementary exercise to verify the following identities,
which will be used repeatedly throughout this paper:
\begin{subequations} \label{Eq_qtswap}
\begin{align} \label{Eq_qtswap_C0}
(a;q,t)_{\la'}&=(-a)^{\abs{\la}} q^{n(\la)} t^{-n(\la')} 
(a^{-1};t,q)_{\la}, \\
C^{-}_{\la'}(a;q,t)&=C^{-}_{\la}(a;t,q), \\
C^{+}_{\la'}(a;q,t)&=(-aq)^{\abs{\la}} q^{3n(\la)} t^{-3n(\la')}
C^{+}_{\la}((aqt)^{-1};t,q),
\end{align}
\end{subequations}
\begin{subequations}
\begin{align} 
(a;q,t)_{2\la}&=(a,aq;q^2,t)_{\la},\label{Eq_double_C0} \\
C^{-}_{2\la}(a;q,t)&=C^{-}_{\la}(a,aq;q^2,t) \label{Eq_double_Cmin},
\end{align}
\end{subequations}
and
\begin{subequations}\label{Eq_Gdual}
\begin{align}\label{Eq_Gdual_C0}
(a;q,t)_{m^n-\la}
&=(-q^{1-m}t^{n-1}/a)^{\abs{\la}} 
q^{n(\la')} t^{-n(\la)}
\frac{(a;q,t)_{m^n}}{(q^{1-m}t^{n-1}/a;q,t)_{\la}}, \\
C^{-}_{m^n-\la}(a;q,t)&=
(-q^{1-m}/a)^{\abs{\la}} q^{n(\la')} t^{-n(\la)}
\frac{(at^{n-1};q,t)_{m^n} C_{\la}^{-}(a;q,t)}
{(at^{n-1},q^{1-m}/a;q,t)_{\la}}.
\end{align}
\end{subequations}

\section{Rogers--Szeg\H{o} polynomials} 
For integers $k,n$ such that $0\leqslant k\leqslant n$, let 
\[
\qbin{n}{k}_q:=\frac{(q;q)_n}{(q;q)_k(q;q)_{n-k}}
\]
be a $q$-binomial coefficient.
Then the Rogers--Szeg\H{o} polynomials $H_m(z;q)$ are defined as
\cite{Szego26}
\begin{equation}\label{RS}
H_m(z;q):=\sum_{k=0}^m z^k \qbin{m}{k}_q,
\end{equation}
for $m$ a nonnegative integer.
They have generating function \cite[page 49]{Andrews76}
\begin{equation}\label{Eq_RS-GF}
\sum_{m\geqslant 0} \frac{H_m(z;q) t^m}{(q;q)_m}=
\frac{1}{(t,tz;q)_{\infty}},
\end{equation}
and satisfy the three-term recurrence relation
\[
H_{m+1}(z;q)=(1+z)H_m(z;q)-(1-q^m) z H_{m-1}(z;q)
\]
subject to the initial conditions $H_{-1}=0$, $H_0=1$.
For $0<q<1$, the Rogers--Szeg\H{o} polynomials
satisfy the orthogonality relation
\[
\frac{1}{2\pi\iup} \int_{\mathbb{T}} 
H_m(zq^{-1/2};q) H_n(\bar{z}q^{-1/2};q)\,
\bigabs{(zq^{1/2};q)_{\infty}}^2\,
\frac{\dup z}{z}=q^{-m}(q;q)_m \, \delta_{m,n},
\]
where $\mathbb{T}$ is the positively-oriented unit circle.

The Rogers--Szeg\H{o} polynomials are closely related to symmetric 
functions and may, for example, be expressed in terms of Schur functions 
as
\begin{equation}\label{Eq_Hm-Schur}
H_m(z;q)=(q)_m \sum_{\la\vdash m} \frac{q^{n(\la)}}{H_{\la}(q)}\,
s_{\la}(1,z),
\end{equation}
with $H_{\la}(q)$ the hook-length polynomial \eqref{Eq_Hhook}. 
Indeed, by \eqref{Eq_RS-GF} and the $n=2$ case of 
\cite[page 66]{Macdonald95}
\[
\sum_{\la} \frac{q^{n(\la)}}{H_{\la}(q)}\, s_{\la}(x_1,\dots,x_n)
=\prod_{i=1}^n \frac{1}{(x_i;q)_{\infty}},
\]
the formula \eqref{Eq_Hm-Schur} immediately follows.

\medskip

We require two generalisations of the Rogers--Szeg\H{o} polynomials 
to polynomials indexed by partitions. First,
\begin{align}\label{Eq_RSBCn}
h_{\la}^{(m)}(a,b;q)&:=
\prod_{\substack{i=1 \\[0.5pt] i \text{ odd}}}^{m-1} 
(-a)^{m_i(\la)} H_{m_i(\la)}(b/a;q)
\prod_{\substack{i=1 \\[0.5pt] i \text{ even}}}^{m-1} H_{m_i(\la)}(ab;q) \\
&\hphantom{:}=(-a)^{\op(\la)} 
\prod_{\substack{i=1 \\[0.5pt] i \text{ odd}}}^{m-1} 
H_{m_i(\la)}(b/a;q)
\prod_{\substack{i=1 \\[0.5pt] i \text{ even}}}^{m-1} H_{m_i(\la)}(ab;q),
\notag 
\end{align}
where $m$ is a nonnegative integer. Compared to earlier definitions
in \cite{BW13,W06} the parameters $a$ and $b$ have been replaced 
by their negatives.
Since $H_m$ is a self-reciprocal polynomial, i.e.,
$z^m H_m(1/z;q)=H_m(z;q)$, it follows that 
$h_{\la}^{(m)}(a,b;q)$ is symmetric in $a$ and $b$:
\[
h_{\la}^{(m)}(a,b;q)=h_{\la}^{(m)}(b,a;q).
\]
Since $m_i(\la)=0$ for $i>\la_1$, the upper bound on the products over 
$i$ in \eqref{Eq_RSBCn} may be dropped if $m\geqslant \la_1+1$.
For such $m$ we simply write $h_{\la}(a,b;q)$. That is,
\begin{equation}\label{Eq_RSBCn-mlim}
h_{\la}(a,b;q):=
\prod_{\substack{i\geqslant 1 \\[0.5pt] i \text{ odd}}}
(-a)^{m_i(\la)} H_{m_i(\la)}(b/a;q)
\prod_{\substack{i\geqslant 1 \\[0.5pt] i \text{ even}}} H_{m_i(\la)}(ab;q).
\end{equation}
We further define
\begin{equation}\label{Eq_RSBn}
h_{\la}^{(m)}(a;t):=h_{\la}^{(m)}(a,-1;t)=
\prod_{i=1}^{m-1} H_{m_i(\la)}(-a;t).
\end{equation}
It follows from
\begin{subequations}\label{Eq_spec}
\begin{align}
H_m(0;t)&=1, \\
H_m(-1;t)&=\begin{cases}(t;t^2)_{m/2} & \text{$m$ even}, \\
0 & \text{$m$ odd,} \end{cases}  \label{Eq_spec2} \\
H_m(-t;t)&=(t;t^2)_{\ceil{m/2}},
\label{Eq_spec3} \\[2mm]
H_m(t^{1/2};t)&=(-t^{1/2};t^{1/2})_m, \label{Eq_spec4}
\end{align}
\end{subequations}
(see \cite[Equation (1.10)]{W06})
that for special values of $a$ and $b$ the polynomials
\eqref{Eq_RSBCn}--\eqref{Eq_RSBn} completely factor.
This will be important in Section~\ref{Sec_char} 
in our discussion of character formulas for affine Lie algebras.

\section{Plethystic notation}
Let $\Symm_n$ be the symmetric group on $n$ letters, 
$\Lambda_n=\F[x_1,\dots,x_n]^{\Symm_n}$ the ring of symmetric
functions in the alphabet $x_1,\dots,x_n$ with coefficients in $\F$, 
and $\Lambda$ the corresponding ring of symmetric functions in 
countably many variables, see e.g., \cite{Macdonald95,Stanley99}. 
We will mostly consider $\F=\Rat(q,t)$ and
$\F=\Rat(q,t,\tees)$, or minor variations thereof.

To facilitate computations in $\Lambda$ we frequently employ 
plethystic or $\la$-ring notation \cite{Lascoux01,Haglund08}. 
This is most easily described in terms of the Newton power sums 
\[
p_r(x):=x_1^r+x_2^r+\cdots,\qquad r\geqslant 1,
\]
with generating function
\[
\Psi_z(x):=\sum_{r=1}^{\infty} z^{r-1} p_r(x)
=\sum_{i\geqslant 1} \frac{x_i}{1-zx_i}.
\]
The power sums form an algebraic basis of $\Lambda$, that is, 
$\Lambda=\F[p_1,p_2,\dots]$.

If $x=(x_1,x_2,\dots)$, we additively write $x=x_1+x_2+\cdots$,
and to indicate the latter notation, we use plethystic brackets:
\[
f(x)=f(x_1,x_2,\dots)=f[x_1+x_2+\cdots]=f[x],\qquad f\in\Lambda.
\]
A power sum whose argument is the sum, difference or Cartesian product 
of two alphabets $x$ and $y$ is then defined as
\begin{subequations}\label{Eq_plet1}
\begin{align}
p_r[x+y]&:=p_r[x]+p_r[y], \\
p_r[x-y]&:=p_r[x]-p_r[y], \\
p_r[xy]&:=p_r[x]p_r[y],
\end{align}
\end{subequations}
respectively.
In particular, if $x$ is the empty alphabet then $p_r[-y]=-p_r[y]$,
(which should not be confused with $p(-y)=p(-y_1,-y_2,\dots)=(-1)^r p_r(y)$),
so that
\begin{equation}\label{Eq_minmin}
f[-(-x)]=f[x].
\end{equation}
Occasionally we need to also use an ordinary minus sign in plethystic
notation. To distinguish this from a plethystic minus sign,
we denote by $\varepsilon$ the alphabet consisting 
of the single letter $-1$, so that for $f\in\Lambda$
\[
f(-x)=f(-x_1,-x_2,\dots)=f[\varepsilon x_1+\varepsilon x_2+\cdots]=
f[\varepsilon x].
\]
Hence
\[
p_r[\varepsilon x]=(-1)^r p_r[x],\qquad p_r[-\varepsilon x]=(-1)^{r-1} p_r[x]
\]
and
\[
f[x+\varepsilon]=f(-1,x_1,x_2,\dots).
\]

For indeterminates $a,b,t$ and $f\in\Lambda$,
we further define $f[(a-b)/(1-t)]$ by
\begin{equation}\label{Eq_plet2}
p_r\Big[\frac{a-b}{1-t}\Big]=\frac{a^r-b^r}{1-t^r},
\end{equation}
and note that $(a-b)/(1-t)$ may be viewed as the
difference between the alphabets $a(1+t+t^2+\cdots)$ and 
$b(1+t+t^2+\cdots)$, where $a(1+t+t^2+\cdots)$ is the
Cartesian product of the single-letter alphabet $a$ and
the infinite alphabet $1+t+t^2+\cdots$. 
Alternatively, $(a-b)/(1-t)$ may be interpreted as
the Cartesian product of $a-b$ and $1+t+t^2+\cdots$.
We can of course combine \eqref{Eq_plet1} and \eqref{Eq_plet2}, 
and for example
\[
p_r\Big[x+\frac{a-b}{1-t}\Big]=p_r[x]+p_r\Big[\frac{a-b}{1-t}\Big].
\]

For $r\geqslant 0$, the complete symmetric functions are defined by
\[
h_r(x):=\sum_{1\leqslant i_1\leqslant i_2\leqslant\cdots\leqslant i_r}
x_{i_1}x_{i_2}\cdots x_{i_r},
\]
and admit the generating function
\begin{equation}\label{Eq_sigmaz}
\sigma_z(x):=\sum_{r\geqslant 0} z^r h_r(x)
=\prod_{i\geqslant 1} \frac{1}{1-zx_i}.
\end{equation}
Since $\Psi_z(x)=\frac{\dup\:}{\dup z}\log \sigma_z(x)$, it follows that
\begin{subequations}\label{Eq_sigma-formulas}
\begin{align}
\sigma_1[x+y]&=\sigma_1[x]\sigma_1[y]
=\prod_{i\geqslant 1} \frac{1}{(1-x_i)(1-y_i)}, \\
\sigma_1[x-y]&=\frac{\sigma_1[x]}{\sigma_1[y]} 
=\prod_{i\geqslant 1} \frac{1-y_i}{1-x_i}, \\
\sigma_1\Big[\frac{a-b}{1-t}\Big]&=
\prod_{k\geqslant 0} \frac{\sigma_1[at^k]}{\sigma_1[bt^k]}=
\frac{(b;t)_{\infty}}{(a;t)_{\infty}}.
\end{align}
\end{subequations}
These three formulas allow various infinite products to be expressed 
in terms of symmetric functions. For example, it follows immediately
from the generating function \eqref{Eq_RS-GF} that the Rogers--Szeg\H{o}
polynomials $H_m(z;q)$ may be identified as 
\[
H_m(z;q)=(q;q)_m\, h_m\Big[\frac{1+z}{1-q}\Big].
\]

Finally, for $\F=\Rat(q,t)$, $\omega_{q,t}$ is the $\F$-algebra
endomorphism of $\Lambda$ defined by \cite[page 312]{Macdonald95}
\[
\omega_{q,t}\, p_r=(-1)^{r-1} \frac{1-q^r}{1-t^r}\, p_r.
\]
Note that $\omega_{t,q}=\omega_{q,t}^{-1}$ and, plethystically,
\begin{equation}\label{Eq_omegaqt}
\omega_{q,t}\, f(x)=f\Big(\Big[{-}\varepsilon \, \frac{1-q}{1-t}\,x\Big]\Big),
\qquad f\in\Lambda.
\end{equation}

\section{Macdonald polynomials} 
Let $\F=\Rat(q,t)$. 
The power sums $p_{\la}:=\prod_{i=1}^{l(\la)} p_{\la_i}$ may be used 
to define Macdonald's $q,t$-analogue of the Hall scalar product on $\Lambda$ 
as \cite{Macdonald95}
\begin{equation*}
\ip{p_{\la}}{p_{\mu}}_{q,t}:=
\delta_{\la\mu} z_{\la} \prod_{i=1}^n
\frac{1-q^{\la_i}}{1-t^{\la_i}},
\end{equation*}
where $z_{\la}=\prod_{i\geqslant 1} m_i(\la)! \, i^{m_i(\la)}$.
The Macdonald polynomials $P_{\la}(q,t)=P_{\la}(x;q,t)$
are the unique family of symmetric functions such that \cite{Macdonald95}
\begin{equation}\label{Eq_P-mon}
P_{\la}(q,t)=m_{\la}+\sum_{\mu<\la} u_{\la\mu}(q,t) m_{\mu}
\end{equation}
and
\[
\ip{P_{\la}(q,t)}{P_{\mu}(q,t)}_{q,t}
=0\qquad \text{if$\quad\la\neq\mu$}.
\]
Here the $m_{\la}$ are the monomial symmetric functions, defined by
\[
m_{\la}(x):=\sum_{\alpha} x^{\alpha},
\]
where $\alpha$ is summed over distinct permutations of 
$\la=(\la_1,\la_2,\dots)$ and 
$x^{\alpha}:=\prod_{i\geqslant 1} x_1^{\alpha_1} x_2^{\alpha_2} \cdots$.
By the triangularity of \eqref{Eq_P-mon} and the fact that the $m_{\la}$
form a basis of $\Lambda$, it immediately follows that the Macdonald 
polynomials form a basis of $\Lambda$ as well. 
When $l(\la)>n$, $P_{\la}(x_1,\dots,x_n;q,t)=0$ and the polynomials
$P_{\la}(x_1,\dots,x_n;q,t)$ indexed by partitions $\la$ of length at most $n$
form a basis of $\Lambda_n$.

The skew Macdonald polynomials $P_{\skew{\la}{\mu}}(q,t)$ are defined by
\[
\ip{P_{\skew{\la}{\mu}}(q,t)}{P_{\nu}(q,t)}_{q,t}=
\ip{P_{\la}(q,t)}{P_{\mu}(q,t) P_{\nu}(q,t)}_{q,t},
\]
and vanish unless $\mu\subseteq\la$.
For $q=t$ the Macdonald polynomials simplify to the Schur 
functions: $P_{\la/\mu}(t,t)=s_{\la/\mu}$. 

For later comparison with the Koornwinder and $(R,S)$ Macdonald polynomials,
we remark that an alternative description of the Macdonald polynomials 
in $n$ variables is as the unique family of
polynomials \eqref{Eq_P-mon} such that \cite{Macdonald95}
\[
\ip{P_{\la}}{P_{\mu}}'_{q,t}=0 \qquad\text{if $\la\neq\mu$},
\]
where, for $\abs{q},\abs{t}<1$, $\ip{\cdot}{\cdot}'_{q,t}$
is the scalar product on $\F[x]=\F[x_1,\dots,x_n]$ defined by
\[
\ip{f}{g}'_{q,t}:=\frac{1}{n! (2\pi\iup)^n}
\int_{\mathbb{T}^n} f(x) g(x^{-1})\Delta(x;q,t)\,
\frac{\dup x_1}{x_1}\cdots \frac{\dup x_n}{x_n}.
\]
Here $f(x^{-1}):=f(x_1^{-1},\dots,x_n^{-1})$, $\Delta(x;q,t)$ is 
the Macdonald density
\begin{equation}\label{Eq_Mac-density}
\Delta(x;q,t):=
\prod_{1\leqslant i<j\leqslant n}   
\frac{(x_i/x_j,x_j/x_i;q)_{\infty}}{(tx_i/x_j,tx_j/x_i;q)_{\infty}}
\end{equation}
and $\mathbb{T}^n$ is the $n$-dimensional complex torus:
\[
\mathbb{T}^n:=
\{(x_1,\dots,x_n)\in\mathbb{C}^n: \abs{x_1}=\cdots=\abs{x_n}=1\}.
\]

Below we list a number of standard results from Macdonald polynomial
theory needed later. 
First of all, defining a second family of Macdonald polynomials 
$Q_{\la/\mu}(q,t)=Q_{\la/\mu}(x;q,t)$ as 
\begin{equation}\label{Eq_Qdef}
Q_{\la/\mu}(q,t):=\frac{b_{\la}(q,t)}{b_{\mu}(q,t)}\, P_{\la/\mu}(q,t),
\end{equation}
with $b_{\la}(q,t)$ given by \eqref{Eq_b-def}, we have the duality
\cite[page 327]{Macdonald95}
\begin{equation}\label{Eq_omegaQ}
\omega_{q,t}\,P_{\la}(q,t)=Q_{\la'}(t,q),
\end{equation}
as well as the orthogonality \cite[page 324]{Macdonald95}
\[
\ip{P_{\la}(q,t)}{Q_{\mu}(q,t)}_{q,t}=\delta_{\la\mu}.
\]
This last equation is equivalent to the Cauchy identity 
\cite[page 324]{Macdonald95}
\[
\sum_{\la} P_{\la}(x;q,t)Q_{\la}(y;q,t)=
\prod_{i,j\geqslant 1} \frac{(tx_iy_j;q)_{\infty}}{(x_iy_j;q)_{\infty}},
\]
which we repeatedly require in the dual form \cite[page 329]{Macdonald95}
\begin{equation}\label{Eq_Mac-Cauchy-a}
\sum_{\la\subseteq m^n} (-1)^{\abs{\la}} P_{\la}(x_1,\dots,x_n;q,t) 
P_{\la'}(y_1,\dots,y_m;t,q)=\prod_{i=1}^n \prod_{j=1}^m (1-x_iy_j).
\end{equation}

We also need the $g$- and $e$-Pieri rules for Macdonald polynomials
\cite[page 340]{Macdonald95},
expressed in generating function form. First, in the $g$-Pieri case
\begin{equation}\label{Eq_g-Pieri}
P_{\mu}(x;q,t) \prod_{i\geqslant 1} 
\frac{(atx_i;q)_{\infty}}{(ax_i;q)_{\infty}}
=\sum_{\la\succ\mu} a^{\abs{\skew{\la}{\mu}}} \varphi_{\la/\mu}(q,t) 
P_{\la}(x;q,t),
\end{equation}
where the Pieri coefficient
$\varphi_{\la/\mu}(q,t)$ is given by \cite[page 342]{Macdonald95}
\begin{multline}\label{Eq_varphi}
\varphi_{\la/\mu}(q,t)=
\prod_{1\leqslant i\leqslant j\leqslant l(\la)}
\frac{(qt^{j-i};q)_{\la_i-\la_j}}{(t^{j-i+1};q)_{\la_i-\la_j}}\cdot
\frac{(qt^{j-i};q)_{\mu_i-\mu_{j+1}}}{(t^{j-i+1};q)_{\mu_i-\mu_{j+1}}} \\
\times
\frac{(t^{j-i+1};q)_{\la_i-\mu_j}}{(qt^{j-i};q)_{\la_i-\mu_j}}\cdot
\frac{(t^{j-i+1};q)_{\mu_i-\la_{j+1}}}{(qt^{j-i};q)_{\mu_i-\la_{j+1}}}.
\end{multline}
Similarly, the $e$-Pieri rule is given by
\begin{equation}\label{Eq_e-Pieri}
P_{\mu}(x;q,t)\prod_{i\geqslant 1} (1+ax_i)=
\sum_{\la'\succ\mu'} a^{\abs{\skew{\la}{\mu}}} 
\psi'_{\skew{\la}{\mu}}(q,t) P_{\la}(x;q,t),
\end{equation}
where \cite[page 336]{Macdonald95}
\begin{equation}\label{Eq_psip}
\psi'_{\skew{\la}{\mu}}(q,t) = \prod
\frac{1-q^{\mu_i-\mu_j}t^{j-i-1}}{1-q^{\mu_i-\mu_j}t^{j-i}}\cdot 
\frac{1-q^{\la_i-\la_j}t^{j-i+1}}{1-q^{\la_i-\la_j}t^{j-i}}.
\end{equation}
The product in \eqref{Eq_psip} is over all $i<j$ such that 
$\la_i=\mu_i$ and $\la_j>\mu_j$.
An alternative expression for $\psi'_{\skew{\la}{\mu}}(q,t)$ is given by
\cite[page 340]{Macdonald95}
\begin{equation}\label{Eq_psip340}
\psi'_{\skew{\la}{\mu}}(q,t)=\prod 
\frac{b_{\la}(s;q,t)}{b_{\mu}(s;q,t)},
\end{equation}
where $b_{\la}(s;q,t)$ is given by \eqref{Eq_blas} and where the product is 
over all squares $s=(i,j)\in\mu\subseteq\la$ such that $i<j$, 
$\mu_i=\la_i$ and $\la'_j>\mu_j'$. 

For $\la$ a partition, define
\begin{equation}\label{Eq_evenarm}
b_{\la}^{\text{ea}}(q,t):=\prod_{\substack{s\in\la \\[1pt] a(s) \text{ even}}}
b_{\la}(s;q,t)
=\prod_{\substack{s\in\la \\[1pt] a(s) \text{ even}}}
\frac{1-q^{a(s)}t^{l(s)+1}}{1-q^{a(s)+1}t^{l(s)}},
\end{equation}
where `ea' stands for `even arm(-length)'.

\begin{lemma}\label{Lem_psip}
For partitions $\la\succ\mu$ such that
\[
\la=2\ceil{\mu/2}:=(2\ceil{\mu_1/2},2\ceil{\mu_2/2},\dots),
\]
we have
\[
\psi'_{\skew{\la}{\mu}}(q,t) \,
=\frac{C^{-}_{\la/2}(t;q^2,t)}{C^{-}_{\la/2}(q;q^2,t)}
\cdot \frac{1}{b_{\mu}^{\textup{ea}}(q,t)}.
\]
\end{lemma}

\begin{proof}
The product in \eqref{Eq_psip340} is over all squares $s$ in $\mu$
for which $\la$ and $\mu$ have the same row length but different 
column lengths.
If $\la=2\ceil{\mu/2}$ then $\la$ is obtained from $\mu$ by
adding a square to each row of odd length.
Hence the product is over all squares $s=(i,j)$ such that
$\mu_i$ and $j$ are even and such that there exists a $k>i$
with $\mu_k=j-1$. For example, if $\mu=(6,4,3,3,2,1)$ then
$\la=(6,4,4,4,2,2)$. In the diagram on the left, shown below, 
the squares contributing to $\prod b_{\la}(s;q,t)/b_{\mu}(s;q,t)$ 
are marked with a blue cross.
Because each marked square $(i,j)$ occurs in a rows of even length and has
even $j$-coordinate, each marked square must have even arm-length.
We can, however, include all other squares of $\la$ and $\mu$
with even arm-length, since their respective contributions to $b_{\la}(s;q,t)$ 
and $b_{\mu}(s;q,t)$ trivially cancel, as indicated by the added red squares
in the diagram on the right:

\smallskip

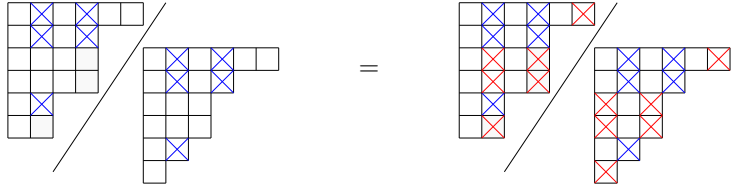
\begin{center}
\begin{tikzpicture}[scale=0.3,line width=0.3pt]
\draw[fill,color=white!95!gray] (1,-6) rectangle (2,-5);
\draw[fill,color=white!95!gray] (3,-4) rectangle (4,-3);
\draw[fill,color=white!95!gray] (3,-3) rectangle (4,-2);
\draw (0,0)--(6,0);
\draw (0,-1)--(6,-1);
\draw (0,-2)--(4,-2);
\draw (0,-3)--(4,-3);
\draw (0,-4)--(4,-4);
\draw (0,-5)--(2,-5);
\draw (0,-6)--(2,-6);
\draw (0,0)--(0,-6);
\draw (1,0)--(1,-6);
\draw (2,0)--(2,-6);
\draw (3,0)--(3,-4);
\draw (4,0)--(4,-4);
\draw (5,0)--(5,-1);
\draw (6,0)--(6,-1);
\draw[blue] (1,0)--(2,-1); \draw[blue] (2,0)--(1,-1);
\draw[blue] (3,0)--(4,-1); \draw[blue] (4,0)--(3,-1);
\draw[blue] (1,-1)--(2,-2); \draw[blue] (2,-1)--(1,-2);
\draw[blue] (3,-1)--(4,-2); \draw[blue] (4,-1)--(3,-2);
\draw[blue] (1,-4)--(2,-5); \draw[blue] (2,-4)--(1,-5);

\draw (7,0)--(2,-7.5);

\draw (6,-2)--(12,-2);
\draw (6,-3)--(12,-3);
\draw (6,-4)--(10,-4);
\draw (6,-5)--(9,-5);
\draw (6,-6)--(9,-6);
\draw (6,-7)--(8,-7);
\draw (6,-8)--(7,-8);
\draw (6,-2)--(6,-8);
\draw (7,-2)--(7,-8);
\draw (8,-2)--(8,-7);
\draw (9,-2)--(9,-6);
\draw (10,-2)--(10,-4);
\draw (11,-2)--(11,-3);
\draw (12,-2)--(12,-3);
\draw[blue] (7,-2)--(8,-3); \draw[blue] (8,-2)--(7,-3);
\draw[blue] (9,-2)--(10,-3); \draw[blue] (10,-2)--(9,-3);
\draw[blue] (7,-3)--(8,-4); \draw[blue] (8,-3)--(7,-4);
\draw[blue] (9,-3)--(10,-4); \draw[blue] (10,-3)--(9,-4);
\draw[blue] (7,-6)--(8,-7); \draw[blue] (8,-6)--(7,-7);

\draw (16,-3) node {$=$};

\draw[fill,color=white!95!gray] (21,-6) rectangle (22,-5);
\draw[fill,color=white!95!gray] (23,-4) rectangle (24,-3);
\draw[fill,color=white!95!gray] (23,-3) rectangle (24,-2);
\draw (20,0)--(26,0);
\draw (20,-1)--(26,-1);
\draw (20,-2)--(24,-2);
\draw (20,-3)--(24,-3);
\draw (20,-4)--(24,-4);
\draw (20,-5)--(22,-5);
\draw (20,-6)--(22,-6);
\draw (20,0)--(20,-6);
\draw (21,0)--(21,-6);
\draw (22,0)--(22,-6);
\draw (23,0)--(23,-4);
\draw (24,0)--(24,-4);
\draw (25,0)--(25,-1);
\draw (26,0)--(26,-1);
\draw[blue] (21,0)--(22,-1); \draw[blue] (22,0)--(21,-1);
\draw[blue] (23,0)--(24,-1); \draw[blue] (24,0)--(23,-1);
\draw[blue] (21,-1)--(22,-2); \draw[blue] (22,-1)--(21,-2);
\draw[blue] (23,-1)--(24,-2); \draw[blue] (24,-1)--(23,-2);
\draw[blue] (21,-4)--(22,-5); \draw[blue] (22,-4)--(21,-5);
\draw[red] (25,0)--(26,-1); \draw[red] (26,0)--(25,-1);
\draw[red] (23,-2)--(24,-3); \draw[red] (24,-2)--(23,-3);
\draw[red] (21,-2)--(22,-3); \draw[red] (22,-2)--(21,-3);
\draw[red] (23,-3)--(24,-4); \draw[red] (24,-3)--(23,-4);
\draw[red] (21,-3)--(22,-4); \draw[red] (22,-3)--(21,-4);
\draw[red] (21,-5)--(22,-6); \draw[red] (22,-5)--(21,-6);

\draw (27,0)--(22,-7.5);

\draw (26,-2)--(32,-2);
\draw (26,-3)--(32,-3);
\draw (26,-4)--(30,-4);
\draw (26,-5)--(29,-5);
\draw (26,-6)--(29,-6);
\draw (26,-7)--(28,-7);
\draw (26,-8)--(27,-8);
\draw (26,-2)--(26,-8);
\draw (27,-2)--(27,-8);
\draw (28,-2)--(28,-7);
\draw (29,-2)--(29,-6);
\draw (30,-2)--(30,-4);
\draw (31,-2)--(31,-3);
\draw (32,-2)--(32,-3);
\draw[blue] (27,-2)--(28,-3); \draw[blue] (28,-2)--(27,-3);
\draw[blue] (29,-2)--(30,-3); \draw[blue] (30,-2)--(29,-3);
\draw[blue] (27,-3)--(28,-4); \draw[blue] (28,-3)--(27,-4);
\draw[blue] (29,-3)--(30,-4); \draw[blue] (30,-3)--(29,-4);
\draw[blue] (27,-6)--(28,-7); \draw[blue] (28,-6)--(27,-7);
\draw[red] (31,-2)--(32,-3); \draw[red] (32,-2)--(31,-3);
\draw[red] (28,-4)--(29,-5); \draw[red] (29,-4)--(28,-5);
\draw[red] (26,-4)--(27,-5); \draw[red] (27,-4)--(26,-5);
\draw[red] (28,-5)--(29,-6); \draw[red] (29,-5)--(28,-6);
\draw[red] (26,-5)--(27,-6); \draw[red] (27,-5)--(26,-6);
\draw[red] (26,-7)--(27,-8); \draw[red] (27,-7)--(26,-8);

\end{tikzpicture}
\end{center}
Hence
\begin{align*}
\psi'_{\skew{\la}{\mu}}(q,t)
&=\prod_{\substack{s\in\la \\[1pt] a(s) \textup{ even}}} b_{\la}(s;q,t)
\prod_{\substack{s\in\mu \\[1pt] a(s) \textup{ even}}} 
\frac{1}{b^{\text{ea}}_{\mu}(s;q,t)} \\[1mm]
&=\bigg(\prod_{s\in\la/2}
\frac{1-q^{2a(s)}t^{l(s)+1}}{1-q^{2a(s)+1}t^{l(s)}}\bigg)
\cdot
\frac{1}{b^{\text{ea}}_{\mu}(q,t)},
\end{align*}
where the second equality uses the fact that $\la$ is an even partition.
By \eqref{Eq_qshift} we are done.
\end{proof}

If in \eqref{Eq_e-Pieri} we set $x_i=0$ for $i>n$ and
equate terms of degree $\abs{\mu}+n$, we obtain
\[
P_{\mu}(x_1,\dots,x_n;q,t) \, x_1\cdots x_n = 
\psi'_{\skew{(\mu+1^n)}{\mu}}(q,t) P_{\mu+1^n}(x_1,\dots,x_n;q,t).
\]
By \eqref{Eq_psip}, the Pieri coefficient on the right is $1$, so that
\cite[page 325]{Macdonald95}
\begin{equation}\label{Eq_add1}
P_{\mu}(x_1,\dots,x_n;q,t) \, x_1\cdots x_n = 
P_{\mu+1^n}(x_1,\dots,x_n;q,t).
\end{equation}

Closely related to the Pieri formulas is the branching rule
\cite[page 346]{Macdonald95}
\begin{equation}\label{Eq_BR}
P_{\la}(x_1,\dots,x_n;q,t)=
\sum_{\mu\prec\la} x_n^{\abs{\skew{\la}{\mu}}} 
\psi_{\skew{\la}{\mu}}(q,t) P_{\mu}(x_1,\dots,x_{n-1};q,t),
\end{equation}
where \cite[page 341]{Macdonald95}
\begin{equation}\label{Eq_psipsip}
\psi_{\skew{\la}{\mu}}(q,t)=\psi'_{\skew{\la'}{\mu'}}(t,q).
\end{equation}
Together with the initial condition $P_{\mu}(\text{--}\,;q,t)=\delta_{\mu,0}$,
this uniquely determines the Macdonald polynomials.

We conclude our list of results for Macdonald polynomials
with the principal specialisation formula 
\cite[page 338]{Macdonald95}
\begin{equation}\label{Eq_Pspec}
P_{\la}(1,t,\dots,t^{n-1};q,t)=
P_{\la}\Big(\Big[\frac{1-t^n}{1-t}\Big];q,t\Big)=
t^{n(\la)} \frac{(t^n;q,t)_{\la}}{C^{-}_{\la}(t;q,t)}.
\end{equation}

\section{Koornwinder polynomials}
\subsection{Koornwinder polynomials}
The Koornwinder polynomials \cite{Koornwinder92} are a generalisation of
the Macdonald polynomials to the root system $\mathrm{BC}_n$.
They depend on six parameters, except for $n=1$ when they correspond to
the $5$-parameter Askey--Wilson polynomials \cite{AW85}.

Throughout this section $x=(x_1,\dots,x_n)$. 
Then the Koornwinder density is given by
\begin{equation}\label{Eq_Kdensity}
\Delta(x;q,t;\tees):=
\prod_{i=1}^n \frac{(x_i^{\pm 2};q)_{\infty}}
{\prod_{r=0}^3 (t_r x_i^{\pm};q)_{\infty}}
\prod_{1\leqslant i<j\leqslant n} 
\frac{(x_i^{\pm}x_j^{\pm};q)_{\infty}}
{(tx_i^{\pm}x_j^{\pm};q)_{\infty}},
\end{equation}
where
\begin{align*}
(x_i^{\pm};q)_{\infty}&:=(x_i,x_i^{-1};q)_{\infty}, \\
(x_i^{\pm}x_j^{\pm};q)_{\infty}&:=
(x_ix_j,x_ix_j^{-1},x_i^{-1}x_j,x_i^{-1}x_j^{-1};q)_{\infty}.
\end{align*}
For complex $q,t,t_0,\dots,t_3$ such that 
$\abs{q},\abs{t},\abs{t_0},\dots,\abs{t_3}<1$, this
density may be used to define a scalar product on $\Complex[x^{\pm 1}]$ 
as
\[
\ip{f}{g}_{q,t;\tees}^{(n)}:=
\int_{\mathbb{T}^n} f(x) g(x^{-1})\Delta(x;q,t;\tees) \dup T(x),
\]
where 
\begin{equation}\label{Eq_measure}
\dup T(x):=\frac{1}{2^n n! (2\pi\iup)^n}\,
\frac{\dup x_1}{x_1}\cdots \frac{\dup x_n}{x_n}.
\end{equation}
Let $W=\Symm_n\ltimes (\Z/2\Z)^n$ be the hyperoctahedral group
with natural action on $\Complex[x^{\pm}]$. For $\la$ a partition
of length at most $n$, let $m_{\la}^W$ be the $W$-invariant monomial
symmetric function indexed by $\la$:
\[
m_{\la}^W(x):=\sum_{\alpha} x^{\alpha}.
\]
Here the sum is over all $\alpha$ in the $W$-orbit of $\la$.
In analogy with Macdonald polynomials, the Koornwinder 
polynomials $K_{\la}=K_{\la}(x;q,t;\tees)$ 
are defined as the unique family of polynomials in 
$\Lambda^{\mathrm{BC}_n}:=\Complex[x^{\pm}]^W$ such that \cite{Koornwinder92}
\[
K_{\la}=m^W_{\la}+\sum_{\mu<\la} c_{\la\mu} m^W_{\mu}
\]
and
\begin{equation}\label{Eq_KKnul}
\ip{K_{\la}}{K_{\mu}}_{q,t;\tees}^{(n)}
=0 \qquad\text{if $\la\neq\mu$}.
\end{equation}
From the definition it follows that the $K_{\la}$ are symmetric under 
permutation of the $t_r$. 
The quadratic norm was first evaluated in \cite{vDiejen96} (self-dual case) 
and \cite{Sahi99} (general case). 
For our purposes we only need
\begin{equation}\label{Eq_Gus}
\ip{1}{1}_{q,t;\tees}^{(n)} 
=\prod_{i=1}^n 
\frac{(t,t_0t_1t_2t_3t^{n+i-2};q)_{\infty}}
{(q,t^i;q)_{\infty}\prod_{0\leqslant r<s\leqslant 3}(t_rt_st^{i-1};q)_{\infty}},
\end{equation}
known as Gustafson's integral \cite{Gustafson90}.

The $\mathrm{BC}_n$ analogue of the Cauchy identity \eqref{Eq_Mac-Cauchy-a} 
is given by \cite[Theorem 2.1]{Mimachi01}
\begin{multline}\label{Eq_Mim}
\sum_{\la\subseteq m^n} (-1)^{\abs{\la}} 
K_{m^n-\la}(x;q,t;\tees) K_{\la'}(y;t,q;\tees) \\[-1mm] 
=\prod_{i=1}^n\prod_{j=1}^m \big(x_i+x_i^{-1}-y_j-y_j^{-1}\big)
=\prod_{i=1}^n\prod_{j=1}^m x_i^{-1} \big(1-x_iy_j^{\pm}\big),
\end{multline}
where $y=(y_1,\dots,y_m)$ and $(1-ab^{\pm}):=(1-ab)(1-ab^{-1})$.

\subsection{Lifted and virtual Koornwinder polynomials}
The lifted Koornwinder polynomials 
$\tilde{K}_{\la}=\tilde{K}_{\la}(q,t,T;\tees)=
\tilde{K}_{\la}(x;q,t,T;\tees)$
are a $7$-parameter family of inhomogeneous symmetric functions
\cite{Rains05}.
They are invariant under permutations of the $t_r$ and
form a $\Rat(q,t,T,\tees)$-basis of $\Lambda$.
For example, $\tilde{K}_0=1$ and 
\[
\tilde{K}_1=m_1+\frac{1-T}{(1-t)(1-t_0t_1t_2t_3T^2/t^2)}
\sum_{r=0}^3 \Big(\frac{t_0t_1t_2t_3T}{t_rt}-t_r\Big).
\]
As a function of the $t_r$, the lifted Koornwinder polynomial
$\tilde{K}_{\la}$ has poles at
\begin{equation}\label{Eq_poles}
t_0t_1t_2t_3=
q^{1-\hat{a}(s)} t^{\hat{l}(s)+1}T^{-2}=
q^{2-\la_i-j} t^{i+\la'_j}T^{-2}
\end{equation}
for all $s=(i,j)\in\la$.
Importantly, according to \cite[Theorem 7.1]{Rains05}, 
for generic $q,t,t_0,\dots,t_3$ (so as to avoid potential poles)
\begin{multline}\label{Eq_K_lifted}
\tilde{K}_{\la}(x_1^{\pm},\dots,x_n^{\pm};q,t,t^n;\tees) \\
=\begin{cases}
K_{\la}(x_1,\dots,x_n;q,t;\tees) & \text{if $l(\la)\leqslant n$}, \\[2mm]
0 & \text{otherwise},
\end{cases}
\end{multline}
where, for $f\in\Lambda$ (or $f\in\hat{\Lambda}$, see below), 
$f(x_1,\dots,x_k):=f(x_1,\dots,x_k,0,0,\dots)$.

Let $\hat{\Lambda}$ be the completion of the ring of symmetric 
functions with respect to the natural grading by degree, i.e.,
$\hat{\Lambda}$ is the inverse limit of $\Lambda_n$ relative to
the homomorphism $\rho_{m,n}:\Lambda_m\to\Lambda_n$ $(m\geqslant n$)
which sends $m_{\la}(x_1,\dots,x_m)$ to 
$m_{\la}(x_1,\dots,x_n)$ for $l(\la)\leqslant n$ and to $0$ otherwise.
Then the virtual Koornwinder polynomials 
$\hat{K}_{\la}=\hat{K}_{\la}(q,t,Q;\tees)= 
\hat{K}_{\la}(x;q,t,Q;\tees)$ 
(which are again symmetric in the $t_r$)
form a $\Rat(q,t,Q,\tees)$-basis of $\hat{\Lambda}$, such that
for $\la\subseteq m^n$,
\begin{multline}\label{Eq_virtualK}
\hat{K}_{\la}(x_1,\dots,x_n;q,t,q^m;\tees) \\
=(x_1\cdots x_n)^m K_{m^n-\la}(x_1,\dots,x_n;q,t;\tees).
\end{multline}
When $Q=0$ the virtual Koornwinder polynomials can be expressed in terms 
of Macdonald polynomials as \cite[Corollary 7.21]{Rains05}
\[
\hat{K}_{\la}(x;q,t,0;\tees)=P_{\la}(x;q,t)
\prod_{i\geqslant 1} \frac{\prod_{r=0}^3 (t_r x_i;q)_{\infty}}
{(x_i^2;q)_{\infty}}
\prod_{i<j} \frac{(tx_ix_j;q)_{\infty}}{(x_ix_j;q)_{\infty}},
\]
from which it follows that
\begin{multline}\label{Eq_largem}
\lim_{m\to\infty} (x_1\cdots x_n)^m 
K_{m^n-\la}(x_1,\dots,x_n;q,t;\tees) \\[-1mm]
=P_{\la}(x_1,\dots,x_n;q,t)
\prod_{i=1}^n \frac{\prod_{r=0}^3 (t_r x_i;q)_{\infty}}{(x_i^2;q)_{\infty}}
\prod_{1\leqslant i<j\leqslant n} 
\frac{(tx_ix_j;q)_{\infty}}{(x_ix_j;q)_{\infty}}.
\end{multline}

The lifted and virtual Koornwinder polynomials admit a lift of the
Cauchy identity \eqref{Eq_Mim} to $\hat{\Lambda}_x\otimes\Lambda_y$ 
as follows, see \cite[Theorem 7.14]{Rains05}: 
\begin{multline}\label{Eq_Cauchy714}
\sum_{\la} (-1)^{\abs{\la}}
\hat{K}_{\la}(x;t,q,T;\tees) 
\tilde{K}_{\la'}(y;q,t,T;\tees) \\
=\prod_{i,j\geqslant 1}(1-x_iy_j).
\end{multline}
This may be used to derive the following symmetry relation for
virtual Koornwinder polynomials.

\begin{proposition}\label{Prop_propsymm}
The virtual Koornwinder polynomials satisfy
\begin{multline*}
\hat{K}_{\la}(x;q,t,Q;\tees) \\
=\hat{K}_{\lambda}(x;q,t,Qt_0t_1/q;q/t_0,q/t_1,t_2,t_3)
\prod_{i\geqslant 1}  \frac{(t_0 x_i,t_1 x_i;q)_{\infty}}
{(qx_i/t_0,qx_i/t_1;q)_{\infty}}.
\end{multline*}
\end{proposition}

\begin{proof}
We start with \eqref{Eq_Cauchy714} and identify 
the double product on the right as $\sigma_1[-xy]$. 
Carrying out the plethystic substitution 
\[
y\mapsto y+\frac{t_0-t/t_0}{1-t}+\frac{t_1-t/t_1}{1-t}=:y',
\]
and applying the symmetry \cite[Equation (7.2)]{Rains05}
\begin{multline}\label{Eq_Ktilde-symm}
\tilde{K}_{\la}\Big(\Big[y+\frac{t_0-t/t_0}{1-t}+\frac{t_1-t/t_1}{1-t}\Big];
q,t,T;\tees\Big) \\
=\tilde{K}_{\la}(y;q,t,Tt_0t_1/t;t/t_0,t/t_1,t_2,t_3),
\end{multline}
we obtain
\begin{multline}\label{Eq_KK-transformed}
\sum_{\la} (-1)^{\abs{\la}} \hat{K}_{\la}(x;t,q,T;\tees) 
\tilde{K}_{\la'}(y;q,t,Tt_0t_1/t;t/t_0,t/t_1,t_2,t_3) \\
=\prod_{i\geqslant 1} \frac{(t_0 x_i,t_1 x_i;t)_{\infty}}
{(tx_i/t_0,tx_i/t_1;t)_{\infty}} \prod_{i,j\geqslant 1}(1-x_iy_j).
\end{multline}
Here the right-hand side follows from \eqref{Eq_minmin}
and \eqref{Eq_sigma-formulas}:
\begin{align*}
\sigma_1[-xy']&=\sigma_1\Big[{-}x\Big(y+\frac{t_0-t/t_0}{1-t}+
\frac{t_1-t/t_1}{1-t}\Big)\Big] \\
&=\sigma_1[-xy]\sigma_1\Big[x\,\frac{t/t_0-t_0}{1-t}\Big]
\sigma_1\Big[x\,\frac{t/t_1-t_1}{1-t}\Big] \\
&=\prod_{i,j\geqslant 1}(1-x_iy_j)
\prod_{i\geqslant 1} \frac{(t_0 x_i,t_1 x_i;t)_{\infty}}
{(tx_i/t_0,tx_i/t_1;t)_{\infty}}.
\end{align*}
After replacing $(t_0,t_1,T)\mapsto (t/t_0,t/t_1,Tt_0t_1/t)$,
the identity \eqref{Eq_KK-transformed} takes the equivalent 
form
\begin{multline*}
\sum_{\la} (-1)^{\abs{\la}}
\hat{K}_{\la}(x;t,q,Tt_0t_1/t;t/t_0,t/t_1,t_2,t_3) 
\tilde{K}_{\la'}(y;q,t,T;\tees) \\
=\prod_{i\geqslant 1} 
\frac{(tx_i/t_0,tx_i/t_1;t)_{\infty}}{(t_0 x_i,t_1 x_i;t)_{\infty}}
\prod_{i,j\geqslant 1}(1-x_iy_j).
\end{multline*}
Expanding the double product on the right by the Cauchy identity
\eqref{Eq_Cauchy714}, and extracting coefficients of 
$\tilde{K}_{\la'}(y;q,t,T;\tees)$, yields
\begin{multline*}
\hat{K}_{\la}(x;t,q,Tt_0t_1/t;t/t_0,t/t_1,t_2,t_3)  \\
=\hat{K}_{\la}(x;t,q,T;\tees) 
\prod_{i\geqslant 1} 
\frac{(tx_i/t_0,tx_i/t_1;t)_{\infty}}{(t_0 x_i,t_1 x_i;t)_{\infty}}.
\end{multline*}
After the substitution $(t,q,T)\mapsto (q,t,Q)$ we are done.
\end{proof}

\section{Macdonald--Koornwinder polynomials}
\subsection{Macdonald polynomials on root systems}\label{Sec_Mac-RS}
Below we closely follow Macdonald's exposition in \cite{Macdonald00}.
For basic definitions and facts pertaining to root systems we refer
the reader to \cite[Chapter III]{Humphreys72}.

Let $E$ be a Euclidean space with positive-definite symmetric bilinear form
$\ip{\cdot}{\cdot}$ and $R$ a root system spanning $E$. The rank of $R$
is the dimension of $E$.
All root systems will be assumed to be irreducible, but not
necessarily reduced.
The root system dual to $R$, denoted $R^{\vee}$, is given by
\[
R^{\vee}=\{\alpha^{\vee}:~\alpha\in R\},
\]
where $v^{\vee}:=2v/\ip{v}{v}=2v/\|v\|^2$ for $v\in E$.

A pair of root systems $(R,S)$ in $E$ is called admissible if $R$ and $S$
share the same Weyl group, $W$, and $S$ is reduced.
Given such an admissible pair and $\alpha\in R$, 
there exists a unique $u_{\alpha}>0$ such that $u_{\alpha}^{-1} \alpha\in S$. 
Moreover, the map $\alpha\mapsto u_{\alpha}^{-1} \alpha$ from $R$ to $S$
is surjective (injective if $R$ is reduced) and commutes with the action 
of $W$. Hence $u_{\alpha}$ is fixed along Weyl orbits, and 
$u_{2\alpha}=2u_{\alpha}$ if $\alpha,2\alpha$ are both in $R$.
Two admissible pairs $(R,S)$ and $(R',S')$ are said to 
be similar if there exist positive real numbers $a,b$ such that $R'=aR$ and 
$S'=bS$. Then $bu'_{\alpha'}=au_{\alpha}$ for $\alpha'\in R'$ and
$\alpha\in R$ such that $\alpha'=a\alpha$.
Since we only require the classification of admissible pairs of root 
systems up to similarity, and since roots of equal length are
conjugate under the action of the Weyl group and hence
in the same Weyl orbit, we may fix the value of $u_{\alpha}$
for roots of shortest length.
The classification then breaks up into three cases.
\begin{itemize}
\item[(1)] $R$ is reduced and $S=R$ (and hence $u_{\alpha}=1$).
\item[(2)] $R$ is reduced but not simply laced, and $S=R^{\vee}$.
Unlike Macdonald, who normalises the length of short roots in 
non-simply-laced reduced root systems as $\sqrt{2}$, 
we take the length of the short roots to be $1$ when $R=\mathrm{B}_n$ 
and $\sqrt{2}$ in all other cases.
In particular this implies that $\mathrm{B}_n^{\vee}=\mathrm{C}_n$.
Writing $u_{\longroot}$ and $u_{\shortroot}$ for $u_{\alpha}$ 
indexed by long and short roots respectively, we then have:
(i) $u_{\longroot}=1$ and $u_{\shortroot}=1/2$ for 
$(R,S)=(\mathrm{B}_n,\mathrm{C}_n)$, 
(ii) $u_{\shortroot}=1$ and $u_{\longroot}=2$ for 
$(R,S)=(\mathrm{C}_n,\mathrm{B}_n)$ or
$(R,S)=(\mathrm{F}_4,\mathrm{F}_4^{\vee})$,
(iii) $u_{\shortroot}=1$ and $u_{\longroot}=3$ for 
$(R,S)=(\mathrm{G}_4,\mathrm{G}_4^{\vee})$.
\item[(3)] $R$ is the non-reduced root system $\mathrm{BC}_n$ and $S$ is one
of $\mathrm{B}_n, \mathrm{C}_n$. 
In both cases we fix $S\subseteq R$, so that 
$u_{\alpha}\in\{1,2\}$ for $S=\mathrm{B}_n$ and
$u_{\alpha}\in\{1/2,1\}$ for $S=\mathrm{C}_n$.
\end{itemize}
For $(R,S)$ an admissible pair of root systems of rank $r$, we fix a basis
of simple roots $\Delta=\{\alpha_1,\dots,\alpha_r\}$ of $R$ and write 
$\alpha>0$ if $\alpha\in R$ is a positive root with respect to $\Delta$.
The fundamental weights $\fwc_1,\dots,\fwc_r$ of $R$ are 
given by $\ip{\alpha_i^{\vee}}{\fwc_j}=\delta_{ij}$.
As usual, we denote the root, coroot and weight lattices of $R$ by 
$Q$, $Q^{\vee}$ and $P$ respectively. We also write
$Q_{+}=\sum_{i=1}^r \Nat\alpha_i$ 
for the cone in $Q$ spanned by the simple roots,
and $P_{+}=\sum_{i=1}^r \Nat\fwc_i$ for the set of dominant (integral) weights.

Let $A$ be the group algebra over $\Real$ of $P$, with elements $\eup^{\la}$,
and $A^W$ the algebra of $W$-invariant elements of $A$. A basis
of $A^W$ is given by the monomial symmetric functions
\[
m_{\la}^W=\sum_{\mu} \eup^{\mu}, \qquad \la\in P_{+},
\]
with $\mu$ summed over the $W$-orbit of $\la$.

The $(R,S)$ Macdonald polynomials defined below depend on the variables $q$ 
and $t_{\alpha}$, $\alpha\in R$, such that $t_{\alpha}$ is constant 
along Weyl orbits. 
Hence there is only one $t_{\alpha}$ in case (1), two in case (2),
and three in case (3). In each case we write this set of $t_{\alpha}$'s by
$\underline{t}$. The generalised Macdonald density (compare with
\eqref{Eq_Mac-density}) is then
\begin{equation}\label{Eq_PRS-density}
\Delta(q,\underline{t}):=\prod_{\alpha\in R} 
\frac{(t_{2\alpha}^{1/2}\eup^{\alpha};q^{u_{\alpha}})_{\infty}}
{(t_{\alpha} t_{2\alpha}^{1/2}\eup^{\alpha};q^{u_{\alpha}})_{\infty}},
\end{equation}
where $t_{2\alpha}:=1$ if $2\alpha\not\in R$.
Assuming $\abs{q},\abs{t_{\alpha}}<1$, this defines the following
scalar product on $A$:
\[
\ip{f}{g}_{q,\underline{t}}:=
\frac{1}{\abs{W}}\int_T f\,\bar{g}\,\Delta(q,\underline{t})\dup T,
\]
where integration is with respect to Haar measure on the torus 
$T=E/Q^{\vee}$ and, for $f=\sum_{\la\in P} f_{\la} \eup^{\la}$,
$\bar{f}:=\sum_{\la\in P} f_{\la} \eup^{-\la}$.
The $(R,S)$ Macdonald polynomials $P_{\la}(q,\underline{t})$,
indexed by $\la\in P_{+}$, are the unique family of 
$W$-symmetric functions 
\begin{equation}\label{Eq_PRSm}
P_{\la}(q,\underline{t})=m_{\la}^W+
\sum_{\mu<\la} u_{\la\mu}(q,\underline{t}) m_{\mu}^W
\end{equation}
such that 
\[
\ip{P_{\la}}{P_{\mu}}_{q,\underline{t}}=0\qquad \text{if$\quad\la\neq\mu$}.
\]
The sum in \eqref{Eq_PRSm} is with respect to the dominance (partial)
order on $P_{+}$ defined by $\la\geqslant\mu$ if $\la-\mu\in Q_{+}$.
Of course, when dealing with the $P_{\la}(q,\underline{t})$ as polynomials,
the restrictions $\abs{q},\abs{t_{\alpha}}<1$ imposed above may be
dropped, and typically we view $q$ and the $t_{\alpha}$ as indeterminates.

Below we are interested in the generalised Macdonald polynomials for 
$(R,S)$ one of the four admissible pairs 
$(\mathrm{B}_n,\mathrm{B}_n)$, $(\mathrm{B}_n,\mathrm{C}_n)$, 
$(\mathrm{C}_n,\mathrm{B}_n)$ and $(\mathrm{D}_n,\mathrm{D}_n)$.
Moreover, in the Hall--Littlewood limit, $q\to 0$, 
(in which case the $S$-dependence drops out) we also consider 
$R=\mathrm{BC}_n$.
In the following we assume the standard realisation in $\Real^n$
of $\mathrm{B}_n$, $\mathrm{C}_n$ and $\mathrm{D}_n$, consistent with 
our normalisation of short roots in (2):
\begin{subequations}\label{Eq_Delta}
\begin{align}\label{Eq_Delta-B}
\Delta&=\{\alpha_1,\dots,\alpha_n\}
=\big\{\epsilon_1-\epsilon_2,\dots,
\epsilon_{n-1}-\epsilon_n,\epsilon_n\big\},  
& R&=\mathrm{B}_n=\mathrm{C}_n^{\vee}, \\
\Delta&=\{\alpha_1,\dots,\alpha_n\}
=\big\{\epsilon_1-\epsilon_2,\dots,
\epsilon_{n-1}-\epsilon_n,\epsilon_{n-1}+\epsilon_n\big\},
& R&=\mathrm{D}_n.
\end{align}
\end{subequations}
We further parametrise the set of dominant weights $P_{+}$ as 
(see e.g., \cite[page 470]{KT87})
\begin{equation}\label{Eq_dominant-weights}
\la_1\epsilon_1+\cdots+\la_n\epsilon_n,
\end{equation}
where $\la=(\la_1,\dots,\la_n)$ is a partition in the case of
$\mathrm{C}_n$, and a partition or half-partition in the case
of $\mathrm{B}_n$, $\mathrm{D}_n$, with the exception that for
$\mathrm{D}_n$ the part $\la_n$ can be negative: 
\label{page_Dn}
$-\la_{n-1}\leqslant\la_n\leqslant\la_{n-1}$.\footnote{The
map $\la_n\mapsto -\la_n$ corresponds to the Dynkin diagram automorphism 
interchanging $\fwc_{n-1}$ and $\fwc_n$.}
It is not difficult to check that \eqref{Eq_dominant-weights} can be 
rewritten in terms of the fundamental weights as
\begin{subequations}\label{Eq_Pplus}
\begin{align}
&(\la_1-\la_2)\fwc_1+\cdots+(\la_{n-1}-\la_n)\fwc_{n-1}+2\la_n\fwc_n,  
& R&=\mathrm{B}_n, \label{Eq_Pplus-B} \\
&(\la_1-\la_2)\fwc_1+\cdots+(\la_{n-1}-\la_n)\fwc_{n-1}+\la_n\fwc_n,  
& R&=\mathrm{C}_n, \label{Eq_Pplus-C} \\
&(\la_1-\la_2)\fwc_1+\cdots+(\la_{n-1}-\la_n)\fwc_{n-1}+
(\la_{n-1}+\la_n)\fwc_n,
& R&=\mathrm{D}_n.
\end{align}
\end{subequations}
Finally, writing $x_i=\eup^{-\epsilon_i}$ (for $1\leqslant i\leqslant n$),
we will denote the four families of interest by
\begin{align*}
P_{\la}^{(\mathrm{B}_n,\mathrm{B}_n)}(x;q,t,t_2), &&& 
P_{\la}^{(\mathrm{B}_n,\mathrm{C}_n)}(x;q,t,t_2), \\
P_{\la}^{(\mathrm{C}_n,\mathrm{B}_n)}(x;q,t,t_2), &&&
P_{\la}^{(\mathrm{D}_n,\mathrm{D}_n)}(x;q,t),
\end{align*}
where $x=(x_1,\dots,x_n)$ and 
\begin{equation}\label{Eq_tt2}
t=t_{\alpha_1}, \qquad t_2=t_{\alpha_n}.
\end{equation}
There are several relations between these polynomials.
For example \cite[Equation (5.60)]{vDiejen95},
\begin{subequations} \label{Eq_PBBDD}
\begin{equation}
P_{\la}^{(\mathrm{D}_n,\mathrm{D}_n)}(x;q,t)
=P_{\la}^{(\mathrm{B}_n,\mathrm{B}_n)}(x;q,t,1)
\end{equation}
if $l(\la)<n$, and
\begin{align}
P_{\la}^{(\mathrm{D}_n,\mathrm{D}_n)}(x;q,t)+
P_{\bar{\la}}^{(\mathrm{D}_n,\mathrm{D}_n)}(x;q,t)
&=P_{\la}^{(\mathrm{B}_n,\mathrm{B}_n)}(x;q,t,1), \label{Eq_DDDDBB} \\[1mm]
P_{\la}^{(\mathrm{D}_n,\mathrm{D}_n)}(x;q,t)-
P_{\bar{\la}}^{(\mathrm{D}_n,\mathrm{D}_n)}(x;q,t)
&=P_{\la-(\frac{1}{2})^n}^{(\mathrm{B}_n,\mathrm{C}_n)}(x;q,t,q^{1/2}) 
\label{Eq_DDDDBB2} \\[-1pt]
& \qquad \qquad \times
\prod_{i=1}^n \big(x_i^{-1/2}-x_i^{1/2}\big)\notag 
\end{align}
\end{subequations}
if $\la$ is a partition or half-partition of length $n$.
Here $\bar{\la}:=(\la_1,\dots,\la_{n-1},-\la_n)$.
Hence
\begin{equation}\label{Eq_PDDbar}
P_{\bar{\la}}^{(\mathrm{D}_n,\mathrm{D}_n)}(x;q,t)
=P_{\la}^{(\mathrm{D}_n,\mathrm{D}_n)}(\bar{x};q,t),
\qquad \bar{x}:=(x_1,\dots,x_{n-1},x_n^{-1}).
\end{equation}
Similarly, comparing the Koornwinder density \eqref{Eq_Kdensity} with 
\eqref{Eq_PRS-density}, it follows that 
\begin{equation}\label{Eq_PCB-K}
P_{\la}^{(\mathrm{C}_n,\mathrm{B}_n)}(x;q,t,t_2)=
K_{\la}\big(x;q,t;\pm q^{1/2},\pm t_2^{1/2}\big)
\end{equation}
(see also \cite{vDiejen95}). Although the $(\mathrm{B}_n,\mathrm{B}_n)$ 
and $(\mathrm{B}_n,\mathrm{C}_n)$ Macdonald polynomials 
are indexed by partitions or half-partitions, they too
can be related to Koornwinder polynomials \cite{vDiejen95}.
Before describing this relation, we briefly discuss another family of 
polynomials incorporating both $\mathrm{B}_n$ families.

\subsection{The Macdonald--Koornwinder polynomials
\texorpdfstring{$K_{\la}(x;q,t;t_2,t_3)$}{K}}
\label{Sec_Kt2t3}

Our description of Macdonald polynomials attached to root systems is by 
no means the most general and modern setup, see e.g., 
\cite{Cherednik05,Macdonald03,Stokman11}.
Beyond Macdonald's original approach we have already covered the 
Koornwinder polynomials, and in this section we discuss one further 
family of $\mathrm{B}_n$-like polynomials, denoted by 
$K_{\la}(x;q,t;t_2,t_3)$. 
In the notation of \cite[Definition 3.21]{Stokman11} they correspond to 
the Macdonald--Koornwinder polynomials $P_{\la}^{+}$ with initial data
given by the quintuple $D=(\mathrm{B}_n,\Delta,t,P,Q)$. 
Here $\Delta$ is the basis of simple roots of $\mathrm{B}_n$ given in 
\eqref{Eq_Delta-B}, $P$ and $Q$ again denote the weight and root 
lattices of $\mathrm{B}_n$, and $t$ stands for `twisted'.

The polynomials $K_{\la}(x;q,t;t_2,t_3)$, where
$x=(x_1,\dots,x_n)$ and $\la=(\la_1,\dots,\la_n)$ is
a partition or half-partition, are $\mathrm{B}_n$-symmetric functions
in the sense of \eqref{Eq_PRSm} such that
\begin{equation}\label{Eq_KK0}
\ipbig{K_{\la}(x;q,t;t_2,t_3)}
{K_{\mu}(x;q,t;t_2,t_3)}_{q,t;t_2,t_3}^{(n)}=0
\quad\text{for}\quad \la\neq\mu.
\end{equation}
Here, for $f,g\in\Lambda^{\mathrm{BC}_n}$,
\begin{equation}\label{Eq_Bnip}
\ip{f}{g}_{q,t;t_2,t_3}^{(n)}:=
\int_{\mathbb{T}^n} f(x)g(x)\Delta(x;q,t;t_2,t_3)\dup T(x)
\end{equation}
with
\[
\Delta(x;q,t;t_2,t_3):=
\prod_{i=1}^n \frac{(x_i^{\pm};q^{1/2})_{\infty}}
{(t_2 x_i^{\pm},t_3 x_i^{\pm};q)_{\infty}}
\prod_{1\leqslant i<j\leqslant n} 
\frac{(x_i^{\pm}x_j^{\pm};q)_{\infty}}
{(tx_i^{\pm}x_j^{\pm};q)_{\infty}}.
\]
In the notation of the previous section this corresponds to 
$(R,S)=(\mathrm{B}_n,\mathrm{B}_n)$ and
\[
\Delta(q,\underline{t})=
\prod_{\alpha \text{ short}}
\frac{(\eup^{\alpha},q^{1/2} \eup^{\alpha};q)_{\infty}}
{(t_{\alpha}\eup^{\alpha},\bar{t}_{\alpha}\eup^{\alpha};q)_{\infty}}
\prod_{\alpha \text{ long}}
\frac{(\eup^{\alpha};q)_{\infty}}
{(t_{\alpha}\eup^{\alpha};q)_{\infty}},
\]
where $\underline{t}=(t_{\alpha_1},t_{\alpha_n},\bar{t}_{\alpha_n})=
(t,t_2,t_3)$.
It thus follows that
\begin{subequations}\label{Eq_PB-K}
\begin{align}\label{Eq_PBB-K}
P_{\la}^{(\mathrm{B}_n,\mathrm{B}_n)}(x;q,t,t_2),
&=K_{\la}(x;q,t;t_2,q^{1/2}) \\
P_{\la}^{(\mathrm{B}_n,\mathrm{C}_n)}(x;q,t,t_2)
&=K_{\la}(x;q,t;t_2,t_2q^{1/2}).
\label{Eq_PBC-K}
\end{align}
\end{subequations}

The next lemma shows that the $K_{\la}(x;q,t;t_2,t_3)$ can be expressed
in terms of Koornwinder polynomials, allowing us to prove results for
the former using the latter.

\begin{lemma}\label{Lem_KBn}
For $\la$ a partition or half-partition
\begin{multline*}
K_{\la}(x;q,t;t_2,t_3) \\= 
\begin{cases} K_{\la}(x;q,t;-1,-q^{1/2},t_2,t_3) & 
\text{$\la$ a partition}, \\[1mm]
\displaystyle
K_{\la-(\frac{1}{2})^n}(x;q,t;-q,-q^{1/2},t_2,t_3)
\prod_{i=1}^n \big(x_i^{1/2}+x_i^{-1/2}\big) & \text{otherwise}.
\end{cases}
\end{multline*}
\end{lemma}
For $t_3=q^{1/2}$ or $ t_3=t_2q^{1/2}$ this is equivalent to
\cite[Equations (5.50) \& (5.51)]{vDiejen95}.

\begin{proof}
The triangularity with respect to the monomial symmetric functions
for $\mathrm{B}_n$ is clear, and in the following we show that 
$K_{\la}(x;q,t;t_2,t_3)$ as given by the lemma satisfies \eqref{Eq_KK0}. 
Viewing the integral on the right of \eqref{Eq_Bnip} as a constant term
evaluation, it follows that \eqref{Eq_KK0} holds when $\la$ is a partition
and $\mu$ a half-partition.
By \eqref{Eq_KKnul} with $\{t_0,t_1\}=\{-1,-q^{1/2}\}$ it is also clear
that it holds when $\la$ and $\mu$ are both partitions.
In the case of two half-partitions
\begin{align*}
&\ipbig{K_{\la}(x;q,t;t_2,t_3)}{K_{\mu}(x;q,t;t_2,t_3)}_{q,t;t_2,t_3}^{(n)}\\
&\quad=\int_{\mathbb{T}^n} 
K_{\nu}(x;q,t;-q,-q^{1/2},t_2,t_3)
K_{\omega}(x;q,t;-q,-q^{1/2},t_2,t_3) \\
& \quad\qquad \times \prod_{i=1}^n \big(x_i^{1/2}+x_i^{-1/2}\big)^2 
\frac{(x_i^{\pm};q^{1/2})_{\infty}}
{(t_2 x_i^{\pm},t_3 x_i^{\pm};q)_{\infty}}
\prod_{1\leqslant i<j\leqslant n} \frac{(x_i^{\pm}x_j^{\pm};q)_{\infty}}
{(tx_i^{\pm}x_j^{\pm};q)_{\infty}}\: \dup T(x), 
\end{align*}
where $\nu:=\la-(\frac{1}{2})^n$ and $\omega:=\mu-(\frac{1}{2})^n$.
Since 
\[
\big(x_i^{1/2}+x_i^{-1/2}\big)^2
=\frac{(-x_i^{\pm};q^{1/2})_{\infty}}
{(-qx_i^{\pm},-q^{1/2}x_i^{\pm};q)_{\infty}}
\]
and
\[
(-x_i^{\pm},x_i^{\pm};q^{1/2})_{\infty}=(x_i^{\pm 2};q)_{\infty},
\]
the second line of the integrand is precisely the Koornwinder density 
\[
\Delta(x;q,t;-q,-q^{1/2},t_2,t_3).
\]
Hence
\begin{multline*}
\ipbig{K_{\la}(x;q,t;t_2,t_3)}{K_{\mu}(x;q,t;t_2,t_3)}_{q,t;t_2,t_3}^{(n)}\\
=\ipbig{K_{\nu}(x;q,t;-q,-q^{1/2},t_2,t_3)}
{K_{\omega}(x;q,t;-q,-q^{1/2},t_2,t_3)}_{q,t;-q,-q^{1/2},t_2,t_3}^{(n)}.
\end{multline*}
By \eqref{Eq_KKnul} this vanishes unless $\nu=\omega$, i.e., unless
$\la=\mu$.
\end{proof}

That
\[
K_{\la-(\frac{1}{2})^n}(x;q,t;-q,-q^{1/2},t_2,t_3)
\prod_i \big(x_i^{1/2}+x_i^{-1/2}\big)
\]
is the natural extension of 
$K_{\la}\big(x;q,t;-1,-q^{1/2},t_2,t_3\big)$ to half-partitions $\la$
may also be understood from the point of view of virtual Koornwinder 
polynomials. Taking
\[
(x;Q,t_0,t_1)=(x_1,\dots,x_n;q^m,-1,-q^{1/2}) 
\]
in Proposition~\eqref{Prop_propsymm} leads to
\begin{multline}\label{Eq_Khat2-xprod}
\hat{K}_{\la}(x;q,t,q^m;-1,-q^{1/2},t_2,t_3)  \\
=\hat{K}_{\la}(x;q,t,q^{m-1/2};-q,-q^{1/2},t_2,t_3)
\prod_{i=1}^n(1+x_i).
\end{multline}
Since the left-hand side of \eqref{Eq_virtualK}
is well-defined for $\la$ a partition and $m$ a half-integer, we can use
that equation to eliminate the virtual Koornwinder polynomials in 
\eqref{Eq_Khat2-xprod}.
Also replacing $m^n-\la$ by $\la$, this results in
\[
K_{\la}(x;q,t;-1,-q^{1/2},t_2,t_3) 
=K_{\la-(\frac{1}{2})^n}(x;q,t;-q,-q^{1/2},t_2,t_3)
\prod_{i=1}^n \big(x_i^{1/2}+x_i^{-1/2}\big).
\]

We conclude this section with the analogue of
\eqref{Eq_largem} for $K_{\la}(x;q,t;t_2,t_3)$.

\begin{lemma}\label{Lem_largem-2}
For $m$ a nonnegative integer or half-integer and $\la$ a partition,
\begin{multline}\label{Eq_largem-2}
\lim_{m\to\infty} (x_1\cdots x_n)^m K_{m^n-\la}(x;q,t;t_2,t_3) \\[-1mm]
=P_{\la}(x;q,t)
\prod_{i=1}^n \frac{(t_2 x_i,t_3 x_i;q)_{\infty}}
{(x_i;q^{1/2})_{\infty}}
\prod_{1\leqslant i<j\leqslant n} \frac{(tx_ix_j;q)_{\infty}}{(x_ix_j;q)_{\infty}}.
\end{multline}
\end{lemma}

\begin{proof}
For $m$ an integer this is just \eqref{Eq_largem} with 
$\{t_0,t_1\}=\{-1,-q^{1/2}\}$.
For $m$ a half-integer it follows from Lemma~\ref{Lem_KBn} that
\begin{multline*}
(x_1\cdots x_n)^m K_{m^n-\la}(x;q,t;t_2,t_3) \\
=(x_1\cdots x_n)^{m-\frac{1}{2}} 
K_{(m-\frac{1}{2})^n-\la}(x;q,t;-q,-q^{1/2},t_2,t_3)
\prod_{i=1}^n (1+x_i).
\end{multline*}
Letting $m$ tend to infinity using \eqref{Eq_largem} with
$\{t_0,t_1\}=\{-q,-q^{1/2}\}$ results in the right-hand side
of \eqref{Eq_largem-2}.
\end{proof}

\section{Hall--Littlewood polynomials}
\subsection{Hall--Littlewood of type $R$}
The ordinary (or $\mathrm{A}_{n-1}$) Hall--Littlewood polynomials are
defined as \cite[page 208]{Macdonald95}
\begin{equation}\label{Eq_HLPA}
P_{\la}(x_1,\dots,x_n;t):=\frac{1}{v_{\la}(t)} 
\sum_{w\in\Symm_n} w\bigg( x_1^{\la_1}\cdots x_n^{\la_n} 
\prod_{1\leqslant i<j\leqslant n} \frac{x_i-tx_j}{x_i-x_j} \bigg),
\end{equation}
where $\la$ is a partition of length at most $n$,
$v_{\la}(t):=\prod_{i\geqslant 0} (t;t)_{m_i(\la)}/(1-t)^{m_i(\la)}$
\label{Page_vla}
and $m_0(\la):=n-l(\la)$. 
They correspond to the Macdonald polynomials for $q=0$, i.e.,
$P_{\la}(x;t)=P_{\la}(x;0,t)$, and for $t=0$
reduce to the Schur functions.

Taking $q=0$ in the $(R,S)$ Macdonald polynomials of Section~\ref{Sec_Mac-RS}
yields the more general Hall--Littlewood polynomials of type $R$.
(The root system $S$ no longer plays a role when $q=0$.)
More simply, however, the Hall--Littlewood polynomials of type $R$ can
be explicitly computed from
\begin{equation}\label{Eq_MacP}
P_{\la}(t)=\frac{1}{W_{\la}(t)}
\sum_{w\in W} w\bigg( \! \eup^{\la}
\prod_{\alpha>0} \frac{1-t_{\alpha} t_{2\alpha}^{1/2} \eup^{-\alpha}}
{1-t_{2\alpha}^{1/2} \eup^{-\alpha}}\bigg),\qquad \la\in P_{+}.
\end{equation}
The normalising factor $W_{\la}(t)$ is the Poincar\'e
polynomial of the stabilizer of $\la$ in $W$.

For our purposes we need to consider \eqref{Eq_MacP} for $R$ one of 
$\mathrm{BC}_n,\mathrm{B}_n,\mathrm{C}_n,\mathrm{D}_n$. 
As in Section~\ref{Sec_Mac-RS}, we express these four families
using variables $x_1,\dots,x_n$ and partitions or half-partitions $\la$, 
rather than roots and dominant weights. Accordingly we write
$P_{\la}^{(R)}(x_1,\dots,x_n;q,t,t_2,\dots,t_r)$, where $r=3$ in the 
case of $\mathrm{BC}_n$, $r=2$ for $\mathrm{B}_n$ and $\mathrm{C}_n$ and
$r=1$ for $\mathrm{D}_n$. In each case we again assume
\eqref{Eq_Delta} and \eqref{Eq_Pplus}, and identify $x_i=\exp(-\epsilon_i)$.
As a basis of simple roots for the root system $\mathrm{BC}_n$ we take the 
$\mathrm{B}_n$ basis \eqref{Eq_Delta-B} (as opposed to a $\mathrm{C}_n$
basis), and identify 
$(t_{\alpha_1},t_{\alpha_n},t^{1/2}_{2\alpha_n})=(t,-t_2/t_3,-t_3)$.
Then \eqref{Eq_MacP} takes the equivalent form
\begin{multline}\label{Eq_BCHL2}
P^{(\mathrm{BC}_n)}_{\la}(x;t,t_2,t_3)
=\frac{1}{(t_2t_3;t)_{n-l(\la)}v_{\la}(t)} \\
\times \sum_{w\in W} w\bigg( x^{-\la} \prod_{i=1}^n
\frac{(1-t_2x_i)(1-t_3x_i)}{1-x_i^2}
\prod_{1\leqslant i<j\leqslant n}
\frac{(tx_i-x_j)(1-tx_ix_j)}{(x_i-x_j)(1-x_ix_j)} \bigg),
\end{multline}
with $W$ the hyperoctahedral group and $\la$ a partition of length at most $n$.
Alternatively (see e.g., \cite{Venkateswaran15}),
\begin{equation}\label{Eq_BCHL}
P^{(\mathrm{BC}_n)}_{\la}(x;t,t_2,t_3)=K_{\la}(x;0,t;0,0,t_2,t_3).
\end{equation}
For $t=0$ this admits the determinantal form
\begin{multline}\label{Eq_detBC}
P^{(\mathrm{BC}_n)}_{\la}(x;0,t_2,t_3)=\frac{1}{\Delta_{\mathrm{C}}(x)} \\
\times \det_{1\leqslant i,j\leqslant n} 
\big(x_i^{-\la_j+j-1}(1-t_2x_i)(1-t_3x_i)-
x_i^{\la_j+2n-j-1}(x_i-t_2)(x_i-t_3)\big),
\end{multline}
where 
\begin{equation}\label{Eq_VdMC}
\Delta_{\mathrm{C}}(x):=\prod_{i=1}^n (1-x_i^2)
\prod_{1\leqslant i<j\leqslant n} (x_i-x_j)(x_ix_j-1) 
=\det_{1\leqslant i,j\leqslant n}\big(x_i^{j-1}-x_i^{2n-j+1}\big)
\end{equation}
is the $\mathrm{C}_n$ Vandermonde product.

\begin{lemma}\label{Lem_littlelemma}
Let
\begin{equation}\label{Eq_Phi}
\Phi(x;t,t_2,t_3):=\prod_{i=1}^n \frac{(1-t_2x_i)(1-t_3x_i)}{1-x_i^2}
\prod_{1\leqslant i<j\leqslant n}\frac{1-tx_ix_j}{1-x_ix_j}.
\end{equation}
For $m$ a positive integer,
\begin{equation}\label{Eq_PBC-Phi}
P^{(\mathrm{BC}_n)}_{m^n}(x;t,t_2,t_3)
=\sum_{\varepsilon\in \{\pm 1\}^n}\Phi(x^{\varepsilon};t,t_2,t_3)
\prod_{i=1}^n x_i^{-\varepsilon_i m}.
\end{equation}
\end{lemma}

\begin{proof}
By \eqref{Eq_HLPA} and \eqref{Eq_Phi},
\begin{align*}
&\sum_{w\in\Symm_n} w\bigg( x^{-\la} \prod_{i=1}^n
\frac{(1-t_2x_i)(1-t_3x_i)}{1-x_i^2} 
\prod_{1\leqslant i<j\leqslant n}
\frac{(tx_i-x_j)(1-tx_ix_j)}{(x_i-x_j)(1-x_ix_j)} \bigg) \\
&\quad=\Phi(x;t,t_2,t_3) \sum_{w\in\Symm_n} 
w\bigg(x^{-\la} \prod_{1\leqslant i<j\leqslant n}
\frac{tx_i-x_j}{x_i-x_j} \bigg) \\[1mm]
&\quad=v_{\la}(t) \Phi(x;t,t_2,t_3) P_{\la}(x^{-1};t).
\end{align*}
Hence
\[
P^{(\mathrm{BC}_n)}_{\la}(x;t,t_2,t_3)=\frac{1}{(t_2t_3;t)_{n-l(\la)}}
\sum_{\varepsilon\in\{\pm 1\}^n} 
\Phi(x^{\varepsilon};t,t_2,t_3) P_{\la}(x^{-\varepsilon};t).
\]
The claim now follows from $P_{m^n}(x;t)=\prod_{i=1}^n x_i^m$,
and the fact that $n-l(\la)=0$ for $\la=m^n$ with $m\geqslant 1$.
\end{proof}

In the case of $\mathrm{C}_n$ we can be brief. 
The identification of parameters \eqref{Eq_tt2} again applies, 
and from the $q=0$ case of \eqref{Eq_PBC-K}, 
\begin{align}\label{Eq_CHL}
P^{(\mathrm{C}_n)}_{\la}(x;t,t_2)&=
P^{(\mathrm{BC}_n)}_{\la}\big(x;t,\pm t_2^{1/2}\big) \\
&=P^{(\mathrm{C}_n,\mathrm{B}_n)}_{\la}(x;0,t,t_2)=
K_{\la}\big(x;0,t;0,0,\pm t_2^{1/2}\big).\notag
\end{align}
Accordingly, \eqref{Eq_detBC} for $t_3=-t_2$ is a one-parameter
deformation of the symplectic Schur function \cite{Littlewood50}
\begin{equation}\label{Eq_symplectic}
\symp_{2n,\la}(x):=\frac{1}{\Delta_{\mathrm{C}}}\,
\det_{1\leqslant i,j\leqslant n} 
\big(x_i^{-\la_j+j-1}-x_i^{\la_j+2n-j+1}\big).
\end{equation}

The Hall--Littlewood polynomials $P^{(\mathrm{B}_n)}_{\la}(x;t,t_2)$ are
given by the \eqref{Eq_BCHL2} for $t_3=-1$, where now $\la$ is a partition
or half-partition.
In the latter case $v_{\la}(t)$ is as defined on 
page~\pageref{Page_vla} but with $\prod_{i\geqslant 0}$ a product over 
half-integers. By \eqref{Eq_PB-K} we also have 
\begin{equation}\label{Eq_BHL1}
P^{(\mathrm{B}_n)}_{\la}(x;t,t_2)=
P^{(\mathrm{B}_n,\mathrm{B}_n)}_{\la}(x;0,t,t_2)=K_{\la}(x;0,t;t_2,0).
\end{equation}
When $t=t_2=0$ the $\mathrm{B}_n$ Hall--Littlewood polynomials
simplify to the odd orthogonal Schur functions \cite{Littlewood50}
\begin{equation}\label{Eq_odd-orthogonal}
\so_{2n+1,\la}(x):=\frac{1}{\Delta_{\mathrm{B}}(x)}\,
\det_{1\leqslant i,j\leqslant n} 
\big(x_i^{-\la_j+j-1}-x_i^{\la_j+2n-j}\big),
\end{equation}
where 
\begin{equation}\label{Eq_VdMB}
\Delta_{\mathrm{B}}(x):=\prod_{i=1}^n (1-x_i)
\prod_{1\leqslant i<j\leqslant n} (x_i-x_j)(x_ix_j-1) 
=\det_{1\leqslant i,j\leqslant n}(x_i^{j-1}-x_i^{2n-j}).
\end{equation}

The $\mathrm{B}_n$ analogue of Lemma~\ref{Lem_littlelemma} is
given as follows.

\begin{lemma}\label{Lem_littlelemmaB}
For $m$ a positive integer,
\[
P^{(\mathrm{B}_n)}_{\halfm{n}}(x;t,t_2)
=\sum_{\varepsilon\in \{\pm 1\}^n}\Phi(x^{\varepsilon};t,t_2,-1)
\prod_{i=1}^n x_i^{-\varepsilon_i m/2}.
\]
\end{lemma}

Finally, the $\mathrm{D}_n$ Hall--Littlewood polynomials
$P^{(\mathrm{D}_n)}_{\la}(x;t)$ are given by \eqref{Eq_BCHL2} with
$(t_2,t_3)=(-1,1)$, provided we multiply the right-hand side by $2$
when $l(\la)<n$, and $W$ is taken to be the group of even
signed-permutations, i.e., $W=\Symm_n\ltimes (\Z/2\Z)^{n-1}$.
As described on page~\pageref{page_Dn}, $\la$ can be a partition or
half-partition with $\la_n$ an exceptional part that can take on negative
integer or half-integer values as long as $\abs{\la_n}\leqslant\la_{n-1}$.
When $t=0$ this yields the even orthogonal Schur functions \cite{Okada98}
\begin{multline*}
\so_{2n,\la}(x):=\frac{1}{2\Delta_{\mathrm{D}}(x)}
\Big(\det_{1\leqslant i,j\leqslant n}
\big(x_i^{\la_j+2n-j-1}+x_i^{-\la_j+j-1}\big) \\
-\det_{1\leqslant i,j\leqslant n}
\big(x_i^{\la_j+2n-j-1}-x_i^{-\la_j+j-1}\big)\Big),
\end{multline*}
where
\begin{equation}\label{Eq_VdMD}
\Delta_{\mathrm{D}}(x):=
\prod_{1\leqslant i<j\leqslant n} (x_i-x_j)(x_ix_j-1)
=\tfrac{1}{2}\det_{1\leqslant i,j\leqslant n}(x_i^{j-1}+x_i^{2n-j-1}).
\end{equation}

Again we have a simple analogue of Lemma~\ref{Lem_littlelemma}.

\begin{lemma}\label{Lem_littlelemma2}
Let $\varepsilon=(\varepsilon_1,\dots,\varepsilon_n)\in\{-1,1\}^n$
and $\sgn(\varepsilon):=\prod_{i=1}^n \varepsilon_i$. Then,
for $m$ a positive integer,
\[
P^{(\mathrm{D}_n)}_{\halfm{n}}(x;t)
=\sum_{\substack{\varepsilon\in \{\pm 1\}^n \\ \sgn(\varepsilon)=1}}
\Phi(x^{\varepsilon};t,1,-1) \prod_{i=1}^n x_i^{-\varepsilon_i m/2}.
\]
\end{lemma}

\subsection{Modified Hall--Littlewood polynomials}
The simplest definition of the modified Hall--Littlewood polynomials
is as a plethystically substituted ordinary Hall--Littlewood polynomial
(see e.g.,\cite[page 234]{Macdonald95}):
\begin{equation}\label{Eq_HLPp}
P'_{\la}(x;t):=P_{\la}\Big(\Big[\frac{x}{1-t}\Big];t\Big)\quad\text{and}\quad
Q'_{\la}(x;t):=Q_{\la}\Big(\Big[\frac{x}{1-t}\Big];t\Big),
\end{equation}
where $Q_{\la}(x;t):=b_{\la}(t) P_{\la}(x;t)$ and
\[
b_{\la}(t):=\prod_{i\geqslant 1} (t;t)_{m_i(\la)}.
\]
This definition obscures the important fact that the modified 
Hall--Littlewood polynomials, and hence sums such as
\begin{equation}\label{Eq_Pp-char}
\sum_{\substack{\la \\[1pt] \la_1\leqslant m}}
t^{\abs{\la}/2} P'_{\la}\big(x_1^{\pm},\dots,x_n^{\pm};t\big)
\end{equation}
(see \eqref{Eq_char-A2n2}), are Schur positive.
To exhibit the combinatorial nature of \eqref{Eq_Pp-char}
and similar such expressions for the characters of affine
Lie algebras given in Section~\ref{Sec_char}, 
we need the Lascoux--Sch\"utzenberger 
description of the modified Hall--Littlewood polynomials in terms
of the charge statistic on tableaux~\cite{LS78}.

A filling of the Young diagram of a partition $\la$ with positive integers
such that rows are weakly increasing from left to right and columns are 
strictly increasing from top to bottom is called a semistandard Young 
tableau of shape $\la$, see e.g., \cite{Fulton97,Stanley99,Macdonald95}.
The (weak) composition $\mu=(\mu_1,\mu_2,\dots)$ such that $\mu_i$ counts the
number of squares filled with the number $i$ is called the weight 
(or filling/content/type) of the tableau. 
If we denote the set of semistandard Young tableaux of shape $\la$ and 
weight $\mu$ by $\Tab(\la,\mu)$, then the Schur function $s_{\la}$ may be 
expressed as
\begin{equation}\label{Eq_Schur-T}
s_{\la}(x_1,x_2,\dots)=\sum_{T\in\Tab(\la,\cdot)} x^T.
\end{equation}
Here $x^T$ is shorthand for $x^{\weight(T)}$, so that
$x^T=x^{\mu}=x_1^{\mu_1}x_2^{\mu_2}\cdots$ if $T\in\Tab(\cdot,\mu)$.
Lascoux and Sch\"utzenberger equipped the set of semistandard Young tableaux 
with a statistic, called charge. For our purposes it suffices to
consider tableaux in $\Tab(\cdot,\mu)$ such that $\mu$ is a partition. 
To compute the charge of such a tableau, we first
form the reverse reading word, $w=w(T)$, by consecutively reading the
rows of $T$ from right to left, starting with the top row
and ending with the bottom row.
For example, the reverse reading word of 
\medskip

\begin{center}
\begin{tikzpicture}[scale=0.4,line width=0.3pt]
\draw (0,0)--(6,0);
\draw (0,-1)--(6,-1);
\draw (0,-2)--(3,-2);
\draw (0,-3)--(2,-3);
\draw (0,-4)--(1,-4);
\draw (0,0)--(0,-4);
\draw (1,0)--(1,-4);
\draw (2,0)--(2,-3);
\draw (3,0)--(3,-2);
\draw (4,0)--(4,-1);
\draw (5,0)--(5,-1);
\draw (6,0)--(6,-1);
\draw (0.5,-0.5) node {$1$};
\draw (1.5,-0.5) node {$1$};
\draw (2.5,-0.5) node {$1$};
\draw (3.5,-0.5) node {$2$};
\draw (4.5,-0.5) node {$3$};
\draw (5.5,-0.5) node {$4$};
\draw (0.5,-1.5) node {$2$};
\draw (1.5,-1.5) node {$2$};
\draw (2.5,-1.5) node {$3$};
\draw (0.5,-2.5) node {$4$};
\draw (1.5,-2.5) node {$6$};
\draw (0.5,-3.5) node {$5$};
\end{tikzpicture}
\end{center}
\smallskip
is $432111 322 64 5$. 
In a clockwise manner, wrap the letters of $w=w_1\dots w_k$ 
around a circle, putting a marker between the first letter $w_1$
and last letter $w_k$. Repeatedly going round the circle in
clockwise manner, read the letters of $w$ starting with $w_1$.
Label a total of $\mu'_1$ letters as follows. If a letter $i$ has just
been labelled $k$, then the first letter $i+1$ read after that $i$ is
labelled $k$ if it occurs before the marker and $k+1$ if it occurs after
the marker.
To get started, label the first $1$ that is read by $0$. 
For the reading word in our example this gives

\medskip
\begin{center}
\begin{tikzpicture}[scale=1.2,line width=0.3pt]
\foreach \x in {75} \draw ({cos(\x)},{sin(\x)}) node {$4$};
\foreach \x in {45} \draw ({cos(\x)},{sin(\x)}) node {$3_{\red{1}}$};
\foreach \x in {15} \draw ({cos(\x)},{sin(\x)}) node {$2$};
\foreach \x in {345} \draw ({cos(\x)},{sin(\x)}) node {$1_{\red{0}}$};
\foreach \x in {315} \draw ({cos(\x)},{sin(\x)}) node {$1$};
\foreach \x in {285} \draw ({cos(\x)},{sin(\x)}) node {$1$};
\foreach \x in {255} \draw ({cos(\x)},{sin(\x)}) node {$3$};
\foreach \x in {225} \draw ({cos(\x)},{sin(\x)}) node {$2_{\red{0}}$};
\foreach \x in {195} \draw ({cos(\x)},{sin(\x)}) node {$2$};
\foreach \x in {165} \draw ({cos(\x)},{sin(\x)}) node {$6_{\red{2}}$};
\foreach \x in {135} \draw ({cos(\x)},{sin(\x)}) node {$4_{\red{1}}$};
\foreach \x in {105} \draw ({cos(\x)},{sin(\x)}) node {$5_{\red{1}}$};
\draw (0,0.8)--(0,1.2);
\end{tikzpicture}
\end{center}
\smallskip
Keep repeating the above labelling procedure with the remaining unlabelled 
letters of $w$ until all letter are labelled, in such a way that $\mu'_i$
letter are labelled in the $i$th step.
The completed labelling of the word in our example is
\medskip
\begin{center}
\begin{tikzpicture}[scale=1.2,line width=0.3pt]
\foreach \x in {75} \draw ({cos(\x)},{sin(\x)}) node {$4_{\blue{2}}$};
\foreach \x in {45} \draw ({cos(\x)},{sin(\x)}) node {$3_{\red{1}}$};
\foreach \x in {15} \draw ({cos(\x)},{sin(\x)}) node {$2_{\green{1}}$};
\foreach \x in {345} \draw ({cos(\x)},{sin(\x)}) node {$1_{\red{0}}$};
\foreach \x in {315} \draw ({cos(\x)},{sin(\x)}) node {$1_{\blue{0}}$};
\foreach \x in {285} \draw ({cos(\x)},{sin(\x)}) node {$1_{\green{0}}$};
\foreach \x in {255} \draw ({cos(\x)},{sin(\x)}) node {$3_{\blue{1}}$};
\foreach \x in {225} \draw ({cos(\x)},{sin(\x)}) node {$2_{\red{0}}$};
\foreach \x in {195} \draw ({cos(\x)},{sin(\x)}) node {$2_{\blue{0}}$};
\foreach \x in {165} \draw ({cos(\x)},{sin(\x)}) node {$6_{\red{2}}$};
\foreach \x in {135} \draw ({cos(\x)},{sin(\x)}) node {$4_{\red{1}}$};
\foreach \x in {105} \draw ({cos(\x)},{sin(\x)}) node {$5_{\red{1}}$};
\draw (0,0.8)--(0,1.2);
\end{tikzpicture}
\end{center}
\smallskip
or, in one-line, notation
$4_{\blue{2}} 3_{\red{1}}  2_{\green{1}} 
1_{\red{0}}  1_{\blue{0}} 1_{\green{0}} 
3_{\blue{1}} 2_{\red{0}}  2_{\blue{0}} 
6_{\red{2}}  4_{\red{1}}  5_{\red{1}}$.
The charge, $\charge(T)$, of $T$ is the sum of the 
labels of its reverse reading word. For the tableau in the
example, $\charge(T)=\blue{2}+\red{1}+\green{1}+\red{0}+\blue{0}+\green{0}+
\blue{1}+\red{0}+\blue{0}+\red{2}+\red{1}+\red{1}=9$.

Using the charge statistic, the modified Hall--Littlewood 
polynomial $Q'_{\mu}$ can be expressed as \cite{Butler94,DLT94,LS78}
\[
Q'_{\mu}(x;t)=\sum_{T\in\Tab(\cdot,\mu)} t^{\charge(T)} s_{\shape(T)}(x),
\]
or, equivalently, as
\[
Q'_{\mu}(x;t)=\sum_{\la} K_{\la\mu}(t) s_{\la}(x),
\]
where $K_{\la\mu}(t)$ is the Kostka--Foulkes polynomial
\[
K_{\la\mu}(t):=\sum_{T\in\Tab(\la,\mu)} t^{\charge(T)}. 
\]
The coefficients $K_{\la\mu}(t)/b_{\mu}(t)$
in the Schur expansion of the modified Hall--Little\-wood 
polynomials $P'_{\mu}(x;t)$ are rational functions instead of polynomials,
but, viewed as a formal power series in $t$,
\[
\frac{K_{\la\mu}(t)}{b_{\mu}(t)}=\sum_{k\geqslant 0}a_k t^k,
\quad a_k\in\mathbb{Z}_{\geqslant 0}.
\]

Alternatively, we have the combinatorial expression \cite{Kirillov00,WZ12}
\begin{equation}\label{Eq_Qp-qbinom}
Q'_{\la}(x_1,\dots,x_n;t)=
\sum_{0=\mu^{(n)}\subseteq\cdots
\subseteq\mu^{(1)}\subseteq\mu^{(0)}=\la}
\prod_{i=1}^n g_{\skew{\mu^{(i-1)}}{\mu^{(i)}}}(x_i;t),
\end{equation}
where
\begin{equation}\label{Eq_glamu}
g_{\skew{\la}{\mu}}(z;t):=
Q_{\skew{\la}{\mu}}\Big(\Big[z\, \frac{1-q}{1-t}\Big];t\Big)=
z^{\abs{\skew{\la}{\mu}}}t^{n(\skew{\la}{\mu})}
\prod_{i\geqslant 1} \qbin{\la'_i-\mu'_{i+1}}{\la'_i-\mu'_i}_t
\end{equation}
is the $h$-Pieri coefficient for Hall--Littlewood polynomials
\[
P_{\mu}(x;t)\prod_{i\geqslant 1} \frac{1}{1-zx_i}=
\sum_{\la\supseteq\mu} g_{\mu/\nu}(z;t) P_{\la}(x;t),
\]
see \cite{Kirillov00,KL13,W13}.
Note that for $n=1$ this trivialises to the principal specialisation
formula \cite[page 213]{Macdonald95}
\begin{equation}\label{Eq_Qp-een}
Q'_{\la}(z;t)=z^{\abs{\la}} Q_{\la}(1,q,q^2,\dots;q)=
z^{\abs{\la}}t^{n(\la)}.
\end{equation}
In Section~\ref{Sec_char} we will show how \eqref{Eq_Qp-qbinom} can be used 
to simplify some of our character formulas when restricted to the
basic representation.


\chapter{Virtual Koornwinder integrals}\label{Ch_virtual}
\section{Basic definitions}
For $\{f_{\la}\}$ a basis of $\Lambda_n$, $\Lambda$ or
$\Lambda^{\mathrm{BC}_n}$ and $g$ an arbitrary element of one of these
spaces, we write $[f_{\la}] g$ for the coefficient $c_{\la}$ in
$g=\sum_{\la} c_{\la} f_{\la}$.
Although typically $f_0=1$ we will still write $[f_0]g$ to avoid ambiguity 
as to the choice of basis.

\medskip

Our approach to bounded Littlewood identities relies crucially 
on properties of two linear functionals, denoted $I_K$ and $I_K^{(n)}$
and referred to as virtual Koornwinder integrals,
acting on $\Lambda$ and $\Lambda^{\mathrm{BC}_n}$ respectively. 
Let $f\in\Lambda$. Then the virtual Koornwinder integral 
$I_K$ is defined as \cite[page 110]{Rains05}
\begin{equation}\label{Eq_functional-IK}
I_K\big(f;q,t,T;\tees\big):=\big[\tilde{K}_0(q,t,T;\tees)]f,
\end{equation}
where $\tilde{K}_{\la}$ is the lifted Koornwinder polynomial.
Similarly, for $f\in\Lambda^{\mathrm{BC}_n}$ and $x=(x_1,\dots,x_n)$
\cite[page 95]{Rains05}
\begin{equation}\label{Eq_functional-IKn}
I_K^{(n)}\big(f;q,t;\tees\big):=\big[K_0(x;q,t;\tees)] f.
\end{equation}

By \eqref{Eq_KKnul} and $K_0=1$, it follows that
\begin{align}\label{Eq_IKn-ip}
I_K^{(n)}\big(f;q,t;\tees\big)
&=\frac{\ip{1}{f(x)}}{\ip{1}{1}} \\
&=\frac{1}{\ip{1}{1}}\int_{\mathbb{T}^n} f(x)
\Delta(x;q,t;\tees)\dup T(x). \notag
\end{align}
Also, by \eqref{Eq_K_lifted}, for $f\in\Lambda$ and generic $q,t,\tees$,
\begin{equation}\label{Eq_functional-IK-IKn}
I_K^{(n)}\big(f(x_1^{\pm},\dots,x_n^{\pm};q,t;\tees\big)=
I_K\big(f;q,t,t^n;\tees\big).
\end{equation}

\begin{remark}
Invoking the ring homomorphism
$\varphi: \Lambda_{2n}\to \Lambda^{\mathrm{BC}_n}$ given by 
\[
\varphi\big(m_{\la}(x_1,\dots,x_{2n})\big)=
m_{\la}(x_1^{\pm},\dots,x_n^{\pm 1})
\]
(so that $\ker(\varphi)=\langle e_i-e_{2n-i}:~0\leqslant i<n\rangle$),
it will be convenient to extend $I_K^{(n)}$
to also act on $\Lambda_{2n}$ (or more simply $f\in\Lambda$).
That is, we set
\begin{multline}\label{Eq_functional-IKn-2}
I_K^{(n)}\big(f;q,t;\tees\big):=\big[K_0(x;q,t;\tees)] 
f(x_1^{\pm},\dots,x_n^{\pm}), \\ 
\text{for $f\in\Lambda_{2n}$, or $f\in\Lambda)$}.
\end{multline}
Since for such $f$
\[
I_K^{(n)}\big(f(x_1,\dots,x_{2n});q,t;\tees\big)=
I_K^{(n)}\big(f(x_1^{\pm},\dots,x_n^{\pm});q,t;\tees\big),
\]
we will often not distinguish between \eqref{Eq_functional-IKn} and 
\eqref{Eq_functional-IKn-2} and simply write 
\[
I_K^{(n)}\big(f;q,t;\tees\big),
\]
where $f$ can be either a symmetric or $\mathrm{BC}_n$-symmetric function.
\end{remark}

\medskip

The reader is warned that for specialisations that hit the poles 
\eqref{Eq_poles} of the lifted Koornwinder 
polynomials\footnote{Since $q^{2-\la_i-j}$ is a nonpositive integer 
power of $q$ this can never happen when the product $t_0t_1t_2t_3$ 
contains a positive power of $q$.},  
\eqref{Eq_functional-IK-IKn} should be treated with great caution
as the right-hand side may not be well-defined. \label{page_poles}
In such cases $T$ needs to be specialised before the $t_r$ 
(or at least before some of the $t_r$).
For example, if $T=t^n$ and 
$\{\tees\}=\{1,-1,t^{1/2},-t^{1/2}\}=:\{\pm 1,\pm t^{1/2}\}$,
then the lifted Koornwinder polynomial $\tilde{K}_{\la}$ 
is ill-defined if
\[
l(\lambda)-\tfrac{1}{2}m_1(\lambda)\leqslant n<l(\lambda).
\]
Accordingly, for \eqref{Eq_functional-IK-IKn} to hold
we must first specialise $T=t^n$ before specialising the $t_r$.
For example, since
\[
\tilde{K}_{1^2}(q,t,T;\pm 1,\pm t_2)=
m_{1^2}-\frac{(1-T)(T/t+t)(t-t_2^2)}{(1-t)(1+t)(t-t_2^2(T/t)^2)}
\]
(and $\tilde{K}_0=1$), we have
\[
I_K(m_{1^2};q,t,T;\pm 1,\pm t_2)=
\frac{(1-T)(T/t+t)(t-t_2^2)}{(1-t)(1+t)(t-t_2^2(T/t)^2)}.
\]
This gives $I_K(m_{1^2}(q,t,t;\pm 1,\pm t_2)=1$,
which trivially agrees with
\begin{equation}\label{Eq_IKm11-a}
I_K^{(1)}(m_{1^2};q,t;\pm 1,\pm t_2)=I_K^{(1)}(1;q,t;\pm 1,\pm t_2)=1.
\end{equation}
However,
\begin{equation}\label{Eq_IKm11-b}
I_K(m_{1^2};q,t,T;\pm 1,\pm t^{1/2})=0.
\end{equation}

For specialising in the `wrong' order we can use
\cite[Lemma 5.10]{Rains14}.

\begin{lemma}\label{Lem_wrong-order}
For fixed $n$, let $t_0\cdots t_3 t^{n-2}=t^k$, 
where $k$ is a nonnegative integer. Then
\begin{align*}
\lim_{T\to t^n} & I_K(f;q,t,T;\tees)  \\
&=\tfrac{1}{2} I_K^{(n)}(f;q,t;\tees) \\
&\quad +\tfrac{1}{2} I_K^{(k)}
\Big(f\Big[x_1^{\pm}+\cdots+x_k^{\pm}+\sum_{r=0}^3 \frac{t_r-t/t_r}{1-t}
\Big];q,t;t/t_0,t/t_1,t/t_2,t/t_3\Big).
\end{align*}
\end{lemma}

\begin{remark}
To be consistent with our convention that 
\[
f(x_1^{\pm},\dots,x_k^{\pm})=f(x_1,x_1^{-1},\dots,x_k,x_k^{-1})
\]
and 
$f(x_1+\cdots+x_k)=f[x_1+\cdots+x_k]$, we interpret
$f[x_1^{\pm}+\cdots+x_k^{\pm}]$ as
$f[x_1+x_1^{-1}+\cdots+x_k+x_k^{-1}]$.
\end{remark}

Continuing our previous example, for
\begin{equation}\label{Eq_teespmpm}
\{t_0,t_1,t_2,t_3\}=\{\pm 1,\pm t^{1/2}\}=
\{1,\varepsilon,t^{1/2},\varepsilon t^{1/2}\}
\end{equation}
we plethystically have
\begin{equation}\label{Eq_plet1pluseps}
\sum_{r=0}^3 \frac{t_r-t/t_r}{1-t}
=\frac{1-t}{1-t}
+\frac{\varepsilon-\varepsilon t}{1-t}
+\frac{t^{1/2}-t^{1/2}}{1-t}
+\frac{\varepsilon t^{1/2}-\varepsilon t^{1/2}}{1-t}
=1+\varepsilon,
\end{equation}
so that
\begin{align*}
\lim_{T\to t} &I_K\big(f;q,t,T;\pm 1,\pm t^{1/2}\big) \\
&=\tfrac{1}{2} I_K^{(1)}\big(f;q,t,\pm 1,\pm t^{1/2}\big)
+\tfrac{1}{2} I_K^{(0)}\big(f[1+\varepsilon];q,t,\pm t,\pm t^{1/2}\big) \\[1mm]
&=\tfrac{1}{2} I_K^{(1)}\big(f;q,t,\pm 1,\pm t^{1/2}\big)
+\tfrac{1}{2} I_K^{(0)}\big(f(1,-1);q,t,\pm t,\pm t^{1/2}\big).
\end{align*}
If $f=m_{1^2}$ then $f(1,-1)=-1$. By $I_K^{(0)}(f;q,t,\tees)=f$
and \eqref{Eq_IKm11-a}, this yields 
\[
\lim_{T\to t} I_K\big(m_{1^2};q,t,T;\pm 1,\pm t^{1/2}\big),
=\tfrac{1}{2}(1-1)=0,
\]
in accordance with \eqref{Eq_IKm11-b}.

\medskip

\begin{lemma}\label{Lem_IK}
For $\mu$ a partition and generic $q,t,t_2,t_3$,
\begin{equation}\label{Eq_C76-spec}
I_K\big(f[x+\varepsilon];q,t,T;-t,-t^{1/2},t_2,t_3\big) \\
=I_K\big(f;q,t,t^{1/2}T;-1,-t^{1/2},t_2,t_3\big).
\end{equation}
\end{lemma}

\begin{proof}
Let $f\in\Lambda$. From definition \eqref{Eq_functional-IK} of the virtual
Koornwinder integral and the symmetry \eqref{Eq_Ktilde-symm} of the lifted
Koornwinder polynomials, we infer that
(see also \cite[Equation (7.4)]{Rains05}\footnote{In 
\cite[Equation (7.4)]{Rains05} the denominator term $(1-t)$ should
be corrected to $(1-t^k)$.})
\begin{multline}\label{Eq_C76}
I_K\Big(f\Big[x+\frac{t_0-t/t_0}{1-t}+\frac{t_1-t/t_1}{1-t}\Big];q,t,T;
t/t_0,t/t_1,t_2,t_3\Big)\\
=I_K\big(f;q,t,tT/t_0t_1;\tees\big).
\end{multline}
If we specialise 
$\{t_0,t_1\}=\{-1,-t^{1/2}\}=\{\varepsilon,\varepsilon t^{1/2}\}$,
so that
\[
\frac{t_0-t/t_0}{1-t}+\frac{t_1-t/t_1}{1-t}=
\frac{\varepsilon-\varepsilon t}{1-t}+
\frac{\varepsilon t^{1/2}-\varepsilon t^{1/2}}{1-t}=\varepsilon,
\]
equation \eqref{Eq_C76} simplifies to \eqref{Eq_C76-spec}.
\end{proof}

\section{Closed-form evaluations---the Macdonald case}\label{Sec_closed-form}
In this section we consider several closed-form 
evaluations of virtual Koornwinder integrals over Macdonald polynomials. 

\begin{theorem}\label{Thm_Q1}
For $\mu$ a partition,
\begin{multline}\label{Eq_Q1}
I_K\big(P_{\mu}(q,t);q,t,T;\pm t^{1/2},\pm (qt)^{1/2}\big) \\
=\chi(\mu' \text{ even})\,
\frac{(T^2;q,t^2)_{\nu}}{(qT^2/t;q,t^2)_{\nu}}\cdot
\frac{C_{\nu}^{-}(qt;q,t^2)}{C_{\nu}^{-}(t^2;q,t^2)},
\end{multline}
where, for $\mu'$ an even partition,
$\nu:=(\mu'/2)'=(\mu_1,\mu_3,\dots)$.
\end{theorem}
As usual $(\pm t^{1/2},\pm (qt)^{1/2})$ in the above
is shorthand for
\[
\big(t^{1/2},-t^{1/2},(qt)^{1/2},-(qt)^{1/2}\big).
\]
Theorem~\ref{Thm_Q1}, which for $T=t^n$ is known as the 
$\mathrm{U}(2n)/\mathrm{Sp}(2n)$ vanishing integral,
was conjectured in \cite[Conjecture 1]{Rains05} and 
proven in \cite[Theorem 4.1]{RV07}.

\begin{theorem}\label{Thm_Q6}
For $\mu$ a partition,
\begin{multline}\label{Eq_Q6}
I_K\big(P_{\mu}(q,t);q,t,T;-1,-q^{1/2},-t^{1/2},-(qt)^{1/2}\big) \\
=(-1)^{\abs{\mu}}\,
\frac{(T;q^{1/2},t^{1/2})_{\mu}}{(-q^{1/2}T/t^{1/2};q^{1/2},t^{1/2})_{\mu}}
\cdot \frac{C^{-}_{\mu}(-q^{1/2};q^{1/2},t^{1/2})}
{C^{-}_{\mu}(t^{1/2};q^{1/2},t^{1/2})}.
\end{multline}
\end{theorem}
This theorem was first stated (up to a trivial sign-change) 
as the conjectural \cite[Equation (5.79)]{Rains12}.
By \cite[Theorem 8.5]{Rains14}, which implies
\cite[Conjecture Q6]{Rains12}, it now also has been proven.

From a symmetry of the virtual Koornwinder polynomials,
the virtual Koornwinder integral satisfies the duality
\cite[Corollary 7.6]{Rains05}\footnote{In \cite[Corollary 7.6]{Rains05}
$\tilde{\omega}_{q,t}$ should be corrected to 
$\tilde{\omega}_{t,q}$.}
\[
I_K\big(f;q,t,T;\tees\big) 
=I_K\Big(
f\Big[{-}\varepsilon \Big(\frac{t}{q}\Big)^{1/2}\,\frac{1-q}{1-t}\,x\Big];
t,q,1/T;s_0,s_1,s_2,s_3\Big), 
\]
where $s_r=-(qt)^{1/2}/t_r$.
Applying this to \eqref{Eq_Q1}, and then replacing $(q,t,T,\mu)$ by 
$(t,q,1/T,\mu')$ using \eqref{Eq_qtswap}, \eqref{Eq_double_Cmin}, 
\eqref{Eq_omegaqt}, \eqref{Eq_Qdef} and \eqref{Eq_omegaQ}, 
yields the following dual virtual Koornwinder integral.

\begin{corollary}[{\!\!\cite[page 741]{RV07}}]
For $\mu$ a partition,
\begin{equation}\label{Eq_45}
I_K\big(P_{\mu}(q,t);q,t,T;\pm 1,\pm t^{1/2}\big)=
\chi(\mu \text{ even})\,
\frac{(T^2;q^2,t)_{\mu/2}}{(qT^2/t;q^2,t)_{\mu/2}}\cdot
\frac{C^{-}_{\mu/2}(q;q^2,t)}{C^{-}_{\mu/2}(t;q^2,t)}.
\end{equation}
\end{corollary}

We need two variants of this for $I_K^{(n)}$.
For $\la$ a partition of length at most $n$ or a half-partition of length 
$n$, define
\begin{equation}\label{Eq_Adef-a}
A_{\la}^{(n)}(q,t):=\prod_{1\leqslant i<j\leqslant n}
\frac{(qt^{j-i-1},t^{j-i+1};q^2)_{\la_i-\la_j}}
{(qt^{j-i},t^{j-i};q^2)_{\la_i-\la_j}}.
\end{equation}
It is important to note that $A_{\la}^{(n)}(q,t)$ depends on the 
relative differences between the $\la_i$, and that for $\la$ 
a partition
\begin{equation}\label{Eq_Adef}
A_{\la}^{(n)}(q,t)=
\frac{(t^n;q^2,t)_{\la}}{(qt^{n-1};q^2,t)_{\la}}\cdot
\frac{C^{-}_{\la}(q;q^2,t)}{C^{-}_{\la}(t;q^2,t)}.
\end{equation}

\begin{theorem}\label{Thm_BD}
For $\mu$ a partition of length at most $2n$, let
\[
\tilde{\mu}=(\mu_1-\mu_{2n},\dots,\mu_{2n-1}-\mu_{2n},0).
\] 
Then
\begin{align*}
I_K^{(n)}&\big(P_{\mu}(x_1^{\pm},\dots,x_n^{\pm};q,t);
q,t;\pm 1,\pm t^{1/2}\big) \\
&=(-1)^{\mu_{2n}} I_K^{(n-1)}
\big(P_{\mu}(x_1^{\pm},\dots,x_{n-1}^{\pm},1,-1;q,t);q,t;
\pm t,\pm t^{1/2}\big) \\[1mm]
 &\qquad \qquad \qquad =\begin{cases}
A_{\mu/2}^{(2n)}(q,t)
& \text{if $\tilde{\mu}$ is even}, \\[1mm]
0 & \text{otherwise}, 
\end{cases}
\end{align*}
where $\mu/2$ is a partition or half-partition given by
$(\mu_1/2,\dots,\mu_{2n}/2)$.
\end{theorem}

\begin{theorem}\label{Thm_BD-2}
For $\nu$ a partition of length at most $2n+1$, let
\[
\tilde{\nu}=(\nu_1-\nu_{2n+1},\dots,\nu_{2n}-\nu_{2n+1},0).
\] 
Then
\begin{align*}
I_K^{(n)}&\big(P_{\nu}(x_1^{\pm},\dots,x_n^{\pm},1;q,t);
q,t;-1,t,\pm t^{1/2}\big)\\
&=(-1)^{\nu_{2n+1}}
I_K^{(n)}\big(P_{\nu}(x_1^{\pm},\dots,x_n^{\pm},-1;q,t);
q,t;1,-t,\pm t^{1/2}\big)\\[1mm]
&\qquad \qquad \qquad =\begin{cases}
A_{\nu/2}^{(2n+1)}(q,t)
& \text{if $\tilde{\nu}$ is even}, \\[1mm]
0 & \text{otherwise}.
\end{cases}
\end{align*}
\end{theorem}

\begin{proof}[Proof of Theorem~\ref{Thm_BD}]
From \eqref{Eq_add1} we have
\begin{align*}
P_{\mu}(x_1^{\pm},\dots,x_n^{\pm};q,t)&=
P_{\tilde{\mu}}(x_1^{\pm},\dots,x_n^{\pm};q,t) 
\intertext{and}
P_{\mu}(x_1^{\pm},\dots,x_{n-1}^{\pm},1,-1;q,t)&=
(-1)^{\mu_{2n}} P_{\tilde{\mu}}(x_1^{\pm},\dots,x_{n-1}^{\pm},1,-1;q,t),
\end{align*}
so that 
\begin{subequations}\label{Eq_getridoflastpart}
\begin{multline}
I_K^{(n)}\big(P_{\mu}(x_1^{\pm},\dots,x_n^{\pm};q,t);q,t;\tees\big) \\
=I_K^{(n)}\big(P_{\tilde{\mu}}(x_1^{\pm},\dots,x_n^{\pm};q,t);q,t;\tees\big)
\end{multline}
and
\begin{multline}
I_K^{(n-1)}\big(P_{\mu}(x_1^{\pm},\dots,x_{n-1}^{\pm},1,-1;q,t);
q,t;\tees\big) \\ 
=(-1)^{\mu_{2n}} 
I_K^{(n-1)}\big(P_{\tilde{\mu}}(x_1^{\pm},\dots,x_{n-1}^{\pm},1,-1;q,t);
q,t;\tees\big).
\end{multline}
\end{subequations}
Since $A_{\mu/2}^{(2n)}(q,t)=A_{\tilde{\mu}/2}^{(2n)}(q,t)$,
it thus suffices to prove that
\begin{align}\label{Eq_omega}
I_K^{(n)}&\big(P_{\omega}(x_1^{\pm},\dots,x_n^{\pm};q,t);
q,t;\pm 1,\pm t^{1/2}\big) \\
&=I_K^{(n-1)}
\big(P_{\omega}(x_1^{\pm},\dots,x_{n-1}^{\pm},1,-1;q,t);
q,t;\pm 1,\pm t^{1/2}\big) \notag \\
&\qquad =\chi(\omega \text{ even}) A_{\omega/2}^{(2n)}(q,t),
\notag
\end{align}
for $\omega$ a partition such that $l(\omega)<2n$. 

We now apply Lemma~\ref{Lem_wrong-order} with 
$\{\tees\}=\{\pm 1,\pm t^{1/2}\}$ and $f=P_{\mu}(q,t)$ for $\mu$
a partition of length at most $2n$.
Using \eqref{Eq_teespmpm} and \eqref{Eq_plet1pluseps}, this yields
\begin{align*}
\lim_{T\to t^n} & I_K\big(P_{\mu}(q,t),q,t,T;\pm 1,\pm t^{1/2}\big)  \\
&=\tfrac{1}{2} I_K^{(n)}\big(P_{\mu}(x_1^{\pm},\dots,x_n^{\pm};q,t);
q,t;\pm 1,\pm t^{1/2}\big) \\[1mm]
&\quad +\tfrac{1}{2} 
I_K^{(n-1)}\big(P_{\mu}(x_1^{\pm},\dots,x_{n-1}^{\pm},1,-1;q,t);
q,t;\pm t,\pm t^{1/2}\big).
\end{align*}
By \eqref{Eq_45} and \eqref{Eq_Adef} the left-hand side is equal to
$\chi(\mu \text{ even})\, A_{\mu/2}^{(2n)}(q,t)$, resulting in
\begin{align*}
&\tfrac{1}{2} I_K^{(n)}\big(P_{\mu}(x_1^{\pm},\dots,x_n^{\pm};q,t);
q,t;\pm 1,\pm t^{1/2}\big) \\[1mm]
&\; +\tfrac{1}{2} I_K^{(n-1)}
\big(P_{\mu}(x_1^{\pm},\dots,x_{n-1}^{\pm},1,-1;q,t);
q,t;\pm t,\pm t^{1/2}\big) \\[1mm]
&\qquad =\chi(\mu \text{ even})\, A_{\mu/2}^{(2n)}(q,t). 
\end{align*}
Again using \eqref{Eq_getridoflastpart} as well as 
\[
\chi(\mu \text{ even})\, A_{\mu/2}^{(2n)}(q,t)
=\chi(\mu_{2n} \text{ even})
\chi(\tilde{\mu} \text{ even})\, A_{\tilde{\mu}/2}^{(2n)}(q,t),
\]
and then renaming $\tilde{\mu}_i=\mu_i-\mu_{2n}$ as $\omega_i$ 
for $1\leqslant i\leqslant 2n-1$, and $\mu_{2n}$ as $k$, it follows that
\begin{align*}
&\tfrac{1}{2} I_K^{(n)}\big(P_{\omega}(x_1^{\pm},\dots,x_n^{\pm};q,t);
q,t;\pm 1,\pm t^{1/2}\big) \\[1mm]
&\; +\tfrac{1}{2} (-1)^k I_K^{(n-1)}
\big(P_{\omega}(x_1^{\pm},\dots,x_{n-1}^{\pm},1,-1;q,t);
q,t;\pm t,\pm t^{1/2}\big) 
\\[1mm] &\qquad = \chi(k \text{ even})
\chi\big(\omega \text{ even}\big)\, A_{\omega/2}^{(2n)}(q,t),
\end{align*}
for $\omega$ a partition of length at most $2n-1$ and $k$ an arbitrary
integer. For odd $k$ this implies the first equality in \eqref{Eq_omega},
so that the second equality follows from even $k$.
\end{proof}

\begin{proof}[Proof of Theorem~\ref{Thm_BD-2}]
The proof if analogous to that of Theorem~\ref{Thm_BD} except
that we now need Lemma~\ref{Lem_IK} on top of Lemma \ref{Lem_wrong-order}.

Let $z\in\{-1,1\}$. Since
\begin{multline}\label{Eq_getridoflastpart-2}
I_K^{(n)}\big(P_{\nu}(x_1^{\pm},\dots,x_n^{\pm},z;q,t);q,t;\tees\big) \\
=z^{\nu_{2n+1}} 
I_K^{(n)}\big(P_{\tilde{\nu}}(x_1^{\pm},\dots,x_n^{\pm},z;q,t);q,t;\tees\big),
\end{multline}
and $A_{\nu/2}^{(2n+1)}(q,t)=A_{\tilde{\nu}/2}^{(2n+1)}(q,t)$,
it is enough to show that
\begin{equation}\label{Eq_enough2}
I_K^{(n)}\big(P_{\tau}(x_1^{\pm},\dots,x_n^{\pm},z;q,t);
q,t;-z,z t,\pm t^{1/2}\big)
=\chi(\tau\text{ even}) A_{\tau/2}^{(2n+1)}(q,t),
\end{equation}
for $\tau$ a partition such that $l(\tau)\leqslant 2n$. 

For $\{\tees\}=\{1,\varepsilon t,t^{1/2},\varepsilon  t^{1/2}\}$, 
we have $\sum_{r=0}^3 \frac{t_r-t/t_r}{1-t}=1-\varepsilon$.
Hence, by Lemma~\ref{Lem_wrong-order} with
$f(x)=P_{\nu}\big([x+\varepsilon];q,t\big)$ and $\nu$ a partition of
length at most $2n+1$,
\begin{align*}
\lim_{T\to t^n} & I_K
\big(P_{\nu}\big([x+\varepsilon];q,t\big),q,t,T;1,-t,\pm t^{1/2}\big)  \\
&=\tfrac{1}{2} I_K^{(n)}\big(P_{\nu}(x_1^{\pm},\dots,x_n^{\pm},-1;q,t);
q,t;1,-t,\pm t^{1/2}\big) \\[1mm]
&\quad+\tfrac{1}{2} I_K^{(n)}\big(P_{\nu}(x_1^{\pm},\dots,x_n^{\pm},1;q,t);
q,t;-1,t,\pm t^{1/2}\big) \\[1mm]
&=\tfrac{1}{2} \sum_{z\in\{-1,1\}} 
I_K^{(n)}\big(P_{\nu}(x_1^{\pm},\dots,x_n^{\pm},z;q,t);
q,t;-z,zt,\pm t^{1/2}\big).
\end{align*}
On the other hand, by Lemma~\ref{Lem_IK},
\begin{align*}
\lim_{T\to t^n} & I_K
\big(P_{\nu}\big([x+\varepsilon];q,t\big),q,t,T;1,-t,\pm t^{1/2}\big)  \\
&=\lim_{T\to t^n} 
I_K\big(P_{\nu}(q,t),q,t,t^{1/2} T;\pm 1,\pm t^{1/2}\big) \\
&=\chi(\nu \text{ even})\, A_{\nu/2}^{(2n+1)}(q,t),
\end{align*}
where the second equality follows from \eqref{Eq_45} and \eqref{Eq_Adef}.
Therefore,
\begin{multline*}
\tfrac{1}{2} \sum_{z\in\{-1,1\}} 
I_K^{(n)}\big(P_{\nu}(x_1^{\pm},\dots,x_n^{\pm},z;q,t);
q,t;-z,zt,\pm t^{1/2}\big)  \\[-1mm]
=\chi(\nu \text{ even})\, A_{\nu/2}^{(2n+1)}(q,t).
\end{multline*}
If we define $\tau_i:=\nu_i-\nu_{2n+1}$ for $1\leqslant i\leqslant 2n$ and
$k:=\nu_{2n+1}$, and then use \eqref{Eq_getridoflastpart-2},
the above can also be written as
\begin{multline*}
\tfrac{1}{2} \sum_{z\in\{-1,1\}} z^k 
I_K^{(n)}\big(P_{\tau}(x_1^{\pm},\dots,x_n^{\pm},z;q,t);
q,t;-z,zt,\pm t^{1/2}\big) \\[-1mm]
=\chi(k \text{ even}) \chi(\tau \text{ even})\, A_{\tau/2}^{(2n+1)}(q,t).
\end{multline*}
As before, by considering odd values of $k$ this yields
\begin{multline*}
I_K^{(n)}\big(P_{\tau}(x_1^{\pm},\dots,x_n^{\pm},1;q,t);
q,t;-1,t,\pm t^{1/2}\big) \\ =
I_K^{(n)}\big(P_{\tau}(x_1^{\pm},\dots,x_n^{\pm},-1;q,t);
q,t;1,-t,\pm t^{1/2}\big).
\end{multline*}
Choosing $k$ to be even completes the proof of \eqref{Eq_enough2}.
\end{proof}

\medskip

As our final evaluation of this section we claim the following.

\begin{theorem}\label{Thm_VKI}
For $\mu$ a partition,
\begin{equation}\label{Eq_IK-new}
I_K\big(P_{\mu}(q,t);q,t,T;-1,q,\pm t^{1/2}\big)
=(-1)^{\abs{\mu}} \, 
\frac{(T^2;q^2,t)_{\ceil{\mu/2}}}
{(qT^2/t;q^2,t)_{\ceil{\mu/2}}}
\cdot \frac{1}{b_{\mu}^{\textup{ea}}(q,t)}.
\end{equation}
\end{theorem}

\begin{proof}
From \eqref{Eq_IKn-ip}, definition \eqref{Eq_Kdensity} of the Koorn\-winder
den\-sity and Gustaf\-son's integral \eqref{Eq_Gus}, it follows that
\begin{multline*}
I_K^{(n)}\big(f;q,t;q\tees\big) \\
=I_K^{(n)}\Big(f(x_1^{\pm},\dots,x_n^{\pm})
\prod_{i=1}^n (1-t_0 x_i^{\pm});q,t;\tees\Big)
\prod_{i=1}^n \frac{1-t_0t_1t_2t_3t^{n+i-2}}
{\prod_{r=1}^3 (1-t_0t_r t^{i-1})}.
\end{multline*}
For $f=P_{\mu}(q,t)$ we can use the $e$-Pieri rule \eqref{Eq_e-Pieri}
to expand the integrand. Hence
\begin{multline*}
I_K^{(n)}\big(P_{\mu}(q,t);q,t;q\tees\big) 
=\prod_{i=1}^n 
\frac{1-t_0t_1t_2t_3t^{n+i-2}}{\prod_{r=1}^3 (1-t_0t_r t^{i-1})} \\
\times
\sum_{\la\supseteq\mu} (-t_0)^{\abs{\skew{\la}{\mu}}} 
\psi'_{\skew{\la}{\mu}}(q,t) 
I_K^{(n)}\big(P_{\la}(q,t);q,t;\tees\big) .
\end{multline*}
For $(\tees)=(1,-1,t^{1/2},-t^{1/2})$ this yields
\begin{multline*}
I_K^{(n)}\big(P_{\mu}(q,t);q,t;q,-1,\pm t^{1/2}\big)  \\
=\tfrac{1}{2} \sum_{\la\supseteq\mu} (-1)^{\abs{\skew{\la}{\mu}}} 
\psi'_{\skew{\la}{\mu}}(q,t) 
I_K^{(n)}\big(P_{\la}(q,t);q,t;\pm 1,\pm t^{1/2}\big).
\end{multline*}
The integral in the summand evaluates in closed form 
by Theorem~\ref{Thm_BD}. 
In particular it vanishes unless (i) $\la$ is even or (ii)
$\la$ is odd and $l(\la)=2n$. 
Since $\psi'_{\skew{\la}{\mu}}(q,t)$ is zero unless $\skew{\la}{\mu}$
is a vertical strip, this fixes $\la$ as $\la=2\ceil{\mu/2}=:\nu$ in case
(i) and $\la=2\floor{\mu/2}+1^{2n}=:\omega$ in case (ii).
Noting the three congruences
\[
\abs{\nu}\equiv\abs{\omega}\equiv 0 \pmod{2},\quad
\abs{\skew{\nu}{\mu}}=\op(\mu)\equiv \abs{\mu} \pmod{2},
\]
and
\[
\abs{\skew{\omega}{\mu}}=2n-\op(\mu)\equiv \abs{\mu} \pmod{2},
\]
we obtain
\begin{multline*}
I_K^{(n)}\big(P_{\mu}(q,t);q,t;q,-1,\pm t^{1/2}\big)  \\
=\tfrac{1}{2} (-1)^{\abs{\mu}}\Big(
\psi'_{\skew{\nu}{\mu}}(q,t) A_{\nu/2}^{(2n)}(q,t)
+\psi'_{\skew{\omega}{\mu}}(q,t) A_{\omega/2}^{(2n)}(q,t)\Big).
\end{multline*}

We will now show that the two terms on the right are equal, resulting
in
\begin{equation}\label{Eq_gelijk}
I_K^{(n)}\big(P_{\mu}(q,t);q,t;q,-1,\pm t^{1/2}\big)  \\
=(-1)^{\abs{\mu}} \psi'_{\skew{\nu}{\mu}}(q,t) A_{\nu/2}^{(2n)}(q,t).
\end{equation}
First we note that since $\omega/2=\floor{\mu/2}+(\frac{1}{2})^{2n}$
and $A_{\mu}^{(n)}(q,t)$ 
depends on the relative differences of the $\mu_i$, we have
\[
A_{\omega/2}^{(2n)}(q,t)=A_{\floor{\mu/2}}^{(2n)}(q,t).
\]
Moreover, by \eqref{Eq_Adef-a} and $\nu/2=\ceil{\mu/2}$,
\begin{multline*}
A_{\floor{\mu/2}}^{(2n)}(q,t)=
A_{\nu/2}^{(2n)}(q,t)
\prod_{\substack{ 1\leqslant i<j\leqslant 2n \\
\mu_i \text{ odd},~\mu_j \text{ even}}}
\frac{1-q^{\mu_i-\mu_j}t^{j-i}}{1-q^{\mu_i-\mu_j}t^{j-i-1}}\cdot
\frac{1-q^{\mu_i-\mu_j-1}t^{j-i}}{1-q^{\mu_i-\mu_j-1}t^{j-i+1}} \\
\times\prod_{\substack{ 1\leqslant i<j\leqslant 2n \\
\mu_i \text{ even},~\mu_j \text{ odd}}}
\frac{1-q^{\mu_i-\mu_j}t^{j-i-1}}{1-q^{\mu_i-\mu_j}t^{j-i}}\cdot
\frac{1-q^{\mu_i-\mu_j-1}t^{j-i+1}}{1-q^{\mu_i-\mu_j-1}t^{j-i}}.
\end{multline*}
But from \eqref{Eq_psip} it follows that
\[
\psi'_{\skew{\nu}{\mu}}(q,t) = 
\prod_{\substack{ 1\leqslant i<j\leqslant 2n \\
\mu_i \text{ even},~\mu_j \text{ odd}}}
\frac{1-q^{\mu_i-\mu_j}t^{j-i-1}}{1-q^{\mu_i-\mu_j}t^{j-i}}\cdot 
\frac{1-q^{\mu_i-\mu_j-1}t^{j-i+1}}{1-q^{\mu_i-\mu_j-1}t^{j-i}}
\]
and
\[
\psi'_{\skew{\omega}{\mu}}(q,t) = 
\prod_{\substack{ 1\leqslant i<j\leqslant 2n \\
\mu_i \text{ odd},~\mu_j \text{ even}}}
\frac{1-q^{\mu_i-\mu_j}t^{j-i-1}}{1-q^{\mu_i-\mu_j}t^{j-i}}\cdot 
\frac{1-q^{\mu_i-\mu_j-1}t^{j-i+1}}{1-q^{\mu_i-\mu_j-1}t^{j-i}},
\]
so that
\[
\psi'_{\skew{\omega}{\mu}}(q,t)
A_{\omega/2}^{(2n)}(q,t)=
\psi'_{\skew{\nu}{\mu}}(q,t) A_{\nu/2}^{(2n)}(q,t),
\]
establishing \eqref{Eq_gelijk}.

Since $\nu=2\ceil{\mu/2}$ is even, we can use \eqref{Eq_Adef} 
to write the right side of \eqref{Eq_gelijk} as
\[
(-1)^{\abs{\mu}} \psi'_{\skew{\nu}{\mu}}(q,t) \,
\frac{(t^{2n};q^2,t)_{\nu/2}}{(qt^{2n-1};q^2,t)_{\nu/2}}\cdot
\frac{C^{-}_{\nu/2}(q;q^2,t)}{C^{-}_{\nu/2}(t;q^2,t)}.
\]
By Lemma~\ref{Lem_psip} this is also
\[
(-1)^{\abs{\mu}} 
\frac{(t^{2n};q^2,t)_{\nu/2}}{(qt^{2n-1};q^2,t)_{\nu/2}}
\cdot
\frac{1}{b_{\mu}^{\text{ea}}(q,t)}.
\]
Hence
\begin{equation}\label{Eq_IKn-new}
I_K^{(n)}\big(P_{\mu}(q,t);q,t;-1,q,\pm t^{1/2}\big)
=(-1)^{\abs{\mu}} \, \frac{(t^{2n};q^2,t)_{\ceil{\mu/2}}}
{(qt^{2n-1};q^2,t)_{\ceil{\mu/2}}}
\cdot \frac{1}{b_{\mu}^{\textup{ea}}(q,t)}.
\end{equation}
Since both sides vanish if $l(\mu)>2n$ this holds for all partitions $\mu$. 

For fixed $\mu$ 
\[
I_K\big(P_{\mu}(q,t);q,t,T;\tees\big)
\]
is a rational function in $T$. By \eqref{Eq_functional-IK-IKn}
and \eqref{Eq_IKn-new}, equation \eqref{Eq_IK-new} holds
for $T=t^n$ for all nonnegative integers $n$. Hence it holds 
for arbitrary $T$. 
\end{proof}

\section{Closed-form evaluations---the Hall--Littlewood case}
We present one final virtual Koornwinder integral with 
Hall--Littlewood polynomial argument.
It evaluates in terms of the generalised Rogers--Szeg\H{o} polynomials 
\eqref{Eq_RSBCn}, and does not appear to have a simple $t$-analogue
for Macdonald polynomials.

Let
\[
I_K^{(n)}(P_{\mu}(q,0);q,0;\tees)
:=\lim_{t\to 0} I_K^{(n)}(P_{\mu}(q,t);q,t;\tees). 
\]

\begin{theorem}\label{Thm_newvirtual}
For $\mu$ a partition of length at most $2n$,
\begin{equation}\label{Eq_vanish}
I_K^{(n)}(P_{\mu}(q,0);q,0;0,0,t_2,t_3) \\
=h_{\mu'}^{(2n)}(-t_2,-t_3;q). 
\end{equation}
\end{theorem}

\begin{proof}
Let $x=(x_1,\dots,x_n)$.
By \eqref{Eq_functional-IKn}, equation \eqref{Eq_vanish} may also be stated 
as the rational function identity
\[
\big[K_0(x;q,0;0,0,t_2,t_3)]P_{\mu}(x^{\pm},q,0)=h_{\mu'}^{(2n)}(-t_2,-t_3;q).
\]
Without loss of generality we may thus assume that 
$\abs{t_2},\abs{t_3}<1$ in the following.

Noting that 
\[
h_{\la}^{(2n)}(0,0;q)=\chi(\la\text{ even}),
\]
the $t_2=t_3=0$ case of \eqref{Eq_vanish}, viz.\
\begin{equation}\label{nul}
I^{(n)}_K(P_{\mu}(q,0);q,0;0,0,0,0)=\chi(\mu' \text{ even}),
\end{equation}
follows from \eqref{Eq_Q1} (with $T=t^n)$ in the $t\to 0$ limit.

To include the parameter $t_2$ we use that
\[
\ip{1}{1}_{q,0;0,0,t_2,0}^{(n)}=\ip{1}{1}_{q,0;0,0,0,0}^{(n)}
\]
(see \eqref{Eq_Gus}) and
\[
\Delta(x;q,0;0,0,t_2,0)=
\Delta(x;q,0;0,0,0,0)\prod_{i=1}^n \frac{1}{(t_2 x_i^{\pm};q)_{\infty}}.
\]
From \eqref{Eq_IKn-ip} it thus follows that
\begin{align}\label{Eq_faqdef}
f_{\mu}(t_2;q)&:=I^{(n)}_K\big(P_{\mu}(q,0);q,0;0,0,t_2,0\big) \\[1mm]
&\hphantom{:}=
I^{(n)}_K\bigg(P_{\mu}(x_1^{\pm},\dots,x_n^{\pm};q,0)
\prod_{i=1}^{2n} \frac{1}{(t_2x_i^{\pm};q)_{\infty}};
q,0;0,0,0,0\bigg).  \notag
\end{align}
By the $g$-Pieri rule \eqref{Eq_g-Pieri} for $t=0$, this yields
\begin{align*}
f_{\mu}(t_2;q)&=
\sum_{\nu\succ\mu} t_2^{\abs{\skew{\nu}{\mu}}}\varphi_{\nu/\mu}(q,0) 
I^{(n)}_K\big(P_{\nu}(q,0);q,0;0,0,0,0\big) \\
&=\sum_{\substack{\nu\succ\mu \\ \nu' \text{ even} \\[1pt] l(\nu)\leqslant 2n}}
t_2^{\abs{\skew{\nu}{\mu}}}\varphi_{\nu/\mu}(q,0),
\end{align*}
where the second equality follows from \eqref{nul}.
Since $\varphi_{\nu/\mu}(q,0)$ is zero unless $\skew{\nu}{\mu}$ is a 
horizontal strip and since $\nu'$ must be even, this fixes $\nu$ as 
$\nu_{2i-1}=\nu_{2i}=\mu_{2i-1}$ for $1\leqslant i\leqslant n$.
This is equivalent to $\nu'_i=\mu'_i+\chi(\mu'_i \text{ odd})$,
so that $\abs{\nu/\mu}$ is given by the number of
odd parts of $\mu'$, i.e., by $\op(\mu')$.
Hence
\[
f_{\mu}(t_2;q)=t_2^{\op(\mu')} \varphi_{\nu/\mu}(q,0),
\]
with $\nu$ fixed as above.
From the expression for $\varphi_{\la/\mu}(q,t)$ as given in
\eqref{Eq_varphi} it follows that\footnote{Alternatively, 
this follows from the Pieri 
coefficient $\varphi'_{\la/\mu}(t)$ for Hall--Littlewood polynomials, 
thanks to $\varphi_{\nu/\mu}(q,0)=\varphi'_{\nu'/\mu'}(0,q)=
\varphi'_{\nu'/\mu'}(q)$.}
\begin{align}\label{phi}
\varphi_{\nu/\mu}(q,0)=
\prod_{i\geqslant 1} \frac{(q;q)_{\mu_i-\mu_{i+1}}}
{(q;q)_{\nu_i-\mu_i}(q;q)_{\mu_i-\nu_{i+1}}} 
=\frac{1}{(q;q)_{\nu_1-\mu_1}}
\prod_{i\geqslant 1} \qbin{\mu_i-\mu_{i+1}}{\mu_i-\nu_{i+1}}_q.
\end{align}
When $\nu_{2i-1}=\nu_{2i}=\mu_{2i-1}$ this 
simplifies to $\varphi_{\nu/\mu}(q,0)=1$, 
since either $\nu_{i+1}=\mu_i$ or $\nu_{i+1}=\mu_{i+1}$.
Hence
\begin{equation}\label{faq}
f_{\mu}(t_2;q)=t_2^{\op(\mu')}.
\end{equation}

To also include the parameter $t_3$ we proceed in almost identical fashion.
By
\[
\ip{1}{1}_{q,0;0,0,t_2,t_3}^{(n)}=\frac{1}{(t_2t_3;q)_{\infty}}\, 
\ip{1}{1}_{q,0;0,0,t_2,0}^{(n)}
\]
and
\[
\Delta(x;0,0,t_2,t_3;q,0)=
\Delta(x;0,0,t_2,0;q,0)\prod_{i=1}^n \frac{1}{(t_2x_i^{\pm};q)_{\infty}},
\]
and following the previous steps, we obtain
\begin{align*}
f_{\mu}(t_2,t_3;q)&:=I^{(n)}_K\big(P_{\mu}(q,0);q,0;0,0,t_2,t_3\big) \\[1mm]
&\hphantom{:}=(t_2t_3;q)_{\infty}
\sum_{\nu\succ\mu} t_2^{\abs{\skew{\nu}{\mu}}} f_{\nu}(t_2;q) 
\varphi_{\nu/\mu}(q,0) \\
&\hphantom{:}=(t_2t_3;q)_{\infty}
\sum_{\substack{\nu\succ\mu \\[1pt] l(\nu)\leqslant 2n}}
t_2^{\op(\nu')} t_3^{\abs{\skew{\nu}{\mu}}}\varphi_{\nu/\mu}(q,0).
\end{align*}
Here the second line uses the definition of $f_{\nu}(t_2;q)$ as given in
\eqref{Eq_faqdef}, and the third line uses the evaluation \eqref{faq}. 
To complete the proof we write
$\nu_i=\mu_i+k_i$ for $1\leqslant i\leqslant 2n$, and note that
(see \cite[page 822]{W06})
\begin{equation}\label{p822}
\op(\nu')=\op(\mu')+\sum_{i=1}^{2n} (-1)^{i+1} k_i.
\end{equation}
Once again using \eqref{phi}, we get 
\[
f_{\mu}(t_2,t_3;q)= (t_2t_3;q)_{\infty} \, t_2^{\op(\mu')}
\sum_{k_1,\dots,k_{2n}\geqslant 0}
\frac{1}{(q;q)_{k_1}}
\prod_{i\geqslant 1} t_2^{(-1)^{i+1} k_i} t_3^{k_i} 
\qbin{\mu_i-\mu_{i+1}}{k_{i+1}}_q.
\]
Summing over $k_1$ by \cite[Equation (II.1)]{GR04}
\[
\sum_{k\geqslant 0} \frac{z^k}{(q;q)_k}=\frac{1}{(z;q)_{\infty}}
\qquad \text{for $\abs{z}<1$},
\]
and recalling definition \eqref{RS}, we finally obtain
\begin{align*}
f_{\mu}(t_2,t_3;q)&=t_2^{\op(\mu')}
\prod_{\substack{i=1 \\[0.5pt] i \text{ odd}}}^{2n-1} 
H_{m_i(\mu')}(t_3/t_2;q)
\prod_{\substack{i=1 \\[0.5pt] i \text{ even}}}^{2n-1} 
H_{m_i(\mu')}(t_2t_3;q) \\
&=h_{\mu'}^{(2n)}(-t_2,-t_3;q). \qedhere
\end{align*}
\end{proof}

\chapter{Bounded Littlewood identities}\label{Ch_Bounded}
In this section, which is at the heart of the paper, we use
Macdonald--Koornwin\-der theory and virtual Koornwinder
integrals in particular to prove bounded Littlewood identities 
for Macdonald and Hall--Littlewood polynomials.

\section{Statement of results}
\subsection{$q,t$-Identities}\label{Subsec_qt}
There are five known Littlewood identities for Macdonald polynomials.
By introducing an additional parameter $a$, the first four of these may 
easily be combined to form the pair of identities \cite[Proposition 1.3]{W06}
\begin{equation}\label{Eq_Pb}
\sum_{\la} a^{\op(\la)} b_{\la}^{\textup{oa}}(q,t) P_{\la}(x;q,t)=
\prod_{i=1}^n \frac{(1+ax_i)(qtx_i^2;q^2)_{\infty}}
{(x_i^2;q^2)_{\infty}}
\prod_{1\leqslant i<j\leqslant n}\frac{(tx_ix_j;q)_{\infty}}{(x_ix_j;q)_{\infty}}
\end{equation}
and
\begin{equation}\label{Eq_Pc}
\sum_{\la} a^{\op(\la')}
b_{\la}^{\textup{el}}(q,t) P_{\la}(x;q,t)=
\prod_{i=1}^n\frac{(a t x_i;q)_{\infty}}{(a x_i;q)_{\infty}}
\prod_{1\leqslant i<j\leqslant n}\frac{(tx_ix_j;q)_{\infty}}{(x_ix_j;q)_{\infty}}.
\end{equation}
Here
\[
b_{\la}^{\textup{oa}}(q,t):=
\prod_{\substack{s\in\la \\[1pt] a(s) \text{ odd}}} 
b_{\la}(s;q,t)
\quad\text{and}\quad
b_{\la}^{\textup{el}}(q,t):=
\prod_{\substack{s\in\la \\[1pt] l(s) \text{ even}}} 
b_{\la}(s;q,t),
\]
to be compared with \eqref{Eq_b-def} and \eqref{Eq_evenarm}.
The cases $a=0$ and $a=1$ of \eqref{Eq_Pb} and \eqref{Eq_Pc} correspond 
to Macdonald's original four results, see \cite[page 349]{Macdonald95}.
The fifth identity was first conjectured by Kawanaka \cite{Kawanaka99} 
and subsequently proven in \cite{LSW09} (see also \cite{Rains12}):
\begin{equation}\label{Eq_Kawanaka}
\sum_{\la} b_{\la}^{-}(q,t) P_{\la}(x;q^2,t^2)
=\prod_{i=1}^n \frac{(-tx_i;q)_{\infty}}{(x_i;q)_{\infty}}
\prod_{1\leqslant i<j\leqslant n} \frac{(t^2x_ix_j;q^2)_{\infty}}
{(x_ix_j;q^2)_{\infty}},
\end{equation}
where
\[ 
b_{\la}^{-}(q,t):=
\prod_{s\in\la} \frac{1+q^{a(s)}t^{l(s)+1}}{1-q^{a(s)+1}t^{l(s)}}.
\]
In the following we generalise all of \eqref{Eq_Pb}--\eqref{Eq_Kawanaka}.

\medskip

For $m$ a nonnegative integer and $\la$ a partition, let
\[
b_{\la;m}^{\textup{oa}}(q,t):=
b_{\la}^{\textup{oa}}(q,t)
\prod_{\substack{s\in\la \\[1pt] a'(s) \textup{ odd}}}
\frac{1-q^{2m-a'(s)+1}t^{l'(s)}}{1-q^{2m-a'(s)}t^{l'(s)+1}}.
\]
Note that $b_{\la;m}^{\textup{oa}}(q,t)=0$ if $\la_1>2m+1$ and that
for $\la$ an even partition
\begin{equation}\label{Eq_iflaiseven}
b_{\la;m}^{\textup{oa}}(q,t)=
b_{\la}^{\textup{oa}}(q,t)
\prod_{\substack{s\in\la \\[1pt] a'(s) \textup{ even}}}
\frac{1-q^{2m-a'(s)}t^{l'(s)}}{1-q^{2m-a'(s)-1}t^{l'(s)+1}}.
\end{equation}
Our first bounded Littlewood identity generalises \eqref{Eq_Pb}. 

\begin{theorem}\label{THM_BOUNDED1}
For $x=(x_1,\dots,x_n)$ and $m$ a nonnegative integer,
\[
\sum_{\la} a^{\op(\la)} b_{\la;m}^{\textup{oa}}(q,t) P_{\la}(x;q,t)
=\bigg(\prod_{i=1}^n x_i^m (1+ax_i)\bigg)
P^{(\mathrm{C}_n,\mathrm{B}_n)}_{m^n}(x;q,t,qt).
\]
\end{theorem}

Using \eqref{Eq_PCB-K} to identify
$P^{(\mathrm{C}_n,\mathrm{B}_n)}_{m^n}(x;q,t,qt)$
as a Koornwinder polynomial, and then using \eqref{Eq_largem}
for $\la=0$, it follows that the large-$m$ limit of the right-hand side
simplifies to the right-hand side of \eqref{Eq_Pb}.

When $a=0$ the summand on the left vanishes unless $\la$ is even,
so that\footnote{By \eqref{Eq_add1}, 
the same result may be obtained in the $a\to\infty$ limit.}
\begin{equation}\label{Eq_bounded1a0}
\sum_{\la \text{ even}}
b_{\la;m}^{\textup{oa}}(q,t) P_{\la}(x;q,t)
=(x_1\cdots x_n)^m P^{(\mathrm{C}_n,\mathrm{B}_n)}_{m^n}(x;q,t,qt).
\end{equation}
This is a $q,t$-analogue of the D\'esarm\'enien--Proctor--Stembridge 
determinant formula \cite{Desarmenien86,Proctor90,Stembridge90}
\begin{equation}\label{Eq_DPS}
\sum_{\substack{\la \text{ even}\\[1pt] \la_1\leqslant 2m}} s_{\la}(x)
=\frac{\det_{1\leqslant i,j\leqslant n} \big(x_i^{j-1}-x_i^{2m+2n-j+1}\big)}
{\prod_{i=1}^n (1-x_i^2)\prod_{1\leqslant i<j\leqslant n}(x_i-x_j)(x_ix_j-1)},
\end{equation}
which expresses the symplectic Schur function 
$\symp_{2n,m^n}(x)$ (times $(x_1\cdots x_n)^m$)
in terms of Schur functions. Equivalently, \eqref{Eq_DPS} is a 
branching formula for the character of the symplectic group
$\mathrm{Sp}(n,\Complex)$ indexed by $m\fwc_n$ in terms of 
characters of the general linear group $\mathrm{GL}(n,\Complex)$.
As will be discussed in Section~\ref{Sec_Plane-partitions},
like Macdonald's formula \eqref{Eq_MacB}, the determinant 
\eqref{Eq_DPS} is important in the theory of plane partitions.

Another notable special case follows when $q=0$.
For $s\in\la\subseteq (2m)^n$ such that $a'(s)$ is even we must have
$2m-a'(s)\geqslant 2$, which implies that
$b_{\la;m}^{\textup{oa}}(0,t)=1$.
By \eqref{Eq_CHL} 
the $q=0$ specialisation of \eqref{Eq_bounded1a0} is thus 
\begin{equation}\label{Eq_Stem}
\sum_{\substack{\la \text{ even} \\[1pt] \la_1\leqslant 2m}} P_{\la}(x;t)
=(x_1\cdots x_n)^m P^{(\mathrm{C}_n)}_{m^n}(x;t,0).
\end{equation}
For positive $m$ the right-hand side can be expressed in
terms of the function $\Phi(x;t,0,0)$ by
Lemma~\ref{Lem_littlelemma}. 
The resulting $t$-analogue of the 
D\'esarm\'enien--Proctor--Stembridge determinant
is due to Stembridge \cite[Theorem 1.2]{Stembridge90} who
used it to give new proofs of the Rogers--Ramanujan identities.
We will see in Section~\ref{Sec_RR} that Stembridge's method can be
extended so that identities such as \eqref{Eq_Stem}
yield Rogers--Ramanujan identities for certain
affine Lie algebras $X_N^{(r)}$ for arbitrary $N$.

\medskip

For $m$ a nonnegative integer and $\la$ a partition, let
\[
b_{\la;m}^{\textup{ol}}(q,t):=
\prod_{\substack{s\in\la \\[1pt] l'(s) \textup{ odd}}} 
\frac{1-q^{m-a'(s)}t^{l'(s)-1}}{1-q^{m-a'(s)-1}t^{l'(s)}}
\prod_{\substack{s\in\la \\[1pt] l(s) \textup{ odd}}}
\frac{1-q^{a(s)}t^{l(s)}}{1-q^{a(s)+1}t^{l(s)-1}}.
\]
Note that $b_{\la;m}^{\textup{ol}}(q,t)=0$ if $\la_2>m$,
which implies vanishing for $\la_1>m$ when $\la'$ is even.
Our next theorem contains the first of two bounded analogues 
of the $a=0$ case of \eqref{Eq_Pc}.

\begin{theorem}\label{Thm_bounded2}
For $x=(x_1,\dots,x_n)$ and $m$ a nonnegative integer,
\begin{equation}\label{Eq_bounded2}
\sum b_{\la;m}^{\textup{ol}}(q,t) P_{\la}(x;q,t) 
=(x_1\cdots x_n)^{\frac{m}{2}}
P^{(\mathrm{B}_n,\mathrm{B}_n)}_{\halfm{n}}(x;q,t,1), 
\end{equation}
where the sum is over partitions $\la\subseteq m^n$ such that
$m_i(\la)$ is even for all $1\leqslant i\leqslant m-1$.
\end{theorem}

To see that this generalises \eqref{Eq_Pc} for $a=0$, we first note
that in the large-$m$ limit the right-hand side simplifies to the
right-hand side of \eqref{Eq_Pc} for $a=0$ by \eqref{Eq_PBB-K} 
and the $\la=0$ case of Lemma~\ref{Lem_largem-2}.
Next, to simplify the left-hand side 
we note that there are two types of partitions contributing to the sum.
\begin{description}
\item[Type 1] Partitions $\la$ such that $m_i(\la)$ is even for all 
$1\leqslant i\leqslant m$, i.e., $\la'$ is even.
\item[Type 2] Partitions $\la$ such that $m_i(\la)$ 
is odd for $i=m$ and even for $1\leqslant i<m$, 
i.e., $\la'$ is odd and $\la_1=m$.
\end{description}
Macdonald polynomials indexed by partitions of Type 2 have degree at 
least $m$, so that their contribution vanishes in the large-$m$ limit.
Hence we are left with a sum over partitions of Type 1, for which
\[
\prod_{\substack{s\in\la \\[1pt] l(s) \textup{ odd}}}
\frac{1-q^{a(s)}t^{l(s)}}{1-q^{a(s)+1}t^{l(s)-1}}=
\prod_{\substack{s\in\la \\[1pt] l(s) \textup{ even}}}
\frac{1-q^{a(s)}t^{l(s)+1}}{1-q^{a(s)+1}t^{l(s)}}
=b_{\la}^{\textup{el}}(q,t),
\]
resulting in the $a=0$ case of \eqref{Eq_Pc}.
In fact, \eqref{Eq_PBBDD} can be used to dissect \eqref{Eq_bounded2},
resulting in two bounded Littlewood identities for $\mathrm{D}_n$,
the first of which is our second bounded analogue of \eqref{Eq_Pc}
for $a=0$.

\begin{theorem}\label{Thm_bounded3}
For $x=(x_1,\dots,x_n)$,
$\bar{x}=(x_1,\dots,x_{n-1},x_n^{-1})$ and
$m$ a nonnegative integer,
\begin{subequations}
\begin{align}\label{Eq_bounded3-1}
\sum_{\la' \textup{ even}} b_{\la;m}^{\textup{ol}}(q,t) P_{\la}(x;q,t) 
&=(x_1\cdots x_n)^{\frac{m}{2}}
P^{(\mathrm{D}_n,\mathrm{D}_n)}_{\halfm{n}}(x;q,t) \\[1mm]
\sum_{\substack{\la' \textup{ odd }  \\[1pt] \la_1=m}} 
b_{\la;m}^{\textup{ol}}(q,t) P_{\la}(x;q,t) 
&=(x_1\cdots x_n)^{\frac{m}{2}}
P^{(\mathrm{D}_n,\mathrm{D}_n)}_{\halfm{n}}(\bar{x};q,t).
\end{align}
\end{subequations}
\end{theorem}

Taking $q=0$ in \eqref{Eq_bounded3-1} yields
\begin{equation}\label{Eq_JZ}
\sum_{\substack{\la' \textup{ even} \\[1pt] \la_1\leqslant m}} 
P_{\la}(x;t) \prod_{i=1}^{m-1} (t;t^2)_{m_i(\la)/2} 
=(x_1\cdots x_n)^{\frac{m}{2}} P^{(\mathrm{D}_n)}_{\halfm{n}}(x;t).
\end{equation}
By Lemma~\ref{Lem_littlelemma2} this is equivalent to 
\cite[Theorem 1; Eq.~(7)]{JZ05} of Jouhet and Zeng,
which itself is a $t$-analogue of 
Okada's determinant \cite[Theorem 2.3 (3)]{Okada98}
\[
\sum_{\substack{\la' \textup{ even} \\[1pt] \la_1\leqslant m}} 
s_{\la}(x) =\frac{\sum_{\varepsilon\in\{\pm 1\}}
\det_{1\leqslant i,j\leqslant n}\big(x_i^{j-1}+\varepsilon \, x_i^{m+2n-j-1}\big)}
{2\prod_{i<j} (x_i-x_j)(x_ix_j-1)}.
\]

\medskip

For $m$ a nonnegative integer and $\la$ a partition, let
\begin{equation}\label{Eq_blamel}
b_{\la;m}^{\textup{el}}(q,t):=
b_{\la}^{\textup{el}}(q,t)
\prod_{\substack{s\in\la \\[1pt] l'(s) \textup{ even}}} 
\frac{1-q^{m-a'(s)}t^{l'(s)}}{1-q^{m-a'(s)-1}t^{l'(s)+1}}.
\end{equation}
We note that $b_{\la;m}^{\textup{el}}(q,t)$ vanishes unless $\la_1\leqslant m$.
The next result is \eqref{Eq_qt_Littlewood} from the introduction, 
which bounds \eqref{Eq_Pc} for $a=1$. 
\begin{theorem}\label{Thm_bounded4}
For $x=(x_1,\dots,x_n)$ and $m$ a nonnegative integer,
\begin{equation}\label{Eq_bounded4}
\sum_{\la} b_{\la;m}^{\textup{el}}(q,t) P_{\la}(x;q,t)
=(x_1\cdots x_n)^{\frac{m}{2}}\, 
P^{(\mathrm{B}_n,\mathrm{B}_n)}_{\halfm{n}}(x;q,t,t).
\end{equation}
\end{theorem}
The $q=0$ and $t=q$ specialisations of \eqref{Eq_bounded4} correspond 
to \eqref{Eq_bounded7} below for $t_2=t$, i.e.,
\[
\sum_{\substack{\la \\[1pt] \la_1\leqslant m}} 
\prod_{i=1}^{m-1} \Big( (t;t^2)_{\ceil{m_i(\la)/2}}\Big) P_{\la}(x;t) 
=(x_1\cdots x_n)^{\frac{m}{2}}\, 
P^{(\mathrm{B}_n)}_{\halfm{n}}(x;t,t),
\]
and Macdonald's determinant \eqref{Eq_MacB} respectively.

\medskip

\begin{remark}
Using \eqref{Eq_bounded3-1} it is not hard to prove an
identity that generalises \eqref{Eq_Pc} in full:
\[
\sum_{\la} a^{\op(\la')} \hat{b}_{\la;m}^{\textup{el}}(q,t) 
P_{\la}(x;q,t) = \bigg( \prod_{i=1}^n x_i^{m/2} 
\frac{(a t x_i;q)_{\infty}}{(a x_i;q)_{\infty}}
\bigg) P^{(\mathrm{D}_n,\mathrm{D}_n)}_{\halfm{n}}(x;q,t),
\]
where 
\[
\hat{b}_{\la;m}^{\textup{el}}(q,t):=
b_{\la}^{\textup{el}}(q,t)
\prod_{\substack{s\in\la \\[1pt] l'(s) \textup{ odd}}} 
\frac{1-q^{m-a'(s)}t^{l'(s)-1}}{1-q^{m-a'(s)-1}t^{l'(s)}}.
\]
The largest part of $\la$ in the sum on the left is not bounded, and
unlike \eqref{Eq_bounded2}, \eqref{Eq_bounded3-1} or \eqref{Eq_bounded4}, 
this is not a polynomial identity.
\end{remark}

\medskip

For $m$ a nonnegative integer and $\la$ a partition such that $\la_1\leqslant m$, 
let
\[
b_{\la;m}^{-}(q,t):=b_{\la}^{-}(q,t)\prod_{s\in\la} 
\frac{1-q^{m-a'(s)}t^{l'(s)}}{1+q^{m-a'(s)-1}t^{l'(s)+1}}.
\]
Our final result for Macdonald polynomials is a
bounded analogue of Kawanaka's conjecture \eqref{Eq_Kawanaka}.

\begin{theorem}\label{THM_BOUNDED5}
For $x=(x_1,\dots,x_n)$ and $m$ a nonnegative integer,
\begin{equation}\label{Eq_bounded5}
\sum_{\la} b_{\la;m}^{-}(q,t) P_{\la}(x;q^2,t^2) 
=(x_1\cdots x_n)^{\frac{m}{2}}\,
P^{(\mathrm{B}_n,\mathrm{C}_n)}_{\halfm{n}}(x;q^2,t^2,-t).
\end{equation}
\end{theorem}
For $t=-q$ this simplifies to \eqref{Eq_MacB} and for
$q=0$ it is \eqref{Eq_bounded7} below with $(t,t_2)\mapsto (t^2,-t)$, viz.\
\begin{equation}\label{Eq_IJZ}
\sum_{\substack{\la \\[1pt] \la_1\leqslant m}}
\prod_{i=1}^{m-1} \Big( (-t;t)_{m_i(\la)}\Big) P_{\la}(x;t^2) 
=(x_1\cdots x_n)^{\frac{m}{2}}\, P^{(\mathrm{B}_n)}_{\halfm{n}}(x;t^2,-t).
\end{equation}
Assuming $m$ is positive and rewriting the right-hand side 
using Lemma~\ref{Lem_littlelemmaB} yields \cite[Theorem 1]{IJZ06}
of Ishikawa et al.

\subsection{$t$-Identities}
Our final two theorems do not appear to have simple analogues for 
Macdonald polynomials.

Recall the generalised Rogers--Szeg\H{o} polynomials \eqref{Eq_RSBCn}.

\begin{theorem}\label{Thm_bounded6}
For $x=(x_1,\dots,x_n)$ and $m$ a nonnegative integer,
\begin{equation}\label{Eq_bounded6}
\sum_{\substack{\la \\[1pt] \la_1\leqslant 2m}} 
h_{\la}^{(2m)}(t_2,t_3;t) P_{\la}(x;t) 
=(x_1\cdots x_n)^m P^{(\mathrm{BC}_n)}_{m^n}(x;t,t_2,t_3).
\end{equation}
\end{theorem}

This bounds \cite[Theorem 1.1]{W06}
\begin{equation}\label{Eq_hab-lim}
\sum_{\la} h_{\la}(t_2,t_3;t) P_{\la}(x;t)
=\prod_{i=1}^n \frac{(1-t_2x_i)(1-t_3x_i)}{1-x_i^2}
\prod_{1\leqslant i<j\leqslant n} \frac{1-tx_ix_j}{1-x_ix_j},
\end{equation}
where $h_{\la}(t_2,t_3;t)$ is the Rogers--Szeg\H{o} polynomial
\eqref{Eq_RSBCn-mlim}.
Moreover, if we replace $(t,t_2,t_3)\mapsto (0,-a,-b)$ and use
\eqref{Eq_detBC} and $H_m(z;0)=1+z+\cdots+z^m$, we obtain
the following two-parameter generalisation 
of the D\'esarm\'enien--Proctor--Stembridge determinant \eqref{Eq_DPS}:
\begin{multline*}
\sum_{\substack{\la \\[1pt] \la_1\leqslant 2m}} s_{\la}(x)
\prod_{\substack{i=1 \\[0.5pt] i \text{ odd}}}^{2m-1} 
\frac{a^{m_i(\la)+1}-b^{m_i(\la)+1}}{a-b}
\prod_{\substack{i=1 \\[0.5pt] i \text{ even}}}^{2m-1} 
\frac{1-(ab)^{m_i(\la)+1}}{1-ab} \\ 
=\frac{\det_{1\leqslant i,j\leqslant n} 
\big(x_i^{j-1}(1+ax_i)(1+bx_i)-x_i^{2m+2n-j-1}(x_i+a)(x_i+b)\big)}
{\prod_{i=1}^n (1-x_i^2)\prod_{1\leqslant i<j\leqslant n}(x_i-x_j)(x_ix_j-1)}.
\end{multline*}

\medskip

Recall \eqref{Eq_RSBn}. 
The $t_3=-1$ case of Theorem~\ref{Thm_bounded6} extends as follows.
\begin{theorem}\label{THM_BOUNDED7}
For $x=(x_1,\dots,x_n)$ and $m$ a nonnegative integer,
\begin{equation}\label{Eq_bounded7}
\sum_{\substack{\la \\[1pt] \la_1\leqslant m}} h_{\la}^{(m)}(t_2;t) 
P_{\la}(x;t) =(x_1\cdots x_n)^{\frac{m}{2}}
P_{\halfm{n}}^{(\mathrm{B}_n)}(x;t,t_2).
\end{equation}
\end{theorem}
This is stated without proof in \cite{W06}.
For $(t,t_2)\mapsto (0,-a)$ it simplifies to 
a one-parameter generalisation of Macdonald's
determinant \eqref{Eq_MacB} from the introduction:
\[
\sum_{\substack{\la \\[1pt] \la_1\leqslant m}} 
s_{\la}(x) \prod_{i=1}^{m-1} \frac{1-a^{m_i(\la)+1}}{1-a} \\
=\frac{\det_{1\leqslant i,j\leqslant n} 
\big(x_i^{j-1}(1+ax_i)-x_i^{m+2n-j-1}(x_i+a)\big)}
{\prod_{i=1}^n(1-x_i)\prod_{1\leqslant i<j\leqslant n}(x_i-x_j)(x_ix_j-1)}.
\]

\section{Proofs of Theorems~\ref{THM_BOUNDED1}--\ref{THM_BOUNDED7}}
\label{Sec_Method}
We begin by outlining the general strategy, which is to transform the
problem of proving bounded Littlewood identities into that of 
evaluating virtual Koornwinder integrals.

Recall that if $g\in\Lambda_n$ and $\{f_{\la}\}$ is a basis of $\Lambda_n$,
then $[f_{\la}]g$ is the coefficient of $f_{\la}$ in the expansion of $g$.
Working in full generality, we would like to find a closed-form expression
for
\begin{equation}\label{Eq_PKf}
f_{\la}^{(m)}(q,t;\tees):=
\big[ P_{\la}(x;q,t) \big] (x_1\cdots x_n)^m 
K_{m^n}(x;q,t;\tees),
\end{equation}
where $m$ is a nonnegative integer.
Since 
\[
(x_1\cdots x_n)^m K_{m^n}(x;q,t;\tees)=\sum_{\la\subseteq (2m)^n} 
u_{\la}\, m_{\la}(x),
\]
it follows that $f_{\la}^{(m)}(q,t;\tees)$ vanishes
unless $\la\subseteq (2m)^n$. 

\begin{proposition}\label{Prop_fIK}
For $m$ a nonnegative integer and $\la\subseteq (2m)^n$,
\begin{equation}\label{Eq_rhs}
f_{\la}^{(m)}(q,t;\tees)
=(-1)^{\abs{\la}} I_K^{(m)}\big(P_{\la'}(t,q);t,q;\tees\big).
\end{equation}
\end{proposition}

\begin{proof}
Let $x=(x_1,\dots,x_n)$ and $y=(y_1,\dots,y_m)$.
According to the Cauchy identity for Koornwinder polynomials \eqref{Eq_Mim} 
\begin{multline}\label{Eq_Mim2}
\sum_{\la\subseteq m^n} (-1)^{\abs{\la}} 
(x_1\cdots x_n)^m K_{m^n-\la}(x;q,t;\tees) 
K_{\la'}(y;t,q;\tees) \\
=\sigma_1\big[{-}xy^{\pm}\big].
\end{multline}
If we expand the right-hand side in terms of Macdonald polynomials
using the Cauchy identity \eqref{Eq_Mac-Cauchy-a}, this yields
\begin{multline*}
\sum_{\la\subseteq m^n} (-1)^{\abs{\la}} 
(x_1\cdots x_n)^m K_{m^n-\la}(x;q,t;\tees) 
K_{\la'}(y;t,q;\tees) \\
=\sum_{\la\subseteq (2m)^n} (-1)^{\abs{\la}} P_{\la}(x;q,t) 
P_{\la'}(y^{\pm};t,q).
\end{multline*}
Equating coefficients of $P_{\la}(x;q,t)K_0(y;t,q;\tees)$, we find
\begin{multline*}
[P_{\la}(x;q,t)] (x_1\cdots x_n)^m K_{m^n}(x;q,t;\tees) \\
=(-1)^{\abs{\la}} [K_0(y;t,q;\tees)] P_{\la'}(y^{\pm};t,q),
\end{multline*}
for $\la\subseteq (2m)^n$.
Recalling \eqref{Eq_functional-IKn} and \eqref{Eq_PKf} completes the proof.
\end{proof}

\medskip

Next we consider the problem of computing
\begin{equation}\label{Eq_PKf-2}
f_{\la}^{(m)}(q,t;t_2,t_3):=
\big[ P_{\la}(x;q,t) \big] (x_1\cdots x_n)^m 
K_{m^n}(x;q,t;t_2,t_3),
\end{equation}
where $m$ is a nonnegative integer or half-integer
and $K_{\la}(x;q,t;t_2,t_3)$ is the Mac\-don\-ald--Koornwinder polynomial 
of Section~\ref{Sec_Kt2t3}.

\begin{proposition}\label{Prop_fIK-2}
For $m$ a nonnegative integer or half-integer, $\la\subseteq (2m)^n$
and generic $q,t,t_2,t_3$ 
\begin{equation}\label{Eq_fIK-half}
f_{\la}^{(m)}(q,t;t_2,t_3)=(-1)^{\abs{\la}}
I_K\big(P_{\la'}(t,q);t,q,q^m;-1,-q^{1/2},t_2,t_3\big).
\end{equation}
\end{proposition}

\begin{proof}
When $m$ is an integer we simply have
\[
f_{\la}^{(m)}(q,t;t_2,t_3)=f_{\la}^{(m)}(q,t;-1,-q^{1/2},t_2,t_3).
\]
By \eqref{Eq_rhs} this gives
\[
f_{\la}^{(m)}(q,t;t_2,t_3)=(-1)^{\abs{\la}}
I_K^{(m)}\big(P_{\la'}(t,q);t,q;-1,-q^{1/2},t_2,t_3\big),
\]
which may also be written as \eqref{Eq_fIK-half}.

\medskip

To deal with the half-integer case we set $k=m-1/2$ and
replace $m$ by $k$ in \eqref{Eq_Mim2}, so that now 
$y=(y_1,\dots,y_k)$.
Multiplying both sides by 
$\prod_{i=1}^n (1+x_i)=\sigma_1[-\varepsilon x]$,
using that
\[
\sigma_1[-\varepsilon x]\sigma_1[-xy^{\pm}]=
\sigma_1[-\varepsilon x-xy^{\pm}]=
\sigma_1[-x(y^{\pm}+\varepsilon)],
\]
and finally expanding this by the Cauchy identity \eqref{Eq_Mac-Cauchy-a}, 
we obtain
\begin{multline*}
\prod_{i=1}^n (1+x_i)  \\
\times \sum_{\la\subseteq k^n} (-1)^{\abs{\la}} 
(x_1\cdots x_n)^k K_{k^n-\la}(x;q,t;\tees) K_{\la'}(y;t,q;\tees) \\[-1mm]
=\sum_{\la\subseteq (2m)^n} (-1)^{\abs{\la}} P_{\la}(x;q,t) 
P_{\la'}\big([y^{\pm}+\varepsilon];t,q).
\end{multline*}
After specialising $\{t_0,t_1\}=\{-q,-q^{1/2}\}$ we can apply
Lemma~\ref{Lem_KBn} to rewrite this as
\begin{multline*}
\sum_{\la\subseteq k^n} (-1)^{\abs{\la}} 
(x_1\cdots x_n)^m K_{m^n-\la}(x;q,t;t_2,t_3) 
K_{\la'}(y;t,q;-q,-q^{1/2},t_2,t_3) \\[-1mm]
=\sum_{\la\subseteq (2m)^n} (-1)^{\abs{\la}} P_{\la}(x;q,t) 
P_{\la'}\big([y^{\pm}+\varepsilon];t,q).
\end{multline*}
Equating coefficients of $P_{\la}(x;q,t)K_0(y;t,q;-q,-q^{1/2},t_2,t_3)$
yields
\begin{equation}\label{Eq_generic}
f_{\la}^{(m)}(q,t;t_2,t_3)
=(-1)^{\abs{\la}} 
I_K^{(k)}\big(P_{\la'}([y+\varepsilon];t,q);t,q;-q,-q^{1/2},t_2,t_3\big)
\end{equation}
for $\la\subseteq (2m)^n$.
For generic $q,t,t_2,t_3$ we can write the integral on the right as
\[
I_K\big(P_{\la'}([y^{\pm}+\varepsilon];t,q);t,q,q^k;-q,-q^{1/2},t_2,t_3\big),
\]
where now $y=(y_1,y_2,\dots)$.
By Lemma~\ref{Lem_IK} this is also
\[
I_K\big(P_{\la'}(t,q);t,q,q^m;-1,-q^{1/2},t_2,t_3\big). \qedhere
\]
\end{proof}

\medskip

We are now ready to prove Theorems~\ref{THM_BOUNDED1}--\ref{THM_BOUNDED7}.

\begin{proof}[Proof of Theorem~\ref{THM_BOUNDED1}]
\label{proof-THM_BOUNDED1}
We first prove the $a=0$ case, given in \eqref{Eq_bounded1a0}.
By definition \eqref{Eq_PKf} and equation \eqref{Eq_PCB-K},
this is equivalent to proving that for $\la\subseteq (2m)^n$
\[
f_{\la}^{(m)}\big(q,t;{\pm} q^{1/2},\pm (qt)^{1/2}\big)=
\begin{cases}
b_{\la;m}^{\textup{oa}}(q,t) & \text{$\la$ is even}, \\[1mm]
0 & \text{otherwise}.
\end{cases}
\]

From \eqref{Eq_rhs} we have
\[
f_{\la}^{(m)}\big(q,t;\pm q^{1/2},\pm (qt)^{1/2}\big)
=(-1)^{\abs{\la}}
I_K^{(m)}\big(P_{\la'}(t,q);t,q;\pm q^{1/2},\pm (qt)^{1/2}\big).
\]
Taking $(T,\mu)=(t^m,\la')$ in \eqref{Eq_Q1} and then interchanging
$q$ and $t$, it follows that the virtual Koornwinder integral on the right
vanishes unless $\la$ is even. Moreover, for even $\la$ it evaluates in 
closed form to
\[
\frac{(q^{2m};t,q^2)_{(\la/2)'}}{(q^{2m-1}t;t,q^2)_{(\la/2)'}}\cdot
\frac{C_{(\la/2)'}^{-}(qt;t,q^2)}{C_{(\la/2)'}^{-}(q^2;t,q^2)}.
\]
Using \eqref{Eq_qtswap}, we thus find
\[
f_{\la}^{(m)}\big(q,t;\pm q^{1/2},\pm (qt)^{1/2}\big)
=\Big(\frac{q}{t}\Big)^{\abs{\la}/2}
\frac{(q^{-2m};q^2,t)_{\la/2}}{(q^{1-2m}/t;q^2,t)_{\la/2}}\cdot
\frac{C_{\la/2}^{-}(qt;q^2,t)}{C_{\la/2}^{-}(q^2;q^2,t)},
\]
for $\la$ even and zero otherwise. 
By \eqref{Eq_qshift} this can be also be written as
\begin{align}\label{Eq_al-form}
f_{\la}^{(m)}&\big(q,t;\pm q^{1/2},\pm(qt)^{1/2}\big)  \\
&=\prod_{s\in\la/2}\bigg(
\frac{1-q^{2m-2a'(s)}t^{l'(s)}}{1-q^{2m-2a'(s)-1}t^{l'(s)+1}}
\cdot \frac{1-q^{2a(s)+1}t^{l(s)+1}}{1-q^{2a(s)+2}t^{l(s)}}\bigg) \notag \\
&=\prod_{\substack{s\in\la \\[1pt] a(s) \textup{ odd}}}
\bigg(\frac{1-q^{2m-a'(s)}t^{l'(s)}}{1-q^{2m-a'(s)-1}t^{l'(s)+1}}
\cdot \frac{1-q^{a(s)}t^{l(s)+1}}{1-q^{a(s)+1}t^{l(s)}}\bigg). \notag
\end{align}
Since $\la$ is even, odd arms-lengths correspond to even arm-colengths.
The product on the right is thus $b_{\la;m}^{\textup{oa}}(q,t)$ in
the representation given by \eqref{Eq_iflaiseven}, completing the proof 
of \eqref{Eq_bounded1a0}.

\medskip

To obtain the full theorem we multiply both sides of \eqref{Eq_bounded1a0} 
by $\prod_{i=1}^n(1+ax_i)$.
By the $e$-Pieri rule \eqref{Eq_e-Pieri} we must then show that 
\begin{equation}\label{Eq_verticalstrip-even}
b_{\la;m}^{\textup{oa}}(q,t)=
\sum_{\mu \textup{ even}} a^{\abs{\skew{\la}{\mu}}} 
\psi'_{\skew{\la}{\mu}}(q,t) \, b_{\mu;m}^{\textup{oa}}(q,t). 
\end{equation}
Because $\mu$ is even and $\psi'_{\skew{\la}{\mu}}(q,t)$ vanishes
unless $\skew{\la}{\mu}$ is a vertical strip, $\mu$ is fixed as 
\begin{equation}\label{Eq_mufloorla}
\mu=2\floor{\la/2}:=(2\floor{\la_1/2},2\floor{\la_2/2},\dots),
\end{equation}
which implies that $\abs{\skew{\la}{\mu}}=\op(\la)$.
We thus obtain
\[
b_{\la;m}^{\textup{oa}}(q,t)=
\psi'_{\skew{\la}{\mu}}(q,t)\,b_{\mu;m}^{\textup{oa}}(q,t)
\]
with $\mu$ fixed as above.
The $m$-dependent parts on both sides trivially agree since
\[
\prod_{\substack{s\in\la \\[1pt] a'(s) \textup{ odd}}}
f_{a'(s),l'(s)}
=\prod_{\substack{s\in\mu \\[1pt] a'(s) \textup{ odd}}}
f_{a'(s),l'(s)}.
\]
It thus remains to show that
\[
b_{\la}^{\textup{oa}}(q,t)=\psi'_{\skew{\la}{\mu}}(q,t) \,
b_{\mu}^{\textup{oa}}(q,t). 
\]
Replacing $(\la,\mu,q,t)$ by $(\la',\mu',t,q)$,
using 
\[
b_{\nu'}^{\textup{oa}}(t,q)=
\frac{b_{\nu}^{\textup{el}}(q,t)}{b_{\nu}(q,t)}
\]
on both sides, and finally appealing to
\cite[page 341]{Macdonald95}
\[
\psi'_{\la'/\mu'}(t,q)=
\varphi_{\la/\mu}(q,t)\,
\frac{b_{\mu}(q,t)}{b_{\la}(q,t)},
\]
we are left with
\[
b_{\la}^{\textup{el}}(q,t)=\varphi_{\skew{\la}{\mu}}(q,t)
b_{\mu}^{\textup{el}}(q,t)
\]
for $\mu'=2\floor{\la'/2}$.
Since this is \cite[p. 351]{Macdonald95}, we are done.
\end{proof}

Because they are simpler to prove than 
Theorems~\ref{Thm_bounded2} and \ref{Thm_bounded3}, we consider
Theorems~\ref{Thm_bounded4} and \ref{THM_BOUNDED5} first.

\begin{proof}[Proof of Theorem \ref{Thm_bounded4}]
It will be convenient to prove the claim with $m$ replaced by $2m$.
After this change $m$ is a nonnegative integer or half-integer.
It then follows from \eqref{Eq_PKf-2} and \eqref{Eq_PBB-K} that we must 
prove for $\la\subseteq (2m)^n$ that
\[
f_{\la}^{(m)}(q,t;t,q^{1/2})=b_{\la;2m}^{\textup{el}}(q,t).
\]
By Proposition~\ref{Prop_fIK-2},
\[
f_{\la}^{(m)}(q,t;t,q^{1/2})=(-1)^{\abs{\la}}
I_K\big(P_{\la'}(t,q);t,q,q^m;-1,t,\pm q^{1/2}\big).
\]
The integral on the right can be computed by Theorem~\ref{Thm_VKI}
with $(q,t,T,\mu)\mapsto (t,q,q^m,\la')$, resulting in
\[
f_{\la}^{(m)}(q,t;t,q^{1/2})
=\frac{(q^{2m};t^2,q)_{\ceil{\la'/2}}}
{(q^{2m-1}t;t^2,q)_{\ceil{\la'/2}}}
\cdot \frac{1}{b_{\la'}^{\textup{ea}}(t,q)}.
\]
Let $\nu:=\ceil{\la'/2}'=(\la_1,\la_3,\dots)$.
By \eqref{Eq_qtswap_C0} we can write the first factor on the right as
\[
\Big(\frac{q}{t}\Big)^{\abs{\nu}}
\frac{(q^{-2m};q,t^2)_{\nu}}{(q^{1-2m}/t;q,t^2)_{\nu}}.
\]
By \eqref{Eq_qshift} this is also
\[
\prod_{s\in\nu}
\frac{1-q^{2m-a'(s)}t^{2l'(s)}}{1-q^{2m-a'(s)-1}t^{2l'(s)+1}}=
\prod_{\substack{s\in\la \\[1pt] l'(s) \textup{ even}}} 
\frac{1-q^{2m-a'(s)}t^{l'(s)}}{1-q^{2m-a'(s)-1}t^{l'(s)+1}}.
\]
Since under conjugation legs become arms and arms become legs,
we further have
\[
b_{\la'}^{\textup{ea}}(t,q)b_{\la}^{\textup{el}}(q,t)=1.
\]
Hence
\[
f_{\la}^{(m)}(q,t;t,q^{1/2})
=b_{\la}^{\textup{el}}(q,t)
\prod_{\substack{s\in\la \\[1pt] l'(s) \textup{ even}}} 
\frac{1-q^{2m-a'(s)}t^{l'(s)}}{1-q^{2m-a'(s)-1}t^{l'(s)+1}}
=b_{\la;2m}^{\textup{el}}(q,t)
\]
as claimed.
\end{proof}

\begin{proof}[Proof of Theorem \ref{THM_BOUNDED5}]\label{proof-Thm_bounded5}
We closely follow the previous proof and again replace $m$ by $2m$.
This time it follows from \eqref{Eq_PKf-2} and \eqref{Eq_PBC-K} that we must 
prove
\begin{equation}\label{Eq_fThm2}
f_{\la}^{(m)}\big(q^2,t^2;-t,-qt\big)=b_{\la;2m}^{-}(q,t)
\end{equation}
for $\la\subseteq (2m)^n$. By Proposition~\ref{Prop_fIK-2},
\[
f_{\la}^{(m)}\big(q^2,t^2;-t,-qt\big)=
(-1)^{\abs{\la}} I_K\big(P_{\la'}(t^2,q^2);t^2,q^2,q^{2m};-1,-q,-t,-qt\big).
\]
The integral on the right evaluates to
\[
(-1)^{\abs{\la}}\, \frac{(q^{2m};t,q)_{\la'}}{(-q^{2m-1}t;t,q)_{\la'}}
\cdot \frac{C^{-}_{\la'}(-t;t,q)}{C^{-}_{\la'}(q;t,q)}.
\]
by \eqref{Eq_Q6} with $(q,t,T,\mu)\mapsto(t^2,q^2,t^{2m},\la')$.
Also using \eqref{Eq_qtswap}, we find
\[
f_{\la}^{(m)}\big(q^2,t^2;-t,-qt\big)=
\Big({-}\frac{q}{t}\Big)^{\abs{\la}}
\frac{(q^{-2m};q,t)_{\la}}{(-q^{1-2m}/t;q,t)_{\la}}
\cdot \frac{C^{-}_{\la}(-t;q,t)}{C^{-}_{\la}(q;q,t)}.
\]
Equation \eqref{Eq_fThm2} now follows by \eqref{Eq_qshift}.
\end{proof}

\begin{proof}[Proof of Theorem \ref{Thm_bounded2}]
Again we prove the theorem with $m$ replaced by $2m$.
It then follows from \eqref{Eq_PKf-2} and \eqref{Eq_PBB-K} that we must 
prove for $\la\subseteq (2m)^n$ that
$f_{\la}^{(m)}\big(q,t;1,q^{1/2}\big)$ vanishes unless
$m_i(\la)$ is even for all $1\leqslant i\leqslant 2m-1$, in which case
\begin{equation}\label{Eq_tobeprovedD}
f_{\la}^{(m)}\big(q,t;1,q^{1/2}\big)=b_{\la;2m}^{\textup{ol}}(q,t).
\end{equation}
The problem with using Proposition~\ref{Prop_fIK-2} as 
in the proof of Theorem~\ref{THM_BOUNDED5} is that the specialisation
$\{t_2,t_3\}=\{1,q^{1/2}\}$ corresponds to one of the 
non-generic cases discussed on page~\pageref{page_poles}. It would lead to
\begin{equation}\label{Eq_careful}
f_{\la}^{(m)}\big(q,t;1,q^{1/2}\big)
=(-1)^{\abs{\la}} I_K\big(P_{\la'}(t,q);t,q,q^m;\pm 1,\pm q^{1/2}\big),
\end{equation}
where  the integral on the right is not well-defined.
It is still possible to use \eqref{Eq_careful} by interpreting 
the right in an appropriate limiting sense, but instead we proceed
slightly differently.

First, when $m$ is an integer \eqref{Eq_fIK-half} simply says that
\[
f_{\la}^{(m)}(q,t;t_2,t_3)=(-1)^{\abs{\la}}
I_K^{(m)}\big(P_{\la'}(t,q);t,q;-1,-q^{1/2},t_2,t_3\big).
\]
In this equation there is no problem specialising $\{t_2,t_3\}=\{1,q^{1/2}\}$
so that
\[
f_{\la}^{(m)}(q,t;1,q^{1/2})=(-1)^{\abs{\la}}
I_K^{(m)}\big(P_{\la'}(t,q);t,q;\pm 1,\pm q^{1/2}\big).
\]
The right-hand side can be computed by Theorem~\ref{Thm_BD} with 
$(\mu,q,t,n)\mapsto (\la',t,q,m)$ so that
\begin{equation}\label{Eq_integercase}
f_{\la}^{(m)}(q,t;1,q^{1/2})
=\begin{cases}
A_{\la'/2}^{(2m)}(t,q)
& \text{if $\widetilde{\la'}$ is even}, \\[1mm]
0 & \text{otherwise},
\end{cases}
\end{equation}
where we have also used that $\abs{\la}$ is even if 
$\widetilde{\la'}=(\la'_1-\la'_{2m},\dots,\la'_{2m-1}-\la'_{2m},0)$
is even.

When $m=k+1/2$ is a half-integer, we use \eqref{Eq_generic} written as
\begin{equation}\label{Eq_forHLcase}
f_{\la}^{(m)}(q,t;t_2,t_3)
=(-1)^{\abs{\la}} 
I_K^{(k)}\big(P_{\la'}(x_1,\dots,x_k,-1;t,q);t,q;-q,-q^{1/2},t_2,t_3\big)
\end{equation}
instead of \eqref{Eq_fIK-half}. 
Specialising $\{t_2,t_3\}=\{1,q^{1/2}\}$ this gives
\[
f_{\la}^{(m)}(q,t;1,q^{1/2})
=(-1)^{\abs{\la}} 
I_K^{(k)}\big(P_{\la'}(x_1,\dots,x_k,-1;t,q);t,q;1,-q,\pm q^{1/2}\big).
\]
Now the right can be computed by Theorem~\ref{Thm_BD-2} with 
$(\nu,q,t,n)\mapsto (\la',t,q,k)$. 
Since $2k+1=2m$ and $\abs{\la}+\la_{2m}$ is even if 
$\widetilde{\la'}$ is even, this once again results in
\eqref{Eq_integercase}.

To complete the proof we first note that 
\[
A_{\la'/2}^{(2m)}(t,q)=A_{\floor{\la'/2}}^{(2m)}(t,q).
\]
Indeed, either $\la'$ is even, in which case 
$\floor{\la'/2}=\la'/2$ or $\la'$ is odd and $\la_1=2m$,
in which case $\floor{\la'/2}=\la'/2-(1/2)^{2m}$.
Since $A_{\la'/2}^{(2m)}(t,q)$ only depends on the relative
differences between the parts of $\la'/2$ the change is justified.
Denoting $\floor{\la'/2}$ by $\nu'$ we find that 
in the non-vanishing case, that is, when $m_i(\la)$ is even
for all $1\leqslant i\leqslant 2m-1$, 
\begin{align*}
f_{\la}^{(m)}(q,t;1,q^{1/2})
&=A_{\nu'}^{(2m)}(t,q) \\
&=\frac{(q^{2m};t^2,q)_{\nu'}}{(q^{2m-1}t;t^2,q)_{\nu'}}\cdot
\frac{C^{-}_{\nu'}(t;t^2,q)}{C^{-}_{\nu'}(q;t^2,q)} \\
&=\Big(\frac{q}{t}\Big)^{\abs{\nu}}
\frac{(q^{-2m};q,t^2)_{\nu}}{(q^{1-2m}/t;q,t^2)_{\nu}}\cdot
\frac{C^{-}_{\nu}(t;q,t^2)}{C^{-}_{\nu}(q;q,t^2)},
\end{align*}
where the second equality follows from \eqref{Eq_Adef} and
the last equality from \eqref{Eq_qtswap}.
Since $\nu'=\floor{\la'/2}$ we also have
$\nu=\floor{\la'/2}'$, which can be simplified to
$\nu=(\la_2,\la_4,\la_6,\dots)$.
Recalling \eqref{Eq_qshift}, we obtain
\[
f_{\la}^{(m)}(q,t;1,q^{1/2})=
\prod_{s\in\nu} \bigg(\frac{1-q^{2m-a'(s)}t^{2l'(s)}}
{1-q^{2m-a'(s)-1}t^{2l'(s)+1}}\cdot
\frac{1-q^{a(s)}t^{2l(s)+1}}{1-q^{a(s)+1}t^{2l(s)}}\bigg).
\]
To write this without reference to the partition $\nu$ we consider
both factors in the product separately.
The first factor is trivial:
\begin{equation}\label{Eq_fac1}
\prod_{s\in\nu} \frac{1-q^{2m-a'(s)}t^{2l'(s)}}
{1-q^{2m-a'(s)-1}t^{2l'(s)+1}} 
=\prod_{\substack{s\in\la \\[1pt] l'(s) \textup{ odd}}}
\frac{1-q^{2m-a'(s)}t^{l'(s)-1}}
{1-q^{2m-a'(s)-1}t^{l'(s)}}.
\end{equation}
For the second factor we use that for $\la'$ even we must have
$\la_{2i}=\la_{2i-1}$ for all $i$. 
We can therefore redefine $\nu$ as
\[
\nu:=\begin{cases}
(\la_1,\la_3,\dots) & \text{if $\la'$ is even} \\
(\la_2,\la_4,\dots) & \text{if $\la'$ is odd}.
\end{cases}
\]
For such $\nu$, 
\begin{equation}\label{Eq_fac2}
\prod_{s\in\nu} 
\frac{1-q^{a(s)}t^{2l(s)+1}}{1-q^{a(s)+1}t^{2l(s)}} 
=\prod_{s\in\la} 
\frac{1-q^{a(s)}t^{l(s)}}{1-q^{a(s)+1}t^{l(s)}} 
\end{equation}
in both cases. Combining \eqref{Eq_fac1} and \eqref{Eq_fac2}
we obtain \eqref{Eq_tobeprovedD}.
\end{proof}

\begin{proof}[Proof of Theorem \ref{Thm_bounded3}]
When $m$ is odd the result is completely elementary.
By \eqref{Eq_DDDDBB} and \eqref{Eq_PDDbar} we can
write the right-hand side of \eqref{Eq_bounded2} as
\[
(x_1\cdots x_n)^{\frac{m}{2}}
P^{(\mathrm{D}_n,\mathrm{D}_n)}_{\halfm{n}}(x;q,t) + 
(x_1\cdots x_n)^{\frac{m}{2}}
P^{(\mathrm{D}_n,\mathrm{D}_n)}_{\halfm{n}}(\bar{x};q,t).
\]
When $m$ is odd the first term is a polynomial of even degree
whereas the second term is a polynomial of odd degree.
Since partitions $\la$ of Type 1 have even size and partitions of Type 2
have size congruent to $m$ modulo $2$, it follows that for odd $m$ 
we may dissect \eqref{Eq_bounded2} as in Corollary~\ref{Thm_bounded3}.

To prove the theorem for even $m$ we closely follow the proof of
Proposition~\ref{Prop_fIK-2}.
In \eqref{Eq_Mim2} we replace $m$ by $m-1=:k$ (we do not at this point
assume that $m$ is even)
and multiply both sides by 
$\prod_{i=1}^n (1-x_i^2)=\sigma_1[-x(1+\varepsilon)]$. Then 
expanding the right-hand side in terms of Macdonald polynomials
using \eqref{Eq_Mac-Cauchy-a} gives
\begin{multline*}
\prod_{i=1}^n (1-x_i^2)  \\
\times \sum_{\la\subseteq k^n} (-1)^{\abs{\la}} 
(x_1\cdots x_n)^k K_{k^n-\la}(x;q,t;\tees) K_{\la'}(y;t,q;\tees) 
\\[-1mm]
=\sum_{\la\subseteq (2m)^n} (-1)^{\abs{\la}} P_{\la}(x;q,t) 
P_{\la'}\big([y^{\pm}+1+\varepsilon];t,q\big),
\end{multline*}
where $y=(y_1,\dots,y_k)$.
If we specialise $\{\tees\}=\{\pm q,\pm q^{1/2}\}$ and apply 
Lemma~\ref{Lem_KBn} followed by \eqref{Eq_PBC-K}, this leads to
\begin{multline*}
\prod_{i=1}^n \big(x_i^{-1/2}-x_i^{1/2}\big) \\
\times \sum_{\la\subseteq k^n} (-1)^{\abs{\la}} 
(x_1\cdots x_n)^m 
P_{(m-\frac{1}{2})^n-\la}^{(\mathrm{B}_n,\mathrm{C}_n)}(x;q,t,q^{1/2}) 
K_{\la'}(y;t,q;\pm q,\pm q^{1/2}) \\[-1mm]
=\sum_{\la\subseteq (2m)^n} (-1)^{\abs{\la}} P_{\la}(x;q,t) 
P_{\la'}\big([y^{\pm}+1+\varepsilon];t,q).
\end{multline*}
Equating coefficients of $P_{\la}(x;q,t)K_0(y;t,q;\pm q,\pm q^{1/2})$
and then replacing $y$ by $x$ on the right yields
\begin{multline}\label{Eq_BCcase}
\big[P_{\la}(x;q,t)\big] (x_1\cdots x_n)^m 
P_{(m-\frac{1}{2})^n}^{(\mathrm{B}_n,\mathrm{C}_n)}(x;q,t,q^{1/2}) 
\prod_{i=1}^n \big(x_i^{-1/2}-x_i^{1/2}\big)  \\
=(-1)^{\abs{\la}}
I_K^{(m-1)}\big(P_{\la'}
(x_1^{\pm},\dots,x_{m-1}^{\pm},\pm 1;t,q);t,q;\pm q,\pm q^{1/2}\big).
\end{multline}
By the integer-$m$ case of Proposition~\ref{Prop_fIK-2},
\begin{multline*}
\big[P_{\la}(x;q,t)\big] (x_1\cdots x_n)^m K_{m^n}(x;q,t;t_2,t_3) \\
=(-1)^{\abs{\la}} 
I_K^{(m)}\big(P_{\la'}(t,q);t,q;-1,-q^{1/2},t_2,t_3\big).
\end{multline*}
For $\{t_2,t_3\}=\{1,q^{1/2}\}$ this can also be written as
\begin{multline}\label{Eq_BBcase}
\big[P_{\la}(x;q,t)\big] (x_1\cdots x_n)^m
P_{m^n}^{(\mathrm{B}_n,\mathrm{B}_n)}(x;q,t,1) \\
=(-1)^{\abs{\la}}
I_K^{(m)}\big(P_{\la'}(t,q);t,q;\pm 1,\pm q^{1/2}\big)
\end{multline}
thanks to \eqref{Eq_PBB-K}. Taking half the sum of
\eqref{Eq_BCcase} and \eqref{Eq_BBcase} and recalling \eqref{Eq_PBBDD},
it follows that
\begin{align}\label{Eq_DDIK}
\big[P_{\la}&(x;q,t)\big] (x_1\cdots x_n)^m
P_{m^n}^{(\mathrm{D}_n,\mathrm{D}_n)}(x;q,t) 
\\[1mm] &=\tfrac{1}{2}(-1)^{\abs{\la}}
I_K^{(m)}\big(P_{\la'}(x_1^{\pm},\dots,x_m^{\pm};t,q);
t,q;\pm 1,\pm q^{1/2}\big) 
\notag \\[1mm]
& \quad + \tfrac{1}{2}(-1)^{\abs{\la}} I_K^{(m-1)}\big(P_{\la'}
(x_1^{\pm},\dots,x_{m-1}^{\pm},\pm 1;t,q);t,q;\pm q,\pm q^{1/2}\big).
\notag 
\end{align}
Both virtual Koornwinder integrals on the right can be computed by
Theorem~\ref{Thm_BD}. 
Since 
\[
\chi\big(\widetilde{\la'} \text{ even}\big) \big(\tfrac{1}{2}+
\tfrac{1}{2}(-1)^{\la'_{2m}}\big)
=\chi(\la' \text{ even}) 
\] 
for $\widetilde{\la'}=(\la'_1-\la'_{2m},\dots,\la'_{2m-1}-\la'_{2m},0)$,
we find
\[
\big[P_{\la}(x;q,t)\big] (x_1\cdots x_n)^m
P_{m^n}^{(\mathrm{D}_n,\mathrm{D}_n)}(x;q,t)
=\chi(\la' \text{ even}) A_{\la'/2}(t,q).
\]
In the proof of Theorem~\ref{Thm_bounded2} we have already shown that
\[
A_{\la'/2}(t,q)=b_{\la;2m}^{\textup{ol}}(q,t)
\]
for $\la'$ even (or $\la'$ odd and $\la_1=2m$).
Hence
\[
\big[P_{\la}(x;q,t)\big] (x_1\cdots x_n)^m
P_{m^n}^{(\mathrm{D}_n,\mathrm{D}_n)}(x;q,t)
=\chi(\la' \text{ even}) b_{\la;2m}^{\textup{ol}}(q,t).
\]
Replacing $m$ by $m/2$, this proves \eqref{Eq_bounded3-1} for even $m$.

For completeness we remark that the analogue of \eqref{Eq_DDIK}
for half-integer $m$ is easily shown to be
\begin{align*}
\big[P_{\la}&(x;q,t)\big](x_1\cdots x_n)^m
P_{m^n}^{(\mathrm{D}_n,\mathrm{D}_n)}(x;q,t) \\[1mm]
&=\tfrac{1}{2}(-1)^{\abs{\la}}
I_K^{(k)}\big(P_{\la'}(x_1^{\pm},\dots,x_k^{\pm},-1;t,q);
t,q;1,-q,\pm q^{1/2}\big) \\[1mm]
& \quad + \tfrac{1}{2}(-1)^{\abs{\la}} I_K^{(k)}\big(P_{\la'}
(x_1^{\pm},\dots,x_k^{\pm},1;t,q);t,q;-1,q,\pm q^{1/2}\big),
\end{align*}
where $k=m-1/2$. By Theorem~\ref{Thm_BD-2} this again implies that
\[
\big[P_{\la}(x;q,t)\big] (x_1\cdots x_n)^m
P_{m^n}^{(\mathrm{D}_n,\mathrm{D}_n)}(x;q,t)
=\chi(\la' \text{ even}) b_{\la;2m}^{\textup{ol}}(q,t).
\]
Of course, as noted above, this result follows more simply by 
a degree argument.
\end{proof}

\begin{proof}[Proof of Theorem \ref{Thm_bounded6}]
By \eqref{Eq_BCHL} and \eqref{Eq_PKf} we must prove that
\begin{equation}\label{Eq_fab}
f_{\la}^{(m)}(0,t;0,0,t_2,t_3)=h_{\la}^{(2m)}(t_2,t_3;t)
\end{equation}
for $\la\subseteq (2m)^n$.
By \eqref{Eq_rhs}, 
\begin{equation}\label{Eq_fab2}
f_{\la}^{(m)}(0,t;0,0,t_2,t_3)=(-1)^{\abs{\la}}
I_K^{(m)}\big(P_{\la'}(t,0);t,0;0,0,t_2,t_3\big).
\end{equation}
The integral on the right can be evaluated thanks to 
Theorem~\ref{Thm_newvirtual} with $(n,q,\mu)$ replaced by $(m,t,\la')$. 
Hence
\[
f_{\la}^{(m)}(0,t;0,0,t_2,t_3)=(-1)^{\abs{\la}} h_{\la}(-t_2,-t_3;t)
=h_{\la}^{(2m)}(t_2,t_3;t),
\]
where the second equality follows from definition \eqref{Eq_RSBCn} and
\[
(-1)^{\abs{\la}}=
\prod_{\substack{i\geqslant 1\\i \text{ odd}}} (-1)^{m_i(\la)}.
\qedhere
\]
\end{proof}

A proof of \eqref{Eq_hab-lim} (the large-$m$ limit of 
Theorem~\ref{Thm_bounded6}) using virtual Koornwinder integrals 
is due to Venkateswaran \cite{Venkateswaran12}. Her 
approach, however, is not a limiting version of ours. Crucial
difference is that Venkateswaran stays within 
the $t$-world, whereas we have applied the virtual Koornwinder integral 
\eqref{Eq_vanish} over $P_{\la}(0,q)$.

\begin{proof}[Proof of Theorem~\ref{THM_BOUNDED7}]
As in earlier proofs we replace $m$ by $2m$. From \eqref{Eq_BHL1} and 
\eqref{Eq_PKf-2} it follows that we must prove
\begin{equation}\label{Eq_fB}
f_{\la}^{(m)}(0,t;t_2,0)=h_{\la}^{(2m)}(t_2;t)
\end{equation}
for $\la\subseteq (2m)^n$ and $m$ a nonnegative integer or half-integer.
For $m$ an integer, \eqref{Eq_fB} is the $t_3=-1$ case of \eqref{Eq_fab}, 
and in the remainder we assume $m$ is a half-integer.

We will not apply Proposition~\ref{Prop_fIK-2} as it is not suitable
for taking the $q\to 0$ limit.
Instead we take that limit in \eqref{Eq_forHLcase}. Then
\[
f_{\la}^{(m)}(0,t;t_2,t_3)=(-1)^{\abs{\la}} 
I_K^{(k)}\big(P_{\la'}(x_1^{\pm},\dots,x_k^{\pm},-1;t,0);t,0;0,0,t_2,t_3\big),
\]
where $k=m-1/2$.
By the branching rule \eqref{Eq_BR} and relation
\eqref{Eq_psipsip} this becomes
\[
f_{\la}^{(m)}(0,t;t_2,t_3)=
\sum_{\mu\subseteq\la} 
(-1)^{\abs{\mu}} \psi'_{\skew{\la}{\mu}}(0,t)
I_K^{(k)}\big(P_{\mu'}(t,0);t,0;0,0,t_2,t_3\big).
\]
The virtual Koornwinder integral on the right can be computed
by Theorem~\ref{Thm_newvirtual} with $(n,q,\mu)$ replaced by $(k,t,\mu')$.
Hence
\[
f_{\la}^{(m)}(0,t;t_2,t_3)=
\sum_{\substack{\mu\subseteq\la \\[1pt] \mu_1\leqslant 2k}} 
(-1)^{\abs{\mu}} 
\psi'_{\skew{\la}{\mu}}(0,t) \, h_{\mu}^{(2k)}(-t_2,-t_3;t).
\]
We do not know how to evaluate this in closed form for arbitrary $t_3$,
but for $t_3=0$ it follows from \eqref{Eq_RSBCn} that 
\[
h_{\mu}^{(2k)}(-t_2,0;t)=t_2^{\op(\mu)}.
\]
Also using that $(-1)^{\abs{\mu}}=(-1)^{\op(\mu)}$, we find
\[
f_{\la}^{(m)}(0,t;t_2,0)=
\sum_{\substack{\mu\subseteq\la \\[1pt] \mu_1\leqslant 2m-1}} 
(-t_2)^{\op(\mu)} \psi'_{\skew{\la}{\mu}}(0,t).
\]
Since $\psi'_{\skew{\la}{\mu}}(0,t)$ is the $e$-Pieri coefficient
for ordinary Hall--Littlewood polynomials, we have
\cite[page 215]{Macdonald95}
\[
\psi'_{\la/\mu}(0,t)=
\prod_{i\geqslant 1} \qbin{\la'_i-\la'_{i+1}}{\la'_i-\mu'_i}_t.
\]
Therefore
\[
f_{\la}^{(m)}(0,t;t_2,0)=
\sum_{\substack{\mu\subseteq\la \\[1pt] \mu_1\leqslant 2m-1}} 
(-t_2)^{\op(\mu)} 
\prod_{i=1}^{2m-1} \qbin{\la'_i-\la'_{i+1}}{\la'_i-\mu'_i}_t.
\]
Writing $\mu'_i=\la'_i-k_i$ and using 
\eqref{p822} with $(\mu,\nu,n)\mapsto(\mu',\la',m)$, we finally obtain
\begin{align*}
f_{\la}^{(m)}(0,t;t_2,0)&=
(-t_2)^{\op(\la)} \prod_{i=1}^{2m-1} \sum_{k_i=0}^{m_i(\la)} 
(-t_2)^{(-1)^i k_i} \qbin{m_i(\la)}{k_i}_t \\
&=(-t_2)^{\op(\la)} 
\prod_{\substack{i=1 \\[0.5pt] i \text{ odd}}}^{2m-1} 
H_{m_i(\la)}(-1/t_2;t)
\prod_{\substack{i=1 \\[0.5pt] i \text{ even}}}^{2m-1} H_{m_i(\la)}(-t_2;t) \\
&=h_{\la}^{(2m)}(t_2,-1;t)=h_{\la}^{(2m)}(t_2;t). \qedhere
\end{align*}
\end{proof}


\chapter{Applications}\label{Ch_Apps}

\section{Plane partitions}\label{Sec_Plane-partitions}
In this section we will show that our approach to bounded Littlewood 
identities provides an intimate connection between symmetric plane
partitions and the theory of Gelfand pairs.

A plane partition $\pi$ of shape $\la$ is a filling of the
Young diagram of $\la$ with positive integers, called the parts
of $\pi$, such that the resulting tableau is weakly decreasing along rows 
and columns, see e.g., 
\cite{Bressoud99,Krattenthaler15,Stanley86,Stanley89} and references therein.
The size of the plane partition $\pi$, denoted $\abs{\pi}$,
is the sum of the parts, that is
\[
\abs{\pi}=\sum_{(i,j)\in \la} \pi_{ij},
\]
where the part $\pi_{ij}$ is the filling of the square $(i,j)\in\la$.
For example,
\begin{subequations}\label{Eq_PP}
\begin{equation}\label{Eq_PP1}
\raisebox{-10mm}{
\begin{tikzpicture}[scale=0.4,line width=0.3pt]
\draw (0,0)--(5,0);
\draw (0,-1)--(5,-1);
\draw (0,-2)--(3,-2);
\draw (0,-3)--(2,-3);
\draw (0,-4)--(1,-4);
\draw (0,-5)--(1,-5);
\draw (0,0)--(0,-5);
\draw (1,0)--(1,-5);
\draw (2,0)--(2,-3);
\draw (3,0)--(3,-2);
\draw (4,0)--(4,-1);
\draw (5,0)--(5,-1);
\draw (0.5,-0.5) node {$4$};
\draw (1.5,-0.5) node {$3$};
\draw (2.5,-0.5) node {$3$};
\draw (3.5,-0.5) node {$2$};
\draw (4.5,-0.5) node {$1$};
\draw (0.5,-1.5) node {$3$};
\draw (1.5,-1.5) node {$2$};
\draw (2.5,-1.5) node {$1$};
\draw (0.5,-2.5) node {$3$};
\draw (1.5,-2.5) node {$1$};
\draw (0.5,-3.5) node {$2$};
\draw (0.5,-4.5) node {$1$};
\end{tikzpicture}
}
\end{equation}
\smallskip
is a plane partition of shape $(5,3,2,1,1)$ and size $26$.

A part of size $r$ may be thought of as the stacking of $r$ unit cubes, 
providing a geometric interpretation of plane partitions. 
Thus, the plane partition \eqref{Eq_PP1} corresponds to
\begin{equation}
\raisebox{-12mm}{\begin{tikzpicture}[scale=0.3]
\draw[fill=blue] (0,0)--(0.866,-0.5)--(0.866,0.5)--(0,1)--cycle;
\draw[fill=red] (0.866,-0.5)--(1.732,0)--(1.732,1)--(0.866,0.5)--cycle;
\draw[fill=red] (1.732,0)--(2.598,0.5)--(2.598,1.5)--(1.732,1)--cycle;
\draw[fill=blue] (2.598,0.5)--(3.464,0)--(3.464,1)--(2.598,1.5)--cycle;
\draw[fill=red] (3.464,0)--(4.330,0.5)--(4.330,1.5)--(3.464,1)--cycle;
\draw[fill=blue] (4.330,0.5)--(5.196,0)--(5.196,1)--(4.330,1.5)--cycle;
\draw[fill=red] (5.196,0)--(6.062,0.5)--(6.062,1.5)--(5.196,1)--cycle;
\draw[fill=blue] (6.062,0.5)--(6.928,0)--(6.928,1)--(6.062,1.5)--cycle;
\draw[fill=blue] (6.928,0)--(7.794,-0.5)--(7.794,0.5)--(6.928,1)--cycle;
\draw[fill=red] (7.794,-0.5)--(8.660,0)--(8.660,1)--(7.794,0.5)--cycle;
\draw[fill=green] (0,1)--(0.866,0.5)--(1.732,1)--(0.866,1.5)--cycle;
\draw[fill=green] (6.928,1)--(7.794,0.5)--(8.660,1)--(7.794,1.5)--cycle;
\draw[fill=green] (2.598,1.5)--(3.464,1)--(4.330,1.5)--(3.464,2)--cycle;
\draw[fill=green] (4.330,1.5)--(5.196,1)--(6.062,1.5)--(5.196,2)--cycle;
\draw[fill=blue] (0.866,1.5)--(1.732,1)--(1.732,2)--(0.866,2.5)--cycle;
\draw[fill=red] (1.732,1)--(2.598,1.5)--(2.598,2.5)--(1.732,2)--cycle;
\draw[fill=red] (2.598,1.5)--(3.464,2)--(3.464,3)--(2.598,2.5)--cycle;
\draw[fill=blue] (3.464,2)--(4.330,1.5)--(4.330,2.5)--(3.464,3)--cycle;
\draw[fill=red] (4.330,1.5)--(5.196,2)--(5.196,3)--(4.330,2.5)--cycle;
\draw[fill=blue] (5.196,2)--(6.062,1.5)--(6.062,2.5)--(5.196,3)--cycle;
\draw[fill=blue] (6.062,1.5)--(6.928,1)--(6.928,2)--(6.062,2.5)--cycle;
\draw[fill=red] (6.928,1)--(7.794,1.5)--(7.794,2.5)--(6.928,2)--cycle;
\draw[fill=green] (0.866,2.5)--(1.732,2)--(2.598,2.5)--(1.732,3)--cycle;
\draw[fill=green] (3.464,3)--(4.330,2.5)--(5.196,3)--(4.330,3.5)--cycle;
\draw[fill=green] (6.062,2.5)--(6.928,2)--(7.794,2.5)--(6.928,3)--cycle;
\draw[fill=blue] (1.732,3)--(2.598,2.5)--(2.598,3.5)--(1.732,4)--cycle;
\draw[fill=red] (2.598,2.5)--(3.464,3)--(3.464,4)--(2.598,3.5)--cycle;
\draw[fill=red] (3.464,3)--(4.330,3.5)--(4.330,4.5)--(3.464,4)--cycle;
\draw[fill=blue] (4.330,3.5)--(5.196,3)--(5.196,4)--(4.330,4.5)--cycle;
\draw[fill=blue] (5.196,3)--(6.062,2.5)--(6.062,3.5)--(5.196,4)--cycle;
\draw[fill=red] (6.062,2.5)--(6.928,3)--(6.928,4)--(6.062,3.5)--cycle;
\draw[fill=green] (1.732,4)--(2.598,3.5)--(3.464,4)--(2.598,4.5)--cycle;
\draw[fill=green] (2.598,4.5)--(3.464,4)--(4.330,4.5)--(3.464,5)--cycle;
\draw[fill=green] (4.330,4.5)--(5.196,4)--(6.062,4.5)--(5.196,5)--cycle;
\draw[fill=green] (5.196,4)--(6.062,3.5)--(6.928,4)--(6.062,4.5)--cycle;
\draw[fill=blue] (3.464,5)--(4.330,4.5)--(4.330,5.5)--(3.464,6)--cycle;
\draw[fill=red] (4.330,4.5)--(5.196,5)--(5.196,6)--(4.330,5.5)--cycle;
\draw[fill=green] (3.464,6)--(4.330,5.5)--(5.196,6)--(4.330,6.5)--cycle;
\end{tikzpicture}}
\end{equation}
\end{subequations}
\smallskip

A plane partition $\pi$ of shape $\la$ is symmetric if $\la=\la'$
and $\pi_{ij}=\pi_{ji}$ for all $(i,j)\in\la$.
Clearly, the plane partition \eqref{Eq_PP} is symmetric. 

We write $\pi\subseteq \mathrm{B}(n,p,m)$ if the plane partition $\pi$ 
fits in a box $\mathrm{B}(n,p,m)$ of size $n\times p\times m$, i.e., if
the shape of $\pi$ is contained in $p^n$ and no part of $\pi$ exceeds $m$. 
For symmetric plane partitions we may without loss of generality
assume that $p=n$.

MacMahon's famous conjecture \cite{MacMahon98,MacMahon16}
for the generating function of symmetric plane partitions in 
$\mathrm{B}(n,n,m)$ is given by equation \eqref{Eq_MacMahon} of the 
introduction.
The conjecture was proven, independently and almost simultaneously,
by Andrews \cite{Andrews77,Andrews78} and Macdonald \cite{Macdonald79},
almost 80 years after it was first posed in 1898.
Andrew's proof relied on earlier work of Bender and Knuth \cite{BK72},
who obtained two expressions for the generating function \eqref{Eq_GF-pp}
as a determinant over $q$-binomial coefficients, one for even $m$
and one for odd $m$.
Using basic hypergeometric series and clever manipulations of determinants,
Andrews evaluated both determinants, thereby confirming the
conjecture. 
As already mentioned in the introduction, Macdonald's proof \cite{Macdonald79}
(see also \cite{Proctor86} by Proctor for similar ideas) relied first of all 
on the fact that the generating function for symmetric plane partitions can be 
expressed as a sum over specialised Schur functions as 
in \eqref{Eq_GF-pp-Schur}.
This was first noted by Gordon \cite{Gordon71}, who observed the 
equivalent fact that the generating function for symmetric plane partitions
contained in $\mathrm{B}(n,n,m)$ is equal to the generating function for 
column-strict plane partitions contained in $B(\infty,m,2n-1)$, 
all of whose parts are odd. By the description \eqref{Eq_Schur-T}
of the Schur function in terms of semistandard Young tableaux, this 
immediately implies \eqref{Eq_GF-pp-Schur}.
As shown by Okounkov and Reshetikhin, Gordon's observation is 
a simple consequence of the bijective correspondence between symmetric
plane partitions of size $k$ contained in $\mathrm{B}(n,n,m)$ and 
sequences of interlacing partitions
\[
\mu^{(n-1)}\prec\cdots\prec\mu^{(1)}\prec\mu^{(0)}=:\la
\]
such that $\la_1\leqslant m$ and
$\abs{\la}+2\abs{\mu^{(1)}}+\cdots+2\abs{\mu^{(n-1)}}=k$.
Here the `$i$th diagonal slice', $\mu^{(i)}$, is simply given by
$\mu^{(i)}=(\pi_{1,i+1},\pi_{2,i+2},\pi_{3,i+3},\dots)$. Stacking
the slices with $\la$ as base, and all other slices repeated once,
results in a column-strict plane partition of shape $\la$ all of
whose parts are odd and at most $2n-1$, see \cite{OR03} for details.
The next step in Macdonald's proof was to recognise that
\[
\sum_{\substack{\la\\[1pt]\la_1\leqslant m}}s_{\la}(q,q^3,\dots,q^{2n-1})
=\prod_{i=1}^n \frac{1-q^{m+2i-1}}{1-q^{2i-1}}
\prod_{1\leqslant i<j\leqslant n} \frac{1-q^{2(m+i+j-1)}}{1-q^{2(i+j-1)}}
\]
is a specialisation of the bounded Littlewood identity \eqref{Eq_MacB}.
To prove the latter, he developed a method based on partial fraction 
expansions \cite[pages 232--234]{Macdonald95}, which he used to prove
the more general\footnote{By Lemma~\ref{Lem_littlelemmaB}, this is equivalent
to the $t_2=0$ case of Theorem~\ref{THM_BOUNDED7}.}
\begin{equation}\label{Eq_Macdonald-HL}
\sum_{\substack{\la \\[1pt] \la_1\leqslant m}} P_{\la}(x;t)
=\sum_{\varepsilon\in \{\pm 1\}^n}\Phi(x^{\varepsilon};t,0,-1)
\prod_{i=1}^n x_i^{(1-\varepsilon_i)m/2}.
\end{equation}

What we will show in the remainder of this section is that the method 
developed in Section~\ref{Sec_Method}
implies that \eqref{Eq_MacB} and hence MacMahon's formula for symmetric
plane partitions in a box is a consequence of the fact that
$\big(\mathrm{GL}(n,\mathbb{R}),\mathrm{O}(n)\big)$ is a Gelfand pair.
Because it is somewhat simpler to handle, we will first discuss a closely
related theorem for symmetric plane partitions, due to 
Proctor \cite[Theorem 1, (CYH)]{Proctor90} and 
Stembridge \cite[Corollary 4.3, (b)]{Stembridge90}.

\begin{theorem}\label{Thm_PS90}
The generating function for symmetric plane partitions 
$\pi\subseteq \mathrm{B}(n,n,2m)$ such that the parts 
$\pi_{ii}$ \textup{(}$i\geqslant 1$\textup{)} along the main diagonal 
are even is given by
\begin{equation}\label{Eq_PS90}
\sum_{\substack{\pi\subseteq\mathrm{B}(n,n,2m) \\[1pt] \pi \text{ symmetric} 
\\[1.5pt] \pi_{ii} \text{ even}}} 
q^{\abs{\pi}}=
\prod_{i=1}^n \frac{1-q^{2m+2i}}{1-q^{2i}}
\prod_{1\leqslant i<j\leqslant n} \frac{1-q^{2(2m+i+j)}}{1-q^{2(i+j)}}.
\end{equation}
\end{theorem}
The correspondence between symmetric plane partitions and interlacing 
partitions still holds, but now the parts of the zeroth slice, $\la$,
must be even.
Also, because $m$ has been replaced by $2m$, $\la_1\leqslant 2m$.
It thus follows that the generating function on the left-hand side
of \eqref{Eq_PS90} is given by
\[
\sum_{\substack{\la \text{ even} \\[1pt] \la_1\leqslant 2m}}
s_{\la}(q,q^3,\dots,q^{2n-1}),
\]
so that it remains to prove that
\[
\sum_{\substack{\la \text{ even} \\[1pt] \la_1\leqslant 2m}}
s_{\la}(q,q^3,\dots,q^{2n-1})=
\prod_{i=1}^n \frac{1-q^{2m+2i}}{1-q^{2i}}
\prod_{1\leqslant i<j\leqslant n} \frac{1-q^{2(2m+i+j)}}{1-q^{2(i+j)}}.
\]
Again there is a lift to a Littlewood identity, which in this case is 
the D\'esarm\'enien--Proctor--Stembridge determinant formula 
\eqref{Eq_DPS}. For our purposes we write this in terms of 
symplectic Schur functions as
\begin{equation}\label{Eq_Schur-Schursymplectic}
\sum_{\substack{\la \text{ even}\\[1pt] \la_1\leqslant 2m}} s_{\la}(x)
=(x_1\cdots x_n)^m \symp_{2n,m^n}(x).
\end{equation}

If $G$ is a Lie group and $K$ a compact subgroup
such that, for all (continuous and locally convex) irreducible 
representation $\rho$ of $G$, the $K$-invariant subspace $\rho^K$ 
has dimension at most $1$, then $(G,K)$ is called a Gelfand 
pair~\cite{Bump04,Macdonald95}.
Hence, for $(G,K)$ a Gelfand pair and $\chi$ 
a character of an irreducible representation of $G$,
\begin{equation}\label{Eq_GP}
\int_K \chi(g) \dup g =0 \text{ or } 1,
\end{equation}
where the integration is with respect to normalised Haar 
measure on $K$.
The Gelfand pair relevant to Theorem~\ref{Thm_PS90} is 
$G=\mathrm{GL}\big(n,\mathbb{H})$ and $K=U(n,\mathbb{H})$,
the group of invertible $n\times n$ matrices over the 
division ring of quaternions and the quaternionic unitary
group respectively, see \cite[pages 446--456]{Macdonald95}. 
The group $U(n,\mathbb{H})$ is isomorphic to the compact 
symplectic group, $\mathrm{Sp}(n)$, and \eqref{Eq_GP} becomes
\cite[page 451]{Macdonald95}
\begin{equation}\label{Eq_Weyl}
\int_{\mathrm{Sp}(n)} s_{\la}(g) \dup g=
\begin{cases}
1 & \text{if $\la'$ is even}, \\
0 & \text{otherwise},
\end{cases}
\end{equation}
where $\la$ is a partition of length at most $2n$.
By Weyl's integration formula \cite[page 443]{FH91}, 
for $G$ a compact Lie group and $f$ a class function,
\begin{equation}\label{Eq_GtoT}
\int_G f(g) \dup g=\frac{1}{\abs{W}} \int_T f(t) \, 
\abs{\Delta(t)}^2 \, \dup t,
\end{equation}
where $T$ is a maximal torus in $G$, the integration on the right is
with respect to normalised Haar measure on $T$ and 
$\Delta(t)$ is the $G$-Vandermonde determinant
\[
\Delta(\eup^h)=\prod_{\alpha>0} \big(\eup^{\frac{1}{2}\alpha(h)}
-\eup^{-\frac{1}{2}\alpha(h)}\big).
\]
Applying this to \eqref{Eq_Weyl} yields (see e.g., \cite{RV07})
\begin{multline}\label{Eq_Vanishingintegral1}
\int_{\mathbb{T}^n} s_{\la}(x_1^{\pm},\dots,x_n^{\pm})
\prod_{i=1}^n \bigabs{x_i-x_i^{-1}}^2 \\ \times
\prod_{1\leqslant i<j\leqslant n} \bigabs{x_i+x_i^{-1}-x_j-x_j^{-1}}^2
\dup T(x)=\chi(\la' \text{ even}),
\end{multline}
with $\dup T(x)$ as \eqref{Eq_measure}. This may also be written as
\begin{equation}\label{Eq_Vanishingintegral2}
\int_{\mathbb{T}^n} s_{\la}(x_1^{\pm},\dots,x_n^{\pm})
\prod_{i=1}^n (1-x_i^{\pm 2})
\prod_{1\leqslant i<j\leqslant n} (1-x_i^{\pm} x_j^{\pm}) 
\dup T(x)=\chi(\la' \text{ even}),
\end{equation}
which is precisely the $q=t$ case of \eqref{Eq_Q1} for $T^n$, i.e., 
the $q=t$ case of the $\mathrm{U}(2n)/\mathrm{Sp}(2n)$ 
vanishing integral.\footnote{The $\mathrm{Sp}(n)$ appearing in 
\eqref{Eq_Weyl} versus the $\mathrm{Sp}(2n)$ used in the naming of
the vanishing integral is due to a difference in conventions 
between \cite{Macdonald95} and \cite{Rains05,RV07}.}
Since \eqref{Eq_Schur-Schursymplectic} is the $q=t$ case of
Theorem~\ref{THM_BOUNDED1} for $a=0$, and since we showed in 
Section~\ref{Sec_Method} that the latter is a consequence of the full 
$\mathrm{U}(2n)/\mathrm{Sp}(2n)$ vanishing integral, it follows that 
\eqref{Eq_Vanishingintegral2} implies \eqref{Eq_Schur-Schursymplectic}.

\medskip
To use Gelfand pairs to prove \eqref{Eq_MacB}, and hence MacMahon's
conjecture, is slightly more involved.
This time the prerequisite pair is 
$\big(\mathrm{GL}(n,\mathbb{R}),\mathrm{O}(n)\big)$, for which
\eqref{Eq_GP} takes the form
\begin{equation}\label{Eq_Weyl2}
\int_{\mathrm{O}(n)} s_{\la}(g) \dup g=
\begin{cases}
1 & \text{if $\la$ is even}, \\
0 & \text{otherwise},
\end{cases}
\end{equation}
where $l(\la)\leqslant n$, see \cite[pages 414--424]{Macdonald95}.
Again one can use \eqref{Eq_GtoT} to integrate over
a torus and obtain explicit integrals \`{a} la
\eqref{Eq_Vanishingintegral2}. Because of the distinction
between $\mathrm{O}(2n)$ and $\mathrm{O}(2n+1)$, this leads to
four integrals, see \cite[Equations~(1.1)--(1.4)]{RV07}.
Of these four we require
\begin{multline}\label{Eq_RV13}
\int_{\mathbb{T}^n} s_{\la}(x_1^{\pm},\dots,x_n^{\pm},1)
\prod_{i=1}^n \abs{1-x_i}^2 \\ \times
\prod_{1\leqslant i<j\leqslant n} \bigabs{x_i+x_i^{-1}-x_j-x_j^{-1}}^2
\dup T(x)=\chi(\tilde{\la} \text{ even}),
\end{multline}
where $l(\la)\leqslant 2n+1$ and
$\tilde{\la}:=(\la_1-\la_{2n+1},\dots,\la_{2n}-\la_{2n+1},0)$,
as well as
\begin{multline}\label{Eq_RV132}
\int_{\mathbb{T}^{n-1}} s_{\mu}(x_1^{\pm},\dots,x_{n-1}^{\pm},\pm 1)
\prod_{i=1}^{n-1} \abs{x_i-x_i^{-1}}^2 \\ \times
\prod_{1\leqslant i<j\leqslant n-1} \bigabs{x_i+x_i^{-1}-x_j-x_j^{-1}}^2
\dup T(x)=\chi(\tilde{\mu} \text{ even}),
\end{multline}
where $l(\mu)\leqslant 2n$ and
$\tilde{\mu}:=(\mu_1-\mu_{2n},\dots,\mu_{2n}-\mu_{2n},0)$.

It is not difficult to show that \eqref{Eq_RV13} and \eqref{Eq_RV132} 
imply the following closely related integrals.

\begin{lemma}
For $\mu$ a partition of length at most $2n$,
\begin{equation}\label{Eq_Vanishingintegral3}
\int_{\mathbb{T}^n} s_{\mu}(x_1^{\pm},\dots,x_n^{\pm})
\prod_{i=1}^n (1-x_i^{\pm}) 
\prod_{1\leqslant i<j\leqslant n} (1-x_i^{\pm}x_j^{\pm})
\dup T(x)=(-1)^{\abs{\mu}},
\end{equation}
and for $\mu$ a partition of length at most $2n-1$,
\begin{multline}\label{Eq_Vanishingintegral4}
\int_{\mathbb{T}^{n-1}} s_{\mu}(x_1^{\pm},\dots,x_{n-1}^{\pm},-1)
\prod_{i=1}^{n-1} (1-x_i^{\pm 2})  \\ \times
\prod_{1\leqslant i<j\leqslant n-1} (1-x_i^{\pm}x_j^{\pm})
\dup T(x)=(-1)^{\abs{\mu}}.
\end{multline}
\end{lemma}

\begin{proof}
Because the proofs are identical we only consider 
\eqref{Eq_Vanishingintegral3}, and leave \eqref{Eq_Vanishingintegral4}
to the reader.

Denote the integral in \eqref{Eq_Vanishingintegral3} by $I_{\mu}^{(n)}$.
From specialising $x_{n+1}=1$ in the inverse of the branching rule for 
Schur functions (see e.g., \cite[Equation (5.49)]{BR01})
\[
s_{\mu}(x_1,\dots,x_n)=\sum_{\nu'\prec\mu'} (-x_{n+1})^{\abs{\mu/\nu}}
s_{\nu}(x_1,\dots,x_{n+1}),
\]
it follows that 
\begin{multline*}
I_{\mu}{(n)}=\sum_{\nu'\prec\mu'} (-1)^{\abs{\mu/\nu}}
\int_{\mathbb{T}^n} s_{\nu}(x_1^{\pm},\dots,x_n^{\pm},1)
\prod_{i=1}^n (1-x_i^{\pm})  \\ \times
\prod_{1\leqslant i<j\leqslant n} (1-x_i^{\pm}x_j^{\pm})
\dup T(x)
\end{multline*}
The left-hand side of \eqref{Eq_RV13} can be written as
\[
\int_{\mathbb{T}^n} s_{\la}(x_1^{\pm},\dots,x_n^{\pm},1)
\prod_{i=1}^n (1-x_i^{\pm}) 
\prod_{1\leqslant i<j\leqslant n} (1-x_i^{\pm}x_j^{\pm})
\dup T(x),
\]
so that we can compute each of the integrals in the sum.
Thus
\[
I_{\mu}{(n)}=\sum_{\nu'\prec\mu'} (-1)^{\abs{\mu/\nu}}
\chi(\tilde{\nu} \text{ even}).
\]
Since $\nu\subseteq\mu$ and $l(\mu)\leqslant 2n$, we have
$\nu_{2n+1}=0$, resulting in
\[
I_{\mu}{(n)}=\sum_{\nu'\prec\mu'} (-1)^{\abs{\mu/\nu}}
\chi(\nu \text{ even}).
\]
By a repeat of the argument following \eqref{Eq_verticalstrip-even},
it follows that $\nu$ is completely fixed as $\nu_i=2\floor{\mu_i/2}$
(compare with \eqref{Eq_mufloorla}), completing the proof.
\end{proof}

To finally show that this implies the bounded Littlewood identity 
\eqref{Eq_MacB}, and hence MacMahon's conjecture, we note 
that \eqref{Eq_MacB} is the $q=t$ specialisation of 
Theorem~\ref{Thm_bounded4}.
This theorem was proved in Section~\ref{Sec_Method} using the 
virtual Koornwinder integral \eqref{Eq_IK-new} with $T=t^m$,
where $m$ is an integer or half-integer. Setting $q=t$ and assuming
that $m$ is an integer, say $n$, this is the integral
\[
I_K^{(n)}(s_{\mu};q,q;-1,q,\pm q^{1/2})=(-1)^{\abs{\mu}},
\]
which is nothing but \eqref{Eq_Vanishingintegral3} in disguise.
On the other hand, setting $q=t$ and assuming that $m=n-1/2$ 
is a half-integer, we may use Lemma~\ref{Lem_IK} to write this
as the integral
\[
I_K^{(n-1)}(s_{\mu}[x+\varepsilon];q,q;\pm q^{1/2},\pm q)=(-1)^{\abs{\mu}}.
\]
This time this may be may be recognised as \eqref{Eq_Vanishingintegral4}.

\begin{remark}
There are many parallels between the material presented in this section 
and the work of Baik and the first author on algebraic aspects of
increasing subsequences \cite{BR01}.
In particular, \cite[Theorem 5.2]{BR01} contains the two 
Littlewood-type identities
\begin{align*}
\sum_{\substack{\la \text{ even} \\[1pt] l(\la)\leqslant n}}
s_{\la}(x)&=
\int_{\textrm{O}(n)} \det\big(\sigma_g(x)\big) \dup g  \\
\sum_{\substack{\la' \text{ even} \\[1pt] l(\la)\leqslant 2n}}
s_{\la}(x)&=
\int_{\textrm{Sp}(n)} \det\big(\sigma_g(x)\big) \dup g,
\end{align*}
where $x=(x_1,x_2,\dots)$ and
$\sigma_z(x)$ is defined in \eqref{Eq_sigmaz}, so that
\[
\det\big(\sigma_g(x)\big)=\exp\bigg(\sum_{r\geqslant 1}
\frac{p_r(x)}{r} \Trace(g^r) \bigg).
\]
\end{remark}

\section{Character identities for affine Lie algebras}\label{Sec_char}

\subsection{Main results}\label{Sec_char-results}
We will only define a minimum of notation needed to state our results, 
and for a more comprehensive introduction to the representation theory of 
affine Lie algebras we refer the reader to \cite{Carter05,Kac90}.

\begin{figure}[t]\label{Fig}
\begin{center}
\begin{tikzpicture}[scale=0.6]
\draw (-1,2.2) node {$\mathrm{B}_n^{(1)}$};
\draw (1,2.2)--(1,3.2);
\draw (0,2.2)--(5,2.2);
\draw (5,2.27)--(6,2.27);
\draw (5,2.13)--(6,2.13);
\draw (5.4,2.4)--(5.6,2.2)--(5.4,2);
\foreach \x in {0,...,6} \draw[fill=blue] (\x,2.2) circle (0.08cm);
\draw[fill=blue] (1,3.2) circle (0.08cm);
\draw (0,1.8) node {$\sc \alpha_1$};
\draw (1,3.6) node {$\sc \alpha_0$};
\draw (1,1.8) node {$\sc \alpha_2$};
\draw (6,1.8) node {$\sc \alpha_n$};
\draw (0.75,3.2) node {$\sc 1$};
\draw (0,2.6) node {$\sc 1$};
\draw (0.75,2.6) node {$\sc 2$};
\draw (2,2.6) node {$\sc 2$};
\draw (3,2.6) node {$\sc 2$};
\draw (4,2.6) node {$\sc 2$};
\draw (5,2.6) node {$\sc 2$};
\draw (6,2.6) node {$\sc 2$};
\draw (7.8,2.2) node {$\mathrm{A}_{2n-1}^{(2)}$};
\draw (10,2.2)--(10,3.2);
\draw (9,2.2)--(14,2.2);
\draw (14.6,2.4)--(14.4,2.2)--(14.6,2);
\draw (14,2.27)--(15,2.27);
\draw (14,2.13)--(15,2.13);
\draw[fill=blue] (10,3.2) circle (0.08cm);
\foreach \x in {9,...,15} \draw[fill=blue] (\x,2.2) circle (0.08cm);
\draw (9,1.8) node {$\sc \alpha_1$};
\draw (10,3.6) node {$\sc \alpha_0$};
\draw (10,1.8) node {$\sc \alpha_2$};
\draw (15,1.8) node {$\sc \alpha_n$};
\draw (9.75,3.2) node {$\sc 1$};
\draw (9,2.6) node {$\sc 1$};
\draw (9.75,2.6) node {$\sc 2$};
\draw (11,2.6) node {$\sc 2$};
\draw (12,2.6) node {$\sc 2$};
\draw (13,2.6) node {$\sc 2$};
\draw (14,2.6) node {$\sc 2$};
\draw (15,2.6) node {$\sc 1$};
\draw (-1,0) node {$\mathrm{B}_n^{(1)\dagger}$};
\draw (5,0)--(5,1);
\draw (1,0)--(6,0);
\draw (0,0.07)--(1,0.07);
\draw (0,-0.07)--(1,-0.07);
\draw (0.6,0.2)--(0.4,0)--(0.6,-0.2);
\foreach \x in {0,...,6} \draw[fill=blue] (\x,0) circle (0.08cm);
\draw[fill=blue] (5,1) circle (0.08cm);
\draw (0,-0.4) node {$\sc \alpha_0$};
\draw (1,-0.4) node {$\sc \alpha_1$};
\draw (5,1.4) node {$\sc \alpha_n$};
\draw (6,-0.4) node {$\sc \alpha_{n-1}$};
\draw (0,0.4) node {$\sc 2$};
\draw (1,0.4) node {$\sc 2$};
\draw (2,0.4) node {$\sc 2$};
\draw (3,0.4) node {$\sc 2$};
\draw (4,0.4) node {$\sc 2$};
\draw (6,0.4) node {$\sc 1$};
\draw (4.75,1) node {$\sc 1$};
\draw (4.75,0.4) node {$\sc 2$};
\draw (7.7,0) node {$\mathrm{A}_{2n-1}^{(2)\dagger}$};
\draw (14,0)--(14,1);
\draw (10,0)--(15,0);
\draw (9.4,0.2)--(9.6,0)--(9.4,-0.2);
\draw (9,-0.07)--(10,-0.07);
\draw (9,0.07)--(10,0.07);
\draw[fill=blue] (14,1) circle (0.08cm);
\foreach \x in {9,...,15} \draw[fill=blue] (\x,0) circle (0.08cm);
\draw (9,-0.4) node {$\sc \alpha_0$};
\draw (10,-0.4) node {$\sc \alpha_1$};
\draw (14,1.4) node {$\sc \alpha_n$};
\draw (15,-0.4) node {$\sc \alpha_{n-1}$};
\draw (9,0.4) node {$\sc 1$};
\draw (10,0.4) node {$\sc 2$};
\draw (11,0.4) node {$\sc 2$};
\draw (12,0.4) node {$\sc 2$};
\draw (13,0.4) node {$\sc 2$};
\draw (13.75,1) node {$\sc 1$};
\draw (13.75,0.4) node {$\sc 2$};
\draw (15,0.4) node {$\sc 1$};
\draw (-1,-2) node {$\mathrm{C}_n^{(1)}$};
\draw (0,-1.93)--(1,-1.93);
\draw (0,-2.07)--(1,-2.07);
\draw (0.4,-1.8)--(0.6,-2)--(0.4,-2.2);
\draw (1,-2)--(5,-2);
\draw (5.6,-1.8)--(5.4,-2)--(5.6,-2.2);
\draw (5,-1.93)--(6,-1.93);
\draw (5,-2.07)--(6,-2.07);
\foreach \x in {0,...,6} \draw[fill=blue] (\x,-2) circle (0.08cm);
\draw (0,-2.4) node {$\sc \alpha_0$};
\draw (1,-2.4) node {$\sc \alpha_1$};
\draw (6,-2.4) node {$\sc \alpha_n$};
\draw (0,-1.6) node {$\sc 1$};
\draw (1,-1.6) node {$\sc 2$};
\draw (2,-1.6) node {$\sc 2$};
\draw (3,-1.6) node {$\sc 2$};
\draw (4,-1.6) node {$\sc 2$};
\draw (5,-1.6) node {$\sc 2$};
\draw (6,-1.6) node {$\sc 1$};
\draw (7.8,-2) node {$\mathrm{D}_{n+1}^{(2)}$};
\draw (9,-1.93)--(10,-1.93);
\draw (9,-2.07)--(10,-2.07);
\draw (9.6,-1.8)--(9.4,-2)--(9.6,-2.2);
\draw (10,-2)--(14,-2);
\draw (14.4,-1.8)--(14.6,-2)--(14.4,-2.2);
\draw (14,-1.93)--(15,-1.93);
\draw (14,-2.07)--(15,-2.07);
\foreach \x in {9,...,15} \draw[fill=blue] (\x,-2) circle (0.08cm);
\draw (9,-2.4) node {$\sc \alpha_0$};
\draw (10,-2.4) node {$\sc \alpha_1$};
\draw (15,-2.4) node {$\sc \alpha_n$};
\draw (9,-1.6) node {$\sc 1$};
\draw (10,-1.6) node {$\sc 1$};
\draw (11,-1.6) node {$\sc 1$};
\draw (12,-1.6) node {$\sc 1$};
\draw (13,-1.6) node {$\sc 1$};
\draw (14,-1.6) node {$\sc 1$};
\draw (15,-1.6) node {$\sc 1$};
\draw (-1,-4) node {$\mathrm{A}_2^{(2)}$};
\draw (0,-3.9)--(1,-3.9);
\draw (0,-3.966)--(1,-3.966);
\draw (0,-4.033)--(1,-4.033);
\draw (0,-4.1)--(1,-4.1);
\draw (0.6,-3.8)--(0.4,-4)--(0.6,-4.2);
\foreach \x in {0,...,1} \draw[fill=blue] (\x,-4) circle (0.08cm);
\draw (0,-4.4) node {$\sc \alpha_0$};
\draw (1,-4.4) node {$\sc \alpha_1$};
\draw (0,-3.6) node {$\sc 2$};
\draw (1,-3.6) node {$\sc 1$};
\draw (8,-4) node {$\mathrm{A}_2^{(2)\dagger}$};
\draw (9,-3.9)--(10,-3.9);
\draw (9,-3.966)--(10,-3.966);
\draw (9,-4.033)--(10,-4.033);
\draw (9,-4.1)--(10,-4.1);
\draw (9.4,-4.2)--(9.6,-4)--(9.4,-3.8);
\foreach \x in {9,...,10} \draw[fill=blue] (\x,-4) circle (0.08cm);
\draw (9,-4.4) node {$\sc \alpha_0$};
\draw (10,-4.4) node {$\sc \alpha_1$};
\draw (9,-3.6) node {$\sc 1$};
\draw (10,-3.6) node {$\sc 2$};
\draw (-1,-6) node {$\mathrm{A}_{2n}^{(2)}$};
\draw (0,-5.93)--(1,-5.93);
\draw (0,-6.07)--(1,-6.07);
\draw (0.6,-5.8)--(0.4,-6)--(0.6,-6.2);
\draw (1,-6)--(5,-6);
\draw (5.6,-5.8)--(5.4,-6)--(5.6,-6.2);
\draw (5,-5.93)--(6,-5.93);
\draw (5,-6.07)--(6,-6.07);
\foreach \x in {0,...,6} \draw[fill=blue] (\x,-6) circle (0.08cm);
\draw (0,-6.4) node {$\sc \alpha_0$};
\draw (1,-6.4) node {$\sc \alpha_1$};
\draw (6,-6.4) node {$\sc \alpha_n$};
\draw (0,-5.6) node {$\sc 2$};
\draw (1,-5.6) node {$\sc 2$};
\draw (2,-5.6) node {$\sc 2$};
\draw (3,-5.6) node {$\sc 2$};
\draw (4,-5.6) node {$\sc 2$};
\draw (5,-5.6) node {$\sc 2$};
\draw (6,-5.6) node {$\sc 1$};
\draw (8,-6) node {$\mathrm{A}_{2n}^{(2)\dagger}$};
\draw (9,-5.93)--(10,-5.93);
\draw (9,-6.07)--(10,-6.07);
\draw (9.4,-5.8)--(9.6,-6)--(9.4,-6.2);
\draw (10,-6)--(14,-6);
\draw (14.4,-5.8)--(14.6,-6)--(14.4,-6.2);
\draw (14,-5.93)--(15,-5.93);
\draw (14,-6.07)--(15,-6.07);
\foreach \x in {9,...,15} \draw[fill=blue] (\x,-6) circle (0.08cm);
\draw (9,-6.4) node {$\sc \alpha_0$};
\draw (10,-6.4) node {$\sc \alpha_1$};
\draw (15,-6.4) node {$\sc \alpha_n$};
\draw (9,-5.6) node {$\sc 1$};
\draw (10,-5.6) node {$\sc 2$};
\draw (11,-5.6) node {$\sc 2$};
\draw (12,-5.6) node {$\sc 2$};
\draw (13,-5.6) node {$\sc 2$};
\draw (14,-5.6) node {$\sc 2$};
\draw (15,-5.6) node {$\sc 2$};
\end{tikzpicture}
\end{center}
\caption{The Dynkin diagrams of the ``$\mathrm{BC}_n$-type'' affine
Lie algebras with labelling of vertices by simple roots
$\alpha_0,\dots,\alpha_n$ and marks $a_0,\dots,a_n$.}
\end{figure}
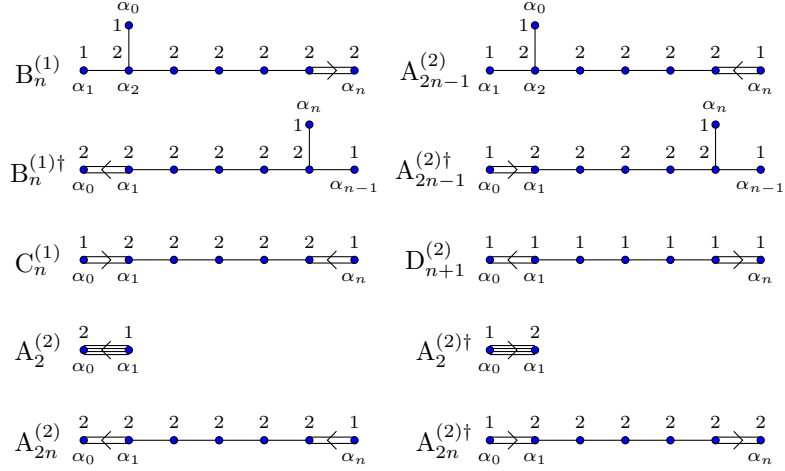
We will be concerned with affine Lie algebras $\mathfrak{g}$ of 
``$\mathrm{BC}_n$ type'',
that is, $\mathfrak{g}=X_N^{(r)}$ with $X_N^{(r)}$ one of
$\mathrm{B}_n^{(1)}$, $\mathrm{C}_n^{(1)}$, $\mathrm{A}_{2n-1}^{(2)}$,
$\mathrm{A}_{2n}^{(2)}$ and $\mathrm{D}_{n+1}^{(2)}$.
Using standard labelling, for these affine Lie algebras
the classical part is either $\mathrm{B}_n$ or $\mathrm{C}_n$.
The relevant Dynkin diagrams are shown in Figure~\ref{Fig}.
For $\mathrm{B}_n^{(1)}$, $\mathrm{A}_{2n-1}^{(2)}$ and
$\mathrm{A}_{2n}^{(2)}$ we also use the nonstandard labelling of
simple roots, indicated by the customary $\dagger$,
obtained by mapping $\alpha_i\mapsto \alpha_{n-i}$ for 
$0\leqslant i\leqslant n$.
Apart from the simple roots $\alpha_0,\dots,\alpha_n$ and fundamental weights
$\fwa_0,\dots,\fwa_n$ we need the null root $\delta$
given by $\delta=\sum_{i=0}^n a_i \alpha_i$, with the $a_i$ the
marks of $\mathfrak{g}$, see Figure~\ref{Fig}.
We are interested in representations of $\mathfrak{g}$ known as 
integrable highest-weight modules. If $P_{+}$ is the set of dominant
integral weights $P_{+}=\sum_{i=0}^n \Nat\fwa_i$
then these modules are indexed by $\Lambda\in P_{+}$, and will be denoted by 
$V(\Lambda)$ in the following.
The character of $V(\Lambda)$ can be computed in closed form by the
Weyl--Kac formula:
\begin{equation}\label{Eq_Weyl-Kac}
\ch V(\Lambda)=
\frac{\sum_{w\in W}\sgn(w) \eup^{w(\Lambda+\rho)-\rho}}
{\prod_{\alpha>0}(1-\eup^{-\alpha})^{\mult(\alpha)}}.
\end{equation}
Here $W$ is the Weyl group of $\gfrak$, $\sgn(w)$ the signature of
$w\in W$, $\rho=\fwa_0+\cdots+\fwa_n$ the Weyl vector,
and $\mult(\alpha)$ the multiplicity of $\alpha$. In the denominator,
the product runs over the positive roots of $\gfrak$.

Below we prove combinatorial character formulas for 
\begin{equation}\label{Eq_chim}
\chi_m(\mathfrak{g}):=\eup^{-m\fwa_0} \ch V(m\fwa_0)
\end{equation}
for $\mathfrak{g}$ one of 
$\mathrm{B}_n^{(1)},\mathrm{B}_n^{(1)\dagger},\mathrm{C}_n^{(1)},
\mathrm{A}_{2n-1}^{(2)},\mathrm{A}_{2n-1}^{(2)\dagger},\mathrm{A}_{2n}^{(2)},
\mathrm{A}_{2n}^{(2)\dagger},\mathrm{D}_{n+1}^{(2)}$.
Since the diagrams of $\mathrm{C}_n^{(1)}$ and $\mathrm{D}_{n+1}^{(2)}$ are
the same when read from left to right as from right to left,
these two algebras occur only once in the above list.
For $\mathrm{A}_{2n-1}^{(2)\dagger},\mathrm{A}_{2n}^{(2)\dagger}$ and
$\mathrm{D}_{n+1}^{(2)}$, however, we obtain two distinct formulas,
making a total of eleven character formulas.

\medskip

Recall that $P_{\la}^{(R)}$ denotes a Hall--Littlewood polynomial of 
type $R$.
Also recall the definition of $X$ in \eqref{Eq_X}, which may 
be written plethystically as
\[
X=(x_1+x_1^{-1}+\dots+x_n+x_n^{-1}) \,\frac{1-t^N}{1-t}.
\]
We complement this with 
\begin{align}\label{Eq_barX}
\bar{X}&:=(x_1^{\pm},tx_1^{\pm},\dots,t^{N-1}x_1^{\pm},\dots\dots,
x_{n-1}^{\pm},tx_{n-1}^{\pm},\dots,t^{N-1}x_{n-1}^{\pm},1,t,\dots,t^{N-1}) \\
&\hphantom{:}=(x_1+x_1^{-1}+\dots+x_{n-1}+x_{n-1}^{-1}+1)
\,\frac{1-t^N}{1-t}. \notag
\end{align}
Finally, in each of the formulas below $m_0(\la):=\infty$.
\begin{theorem}\label{Thm_een}
Let 
\begin{equation}\label{Eq_xit}
x_i:=\eup^{-\alpha_i-\cdots-\alpha_{n-1}-\alpha_n/2}\quad
(1\leqslant i\leqslant n), \qquad\quad t:=\eup^{-\delta},
\end{equation}
and let $m$ and $n$ be positive integers. Then
\begin{align}\label{Eq_Cn-character}
\chi_m\big(\mathrm{C}_n^{(1)}\big)&=\lim_{N\to\infty} t^{mnN^2} 
P^{(\mathrm{C}_{2nN})}_{m^{2nN}}(t^{1/2}X;t,0) \\[1pt] 
&=\sum_{\substack{\la \textup{ even} \\[1pt] \la_1\leqslant 2m}}
t^{\abs{\la}/2} P'_{\la}(x_1^{\pm},\dots,x_n^{\pm};t) \notag  \\[2mm]
\label{Eq_A2n12-character}
\chi_m\big(\mathrm{A}_{2n-1}^{(2)}\big)&=
(t;t^2)_{\infty} \lim_{N\to\infty} t^{\frac{1}{2}mnN^2} 
P^{(\mathrm{D}_{2nN})}_{\halfm{2nN}}(t^{1/2}X;t)  \\
&=\sum_{\substack{\la' \textup{ even} \\[1pt] \la_1\leqslant m}} 
t^{\abs{\la}/2} P'_{\la}(x_1^{\pm},\dots,x_n^{\pm};t)
\prod_{i=0}^{m-1} (t;t^2)_{m_i(\la)/2} \notag 
\intertext{and}
\label{Eq_A2n2-character}
\chi_m\big(\mathrm{A}_{2n}^{(2)}\big)&=\lim_{N\to\infty} t^{\frac{1}{2}mnN^2} 
P^{(\mathrm{B}_{2nN})}_{\halfm{2nN}}(t^{1/2}X;t,0)  \\[1mm]
&=\sum_{\substack{\la\\[1pt] \la_1\leqslant m}}
t^{\abs{\la}/2} P'_{\la}(x_1^{\pm},\dots,x_n^{\pm};t). \notag
\end{align}
\end{theorem}
\begin{theorem}\label{Thm_vier}
Let
\[
x_i:=\eup^{-\alpha_i-\cdots-\alpha_{n-1}+(\alpha_{n-1}-\alpha_n)/2}
\quad (1\leqslant i\leqslant n), \qquad\quad t:=\eup^{-\delta},
\]
and let $m$ and $n$ be positive integers. Then
\begin{align}\label{Eq_A2n12dagger-character} 
\chi_m\big(\mathrm{A}_{2n-1}^{(2)\dagger}\big)
&=(t;t^2)_{\infty} \lim_{N\to\infty} t^{mnN^2} 
P^{(\mathrm{C}_{2nN})}_{m^{2nN}}(t^{1/2} X;t,t) \\
&=\sideset{}{'}\sum_{\substack{\la \\[1pt] \la_1\leqslant 2m}} 
t^{(\abs{\la}+\op(\la))/2}
P'_{\la}(x_1^{\pm},\dots,x_n^{\pm};t) 
\prod_{i=0}^{2m-1} (t;t^2)_{\ceil{m_i(\la)/2}}, \notag
\end{align}
where the prime in the sum over $\la$ denotes the 
restriction that parts of odd size must have even multiplicity.
\end{theorem}
\begin{theorem}\label{Thm_drie}
Let 
\begin{equation}\label{Eq_xitD}
x_i:=\eup^{-\alpha_i-\cdots-\alpha_n}\quad
(1\leqslant i\leqslant n), 
\end{equation}
and let $m$ and $n$ be positive integers. Then
\begin{align}\label{Eq_A2n2-new}
\chi_m\big(\mathrm{A}_{2n}^{(2)\dagger}\big)&=
\lim_{N\to\infty} t^{mnN^2} 
P^{(\mathrm{BC}_{2nN})}_{m^{2nN}}(t^{1/2} X;t,0,-t^{1/2}) \\[2mm]
&=\sum_{\substack{\la \\[1pt] \la_1\leqslant 2m}} t^{(\abs{\la}+\op(\la))/2}
P'_{\la}(x_1^{\pm},\dots,x_n^{\pm};t), \notag
\end{align}
where $t:=\eup^{-\delta}$, and
\begin{align}\label{Eq_Dn-char2}
\chi_m\big(\mathrm{D}_{n+1}^{(2)}\big)&=(-t^{1/2};t^{1/2})_{\infty}
\lim_{N\to\infty} t^{\frac{1}{2}mnN^2} 
P^{(\mathrm{B}_{2nN})}_{\halfm{2nN}}(t^{1/2}X;t,-t^{1/2})  \\[1mm]
&=\sum_{\substack{\la\\[1pt] \la_1\leqslant m}} t^{\abs{\la}/2} 
P'_{\la}(x_1^{\pm},\dots,x_n^{\pm};t)
\prod_{i=0}^{m-1} (-t^{1/2};t^{1/2})_{m_i(\la)}, \notag
\end{align}
where $t^{1/2}:=\eup^{-\delta}$.
\end{theorem}
The identity \eqref{Eq_Cn-character}, without the limiting expression in 
the middle, was first obtained in \cite[Theorem 1.1; (1.4a)]{BW13}.
Equation \eqref{Eq_A2n2-character}, which is \eqref{Eq_char-A2n2} from the 
introduction, extends \cite[Theorem 1.1; (1.4b)]{BW13} from integer to 
half-integer values of $m$. In the same manner, \eqref{Eq_Dn-char2} 
extends \cite[Theorem 5.4]{BW13}. The identity~\ref{Eq_A2n2-new},
again without the limiting expression on the right, 
is~\cite[Theorem 5.3]{BW13}.

In each of the remaining formulas $\eup^{-\alpha_n}$ is specialised.
\begin{theorem}
Let 
\begin{equation}\label{Eq_xit-2}
x_i:=-\eup^{-\alpha_i-\cdots-\alpha_n}\quad
(1\leqslant i\leqslant n), \qquad\quad t:=\eup^{-\delta},
\end{equation}
and specialise $\eup^{-\alpha_n}\mapsto -1$.
Then, for $m$ and $n$ positive integers,
\begin{align}\label{Eq_A2n2dagger-character}
\chi_m\big(\mathrm{A}_{2n}^{(2)\dagger}\big)&=\lim_{N\to\infty} 
t^{\frac{1}{2}m(2n-1)N^2} 
P^{(\mathrm{C}_{(2n-1)N})}_{m^{(2n-1)N}}(t^{1/2}\bar{X};t,0)  \\[1mm]
&=\sum_{\substack{\la \textup{ even} \\[1pt] \la_1\leqslant 2m}}
t^{\abs{\la}/2} P'_{\la}(x_1^{\pm},\dots,x_{n-1}^{\pm},1;t) \notag \\
\intertext{and} 
\label{Eq_Bn1-character}
\chi_m\big(\mathrm{B}_n^{(1)}\big)
&=(t;t^2)_{\infty} \lim_{N\to\infty} t^{\frac{1}{4}m(2n-1)N^2} 
P^{(\mathrm{D}_{(2n-1)N})}_{\halfm{(2n-1)N}}(t^{1/2}\bar{X};t)  \\
&=\sum_{\substack{\la' \textup{ even}\\[1pt] \la_1\leqslant m}}
t^{\abs{\la}/2} P'_{\la}(x_1^{\pm},\dots,x_{n-1}^{\pm},1;t)
\prod_{i=0}^{m-1} (t;t^2)_{m_i(\la)/2}. \notag
\end{align}
\end{theorem}
\begin{theorem}\label{Thm_Dn2-char}
Let
\begin{equation}\label{Eq_xithalf}
x_i:=-\eup^{-\alpha_i-\cdots-\alpha_n}\quad
(1\leqslant i\leqslant n), \qquad\quad t^{1/2}:=-\eup^{-\delta},
\end{equation}
and specialise $\eup^{-\alpha_n}\mapsto -1$.
Then, for $m$ and $n$ positive integers,
\begin{align}\label{Eq_Dn1-character}
\chi_m\big(\mathrm{D}_{n+1}^{(2)}\big)
&=\lim_{N\to\infty} t^{\frac{1}{4}m(2n-1)N^2} 
P^{(\mathrm{B}_{(2n-1)N})}_{\halfm{(2n-1)N}}(t^{1/2}\bar{X};t,0)  \\[1mm]
&=\sum_{\substack{\la \\[1pt] \la_1\leqslant m}}
t^{\abs{\la}/2} P'_{\la}(x_1^{\pm},\dots,x_{n-1}^{\pm},1;t). \notag
\end{align}
\end{theorem}
\begin{theorem}\label{Thm_negen}
Let
\begin{equation}\label{Eq_xit-Bdag} 
x_i:=\eup^{-\alpha_i-\cdots-\alpha_{n-1}+(\alpha_{n-1}-\alpha_n)/2}
\quad (1\leqslant i\leqslant n), \qquad\quad t:=\eup^{-\delta},
\end{equation}
and specialise $\eup^{-\alpha_n}\mapsto \eup^{-\alpha_{n-1}}$.
Then, for $m$ and $n$ positive integers,
\begin{align}\label{Eq_Bndagger-character}
\chi_m\big(\mathrm{B}_n^{(1)\dagger}\big)
&=(-t^{1/2};t^{1/2})_{\infty} \lim_{N\to\infty} t^{\frac{1}{4}m(2n-1)N^2} 
P^{(\mathrm{B}_{(2n-1)N})}_{\halfm{(2n-1)N}}(t^{1/2}\bar{X};t,-t^{1/2})  \\
&=\sum_{\substack{\la\\[1pt] \la_1\leqslant m}} t^{\abs{\la}/2} 
P'_{\la}(x_1^{\pm},\dots,x_{n-1}^{\pm},1;t) 
\prod_{i=0}^{m-1} (-t^{1/2};t^{1/2})_{m_i(\la)} \notag
\end{align}
and
\begin{align}\label{Eq_char-new2}
\chi_m\big(\mathrm{A}_{2n-1}^{(2)\dagger}\big)
&=\lim_{N\to\infty} t^{m(n-1/2)N^2} P^{(\mathrm{BC}_{(2n-1)N})}_{m^{(2n-1)N}} 
(t^{1/2}\bar{X};t,0,-t^{1/2}) \\[1.5pt]
&=\sum_{\substack{\la \\[1pt] \la_1\leqslant 2m}} t^{(\abs{\la}+\op(\la))/2}
P'_{\la}(x_1^{\pm},\dots,x_{n-1}^{\pm},1;t). \notag
\end{align}
\end{theorem}

\smallskip

Thanks to the Macdonald identities, the characters of certain
one-parameter subfamilies of representations admit product forms.
For example, by Corollary~\ref{Cor_Bdagger} of Appendix~\ref{App_A},
it follows that for
$\mathfrak{g}=\mathrm{B}_n^{(1)\dagger}$ and weights
\begin{equation}\label{Eq_La-rho}
\Lambda=(k-1)\rho+k\omega_0, \qquad k\in\mathbb{Z}_{\geqslant 1},
\end{equation}
we have
\[
\eup^{-\Lambda} \ch V(\Lambda)=
\frac{(t^k;t^k)_{\infty}^{n-1}(t^{2k};t^{2k})_{\infty}}{(t;t)_{\infty}^n}
\prod_{i=1}^n \frac{\theta(t^k x_i^{2k};t^{2k})}{\theta(t^{1/2}x_i;t)}
\prod_{1\leqslant i<j\leqslant n} 
\frac{\theta(x_i^kx_j^{\pm k};t^k)}{\theta(x_ix_j^{\pm};t)},
\]
where $x_1,\dots,x_n$ and $t$ are as in \eqref{Eq_xit-Bdag}.
In much the same way, it follows from the Macdonald identity 
for $\mathrm{C}_n^{(1)}$ that 
for $\mathrm{A}_{2n}^{(2)}$ and weights \eqref{Eq_La-rho},
\[
\eup^{-\Lambda} \ch V(\Lambda)=
\frac{(t^k;t^k)_{\infty}^n}{(t;t)_{\infty}^n}
\prod_{i=1}^n \frac{\theta(x_i^{2k};t^k)}
{\theta(t^{1/2}x_i;t)\theta(x_i^2;t^2)}
\prod_{1\leqslant i<j\leqslant n} 
\frac{\theta(x_i^kx_j^{\pm k};t^k)}{\theta(x_ix_j^{\pm};t)},
\]
where $x_1,\dots,x_n$ and $t$ are given by \eqref{Eq_xit}.
Similarly, by the Macdonald identity for $\mathrm{A}_{2n}^{(2)\dagger}$ 
we have the $\mathrm{D}_{n+1}^{(2)}$ identity
\begin{multline*}
\eup^{-\Lambda} \ch V(\Lambda)=
\frac{(t^k;t^k)_{\infty}^n}{(t;t)_{\infty}^{n-1}(t^{1/2};t^{1/2})_{\infty}} \\
\times \prod_{i=1}^n 
\frac{\theta(x_i^k;t^k)\theta(t^kx_i^{2k};t^{2k})}{\theta(x_i;t^{1/2})}
\prod_{1\leqslant i<j\leqslant n} 
\frac{\theta(x_i^kx_j^{\pm k};t^k)}{\theta(x_ix_j^{\pm};t)},
\end{multline*}
where $\Lambda$ is again given by \eqref{Eq_La-rho},
and $x_1,\dots,x_n$ and $t$ are as in \eqref{Eq_xitD}.

All three families include the basic representation, $V(\fwa_0)$, obtained
by taking $k=1$:
\begin{equation}\label{Eq_basic-rep}
\eup^{-\fwa_0} \ch V(\fwa_0)=\kappa_{\mathfrak{g}}(t)
\prod_{i=1}^n \theta(-t^{1/2}x_i;t),
\end{equation}
where 
\[
\kappa_{\mathfrak{g}}(t):=
\begin{cases} 
(-t;t)_{\infty}
& \text{for $\mathfrak{g}=\mathrm{B}_n^{(1)\dagger}$}, \\[1mm]
1 & \text{for $\mathfrak{g}=\mathrm{A}_{2n}^{(2)}$}, \\[1mm]
(-t^{1/2};t^{1/2})_{\infty}
& \text{for $\mathfrak{g}=\mathrm{D}_{n+1}^{(2)}$}. 
\end{cases}
\]

Using the representation \eqref{Eq_Qp-qbinom} for the modified 
Hall--Littlewood polynomials, each of the character formulas
in Theorems~\ref{Thm_een}--\ref{Thm_drie} can be written as a multiple 
basic hypergeometric series. 
For $\mathrm{B}_n^{(1)\dagger}, \mathrm{A}_{2n}^{(2)}$ and 
$\mathrm{D}_{n+1}^{(2)}$ we can then restrict to the basic 
representation and, as a consistency check,
compare with \eqref{Eq_basic-rep}.
To this end we define $f$ by
\begin{align*}
f(x_1,\dots,x_n;t)&:=
\sum_{\substack{\la\\[1pt] \la_1\leqslant 1}}
t^{\abs{\la}/2} P'_{\la}(x_1,\dots,x_n;t) \\
&\hphantom{:}=\sum_{r=0}^{\infty} \frac{t^{r/2}}{(t;t)_r}\,
Q'_{1^r}(x_1,\dots,x_n;t).
\end{align*}
Replacing $r$ by $r_1$, and using \eqref{Eq_Qp-qbinom} and
\eqref{Eq_glamu}, this yields
\[
f(x_1,\dots,x_n;t)=
\sum_{r_1\geqslant\cdots\geqslant r_n\geqslant 0} 
\frac{t^{r_1/2}}{(t;t)_{r_1}}\,
\prod_{i=1}^n x_i^{r_i-r_{i+1}} t^{\binom{r_i-r_{i+1}}{2}}
\qbin{r_i}{r_{i+1}}_t,
\]
where $r_{n+1}:=0$. Introducing new summation indices
$k_1,\dots,k_n$ by $k_i=r_i-r_{i+1}$, the $n$-fold sum factors 
as\footnote{This may also be proved by taking $m=1$ and $q=0$ in 
\eqref{Eq_Mac-Cauchy-a}, and by carrying out the plethystic substitution
$x\mapsto x/(1-t)$.}
\begin{equation}\label{Eq_f-product}
f(x_1,\dots,x_n;t)=\prod_{i=1}^n\bigg(\sum_{k_i=0}^{\infty}
\frac{(t^{1/2}x_i)^{k_i} t^{\binom{k_i}{2}}}{(t;t)_{k_i}}\bigg)
=\prod_{i=1}^n (-t^{1/2}x_i;t)_{\infty},
\end{equation}
where the second equality follows from \cite[Equation (II.2)]{GR04}
\[
\sum_{k\geqslant 0} \frac{z^k q^{\binom{k}{2}}}{(q;q)_k}=
(-z;q)_{\infty}.
\]
We now observe that for $m=1$ the identities \eqref{Eq_A2n2-character},
\eqref{Eq_Dn-char2} and \eqref{Eq_Bndagger-character} simplify to
\[
\eup^{-\fwa_0} \ch V(\fwa_0)=
\begin{cases}
(-t^{1/2};t^{1/2})_{\infty}
f(x_1^{\pm},\dots,x_{n-1}^{\pm},1;t) & 
\text{for $\mathfrak{g}=\mathrm{B}_n^{(1)\dagger}$} \\[1mm]
f(x_1^{\pm},\dots,x_n^{\pm};t) & 
\text{for $\mathfrak{g}=\mathrm{A}_{2n}^{(2)}$}, \\[1mm]
(-t^{1/2};t^{1/2})_{\infty}
f(x_1^{\pm},\dots,x_n^{\pm};t) & 
\text{for $\mathfrak{g}=\mathrm{D}_{n+1}^{(2)}$}.
\end{cases}
\]
By \eqref{Eq_f-product} this indeed gives \eqref{Eq_basic-rep}
for $\mathrm{A}_{2n}^{(2)}$ and $\mathrm{D}_{n+1}^{(2)}$.
For $\mathrm{B}_n^{(1)\dagger}$ it leads to
\begin{align*}
\eup^{-\fwa_0} \ch V(\fwa_0)
|_{\eup^{-\alpha_n}\mapsto \eup^{-\alpha_{n-1}}}&=
(-t^{1/2};t^{1/2})_{\infty}(-t^{1/2};t)_{\infty} 
\prod_{i=1}^{n-1} \theta(-t^{1/2}x_i;t) \\
&=(-t;t)_{\infty} \prod_{i=1}^n \theta(-t^{1/2}x_i;t)|_{x_n=1},
\end{align*}
again in agreement with \eqref{Eq_basic-rep}.
To show that the $m=1$ case of the second $\mathrm{D}_{n+1}^{(2)}$ identity
\eqref{Eq_Dn1-character} is also in accordance with \eqref{Eq_basic-rep}, 
we replace $x_i\mapsto -x_i$ and $t^{1/2}\mapsto -t^{1/2}$ on the right-hand 
side of \eqref{Eq_basic-rep} and then specialise $x_n=1$.
This yields the expression
\begin{align*}
(-t^{1/2};t)_{\infty} \prod_{i=1}^{n-1} \theta(-t^{1/2}x_i;t)
&=f(x_1^{\pm},\dots,x_{n-1}^{\pm},1;t) \\ 
&=\sum_{\substack{\la \\[1pt] \la_1\leqslant 1}}
t^{\abs{\la}/2} P'_{\la}(x_1^{\pm},\dots,x_{n-1}^{\pm},1;t),
\end{align*}
as required.

\subsection{Proof of Theorems~\ref{Thm_een}--\ref{Thm_negen}}
Below we present proofs of the eleven character formulas of 
the previous section.

Recall the Vandermonde determinants of type $\mathrm{B}_n, \mathrm{C}_n$
and $\mathrm{D}_n$ given in \eqref{Eq_VdMC}, \eqref{Eq_VdMB} and
\eqref{Eq_VdMD}.

\begin{proposition}\label{Prop_MPS}
For $x=(x_1,\dots,x_n)$, $m$ a positive integer and $N$ a
nonnegative integer,
\begin{align}\label{Eq_MPS}
&\sum_{\substack{\la \\[1pt] \la_1\leqslant 2m}}  t^{\abs{\la}/2}
h_{\la}^{(2m)}(t_2,t_3;t) 
P_{\la}\bigg(\bigg[x\, \frac{1-t^N}{1-t}\bigg];t\bigg)  \\
&\quad =(x_1\cdots x_n)^{mN} t^{mnN^2/2}
P^{(\mathrm{BC}_{nN})}_{m^{nN}}
\bigg(\bigg[t^{1/2}x \, \frac{1-t^N}{1-t}\bigg];t,t_2,t_3\bigg) \notag \\
&\quad=\frac{\prod_{i=1}^n (t^{1/2}t_2x_i,t^{1/2}t_3x_i;t)_N}
{\prod_{i,j=1}^n (tx_ix_j;t)_N}
\prod_{1\leqslant i<j\leqslant n} (tx_ix_j;t)_{2N} \notag \\
&\quad\quad \times
\sum_{r_1,\dots,r_n\geqslant 0}\,
\frac{\Delta_{\mathrm{C}}(xt^r)}{\Delta_{\mathrm{C}}(x)}
\prod_{i=1}^n 
\frac{(t^{1/2}t_2^{-1}x_i,t^{1/2}t_3^{-1}x_i;t)_{r_i}}
{(t^{1/2}t_2x_i,t^{1/2}t_3x_i;t)_{r_i}} \,
(t_2t_3)^{r_i} (x_i^2 t^{r_i})^{mr_i} \notag \\
&\quad\qquad\qquad\qquad\times 
\prod_{i,j=1}^n \frac{(t^{-N}x_i/x_j,x_ix_j;t)_{r_i}}
{(tx_i/x_j,t^{N+1}x_ix_j;t)_{r_i}}\, t^{N r_i}. \notag
\end{align}
\end{proposition}

\begin{remark}
A more general hypergeometric identity than \eqref{Eq_MPS} holds,
obtained by replacing 
\[
x\mapsto t^{1/2}\bigg( 
x_1\frac{1-t^{N_1}}{1-t}+x_2\frac{1-t^{N_2}}{1-t}
+\cdots+x_n\frac{1-t^{N_n}}{1-t}\bigg)
\]
in \eqref{Eq_bounded6}. From a hypergeometric point of view this more general
identity, which on the right features the $\mathrm{C}_n$ hypergeometric series
\begin{multline*}
\sum_{r_1,\dots,r_n\geqslant 0}\,
\frac{\Delta_{\mathrm{C}}(xt^r)}{\Delta_{\mathrm{C}}(x)}
\prod_{i=1}^n 
\frac{(t^{1/2}t_2^{-1}x_i,t^{1/2}t_3^{-1}x_i;t)_{r_i}}
{(t^{1/2}t_2x_i,t^{1/2}t_3x_i;t)_{r_i}} \,
(t_2t_3)^{r_i} (x_i^2 t^{r_i}\big)^{mr_i} \\
\times \prod_{i,j=1}^n \frac{(t^{-N_j}x_i/x_j,x_ix_j;t)_{r_i}}
{(tx_i/x_j,t^{N_j+1}x_ix_j;t)_{r_i}}\, t^{N_j r_i},
\end{multline*}
is more natural. 
For our purposes, however, we do not require this greater degree of generality.
\end{remark}

\begin{proof}
Identity \eqref{Eq_MPS} follows from \eqref{Eq_bounded6} by the substitution 
\begin{align}\label{Eq_x-change}
x &\mapsto 
t^{1/2}(x_1,x_1t,\dots,x_1t^{N-1},\dots\dots,x_n,x_nt,\dots,x_nt^{N-1}) \\
& \qquad\quad=t^{1/2}x \, \frac{1-t^N}{1-t} \notag
\end{align}
(so that, implicitly, $n\mapsto nN$).
The two left-most expressions immediately follow from
\eqref{Eq_MPS}, but to show equality with the hypergeometric 
sum on the right some work is required.

First we use Lemma~\ref{Lem_littlelemma} to trade the right-hand side of 
\eqref{Eq_bounded6} for 
\[
\sum_{\varepsilon\in \{\pm 1\}^n}\Phi(x^{\varepsilon};t_2,t_3;t)
\prod_{i=1}^n x_i^{(1-\varepsilon_i)m}.
\]
Next we observe that
$\Phi(x^{\varepsilon};t_2,t_3;t)$ contains the factor 
\[
\prod_{1\leqslant i<j\leqslant n}(1-t x_i^{\varepsilon_i}x_j^{\varepsilon_j})
\]
which vanishes if there exists an $i$ ($1\leqslant i\leqslant n-1$) such that
\[
tx_i^{\varepsilon_i} x_{i+1}^{\varepsilon_{i+1}}=1.
\]
Therefore, by the substitution \eqref{Eq_x-change},
the summand vanishes if for some $i,u,p$, 
\[
(\varepsilon_i,\varepsilon_{i+1})=(1,-1) \quad\text{and}\quad
(x_i,x_{i+1})\mapsto (x_u t^p,x_u t^{p+1}).
\] 
Consequently, the only sequences $\varepsilon$ that yield a
non-vanishing summand are of the form
\[
\varepsilon=\big(
\underbrace{-1,\dots,-1}_{\text{$r_1$ times}},
\underbrace{1,\dots,1}_{\text{$N-r_1$ times}},
\underbrace{-1,\dots,-1}_{\text{$r_2$ times}},
\underbrace{1,\dots,1}_{\text{$N-r_2$ times}},
\dots,
\underbrace{-1,\dots,-1}_{\text{$r_n$ times}},
\underbrace{1,\dots,1}_{\text{$N-r_n$ times}}\!\!\!\!\big).
\]
The $r_i$ are exactly the summation indices of \eqref{Eq_MPS}.
The rest of the proof is tedious but elementary and left to the reader.
\end{proof}

Replacing
\begin{equation}\label{Eq_subxN}
x\mapsto (x_1,y_1,\dots,x_n,y_n)
\end{equation}
in \eqref{Eq_MPS} (so that $n\mapsto 2n$), and then using 
\cite[Proposition 5.1]{BW13} to take the limit $y_i\mapsto x_i^{-1}$ 
for all $1\leqslant i\leqslant n$, we obtain the following corollary
of Proposition~\ref{Prop_MPS}.
\begin{corollary}
Let $m$ a positive integer, $N$ a nonnegative integer and $X$ 
the alphabet \eqref{Eq_X}. Then
\begin{align*}
\sum_{\substack{\la \\[1pt] \la_1\leqslant 2m}} & t^{\abs{\la}/2} 
h_{\la}^{(2m)}(t_2,t_3;t) P_{\la}(X;t) \\
&=t^{mnN^2} P^{(\mathrm{BC}_{2nN})}_{m^{2nN}}(t^{1/2}X;t,t_2,t_3) \\[2mm]
&=\prod_{i=1}^n \frac{(t^{1/2}t_2x_i^{\pm},t^{1/2}t_3x_i^{\pm};t)_N}
{(tx_i^{\pm 2};t)_N}\,\qbin{2N}{N}_t 
\prod_{1\leqslant i<j\leqslant n} 
\frac{(tx_i^{\pm}x_j^{\pm};t)_{2N}}{(tx_i^{\pm} x_j^{\pm};t)_N^2} \\
&\quad \times
\sum_{r\in\Z^n} \frac{\Delta_{\mathrm{C}}(xt^r)}{\Delta_{\mathrm{C}}(x)}
\prod_{i=1}^n \frac{(t^{1/2}t_2^{-1}x_i,t^{1/2}t_3^{-1}x_i;t)_{r_i}}
{(t^{1/2}t_2x_i,t^{1/2}t_3x_i;t)_{r_i}} \,
(t_2t_3)^{r_i} (x_i^2 t^{r_i})^{mr_i} \\
&\quad\quad\qquad\times 
\prod_{i,j=1}^n \frac{(t^{-N}x_ix_j^{\pm};t)_{r_i}}
{(t^{N+1}x_ix_j^{\pm};t)_{r_i}}\, t^{2Nr_i}. 
\end{align*}
\end{corollary}
Since 
\[
\lim_{N\to\infty} P_{\la}(X;t)=
P_{\la}\Big(\Big[\frac{x_1^{\pm}+\cdots+x_n^{\pm}}{1-t}\Big];t\Big)
=P'_{\la}(x_1^{\pm},\cdots x_n^{\pm};t),
\]
the above corollary is a bounded analogue of 
\cite[Theorem 5.2; (5.6a)]{BW13}, which states (without the
second line) that
\begin{align}\label{Eq_id56a}
&\sum_{\substack{\la \\[1pt] \la_1\leqslant 2m}} t^{\abs{\la}/2}
h_{\la}^{(2m)}(t_2,t_3;t) 
P'_{\la}(x_1^{\pm},\dots,x_n^{\pm};t) \\
&\quad =\lim_{N\to\infty} t^{mnN^2} P^{(\mathrm{BC}_{2nN})}_{m^{2nN}}
(t^{1/2} X;t,t_2,t_3) \notag \\
&\quad=\frac{1}{(t;t)_{\infty}^n}
\prod_{i=1}^n \frac{(t^{1/2}t_2x_i^{\pm},t^{1/2}t_3x_i^{\pm};t)_{\infty}}
{(tx_i^{\pm 2};t)_{\infty}} 
\prod_{1\leqslant i<j\leqslant n} \frac{1}{(tx_i^{\pm}x_j^{\pm};t)_{\infty}} 
\notag \\ &\quad\quad \times
\sum_{r\in\Z^n} \frac{\Delta_{\mathrm{C}}(xt^r)}{\Delta_{\mathrm{C}}(x)}
\prod_{i=1}^n \frac{(t^{1/2}t_2^{-1}x_i,t^{1/2}t_3^{-1}x_i;t)_{r_i}}
{(t^{1/2}t_2x_i,t^{1/2}t_3x_i;t)_{r_i}} \,
(t_2t_3t^{-n})^{r_i} (x_i^2t^{r_i})^{(m+n)r_i}, \notag
\end{align}
for $m$ a positive integer.

If instead of \eqref{Eq_subxN} we make the substitution 
\begin{equation}\label{Eq_oddchange}
x\mapsto (x_1,y_1,\dots,x_{n-1},y_{n-1},x_n)
\end{equation}
in \eqref{Eq_MPS} (so that $n\mapsto 2n-1$) and then take the limit
$y_i\mapsto x_i^{-1}$ for all $1\leqslant i\leqslant n-1$ and $x_n\to 1$
using \cite[Proposition 5.1]{BW13}, we obtain a bounded version of
\cite[Theorem 5.2; (5.6b)]{BW13}\footnote{Taking $x_n\to t^{1/2}$
instead of $x_n\to 1$ yields additional character identities to those 
of Theorems~\ref{Thm_een} and \ref{Thm_negen}.}
\begin{align}\label{Eq_id56b}
&\sum_{\substack{\la \\[1pt] \la_1\leqslant 2m}} t^{\abs{\la}/2}
h_{\la}^{(2m)}(t_2,t_3;t) 
P'_{\la}(x_1^{\pm},\dots,x_{n-1}^{\pm},1;t) \\
&\quad =\lim_{N\to\infty} t^{m(n-1/2)N^2} 
P^{(\mathrm{BC}_{(2n-1)N})}_{m^{(2n-1)N}} 
(t^{1/2} \bar{X};t,t_2,t_3) \notag \\
&\quad=\frac{1}{(t;t)_{\infty}^n(t^{1/2}t_2,t^{1/2}t_3;t)_{\infty}} 
\prod_{i=1}^n \frac{(t^{1/2}t_2x_i^{\pm},t^{1/2}t_3x_i^{\pm};t)_{\infty}}
{(-tx_i^{\pm};t)_{\infty}(tx_i^{\pm 2};t^2)_{\infty}}  \notag \\
& \qquad\;\; \times 
\prod_{1\leqslant i<j\leqslant n} \frac{1}{(tx_i^{\pm}x_j^{\pm};t)_{\infty}} 
\sum_{r\in\Z^n} \bigg(\frac{\Delta_{\mathrm{B}}(-xt^r)}
{\Delta_{\mathrm{B}}(-x)} \notag \\ &\quad \qquad\times
\prod_{i=1}^n \frac{(t^{1/2}t_2^{-1}x_i,t^{1/2}t_3^{-1}x_i;t)_{r_i}}
{(t^{1/2}t_2x_i,t^{1/2}t_3x_i;t)_{r_i}} \big({-}t_2t_3t^{1/2-n}\big)^{r_i} 
(x_i^2t^{r_i})^{(m+n-1/2)r_i} \bigg), \notag  
\end{align}
where $m$ is a positive integer and $x_n:=1$.

We are now ready to prove Theorems~\ref{Thm_een}--\ref{Thm_negen}.
Our first two character formulas follow from \eqref{Eq_id56a} 
and \eqref{Eq_id56b} by letting $t_2$ and $t_3$ tend to zero. 
The hypergeometric sums on the right can then be identified with
$\chi_m(\mathrm{C}_n^{(1)})$ and $\chi_m(\mathrm{A}_{2n}^{(2)\dagger})$
respectively, by \cite[Lemmas 2.1 \& 2.3]{BW13} and \cite[Lemma 2.3]{BW13}
(which are simple rewritings of the Weyl--Kac formula
for $\mathrm{C}_n^{(1)}$ and $\mathrm{A}_{2n}^{(2)\dagger}$).
In the latter case this identification requires the specialisation
$\eup^{-\alpha_n}\mapsto -1$, corresponding to the condition
$x_n:=1$ in \eqref{Eq_id56b}.
In the $t_2,t_3\to 0$ limit the left-hand sides simplify since 
$h_{\la}^{(2m)}(0,0;t)=\chi(\la \text{ even})$.
We thus obtain \eqref{Eq_Cn-character} and \eqref{Eq_A2n2dagger-character}.

Next we specialise $t_3=-t^{1/2}$ in \eqref{Eq_id56a} and 
\eqref{Eq_id56b}. Using 
\begin{subequations}\label{Eq_Delta-rel}
\begin{align}
\frac{\Delta_{\mathrm{C}}(xt^r)}{\Delta_{\mathrm{C}}(x)}
\prod_{i=1}^n \frac{(-x_i;t)_{r_i}}{(-tx_i;t)_{r_i}}
&=\frac{\Delta_{\mathrm{B}}(xt^r)}{\Delta_{\mathrm{B}}(x)} \\
\frac{\Delta_{\mathrm{B}}(-xt^r)}{\Delta_{\mathrm{B}}(-x)}
\prod_{i=1}^n \frac{(-x_i;t)_{r_i}}{(-tx_i;t)_{r_i}} 
&=\frac{\Delta_{\mathrm{D}}(xt^r)}{\Delta_{\mathrm{D}}(x)} 
\label{Delta-relBD}
\end{align}
\end{subequations}
respectively, this yields
\begin{align}\label{Eq_id56a-new}
&\sum_{\substack{\la \\[1pt] \la_1\leqslant 2m}} t^{\abs{\la}/2}
h_{\la}^{(2m)}(t_2,-t^{1/2};t) 
P'_{\la}(x_1^{\pm},\dots,x_n^{\pm};t) \\
&\quad =\lim_{N\to\infty} t^{mnN^2} P^{(\mathrm{BC}_{2nN})}_{m^{2nN}}
(t^{1/2} X;t,t_2,-t^{1/2}) \notag \\
&\quad=\frac{1}{(t;t)_{\infty}^n}
\prod_{i=1}^n \frac{(t^{1/2}t_2x_i^{\pm};t)_{\infty}}
{(tx_i^{\pm};t)_{\infty}(tx_i^{\pm 2};t^2)_{\infty}} 
\prod_{1\leqslant i<j\leqslant n} \frac{1}{(tx_i^{\pm}x_j^{\pm};t)_{\infty}} 
\notag \\ &\quad\quad \times
\sum_{r\in\Z^n} \frac{\Delta_{\mathrm{B}}(xt^r)}{\Delta_{\mathrm{B}}(x)}
\prod_{i=1}^n \frac{(t^{1/2}t_2^{-1}x_i;t)_{r_i}}
{(t^{1/2}t_2x_i;t)_{r_i}} \,
(-t_2t^{1/2-n})^{r_i} (x_i^2t^{r_i})^{(m+n)r_i} \notag
\end{align}
and
\begin{align}\label{Eq_id56b-new}
&\sum_{\substack{\la \\[1pt] \la_1\leqslant 2m}} t^{\abs{\la}/2}
h_{\la}^{(2m)}(t_2,-t^{1/2};t) 
P'_{\la}(x_1^{\pm},\dots,x_{n-1}^{\pm},1;t) \\
&\quad =\lim_{N\to\infty} t^{m(n-1/2)N^2}
P^{(\mathrm{BC}_{(2n-1)N})}_{m^{(2n-1)N}} 
(t^{1/2} \bar{X};t,t_2,-t^{1/2}) \notag \\
&\quad=\frac{1}{(t;t)_{\infty}^{n-1}(t^2;t^2)_{\infty}(t^{1/2}t_2;t)_{\infty}} 
\prod_{i=1}^n \frac{(t^{1/2}t_2x_i^{\pm};t)_{\infty}}
{(tx_i^{\pm 2};t^2)_{\infty}}  
\prod_{1\leqslant i<j\leqslant n} \frac{1}{(tx_i^{\pm}x_j^{\pm};t)_{\infty}} 
\notag \\ &\quad\quad \times
\sum_{r\in\Z^n} \bigg(\frac{\Delta_{\mathrm{D}}(xt^r)}{\Delta_{\mathrm{D}}(x)}
\prod_{i=1}^n \frac{(t^{1/2}t_2^{-1}x_i;t)_{r_i}}
{(t^{1/2}t_2x_i;t)_{r_i}} \big(t_2t^{1-n}\big)^{r_i} 
(x_i^2t^{r_i})^{(m+n-1/2)r_i} \bigg), \notag  
\end{align}
where $x_n:=1$ in the second identity.
Taking the $t_2\to 0$ limit using
\[
h_{\la}^{(m)}(0,b;t)=(-b)^{\op(\la)},
\]
and identifying the respective right-hand sides as 
$\chi_m(\mathrm{A}_{2n}^{(2)\dagger})$ and
$\chi_m(\mathrm{A}_{2n-1}^{(2)\dagger})$
by \cite[Lemma 2.3]{BW13} and Lemma~\ref{Lem_Bn-WK-2},
results in \eqref{Eq_A2n2-new} and \eqref{Eq_char-new2}.
In particular we note that $x_n$ being equal to $1$ in \eqref{Eq_id56b-new}
implies the specialisation $\eup^{-\alpha_n}\mapsto\eup^{-\alpha_{n-1}}$,
see \eqref{Eq_qxi-3}.
In similar manner we specialise $t_2=t^{1/2}$ in \eqref{Eq_id56a-new}
(considering \eqref{Eq_id56b-new} does not lead to a character identity).
Using \eqref{Delta-relBD} with $x\mapsto -x$,
it follows from Lemma~\ref{Lem_Bn-WK-2} that the right-hand side
simplifies to 
\[
(-t;t)_{\infty} \, \chi_m\big(\mathrm{A}_{2n-1}^{(2)\dagger}\big).
\]
By \eqref{Eq_RSBCn}, the Rogers--Szeg\H{o} polynomial in the summand on the 
left becomes
\[
h_{\la}^{(2m)}(-t^{1/2},t^{1/2};t)=t^{\op(\la)/2} 
\prod_{\substack{i=1 \\[0.5pt] i \text{ odd}}}^{m-1} 
H_{m_i(\la)}(-1;t)
\prod_{\substack{i=1 \\[0.5pt] i \text{ even}}}^{m-1} H_{m_i(\la)}(-t;t).
\]
Using \eqref{Eq_spec2} and \eqref{Eq_spec3} this yields
\[
h_{\la}^{(2m)}(-t^{1/2},t^{1/2};t)=
\prod_{i=1}^{2m-1} (t;t^2)_{\ceil{m/2}}
\]
if $m_i(\la)$ is even for all $i=1,3,\dots,2m-1$, and zero otherwise.
Finally noting that $(-t;t)_{\infty}(t;t^2)_{\infty}=1$, 
Theorem~\ref{Thm_vier} follows.
There is one further specialisation of \eqref{Eq_id56a-new} and
\eqref{Eq_id56b-new} that leads to character identities, namely
$t_2=-1$. We will consider this case as part of a 
more general treatment of \eqref{Eq_id56a} and \eqref{Eq_id56b} 
for $t_3=-1$.

Recalling that Theorem~\ref{THM_BOUNDED7} extends the $t_3=-1$ case of
Theorem~\ref{Thm_bounded6} to half-integer values of $m$, the 
$t_3=-1$ specialisations of \eqref{Eq_id56a} and \eqref{Eq_id56b}
lead to
\begin{align}\label{Eq_id56a-2}
&\sum_{\substack{\la \\[1pt] \la_1\leqslant m}} t^{\abs{\la}/2}
h_{\la}^{(m)}(t_2;t) P'_{\la}(x_1^{\pm},\dots,x_n^{\pm};t) \\[-1mm]
&\quad =\lim_{N\to\infty} t^{mnN^2/2} P^{(\mathrm{B}_{2nN})}_{(m/2)^{2nN}}
(t^{1/2} X;t,t_2) \notag \\
&\quad=\frac{1}{(t;t)_{\infty}^n}
\prod_{i=1}^n \frac{(t^{1/2}t_2x_i^{\pm};t)_{\infty}}
{(t^{1/2}x_i^{\pm};t)_{\infty}(t^2x_i^{\pm 2};t^2)_{\infty}} 
\prod_{1\leqslant i<j\leqslant n} \frac{1}{(tx_i^{\pm}x_j^{\pm};t)_{\infty}} 
\notag \\ &\quad\quad \times
\sum_{r\in\Z^n} \frac{\Delta_{\mathrm{C}}(xt^r)}{\Delta_{\mathrm{C}}(x)}
\prod_{i=1}^n \frac{(t^{1/2}t_2^{-1}x_i;t)_{r_i}}{(t^{1/2}t_2x_i;t)_{r_i}} \,
(-t_2t^{-n})^{r_i} (x_i^2t^{r_i})^{(m+2n)r_i/2} \notag
\end{align}
and
\begin{align}\label{Eq_id56b-2}
&\sum_{\substack{\la \\[1pt] \la_1\leqslant m}} t^{\abs{\la}/2}
h_{\la}^{(m)}(t_2;t) P'_{\la}(x_1^{\pm},\dots,x_{n-1}^{\pm},1;t) \\[-1mm]
&\quad =\lim_{N\to\infty} t^{m(2n-1)N^2/4} 
P^{(\mathrm{B}_{(2n-1)N})}_{(m/2)^{(2n-1)N}}
(t^{1/2} \bar{X};t,t_2) \notag \\
&\quad=\frac{1}{(t;t)_{\infty}^n(t^{1/2}t_2,-t^{1/2};t)_{\infty}} 
\prod_{i=1}^n \frac{(t^{1/2}t_2x_i^{\pm};t)_{\infty}}
{(t^{1/2}x_i^{\pm},-tx_i^{\pm};t)_{\infty}}
\prod_{1\leqslant i<j\leqslant n} \frac{1}{(tx_i^{\pm}x_j^{\pm};t)_{\infty}} 
\notag \\ 
&\quad\quad\times
\sum_{r\in\Z^n} \frac{\Delta_{\mathrm{B}}(-xt^r)}{\Delta_{\mathrm{B}}(-x)}
\prod_{i=1}^n \frac{(t^{1/2}t_2^{-1}x_i;t)_{r_i}}{(t^{1/2}t_2x_i;t)_{r_i}} \,
(t_2t^{1/2-n})^{r_i} (x_i^2t^{r_i})^{(m+2n-1)r_i/2}, \notag  
\end{align}
where $x_n:=1$.
If we now let $t_2$ tend to zero, use that $h_{\la}^{(m)}(0;t)=1$, 
and further use \cite[Lemmas 2.2 \& 2.4]{BW13}
(see also \eqref{Eq_L22} for the former) to identify the right-hand sides as
$\chi_m(\mathrm{A}_{2n}^{(2)})$ and $\chi_m(\mathrm{D}_{n+1}^{(2)})$, 
we obtain \eqref{Eq_A2n2-character} and \eqref{Eq_Dn1-character}.
We again note that the condition $x_n:=1$ in 
\eqref{Eq_id56b-2} implies that in the $\mathrm{D}_{n+1}^{(2)}$ case we 
must specialise $\eup^{-\alpha_n}\mapsto -1$.
Two further cases, already mentioned in relation with
\eqref{Eq_id56a-new} and \eqref{Eq_id56b-new},
arise from \eqref{Eq_id56a-2} and \eqref{Eq_id56b-2} 
by specialising $t_2=-t^{1/2}$.
On the right we can once again use \eqref{Eq_Delta-rel}
as well as \cite[Lemma 2.4]{BW13} and Lemma~\ref{Lem_Bn-WK-2}
to recognise the hypergeometric sums as
\[
(-t^{1/2};t^{1/2})_{\infty} \, \chi_m(\mathfrak{g})
\]
for $\mathfrak{g}=\mathrm{D}_{n+1}^{(2)}$ and $\mathrm{B}_n^{(1)\dagger}$
respectively. 
In the latter case we must assume the specialisation
$\eup^{-\alpha_n}\mapsto \eup^{-\alpha_{n-1}}$.
On the left we use \eqref{Eq_RSBn} and \eqref{Eq_spec4} to find
\[
h_{\la}^{(m)}(t^{1/2};t)=\prod_{i=1}^{m-1} 
(-t^{1/2};t^{1/2})_{m_i(\la)},
\]
completing the proofs of \eqref{Eq_Dn-char2} and 
\eqref{Eq_Bndagger-character}.
To prove our final two results we consider \eqref{Eq_id56a-2} and 
\eqref{Eq_id56b-2} for $t_2=1$. Then, by \eqref{Eq_JZ},
we can add the additional restriction ``$\la'$ is even''
to the sum on the left and $\abs{r}\equiv 0 \pmod{2}$ to the
sum on the right.
Using Lemmas~\ref{Lem_A2n12-WK} and \ref{Lem_Bn-WK}
this proves \eqref{Eq_A2n12-character} and \eqref{Eq_Bn1-character}.

\section{Rogers--Ramanujan identities}\label{Sec_RR}
Starting with the
pioneering series of papers \cite{LM78a,LM78b,LW81a,LW81b,LW82,LW84},
the link between affine Lie algebras and vertex operator algebras on 
the one hand and Rogers--Ramanujan identities on the other is by now 
well established, see also
\cite{ASW99,BCFK14,Capparelli96,CF13,CKLMQRS14,FF93,FM12,KLRS14,MP87,WZ12}.
Nonetheless, examples of $q$-series identities (as opposed to combinatorial
identities) that lift classical Rogers--Ramanujan-type identities to affine
Lie algebra at arbitrary rank and level while still
permitting a product form, are rare.
Recently Griffin et al.~\cite{GOW14} showed how to
use combinatorial character formulas of the type proven in 
Section~\ref{Sec_char} to obtain doubly-infinite families of 
Rogers--Ramanujan identities, including a generalisation of the 
Rogers--Ramanujan \cite{Rogers94} and Andrews--Gordon
\cite{Andrews74,Gordon61} identities to the affine Lie 
algebra $\mathrm{A}_{2n}^{(2)}$. Following the approach of \cite{GOW14},
we prove several new doubly-infinite families of Rogers--Ramanujan 
identities, including a $\mathrm{B}_n^{(1)}$ generalisation of Bressoud's
Rogers--Ramanujan identities for even moduli \cite{Bressoud80a,Bressoud80b}.

\medskip

\begin{theorem}[$\mathrm{A}_{2n}^{(2)}$ Rogers--Ramanujan identities]
Let $m,n$ be positive integers. Then
\begin{multline}\label{Eq_RR1}
\sum_{\substack{\la \\[1pt] \la_1\leqslant m}}
q^{\abs{\la}/2} P_{\la}(1,q,q^2,\dots;q^{2n}) \\[-1mm]
=\frac{(q^{\kappa};q^{\kappa})_{\infty}^{n-1}
(q^{\kappa/2};q^{\kappa/2})_{\infty}}
{(q;q)_{\infty}^{n-1}(q^{1/2};q^{1/2})_{\infty}}
\prod_{i=1}^n \theta(q^i;q^{\kappa/2})
\prod_{1\leqslant i<j\leqslant n} \theta(q^{j-i},q^{i+j};q^{\kappa})
\end{multline}
for $\kappa:=m+2n+1$, and
\begin{multline}\label{Eq_RR1-new}
\sum_{\substack{\la \\[1pt] \la_1\leqslant 2m}} 
t^{\abs{\la}/2+n\op(\la)}
P_{\la}(1,q,q^2,\dots;q^{2n}) \\[-1mm]
=\frac{(q^{\kappa};q^{\kappa})_{\infty}^{n-1}
(q^{\kappa/2};q^{\kappa/2})_{\infty}}
{(q;q)_{\infty}^{n-1}(q^{1/2};q^{1/2})_{\infty}}
\prod_{i=1}^n \theta(q^{i+m};q^{\kappa/2})
\prod_{1\leqslant i<j\leqslant n} \theta(q^{j-i},q^{i+j-1};q^{\kappa})
\end{multline}
for $\kappa:=2m+2n+1$.
\end{theorem}

\begin{proof}
To prove \eqref{Eq_RR1} we apply the specialisation
\[
F:~\Complex[[\eup^{-\alpha_0},\dots,\eup^{-\alpha_n}]]\to
\Complex[[q^{1/2}]]
\]
given by
\begin{equation}\label{Eq_F}
F(\eup^{-\alpha_0})=q^{1/2}\quad\text{and}\quad
F(\eup^{-\alpha_i})=q \quad \text{for $1\leqslant i\leqslant n$}
\end{equation}
to the $\mathrm{A}_{2n}^{(2)}$ character identity \eqref{Eq_A2n2-character}.
Since the the null root for $\mathrm{A}_{2n}^{(2)}$ is
given by $\delta=2\alpha_0+\cdots+2\alpha_{n-1}+\alpha_n$,
it follows from \eqref{Eq_xit} that
\[
F(x_i)=q^{n-i+1/2} \quad (1\leqslant i\leqslant n)\quad\text{and}\quad
F(t)=q^{2n}.
\]
Hence
\begin{align*}
F\Bigg(\;\sum_{\substack{\la \\[1pt] \la_1\leqslant m}} & t^{\abs{\la}/2} 
P'_{\la}(x_1^{\pm},\dots,x_n^{\pm};t) \Bigg)\\
&=\sum_{\substack{\la \\[1pt] \la_1\leqslant m}} q^{n\abs{\la}} 
P'_{\la}(q^{-n+1/2},q^{-n+3/2},\dots,q^{n-1/2};q^{2n}) \\
&=\sum_{\substack{\la \\[1pt] \la_1\leqslant m}} q^{\abs{\la}/2} 
P'_{\la}(1,q,\dots,q^{2n-1};q^{2n}) \\
&=\sum_{\substack{\la \\[1pt] \la_1\leqslant m}} q^{\abs{\la}/2} 
P_{\la}(1,q,q^2,\dots;q^{2n}).
\end{align*}
Here the second equality uses the homogeneity of the modified 
Hall--Littlewood polynomial and the the third equality follows from
\eqref{Eq_HLPp} and
\[
f\Big[\frac{1+q+\cdots+q^{2n-1}}{1-q^{2n}}\Big]=
f\Big[\frac{1}{1-q}\Big]=
f[1+q+q^2+\cdots]
\quad \text{ for $f\in\Lambda$}.
\]
To apply $F$ to the left-hand side of \eqref{Eq_A2n2-character} we use
that for arbitrary $\Lambda\in P_{+}$ parametrised as 
(compare with \eqref{Eq_Pplus-C})
\[
\Lambda=c_0\fwa_0+(\la_1-\la_2)\fwa_1+\cdots+
(\la_{n-1}-\la_n)\fwa_{n-1}+\la_n\fwa_n
\]
with $\la=(\la_1,\dots,\la_n)$ a partition and $c_0$ a nonnegative
integer, we can rewrite the Weyl--Kac formula for $\mathrm{A}_{2n}^{(2)}$ as
\cite[Lemma 2.2]{BW13},
\begin{multline}\label{Eq_L22} 
\eup^{-\Lambda} \ch V(\Lambda)=
\frac{1}{(t;t)_{\infty}^n
\prod_{i=1}^n \theta(t^{1/2}x_i;t)\theta(x_i^2;t^2)
\prod_{1\leqslant i<j\leqslant n} x_j \theta(x_i/x_j,x_ix_j;t)} \\[1mm]
\times \sum_{r\in\Z^n} \widetilde{\symp}_{2n,\la}(xt^r)
\prod_{i=1}^n x_i^{\kappa r_i+\la_i} t^{\kappa r_i^2/2-nr_i}.
\end{multline}
Here $\kappa:=2n+c_0+2\la_1+1$, $\widetilde{\symp}_{2n,\la}$ is the normalised
symplectic Schur function \eqref{Eq_soandsp} and $x_1,\dots,x_n$ and $t$ 
are defined by \eqref{Eq_xit}. Using 
\begin{multline*}
F\bigg((t;t)_{\infty}^n
\prod_{i=1}^n \theta(t^{1/2}x_i;t) \theta(x_i^2;t^2)
\prod_{1\leqslant i<j\leqslant n} x_j \theta(x_i/x_j,x_ix_j;t) \bigg) \\
=(q;q)_{\infty}^{n-1} (q^{1/2};q^{1/2})_{\infty}\,
q^{\sum_{i<j}(n-j+1/2)}
\end{multline*}
as well as the equations \eqref{Eq_symplectic}, \eqref{Eq_soandsp}, 
and appealing to multilinearity, yields
\begin{align*}
F&\big(\eup^{-\Lambda} \ch V(\Lambda)\big)=
\frac{q^{-\sum_{i<j} (n-j+1/2)}}{(q;q)_{\infty}^{n-1}
(q^{1/2};q^{1/2})_{\infty}} \\
& \quad \times
\det_{1\leqslant i,j\leqslant n} \bigg( \sum_{r\in\Z}
q^{(\kappa r+\la_i-\la_j+j-1)(n-i+1/2)+n\kappa r^2-2nr(\la_j+n-j+1)}\\
& \qquad \qquad \qquad\quad 
-\sum_{r\in\Z}
q^{(\kappa r+\la_i+\la_j+2n-j+1)(n-i+1/2)+n\kappa r^2+2nr(\la_j+n-j+1)}
\bigg).
\end{align*}
Replacing $(i,j)\mapsto (n-j+1,n-i+1)$ in the determinant
and then changing $r\mapsto -r-1$ in the second sum, we get
\begin{multline}\label{Eq_intermediate}
F\big(\eup^{-\Lambda} \ch V(\Lambda)\big)=
\frac{1}{(q;q)_{\infty}^{n-1}(q^{1/2};q^{1/2})_{\infty}} \\
\times \det_{1\leqslant i,j\leqslant n} \bigg( \sum_{r\in\Z}
x_i^{2nr-i+1}q^{2n\kappa\binom{r}{2}+\kappa r/2}
\Big( (x_iq^{\kappa r})^{j-1} - (x_iq^{\kappa r})^{2n-j} \Big)\bigg),
\end{multline}
where $x_i:=q^{\kappa/2-i-\la_{n-i+1}}$.
Again using multilinearity and recalling the $\mathrm{B}_n$
Vandermonde determinant
\eqref{Eq_VdMB}, this may be written as
\begin{multline*}
F\big(\eup^{-\Lambda} \ch V(\Lambda)\big)=
\frac{1}{(q;q)_{\infty}^{n-1}(q^{1/2};q^{1/2})_{\infty}} \\
\times \sum_{r\in\Z^n} \Delta_{\mathrm{B}}(xq^{\kappa r})
\prod_{i=1}^n x_i^{2nr_i-i+1}q^{2n\kappa\binom{r_i}{2}+\kappa r_i/2}.
\end{multline*}
By the $\mathrm{D}_{n+1}^{(2)}$ Macdonald identity \cite{Macdonald72}
in the form given by \cite[Corollary 6.2]{RS06}, i.e.,
\begin{multline*}
\sum_{r\in\Z^n} \Delta_{\mathrm{B}}(xq^r)
\prod_{i=1}^n x_i^{2nr_i-i+1}q^{2n\binom{r_i}{2}+r_i/2} \\
=(q;q)_{\infty}^{n-1}(q^{1/2};q^{1/2})_{\infty}
\prod_{i=1}^n \theta(x_i;q^{1/2})
\prod_{1\leqslant i<j\leqslant n} \theta(x_ix_j^{\pm};q),
\end{multline*}
we obtain the product formula
\begin{multline*}
F\big(\eup^{-\Lambda} \ch V(\Lambda)\big)=
\frac{(q^{\kappa};q^{\kappa})_{\infty}^{n-1}
(q^{\kappa/2};q^{\kappa/2})_{\infty}}
{(q;q)_{\infty}^{n-1}(q^{1/2};q^{1/2})_{\infty}} 
\prod_{i=1}^n \theta(q^{\kappa/2-i-\la_{n-i+1}};q^{\kappa/2}) \\
\times
\prod_{1\leqslant i<j\leqslant n} \theta(q^{\la_{n-j+1}-\la_{n-i+1}+j-i},
q^{\kappa-i-j-\la_{n-i+1}-\la_{n-j+1}};q^{\kappa}).
\end{multline*}
By $\theta(x;q)=\theta(q/x;q)$ and a reversal of the products,
this simplifies to
\begin{multline}\label{Eq_A2n2-spec}
F\big(\eup^{-\Lambda} \ch V(\Lambda)\big)=
\frac{(q^{\kappa};q^{\kappa})_{\infty}^{n-1}
(q^{\kappa/2};q^{\kappa/2})_{\infty}}
{(q;q)_{\infty}^{n-1}(q^{1/2};q^{1/2})_{\infty}} 
\prod_{i=1}^n \theta(q^{\la_i+n-i+1};q^{\kappa/2}) \\
\times
\prod_{1\leqslant i<j\leqslant n} \theta(q^{\la_i-\la_j-i+j},
q^{\la_i+\la_j+2n-i-j+2};q^{\kappa}).
\end{multline}
For $\la=0$ and $c_0=m$ this gives the claimed right-hand side of
\eqref{Eq_RR1}.

\medskip

To prove \eqref{Eq_RR1-new} we apply the specialisation 
\[
F^{\dagger}:~\Complex[[\eup^{-\alpha_0},\dots,\eup^{-\alpha_n}]]\to
\Complex[[q^{1/2}]]
\]
given by
\begin{equation}\label{Eq_Fdag}
F^{\dagger}(\eup^{-\alpha_n})=q^{1/2}\quad\text{and}\quad
F^{\dagger}(\eup^{-\alpha_i})=q \quad \text{for $0\leqslant i\leqslant n-1$}
\end{equation}
to the $\mathrm{A}_{2n}^{(2)\dagger}$ character identity \eqref{Eq_A2n2-new}.
This implies the same specialisation of $x_1,\dots,x_n$ and $q$ as before, 
i.e.,
\[
F^{\dagger}(x_i)=q^{n-i+1/2} \quad (1\leqslant i\leqslant n)
\quad\text{and}\quad F^{\dagger}(t)=q^{2n},
\]
so that
\begin{multline*}
F^{\dagger}\Bigg(\;\sum_{\substack{\la \\[1pt] \la_1\leqslant 2m}}
t^{(\abs{\la}+\op(\la))/2} 
P'_{\la}(x_1^{\pm},\dots,x_n^{\pm};t) \Bigg) \\
=\sum_{\substack{\la \\[1pt] \la_1\leqslant 2m}} q^{\abs{\la}/2+n\op(\la)} 
P_{\la}(1,q,q^2,\dots;q^{2n}).
\end{multline*}
Moreover, since \eqref{Eq_F} and \eqref{Eq_Fdag} are compatible with the
map from $\mathrm{A}_{2n}^{(2)}$ to $\mathrm{A}_{2n}^{(2)\dagger}$ 
(corresponding to a reversal of the labelling of simple roots) 
we can again use \eqref{Eq_A2n2-spec}:
\begin{align*}
F^{\dagger}\big(&\eup^{-m\fwa_0} \ch V(m\fwa_0)\big)
\big|_{\mathfrak{g}=\mathrm{A}_{2n}^{(2)\dagger}} \\[1mm]
&=F\big(\eup^{-m\fwa_n} 
\ch V(m\fwa_n)\big)\big|_{\mathfrak{g}=\mathrm{A}_{2n}^{(2)}} \\
&=\frac{(q^{\kappa};q^{\kappa})_{\infty}^{n-1}
(q^{\kappa/2};q^{\kappa/2})_{\infty}}
{(q;q)_{\infty}^{n-1}(q^{1/2};q^{1/2})_{\infty}}
\prod_{i=1}^n \theta(q^{i+m};q^{\kappa/2})
\prod_{1\leqslant i<j\leqslant n} \theta(q^{j-i},q^{i+j-1};q^{\kappa}),
\end{align*}
where $\kappa=2m+2n+1$.
\end{proof}

\begin{theorem}[$\mathrm{D}_{n+1}^{(2)}$ Rogers--Ramanujan identities]
For $m,n$ positive integers and $\kappa:=m+2n$,
\begin{multline}\label{Eq_RR2}
\sum_{\substack{\la \\[1pt] \la_1\leqslant m}}
q^{\abs{\la}/2} P_{\la}(1,q,q^2,\dots;q^{2n-1}) \\[-1mm]
=\frac{(q^{\kappa};q^{\kappa})_{\infty}^n} 
{(q;q)_{\infty}^{n-1} (q^{1/2};q)_{\infty}(q^2;q^2)_{\infty}}
\prod_{i=1}^n \theta(q^{i+(m-1)/2};q^{\kappa})
\prod_{1\leqslant i<j\leqslant n}\theta(q^{j-i},q^{i+j-1};q^{\kappa}),
\end{multline}
and
\begin{multline}\label{Eq_RR3}
\sum_{\substack{\la \\[1pt] \la_1\leqslant m}}
q^{\abs{\la}/2} 
\bigg( \prod_{i=1}^{m-1} (-q^n;q^n)_{
m_i(\la)} \bigg)
P_{\la}(1,q,q^2,\dots;q^{2n}) \\
=\frac{(q^{\kappa};q^{\kappa})_{\infty}^{n-1} 
(q^{\kappa/2};q^{\kappa/2})_{\infty}}
{(q;q)_{\infty}^{n-1} (q^{1/2};q)_{\infty}^2 (q^2;q^2)_{\infty}}
\prod_{i=1}^n \theta(q^{i-1/2};q^{\kappa/2})
\prod_{1\leqslant i<j\leqslant n}\theta(q^{j-i},q^{i+j-1};q^{\kappa}).
\end{multline}
\end{theorem}
For $n=1$ the first of these identities is the second equation 
on page 235 of \cite{W03}.

\begin{proof}[Sketch of the proof]
In the character identity \eqref{Eq_Dn1-character} we carry out the 
specialisation 
\[
F:~\Complex[[\eup^{-\alpha_0},\dots,\eup^{-\alpha_n}]]\to \Complex[[q^{1/2}]]
\]
given by
\[
F(\eup^{-\alpha_0})=q^{1/2},\quad
F(\eup^{-\alpha_n})=-1
\quad\text{and}\quad
F(\eup^{-\alpha_i})=q \quad \text{for $1\leqslant i\leqslant n-1$}.
\]
Noting that $F$ applied to \eqref{Eq_xithalf} yields
\[
F(x_i)=q^{n-i}\quad (1\leqslant i\leqslant n-1),
\qquad\quad F(t^{1/2})=q^{n-1/2},
\]
and following the proof of \eqref{Eq_RR1}, it follows that the
right-hand side of \eqref{Eq_Dn1-character} maps to the left-hand side
of \eqref{Eq_RR2}.
If we parametrise $\Lambda\in P_{+}$ as 
(compare with \eqref{Eq_Pplus-B})
\begin{equation}\label{Eq_La-par}
\Lambda=c_0\fwa_0+(\la_1-\la_2)\fwa_1+\cdots+
(\la_{n-1}-\la_n)\fwa_{n-1}+2\la_n\fwa_n,
\end{equation}
with $\la=(\la_1,\dots,\la_n)$ a partition or half-partition
and $c_0$ a nonnegative, and again follow the previous proof, we find
\begin{multline*}
F\big(\eup^{-\Lambda} \ch V(\Lambda)\big)
=\frac{(q^{\kappa};q^{\kappa})_{\infty}^n} 
{(q;q)_{\infty}^{n-1} (q^{1/2};q)_{\infty}(q^2;q^2)_{\infty}}
\prod_{i=1}^n \theta(q^{\la_i+n-i+(\kappa+1)/2};q^{\kappa}) \\
\times\prod_{1\leqslant i<j\leqslant n}\theta(q^{\la_i-\la_j-i+j},
q^{\la_i+\la_j+2n-i-j+1};q^{\kappa}),
\end{multline*}
where $\kappa:=2n+c_0+2\la_1$. The only change compared to the
proof of \eqref{Eq_RR1} is that we have used the $\mathrm{B}_n^{(1)}$
instead of $\mathrm{D}_{n+1}^{(2)}$ Macdonald identity.
For $\Lambda=m\fwa_0$, i.e., $\la=0$ and $c_0=m$ the above product
gives the right-hand side of \eqref{Eq_RR2}, completing the proof.

Similarly, to prove \eqref{Eq_RR3} we apply the 
specialisation $F$ to \eqref{Eq_Dn-char2}, 
where this time
\begin{gather*}
F:~\Complex[[\eup^{-\alpha_0},\dots,\eup^{-\alpha_n}]]\to
\Complex[[q^{1/2}]] \\
F(\eup^{-\alpha_0})=F(\eup^{-\alpha_n})=q^{1/2}
\quad\text{and}\quad
F(\eup^{-\alpha_i})=q \quad \text{for $1\leqslant i\leqslant n-1$}.
\end{gather*}
Applied to \eqref{Eq_xitD} this gives
\[
F(x_i)=q^{n-i+1/2} \quad (1\leqslant i\leqslant n), \qquad\quad 
F(t^{1/2})=q^n,
\]
so that the left side of \eqref{Eq_Dn-char2} specialises to the
the right side of \eqref{Eq_RR3}.
With the same parametrisation of $\Lambda$ as in \eqref{Eq_La-par} 
and once more using the $\mathrm{D}_{n+1}^{(2)}$ Macdonald identity,
it follows that
\begin{multline*}
F\big(\eup^{-\Lambda} \ch V(\Lambda)\big)
=\frac{(q^{\kappa};q^{\kappa})_{\infty}^{n-1} 
(q^{\kappa/2};q^{\kappa/2})_{\infty}(-q^n;q^n)_{\infty}}
{(q;q)_{\infty}^{n-1} (q^{1/2};q)_{\infty}^2 (q^2;q^2)_{\infty}}
\prod_{i=1}^n \theta(q^{\la_i+n-i+1/2};q^{\kappa/2}) \\
\times \prod_{1\leqslant i<j\leqslant n} \theta(q^{\la_i-\la_j-i+j},
q^{\la_i+\la_j+2n-i-j+1};q^{\kappa})
\end{multline*}
with $\kappa$ as before.
For $\Lambda=m\fwa_0$, i.e., $\la=0$ and $c_0=m$ this gives the 
right-hand side of \eqref{Eq_RR3}.
\end{proof}

\begin{theorem}[$\mathrm{B}_n^{(1)}$ Rogers--Ramanujan identity]
\label{Thm_RR-B}
Let $m,n$ be a positive integers and $\kappa:=m+2n-1$. Then
\begin{multline}\label{Eq_RR4}
\sum_{\substack{\la \\[1pt] \la_1\leqslant m}} q^{\abs{\la}/2} 
\bigg( \prod_{i=1}^{m-1} (-q^{n-1/2};q^{n-1/2})_{m_i(\la)} \bigg)
P_{\la}(1,q,q^2,\dots;q^{2n-1}) \\[-1mm]
=\frac{(q^{\kappa};q^{\kappa})_{\infty}^n}{(q;q)_{\infty}^{n-1}
(q^{1/2};q^{1/2})_{\infty}}
\prod_{i=1}^n \theta(q^{i+m/2-1/2};q^{\kappa}) 
\prod_{1\leqslant i<j\leqslant n} \theta(q^{j-i},q^{i+j-2};q^{\kappa}).
\end{multline}
\end{theorem}
By \eqref{Eq_Qp-een},
\[
P_{\la}(1,q,q^2,\dots;q)=\frac{q^{n(\la)}}{b_{\la}(q)}
\,\stackrel{\la_1\leqslant m}{=}\,
\prod_{i=1}^m \frac{q^{\binom{\la'_i}{2}}}{(q;q)_{m_i(\la)}}.
\]
This shows that if we replace $q\mapsto q^2$, $\la'_i\mapsto N_i$
and $m_i(\la)\mapsto n_i$ (so that $N_i=n_i+\cdots+n_m$) 
in the $n=1$ case of Theorem~\ref{Thm_RR-B} we obtain
Bressoud's even modulus identity \cite{Bressoud80a,Bressoud80b}
\[
\sum_{n_1,\cdots,n_m\geqslant 0}
\frac{q^{N_1^2+\cdots+N_m^2}}
{(q;q)_{n_1}\cdots(q;q)_{n_{m-1}}(q^2;q^2)_{n_m}}
=\frac{(q^{m+1},q^{m+1},q^{2m+2};q^{2m+2})_{\infty}}{(q;q)_{\infty}}.
\]

\begin{remark}\label{Remark_Bressoud}
If after the substitution \eqref{Eq_oddchange} we let $x_n$ tend to $t^{1/2}$
instead of $1$, and then specialise $(a,b)=(-t^{1/2},-1)$, we obtain
an identity for $\chi_m(\mathrm{B}_n^{(1)})$ not included in 
Section~\ref{Sec_char}. Upon specialisation this yields a companion to 
\eqref{Eq_RR4} as follows:
\begin{multline*}
\sum_{\substack{\la \\[1pt] \la_1\leqslant m}} q^{\abs{\la}} 
\bigg( \prod_{i=1}^{m-1} (-q^{n-1/2};q^{n-1/2})_{m_i(\la)} \bigg)
P_{\la}(1,q,q^2,\dots;q^{2n-1}) \\[-1mm]
=\frac{(q^{\kappa};q^{\kappa})_{\infty}^n}{(q;q)_{\infty}^{n-1}
(q^{1/2};q^{1/2})_{\infty}}
\prod_{i=1}^n \theta(q^{i-1/2};q^{\kappa}) 
\prod_{1\leqslant i<j\leqslant n} \theta(q^{j-i},q^{i+j-1};q^{\kappa}).
\end{multline*}
For $n=1$ this corresponds to 
\[
\sum_{n_1,\cdots,n_m\geqslant 0}
\frac{q^{N_1^2+\cdots+N_m^2+N_1+\cdots+N_m}}
{(q;q)_{n_1}\cdots(q;q)_{n_{m-1}}(q^2;q^2)_{n_m}}
=\frac{(q,q^{2m+1},q^{2m+2};q^{2m+2})_{\infty}}{(q;q)_{\infty}},
\]
again due to Bressoud \cite{Bressoud80a,Bressoud80b}.
\end{remark}

\begin{proof}[Sketch of the proof of Theorem~\ref{Thm_RR-B}]
We start with the $\mathrm{B}_n^{(1)\dagger}$ formula
\eqref{Eq_Bndagger-character} and make the specialisation
\begin{subequations}
\begin{gather}
F^{\dagger}:~\Complex[[\eup^{-\alpha_0},\dots,\eup^{-\alpha_n}]]\to 
\Complex[[q^{1/2}]] \\
\label{Eq_FB}
F^{\dagger}(\eup^{-\alpha_0})=q^{1/2} \quad\text{and}\quad
F^{\dagger}(\eup^{-\alpha_i})=q \quad \text{for $1\leqslant i\leqslant n$}.
\end{gather}
\end{subequations}
Applied to \eqref{Eq_xit-Bdag} this yields
\[
F^{\dagger}(x_i)=q^{n-i} \quad (1\leqslant i\leqslant n-1), 
\qquad\quad F^{\dagger}(t)=q^{2n-1},
\]
so that, up to a factor $(-q^{n-1/2};q^{n-1/2})_{\infty}$, the 
left-hand side of \eqref{Eq_RR4} follows by application of $F^{\dagger}$.

To obtain the product-form on the right we first consider the more
general $\mathrm{B}_n^{(1)}$ specialisation formula 
\begin{multline*}
F\big(\eup^{-\Lambda} \ch V(\Lambda)\big)=
\frac{(q^{\kappa};q^{\kappa})_{\infty}^n
(-q^{n-1/2};q^{n-1/2})_{\infty}}
{(q;q)_{\infty}^{n-1}(q^{1/2};q^{1/2})_{\infty}} 
\prod_{i=1}^n \theta(q^{\la_i+n-i+1/2};q^{\kappa}) \\ \times
\prod_{1\leqslant i<j\leqslant n} 
\theta(q^{\la_i-\la_j-i+j},q^{\la_i+\la_j+2n-i-j+1};q^{\kappa}),
\end{multline*}
where $\Lambda$ and $\kappa$ are as in Lemma~\ref{Lem_Bn-WK}, and where
$F$ is the specialisation
\[
F(\eup^{-\alpha_n})=q^{1/2} \quad\text{and}\quad
F(\eup^{-\alpha_i})=q \quad \text{for $0\leqslant i\leqslant n-1$}
\]
of $\mathrm{B}_n^{(1)}$.
Proof of this result follows from the $\mathrm{B}_n^{(1)}$ Macdonald identity.
Again $F^{\dagger}$ and $F$ are compatible so that
\begin{align*}
F^{\dagger}\big(&\eup^{-m\fwa_0} \ch V(m\fwa_0)\big)
\big|_{\mathfrak{g}=\mathrm{B}_n^{(1)\dagger}} \\[1mm]
&=F\big(\eup^{-m\fwa_n} 
\ch V(m\fwa_n)\big)\big|_{\mathfrak{g}=\mathrm{B}_n^{(1)}} \\
&=\frac{(q^{\kappa};q^{\kappa})_{\infty}^n(-q^{n-1/2};q^{n-1/2})_{\infty}}
{(q;q)_{\infty}^n (q^{1/2};q)_{\infty}} 
\prod_{i=1}^n \theta(q^{i+m/2-1/2};q^{\kappa})  \\ 
& \qquad\qquad\qquad\qquad\qquad \times 
\prod_{1\leqslant i<j\leqslant n} \theta(q^{j-i},q^{i+j-2};q^{\kappa}),
\end{align*}
where $\kappa=m+2n-1$.
Up to the factor $(-q^{n-1/2};q^{n-1/2})_{\infty}$ this is the right-hand
side of \eqref{Eq_RR4}.
\end{proof}

\begin{theorem}[$\mathrm{A}_{2n-1}^{(2)}$ Rogers--Ramanujan identities]
Let $m,n$ be a positive integers and $\kappa:=2m+2n$. Then
\begin{multline}\label{Eq_RR5-pre}
\sum_{\substack{\la \\[1pt] \la_1\leqslant 2m}} q^{\abs{\la}/2+(n-1/2)\op(\la)}
P_{\la}(1,q,q^2,\dots;q^{2n-1}) \\
=\frac{(q^{\kappa};q^{\kappa})_{\infty}^n}
{(q;q)_{\infty}^n(q;q^2)_{\infty}} 
\prod_{i=1}^n \theta(q^{i+\kappa/2-1};q^{\kappa})
\prod_{1\leqslant i<j\leqslant n} \theta(q^{j-i},q^{i+j-2};q^{\kappa})
\end{multline}
and
\begin{align}\label{Eq_RR5}
\sideset{}{'}\sum_{\substack{\la \\[1pt] \la_1\leqslant 2m}} &
q^{\abs{\la}/2+n\op(\la)} 
\bigg( \prod_{i=1}^{2m-1} (q^{2n};q^{4n})_{\ceil{m_i(\la)/2}} \bigg)
P_{\la}(1,q,q^2,\dots;q^{2n}) \\
&=\frac{(q^{\kappa};q^{\kappa})_{\infty}^n(-q^{\kappa/2};q^{\kappa})_{\infty}}
{2(q;q)_{\infty}^n} \notag \\ 
& \quad \times
\prod_{i=1}^n \theta(-q^{i-1},q^{i+\kappa/2-1};q^{\kappa})
\prod_{1\leqslant i<j\leqslant n}\theta(q^{j-i},q^{i+j-2};q^{\kappa}), \notag
\end{align}
where the prime denotes the restriction $m_i(\la)\equiv 0 \pmod{2}$ 
for $i=1,3,\dots,2m-1$.
\end{theorem}

In the rank-$1$ case \eqref{Eq_RR5-pre} can also be written as
\[
\sum_{n_1\geqslant\cdots\geqslant n_{2m}\geqslant 0}
\frac{q^{\frac{1}{2}(N_1^2+\cdots+N_{2m}^2+n_1+n_3+\cdots+n_{2m-1})}}
{(q;q)_{n_1}\cdots(q;q)_{n_{2m}}}
=\frac{(q^{m+1},q^{m+1},q^{2m+2};q^{2m+2})_{\infty}}
{(q;q)_{\infty}(q;q^2)_{\infty}},
\]
where $N_i=n_i+\cdots+n_{2m}$.
Since \cite[Lemma A.1]{W06}
\begin{multline*}
\sum_{n_1,\dots,n_{2m}\geqslant 0} 
\frac{a_1^{n_1}a_2^{n_2}\cdots a_{2m}^{n_{2m}}
q^{\frac{1}{2}(N_1^2+\cdots+N_{2m}^2)}}
{(q;q)_{n_1}\cdots(q;q)_{n_{2m}}} \\
=\sum_{n_1,\dots,n_m\geqslant 0} 
\frac{a_2^{n_1}a_4^{n_2}\cdots a_{2m}^{n_m}
q^{N_1^2+\cdots+N_m^2}(-q^{1/2-N_1}a_1/a_2;q)_{N_1}}
{(q;q)_{n_1}\cdots(q;q)_{n_m}},
\end{multline*}
provided that $a_{2i}/a_{2i-1}=a_2/a_1$ for all $2\leqslant i\leqslant m$,
this may also be written as
\[
\sum_{n_1,\dots,n_m\geqslant 0} 
\frac{q^{N_1^2+\cdots+N_m^2}(-q^{1-N_1};q)_{N_1}}
{(q;q)_{n_1}\cdots(q;q)_{n_m}}
=\frac{(q^{m+1},q^{m+1},q^{2m+2};q^{2m+2})_{\infty}}
{(q;q)_{\infty}(q;q^2)_{\infty}}.
\]
For $m=1$ this is identity (12) in Slater's list of
Rogers--Ramanujan-type identities \cite{Slater52}.

\begin{proof}[Sketch of the proof]
The first result follows from the principal specialisation 
\cite{Lepowsky79,Lepowsky82,Kac78}
\begin{gather*}
F:~\Complex[[\eup^{-\alpha_0},\dots,\eup^{-\alpha_n}]]\to \Complex[[q]] \\
F(\eup^{-\alpha_i})=q \quad \text{for $0\leqslant i\leqslant n$}
\end{gather*}
applied to the $\mathrm{A}_{2n-1}^{(2)\dagger}$ formula \eqref{Eq_char-new2}.
Since this specialisation does not distinguish between
$\mathrm{A}_{2n-1}^{(2)\dagger}$ and $\mathrm{A}_{2n-1}^{(2)}$, we can use
the general $\mathrm{A}_{2n-1}^{(2)}$ principal specialisation formula
\cite{Kac78,Lepowsky82}
\begin{multline*}
F\big(\eup^{-\Lambda} \ch V(\Lambda)\big)=
\frac{(q^{\kappa};q^{\kappa})_{\infty}^n}
{(q;q)_{\infty}^n(q;q^2)_{\infty}} 
\prod_{i=1}^n \theta(q^{\la_i+n-i+1};q^{\kappa}) \\
\times
\prod_{1\leqslant i<j\leqslant n} \theta(q^{\la_i-\la_j-i+j},
q^{\la_i+\la_j+2n-i-j+2};q^{\kappa}),
\end{multline*}
where
$\kappa=2n+c_0+\la_1+\la_2$ and
\[
\Lambda=c_0\fwa_0+(\la_1-\la_2)\fwa_1+\cdots+
(\la_{n-1}-\la_n)\fwa_{n-1}+\la_n\fwa_n
\]
for $c_0$ a nonnegative integer and $\la=(\la_1,\dots,\la_n)$ a partition.
Taking $c_0=0$ and $\la=m^n$ gives the right-hand side of \eqref{Eq_RR5-pre}.
The left-hand side follows in the usual way, noting that
\[
F(x_i)=q^{n-i+1} \quad (1\leqslant i\leqslant n-1)
\quad\text{and}\quad F(t)=q^{2n-1}.
\]

\medskip

The identity \eqref{Eq_RR5} follows from the specialisation $F$ applied to
the $\mathrm{A}_{2n-1}^{(2)\dagger}$ formula \eqref{Eq_A2n12dagger-character}, 
where now $F$ stands for
\begin{subequations}\label{Eq_FA2nm12}
\begin{gather}
F:~\Complex[[\eup^{-\alpha_0},\dots,\eup^{-\alpha_n}]]\to \Complex[[q]] \\
F(\eup^{-\alpha_i})=q \quad \text{for $0\leqslant i\leqslant n-1$}
\quad\text{and}\quad
F(\eup^{-\alpha_n})=q^2.
\end{gather}
\end{subequations}
According to \eqref{Eq_qxi-3} this yields
\begin{equation}\label{Eq_FAdag}
F(x_i)=q^{n-i+1/2} \quad (1\leqslant i\leqslant n)
\quad\text{and}\quad F(t)=q^{2n},
\end{equation}
which we should apply to the $\mathrm{A}_{2n-1}^{(2)\dagger}$
character given in \eqref{Eq_Bdag-Adag}.
Unlike the previous cases, the steps required to obtain the 
product form are slightly different to those in the proof of \eqref{Eq_RR1}, 
and below we outline the key steps in the derivation.

From \eqref{Eq_FAdag}, the $\mathrm{D}_n$ Vandermonde determinant
and multilinearity, it follows that 
\begin{align*}
F\bigg(&2\sum_{r\in\Z^n} \Delta_{\mathrm{D}}(xt^r)
\prod_{i=1}^n (-1)^{r_i} x_i^{\kappa r_i-i+1}
t^{\frac{1}{2}\kappa r_i^2-(n-1)r_i} \bigg) \\
&=\det_{1\leqslant i,j\leqslant n} \bigg( \sum_{r\in\Z} (-1)^r 
q^{(\kappa r-i+j)(n-i+1/2)+n\kappa r^2-2nr(n-j)} \\
&\qquad \qquad \quad +\sum_{r\in\Z} (-1)^r 
q^{(\kappa r-i-j+2n)(n-i+1/2)+n\kappa r^2+2nr(n-j)}
\bigg).
\end{align*}
After interchanging $i$ and $j$ and negating $r$ in the
first sum, the right-hand side becomes
\begin{align*}
&\det_{1\leqslant i,j\leqslant n} \bigg( \sum_{r\in\Z} (-1)^r 
q^{2n\kappa \binom{r}{2}+\kappa r/2+(2nr-i+1)(n-i)} \\
&\qquad \qquad \qquad \qquad \qquad \qquad 
\times \Big(q^{(\kappa r+n-i)(j-1)}+q^{(\kappa r+n-i)(2n-j)}\Big)
\bigg) \\[1mm]
&\quad =
\det_{1\leqslant i,j\leqslant n} \bigg( \sum_{r\in\Z} (-1)^r 
x_i^{2nr-i+1} q^{2n\kappa \binom{r}{2}+\kappa r/2} 
\Big( (x_iq^{\kappa r})^{j-1}- (x_i q^{\kappa r})^{2n-j}\Big) \bigg),
\end{align*}
where $x_i:=-q^{n-i}$. Up to the change
$q^{\kappa/2}\mapsto q^{-\kappa/2}$, the above determinant
is the same as the one on the right of \eqref{Eq_intermediate}.
From here on we can thus follow the previous computations to find
\begin{multline*}
F\bigg(2\sum_{r\in\Z^n} \Delta_{\mathrm{D}}(xt^r)
\prod_{i=1}^n (-1)^{r_i} x_i^{\kappa r_i-i+1}
t^{\frac{1}{2}\kappa r_i^2-(n-1)r_i} \bigg) \\
=(q^{\kappa};q^{\kappa})_{\infty}^n(-q^{\kappa/2};q^{\kappa})_{\infty}
\prod_{i=1}^n \theta(-q^{i-1},q^{i+\kappa/2-1};q^{\kappa}) \\
\times \prod_{1\leqslant i<j\leqslant n} \theta(q^{j-i},q^{i+j-2};q^{\kappa}).
\qedhere
\end{multline*}
\end{proof}

We conclude this section with two remarks. First of all, we have not
considered the specialisations of \eqref{Eq_Cn-character} and 
\eqref{Eq_A2n2dagger-character} as the resulting
$\mathrm{C}_n^{(1)}$ and $\mathrm{A}_{2n}^{(2)}$ identities 
were already obtained in \cite{GOW14}, the $\mathrm{A}_{2n}^{(2)}$ case
corresponding to a generalisation of the Rogers--Ramanujan and 
Andrews--Gordon identities for odd moduli.
We have also omitted the specialisation of 
\eqref{Eq_A2n12-character} and \eqref{Eq_Bn1-character}, but for
different reasons.
The most natural substitutions on the combinatorial sides
would be $x_i\mapsto q^{n-i+1/2}$, $(1\leqslant i\leqslant n)$,
$t\mapsto t^{2n}$ and $x_i\mapsto q^{n-i+1}$ $(1\leqslant i\leqslant n-1)$,
$t\mapsto t^{2n-1}$ respectively.
This corresponds to the specialisations
\[
F(\eup^{-\alpha_0})=q^2 \quad\text{and}\quad
F(\eup^{-\alpha_i})=q \quad \text{for $1\leqslant i\leqslant n$}
\]
for $\mathrm{A}_{2n-1}^{(2)}$, and 
\[
F(\eup^{-\alpha_i})=q \quad \text{for $1\leqslant i\leqslant n-1$}
\quad\text{and}\quad
F(\eup^{-\alpha_0})=q^2,~F(\eup^{-\alpha_n})=-1
\]
for $\mathrm{B}_n^{(2)}$. However, $F(\eup^{-m\fwa_0}\ch V(m\fwa_0))$ 
does not factor for such $F$.

\section[Quadratic transformations]{Quadratic transformations for Kaneko--Macdonald-type 
basic hypergeometric series}\label{Sec_KM}

\subsection{Kaneko--Macdonald-type basic hypergeometric series}
Basic hypergeometric series of Kaneko--Macdonald type, which were
first introduced in \cite{Kaneko96,Macdonald13}, are an important 
generalisation of ordinary basic hypergeometric series to multiple
series with Macdonald polynomial argument.
They have been extensively studied in
\cite{BF99,Kaneko96,Kaneko98,LRW09,LW11,Macdonald13,Rains05,Rains06} and
applied to problems in enumerative and algebraic combinatorics,
such as the enumeration of Lozenge tilings \cite{Rosengren15}, 
the computation of the major index generating function of standard Young 
tableaux \cite{KS14} and the evaluation of Selberg integrals 
and Dyson-like constant terms identities
\cite{Kaneko96,Kaneko97,Kaneko98,W05,W08,W09,W10}. 

For $x=(x_1,\dots,x_n)$, the Kaneko--Macdonald basic hypergeometric
series $\qhyp{r}{s}{}$ is defined as
\begin{multline*}
\qhyp{r}{s}{}
\bigg[\genfrac{}{}{0pt}{}{a_1,\dots,a_r}{b_1,\dots,b_s};q,t;x\bigg] \\
:=\sum_{\la} \frac{(a_1,\dots,a_r;q,t)_{\la}}{(b_1,\dots,b_s;q,t)_{\la}} 
\Big((-1)^{\abs{\la}} q^{n(\la')} t^{-n(\la)}\Big)^{s-r+1} \,
\frac{t^{n(\la)}P_{\la}(x;q,t)}{C^{-}_{\la}(q;q,t)}.
\end{multline*}
For our purposes it suffices to consider the principal specialisation
\begin{align}\label{Eq_KM-PS}
&\qhyp{r}{s}{(n)}
\bigg[\genfrac{}{}{0pt}{}{a_1,\dots,a_r}{b_1,\dots,b_s};q,t;z\bigg] \\
&\quad:=\qhyp{r}{s}{}\bigg[\genfrac{}{}{0pt}{}
{a_1,\dots,a_r}{b_1,\dots,b_s};q,t;z(1,t,\dots,t^{n-1})\bigg] \notag \\
&\quad\hphantom{:}=\sum_{\substack{\la \\[1pt] l(\la)\leqslant n}} 
\frac{(t^n,a_1,\dots,a_r;q,t)_{\la}}
{(b_1,\dots,b_s;q,t)_{\la}}
\Big((-1)^{\abs{\la}} q^{n(\la')} t^{-n(\la)}\Big)^{s-r+1}\,
\frac{z^{\abs{\la}} t^{2n(\la)}}{C^{-}_{\la}(q,t;q,t)}. \notag
\end{align}
Since $C^{-}_r(q,t;q,t)=(q,t;q)_r$ we have
$\qhyp{r}{s}{(1)}=\qhypc{r}{s}$, with on the right an ordinary
basic hypergeometric series, see \cite{GR04}.
Due to the factor $(t^n;q,t)_{\la}$, the summand vanishes unless 
$l(\la)\leqslant n$. For generic $b_1,\dots,b_r$ the restriction 
in the sum over $\la$ may thus be dropped. 
If $a_r=q^{-m}$ the series terminates, with support given 
by $\la\subseteq m^n$. A $\qhyp{r+1}{r}{(n)}$ series is said
to be balanced if
\[
t^{n-1}a_1\cdots a_{r+1}=b_1\cdots b_r\quad\text{and}\quad z=q.
\]

Assume that $r\geqslant s$.
Replacing $\la$ by its complement with respect to $m^n$,
and using \eqref{Eq_Gdual} as well as
\begin{equation}\label{Eq_amn}
(a;q,t)_{m^n}=(-a)^{mn} q^{n\binom{m}{2}} t^{-m\binom{n}{2}}
(q^{1-m}t^{n-1}/a;q,t)_{m^n},
\end{equation}
gives
\begin{align}\label{Eq_reverse}
&\qhyp{r+1}{s}{(n)}\bigg[\genfrac{}{}{0pt}{}
{a_1,\dots,a_r,q^{-m}}{b_1,\dots,b_s};q,t;z\bigg] \\
&\;=\Big(\frac{z}{q}\Big)^{mn} 
\Big((-1)^{mn} q^{n\binom{m}{2}} t^{-m\binom{n}{2}}\Big)^{s-r-1}\,
\frac{(a_1,\dots,a_r;q,t)_{m^n}}{(b_1,\dots,b_s;q,t)_{m^n}} 
\notag \\[1mm] & \;\quad\times
\qhyp{r+1}{r}{(n)}\bigg[\genfrac{}{}{0pt}{}
{0^{r-s},q^{1-m}t^{n-1}/b_1,\dots,q^{1-m}t^{n-1}/b_s,q^{-m}}{
q^{1-m}t^{n-1}/a_1,\dots,q^{1-m}t^{n-1}/a_r};q,t;
\frac{b_1\cdots b_sq^{m+1}}{a_1\cdots a_rzt^{n-1}}\bigg].  \notag 
\end{align}
Here $0^{r-s}$ represents $r-s$ numerator parameters equal to $0$.
Similarly, replacing $\la$ by its conjugate and applying
\eqref{Eq_qtswap} yields the duality relation
\begin{multline}\label{Eq_duality}
\qhyp{r+1}{s}{(n)}\bigg[\genfrac{}{}{0pt}{}
{a_1,\dots,a_r,q^{-m}}{b_1,\dots,b_s};q,t;z\bigg] \\[1mm]
=\qhyp{r+1}{r}{(m)}\bigg[\genfrac{}{}{0pt}{}
{1/a_1,\dots,1/a_r,t^{-n}}{0^{r-s},1/b_1,\dots,1/b_s};t,q;
\frac{a_1\cdots a_rzt^n}{b_1\cdots b_sq^m}\bigg],
\end{multline}
where it is again assumed that $r\geqslant s$.

The expression of the (monic) Askey--Wilson polynomials as a 
balanced $\qhypc{4}{3}$ series \cite{AW85} 
has an analogue for principally specialised
Koornwinder polynomials indexed by the rectangular partition $m^n$.
\begin{lemma}\label{Lem_Ksquare}
For $m$ a nonnegative integer,
\begin{align}\label{Eq_AWn}
&K_{m^n}\big(z(1,t,\dots,t^{n-1});q,t;\tees\big) \\
&\qquad = t_0^{-mn} t^{-m\binom{n}{2}}\,
\frac{(t_0t_1t^{n-1},t_0t_2t^{n-1},t_0t_3t^{n-1};q,t)_{m^n}}
{(t_0t_1t_2t_3q^{m-1}t^{n-1};q,t)_{m^n}} \notag \\[1mm] 
& \qquad \qquad \times
\qhyp{4}{3}{(n)}\bigg[\genfrac{}{}{0pt}{}
{zt_0t^{n-1},t_0/z,t_0t_1t_2t_3q^{m-1}t^{n-1},q^{-m}}
{t_0t_1t^{n-1},t_0t_2t^{n-1},t_0t_3t^{n-1}};q,t;q\bigg]. \notag
\end{align}
\end{lemma}

For later use we note that by \eqref{Eq_amn} and \eqref{Eq_reverse} 
we may rewrite this as
\begin{align}\label{Eq_Kmaxrectangle}
&K_{m^n}\big(z(1,t,\dots,t^{n-1});q,t;\tees\big) \\
&\qquad = z^{-mn} t^{-m\binom{n}{2}}
(zt_0t^{n-1},zq^{1-m}t^{n-1}/t_0;q,t)_{m^n} \notag \\[1mm]
& \qquad \qquad \times
\qhyp{4}{3}{(n)}\bigg[\genfrac{}{}{0pt}{}
{q^{1-m}/t_0t_1,q^{1-m}/t_0t_2,q^{1-m}/t_0t_3,q^{-m}}
{zq^{1-m}t^{n-1}/t_0,q^{1-m}/zt_0,q^{2-2m}/t_0t_1t_2t_3};q,t;q\bigg]. \notag
\end{align}
We further note that the symmetry of the left-hand side of
\eqref{Eq_AWn} under permutation of the $t_i$ implies the 
multiple Sears transformation \cite[Eq.~(5.10)]{BF99} 
\begin{multline}\label{Eq_Sears}
\qhyp{4}{3}{(n)}\bigg[\genfrac{}{}{0pt}{}
{a,b,c,q^{-m}}{d,e,f};q,t;q\bigg] \\
=\frac{(e/a,f/a;q,t)_{m^n}}{(e,f;q,t)_{m^n}}\, a^{mn}\:
\qhyp{4}{3}{(n)}\bigg[\genfrac{}{}{0pt}{}
{a,d/b,d/c,q^{-m}}{d,de/bc,df/bc};q,t;q\bigg],
\end{multline}
where $t^{n-1}abc=q^{m-1}def$.

\begin{proof}[Proof of Lemma~\ref{Lem_Ksquare}]
In \cite[Theorem 7.10]{Okounkov98} Okounkov gives an
expansion of the Koornwinder polynomials
in terms of $\mathrm{BC}_n$ interpolation polynomials 
$\bar{P}_{\mu}^{\ast}(x;q,t,s)$ \cite{Okounkov98,Rains05}.
The coefficients in this expansion are $\mathrm{BC}_n$
$q$-binomial coefficients $\qbin{\la}{\hspace{0.5pt}\mu\hspace{0.5pt}}_{q,t,s}$ 
(see e.g., \cite[page 64]{Rains05}) times a ratio of
principally specialised Koornwinder polynomials:
\begin{multline}\label{Eq_Binomial-Koornwinder}
K_{\la}(x;q,t;\tees) \\
=\sum_{\mu\subseteq\la}
\qbin{\la}{\mu}_{q,t,s} \,
\frac{K_{\la}\big(t_0(1,t,\dots,t^{n-1});q,t;\tees\big)}
{K_{\mu}\big(t_0(1,t,\dots,t^{n-1});q,t;\tees\big)}\,
\bar{P}_{\mu}^{\ast}(x;q,t,t_0),
\end{multline}
where $s=t^{n-1}\sqrt{t_0t_1t_2t_3/q\,}$.

Specialising $x=z(1,t,\dots,t^{n-1})$ and $\la=m^n$ in 
\eqref{Eq_Binomial-Koornwinder}, and using 
\cite[Proposition 4.1]{Rains05}
\[
\qbin{m^n}{\mu}_{q,t,s}=(-q)^{\abs{\mu}} t^{n(\mu)} q^{n(\mu')}\,
\frac{(t^n,q^{-m},s^2q^mt^{1-n};q,t)_{\mu}}
{C_{\mu}^{-}(q,t;q,t)C_{\mu}^{+}(s^2;q,t)}
\]
as well as \cite[Theorem 3]{vDiejen96}, \cite{Sahi99}
\begin{multline*}
K_{\la}\big(t_0(1,t,\dots,t^{n-1});q,t;\,t_0,t_1,t_2,t_3\big) \\
=\frac{t^{n(\la)}}{(t_0t^{n-1})^{\abs{\la}}} \cdot
\frac{(t^n,t_0t_1t^{n-1},t_0t_2t^{n-1},t_0t_3t^{n-1};q,t)_{\la}}
{C^{-}_{\la}(t;q,t)C^{+}_{\la}(t_0t_1t_2t_3t^{2n-2}/q;q,t)} 
\end{multline*}
and \cite[Corollary 3.11]{Rains05}
\[
\bar{P}_{\mu}^{\ast}\big(z(1,t,\dots,t^{n-1});q,t,s\big)
=\frac{t^{2n(\mu)} q^{-n(\mu')}}{(-st^{n-1})^{\abs{\mu}}}\cdot
\frac{(t^n,s/z,szt^{n-1};q,t)_{\mu}} {C_{\mu}^{-}(t;q,t)},
\]
we obtain \eqref{Eq_Kmaxrectangle}.
\end{proof}

Lemma~\ref{Lem_KBn} implies an analogue of \eqref{Eq_Kmaxrectangle} 
for the Macdonald--Koornwinder polynomial
$K_{m^n}(q,t;t_2,t_3)$.

\begin{lemma}\label{Lem_Ksquare2}
For $m$ a nonnegative integer or half-integer,
\begin{align}\label{Eq_half-Ksquare}
&K_{m^n}\big(z(1,t,\dots,t^{n-1});q,t;t_2,t_3\big) \\
&\qquad = z^{-mn} t^{-m\binom{n}{2}}
(-zq^{1/2-m}t^{n-1};q,t)_{(2m)^n} \notag \\[1mm]
& \qquad \qquad \times
\qhyp{4}{3}{(n)}\bigg[\genfrac{}{}{0pt}{}
{-q^{1/2-m}/t_2,-q^{1/2-m}/t_3,q^{1/2-m},q^{-m}}
{-zq^{1/2-m}t^{n-1},-q^{1/2-m}/z,q^{3/2-2m}/t_2t_3};q,t;q\bigg]. \notag
\end{align}
\end{lemma}

\begin{proof}
When $m$ is an integer the claim follows from 
\eqref{Eq_Kmaxrectangle} by specialising $\{t_0,t_1\}=\{-q^{1/2},-1\}$ 
and applying
\begin{equation}\label{Eq_double}
(a,aq^{-m};q,t)_{m^n}=(aq^{-m};q,t)_{(2m)^n}.
\end{equation}

When $m$ is a half-integer we set $m=k+1/2$, where $k$ a nonnegative integer.
By Lemma~\ref{Lem_KBn}, we then need to show that
\begin{align}\label{Eq_halfintcase}
K_{k^n} &\big(z(1,t,\dots,t^{n-1});q,t;-q,-q^{1/2},t_2,t_3\big) \\
&\qquad = z^{-kn} t^{-k\binom{n}{2}} (-zq^{-k}t^{n-1},-zqt^{n-1};q,t)_{k^n} 
\notag \\ & \qquad \qquad \times
\qhyp{4}{3}{(n)}\bigg[\genfrac{}{}{0pt}{}
{-q^{-k}/t_2,q^{-k}/t_3,q^{-k-1/2},q^{-k}}
{-zq^{-k}t^{n-1},-q^{-k}/z,q^{1/2-2k}/t_2t_3};q,t;q\bigg]. \notag 
\end{align}
Here we have also used 
\[
\prod_{i=1}^n \big(x_i^{1/2}+x_i^{-1/2}\big)\big|_{x_i=zt^{i-1}}=
z^{-n/2} t^{-\binom{n}{2}/2} (-z;t)_n
\]
and
\[
\frac{(-zq^{-k}t^{n-1};q,t)_{(2k+1)^n}}{(-z;t)_n}=
(-zq^{-k}t^{n-1},zqt^{n-1};q,t)_{k^n}.
\]
Since \eqref{Eq_halfintcase} is \eqref{Eq_Kmaxrectangle} with 
$(t_0,t_1,m)\mapsto (-q,-q^{1/2},k)$ we are done.
\end{proof}

\medskip

Equipped with the above lemmas we can specialise
Theorems~\ref{THM_BOUNDED1} and \ref{THM_BOUNDED5} 
to obtain quadratic transformation formulas
for Kaneko--Macdonald-type basic hypergeometric series.

\subsection{The specialisation of Theorem~\ref{THM_BOUNDED1}}
Our first result arises by specialising $x\mapsto z(1,t,\dots,t^{n-1})$ 
in the bounded Littlewood identity~\eqref{Eq_bounded1a0}.

\begin{proposition}\label{Prop_4321-trafo}
For $m$ a nonnegative integer,
\begin{multline}\label{Eq_4321-trafo}
\qhyp{4}{3}{(n)}\bigg[\genfrac{}{}{0pt}{}
{a,-a,-t^n,q^{-m}}{aq^{1/2}t^n,-aq^{1/2}t^n,-q^{-m}/t};q,t;q\bigg] \\[0.5mm]
=\frac{(a^2t)^{mn} q^{m^2n} (qt^{2n};q,t)_{m^n}}
{(a^2qt^{2n};q^2,t^2)_{m^n}(-qt^n;q,t)_{m^n}}\:
\qhyp{2}{1}{(n)}\bigg[\genfrac{}{}{0pt}{}
{q^{-m}/t,q^{-m}}{qt^{2n-1}};q,t^2;\frac{q}{a^2}\bigg].
\end{multline}
\end{proposition}

Before giving the details of the proof, we first present two equivalent
identities which allow us to connect with quadratic transformation
formulas of Cohl et al.~\cite{CCSHW14} and Gasper and Rahman \cite{GR86}. 
To this end we need the very-well poised multiple 
basic hypergeometric series \cite{CG06,Rains05,W02}
\begin{align}\label{Eq_VWP-BHS}
&\Whyp{r+1}{r}{(n)}(a_1;a_4,\dots,a_{r+1};q,t;z) \\[1mm]
&\quad:=\sum_{\substack{\la \\[1pt] l(\la)\leqslant n}}\bigg(
\frac{(a_1q,a_1q^2,a_1q/t,a_1q^2/t;q^2,t^2)_{\la}}
{C_{\la}^{+}(a_1,a_1q/t;q,t)}\notag \\[-3mm] & \quad\quad\qquad\qquad\times
\frac{(t^n,a_4,\dots,a_{r+1};q,t)_{\la}}
{(a_1q/t^n,a_1q/a_4,\dots,a_1q/a_{r+1};q,t)_{\la}} 
\cdot\frac{z^{\abs{\la}} t^{2n(\la)}}{C^{-}_{\la}(q,t;q,t)} \bigg),
\notag 
\end{align}
where $\abs{q},\abs{z}<1$ if the series does not terminate.
For $n=1$ this definition simplifies to that of the 
very-well poised basic hypergeometric series
\[
\Whyp{r+1}{r}{}(a_1;a_4,\dots,a_{r+1};q,z)
:=\sum_{k=0}^{\infty}\frac{1-a_1q^{2k}}{1-a_1}\cdot
\frac{(a_4,\dots,a_{r+1};q)_k\, z^k}
{(q,a_1q/a_4,\dots,a_1q/a_{r+1};q)_k},
\]
see \cite{GR04}.
For nonnegative integers $n$ and $m$ it follows from 
\eqref{Eq_qtswap} that
\begin{multline}\label{Eq_duality2}
\Whyp{r+1}{r}{(n)}(a_1;a_4,\dots,a_r,q^{-m};q,t;z) \\[1mm]
=\Whyp{r+1}{r}{(m)}\bigg((a_1qt)^{-1};a_4^{-1},\dots,a_r^{-1},t^{-n};t,q;
\frac{(a_4\cdots a_r)^2zq^{1-2m}t^{2n-1}}{(a_1q)^{r-3}}\bigg).
\end{multline}

Watson's transformation between a very-well poised ${_8W_7}$ series
and a balanced ${_4\phi_3}$ series (see \cite[Equation (III.18)]{GR04})
has the following multiple analogue \cite{CG06,W02,Rains05}:
\begin{multline}\label{Eq_Watson}
\Whyp{8}{7}{(n)}\big(a;b,c,d,e,q^{-m};q,t;
a^2q^{m+2}t^{1-n}/bcde\big) \\[0.5mm]
=\frac{(aq,aq/de;q,t)_{m^n}}{(aq/d,aq/e;q,t)_{m^n}}\,
\qhyp{4}{3}{(n)}\bigg[\genfrac{}{}{0pt}{}
{aq/bc,d,e,q^{-m}}{aq/b,aq/c,deq^{-m}t^{n-1}/a};q,t;q\bigg].
\end{multline}
Transforming the right-hand side using the multiple Sears transformation
\eqref{Eq_Sears} with 
\[
(a,b,c,d,e,f)\mapsto(e,aq/bc,d,deq^{-m}t^{n-1}/a,aq/b,aq/c)
\]
yields the iterated Watson transformation
\begin{align}\label{Eq_iterated-Watson}
\Whyp{8}{7}{(n)}&\big(a;b,c,d,e,q^{-m};q,t;
a^2q^{m+2}t^{1-n}/bcde\big) \\[0.5mm]
&=\frac{(aq,aq/be,aq/ce,aq/de;q,t)_{m^n}}{(aq/b,aq/c,aq/d,aq/e;q,t)_{m^n}}\,
e^{mn} \notag \\[1.1mm] & \qquad \times
\qhyp{4}{3}{(n)}\bigg[\genfrac{}{}{0pt}{}
{eq^{-m}t^{n-1}/a,bcdeq^{-1-m}t^{n-1}/a^2,e,q^{-m}}
{beq^{-m}t^{n-1}/a,ceq^{-m}t^{n-1}/a,deq^{-m}t^{n-1}/a};q,t;q\bigg].
\notag
\end{align}

Applying \eqref{Eq_iterated-Watson} with
\[
(a,b,c,d,e)\mapsto 
\big(q^{-m}t^{2n-1}/a,-q^{1/2}t^n,q^{-m}/at,q^{1/2}t^n,-t^n)
\]
to the $\qhyp{4}{3}{(n)}$ series in
\eqref{Eq_4321-trafo}, then replacing $a$ by $1/c$, and finally
simplifying some of the $q,t$-shifted factorials using \eqref{Eq_amn},
\[
(a,-a,aq^{1/2},-aq^{1/2};q,t)_{m^n}=(a;q,t^2)_{(2m)^n}
\]
and
\[
(a,at^n;q,t)_{m^n}=(at^n;q,t)_{m^{2n}},
\]
we obtain the following quadratic transformation formula.

\begin{corollary}\label{Cor_multiple-Cohl}
For $m$ a nonnegative integer,
\begin{multline}\label{Eq_Cohl-terminating}
\qhyp{2}{1}{(n)}\bigg[\genfrac{}{}{0pt}{}
{q^{-m}/t,q^{-m}}{qt^{2n-1}};q,t^2;c^2q\bigg]
=\frac{(c^2q^{1-2m}t^{2n-2};q,t^2)_{(2m)^n}}
{(cq^{1-m}t^{2n-1};q,t)_{m^{2n}}} \\[1mm] 
\times \Whyp{8}{7}{(n)}
\big(cq^{-m}t^{2n-1};cq^{-m}/t,-t^n,q^{1/2}t^n,-q^{1/2}t^n,q^{-m};q,t;cq\big).
\end{multline}
\end{corollary}
For $n=1$ this is a terminating version of \cite[Theorem 21]{CCSHW14}
by Cohl et al., suggesting the following nonterminating analogue.

Let 
\[
(a;q,t)_{\infty^n}:=\prod_{i=1}^n (at^{1-i};q)_{\infty}.
\]

\begin{theorem}\label{Thm_Cohl}
For $\abs{q},\abs{cq},\abs{c^2q}<1$,
\begin{multline}\label{Eq_Cohl}
\qhyp{2}{1}{(n)}\bigg[\genfrac{}{}{0pt}{}
{a,a/t}{qt^{2n-1}};q,t^2;c^2q\bigg]
=\frac{(a^2c^2qt^{2n-2};q,t^2)_{\infty^n}}{(c^2qt^{2n-2};q,t^2)_{\infty^n}}
\cdot\frac{(cqt^{2n-1};q,t)_{\infty^{2n}}}
{(acqt^{2n-1};q,t)_{\infty^{2n}}} \\[1mm] 
\times \Whyp{8}{7}{(n)}
\big(act^{2n-1};ac/t,a,-t^n,q^{1/2}t^n,-q^{1/2}t^n;q,t;cq\big).
\end{multline}
\end{theorem}

An easy consistency check is provided by the $ac=t$ case. Then the
$\Whyp{8}{7}{(n)}$ series trivialises to $1$
and the $\qhyp{2}{1}{(n)}$ series can be summed by the multiple Gauss
sum \cite[Proposition 5.4]{Kaneko96} 
\begin{equation}\label{Eq_Gauss}
\qhyp{2}{1}{(n)}\bigg[\genfrac{}{}{0pt}{}
{a,b}{c};q,t;\frac{ct^{1-n}}{ab}\bigg]=
\frac{(c/a,c/b;q,t)_{\infty^n}}{(c,c/ab;q,t)_{\infty^n}},
\end{equation}
for $\abs{q},\abs{ct^{1-n}/ab}<1$.
A full proof of Theorem~\ref{Thm_Cohl} based on a quadratic transformation
formula for elliptic Selberg integrals is given in 
Appendix~\ref{App_B}.

Applying the dualities \eqref{Eq_duality} and \eqref{Eq_duality2} 
to a terminating summation or transformation formula for 
Kaneko--Macdonald-type basic hypergeometric series, 
interchanges the roles of $m$ and $n$ (and $q$ and $t$).
Since, by lack of a parameter $m$, no analogues of these dualities exist 
in the nonterminating setting, a second inequivalent nonterminating
analogue of \eqref{Eq_Cohl-terminating} can be obtain by first dualising,
then interchanging $m$ and $n$ as well as $q$ and $t$, and 
finally, after writing the resulting transformation in a suitable form, 
by replacing $q^{-m}$ by $a$.
To be more precise, the dual of \eqref{Eq_Cohl-terminating}
(after making the substitutions $m\leftrightarrow n$ and 
$q\leftrightarrow t$) is
\begin{multline*}
\qhyp{2}{1}{(n)}\bigg[\genfrac{}{}{0pt}{}
{qt^n,q^{-2m}}{q^{1-2m}/t};q^2,t;c^2t^{-2n}\bigg]
=\frac{(c^2;q^2,t)_{m^{2n}}}{(c;q,t)_{(2m)^n}} \\[1mm]
\times \Whyp{8}{7}{(n)}
\big(q^{-2m}t^{n-1}/c;qt^n/c,q^{-m}t^{-1/2},-q^{-m}t^{-1/2},
q^{-m},-q^{-m};q,t;q/c\big).
\end{multline*}
Using the generalised Watson and Sears transformations \eqref{Eq_Watson}
and \eqref{Eq_Sears}, we can transform the $\Whyp{8}{7}{(n)}$ on the right. 
Also replacing $c\mapsto ct^{1/2}$ this gives our second
corollary.\footnote{Alternatively, \eqref{Eq_GR-terminating}
may be obtained by equating \eqref{Eq_2phi1-left} and \eqref{Eq_3phi3-right}
below, and using \eqref{Eq_Sears} and \eqref{Eq_Watson}.}

\begin{corollary}\label{Cor_multiple-GR}
For $m$ a nonnegative integer,
\begin{multline}\label{Eq_GR-terminating}
\qhyp{2}{1}{(n)}\bigg[\genfrac{}{}{0pt}{}
{qt^n,q^{-2m}}{q^{1-2m}/t};q^2,t;c^2t^{1-2n}\bigg]=
\frac{(c^2q^{-2m};q^2,t)_{m^{2n}}}
{(c^2q^{-m},q^{-2m}/t;q,t)_{m^n}} \\[1mm]
\times \Whyp{8}{7}{(n)}
\big(c^2 q^{-m-1};c,-c,ct^{1/2},-ct^{1/2},q^{-m};q,t;q^{-m}t^{-n}\big).
\end{multline}
\end{corollary}

For $n=1$ this is a terminating analogue of a transformation formula
of Gasper and Rahman \cite[Equation (1.4)]{GR86} (see also
\cite[Equation (3.5.4)]{GR04}). This suggests a second
nonterminating transformation formula.

\begin{theorem}\label{Thm_GR}
For $\abs{q},\abs{c^2t^{1-2n}},\abs{at^{-n}}<1$,
\begin{multline}\label{Eq_GR}
\qhyp{2}{1}{(n)}\bigg[\genfrac{}{}{0pt}{}
{a^2,qt^n}{a^2q/t};q^2,t;c^2t^{1-2n}\bigg]
=\frac{(a^2c^2;q^2,t)_{\infty^{2n}}}
{(c^2;q^2,t)_{\infty^{2n}}} \cdot
\frac{(a/t,c^2;q,t)_{\infty^n}} 
{(a^2/t,ac^2;q,t)_{\infty^n}} \\[1mm]
\times \Whyp{8}{7}{(n)}
\big(ac^2/q;a,c,-c,ct^{1/2},-ct^{1/2};q,t;at^{-n}\big).
\end{multline}
\end{theorem}

For a proof of this result we again refer to
Appendix~\ref{App_B}.

\smallskip

\begin{proof}[Proof of Proposition~\ref{Prop_4321-trafo}]
In \eqref{Eq_bounded1a0} we specialise $x\mapsto z(1,t,\dots,t^{n-1})$
and replace the summation index $\la$ by $2\la$.
On the left side we then use 
\[
b_{2\la;m}^{\textup{oa}}(q,t)=\Big(\frac{q}{t}\Big)^{\abs{\la}}
\frac{(q^{-2m};q^2,t)_{\la}}{(q^{1-2m}/t;q^2,t)_{\la}}\cdot
\frac{C_{\la}^{-}(qt;q^2,t)}{C_{\la}^{-}(q^2;q^2,t)}
\]
(see the proof of Theorem~\ref{THM_BOUNDED1} on 
page~\pageref{proof-THM_BOUNDED1}) and
\begin{align*}
P_{2\la}\big(z(1,t,\dots,t^{n-1});q,t\big)
&\stackrel{\eqref{Eq_Pspec}}{=}
z^{2\abs{\la}} t^{2n(\la)} \frac{(t^n;q,t)_{2\la}}{C^{-}_{2\la}(t;q,t)} \\
&\stackrel{\hphantom{\eqref{Eq_Pspec}}}{=}
z^{2\abs{\la}} t^{2n(\la)} \frac{(t^n,qt^n;q^2,t)_{\la}}
{C^{-}_{\la}(t,qt;q^2,t)}.
\end{align*}
As a result we obtain the $\qhyp{2}{1}{(n)}$ series
\begin{equation}\label{Eq_2phi1-left}
\qhyp{2}{1}{(n)}\bigg[\genfrac{}{}{0pt}{}
{qt^n,q^{-2m}}{q^{1-2m}/t};q^2,t;\frac{z^2q}{t}\bigg].
\end{equation}
On the right we use
\eqref{Eq_Kmaxrectangle} with 
\[
(\tees)=\big(q^{1/2},-q^{1/2},(qt)^{1/2},-(qt)^{1/2}\big)
\]
and the elementary relation \eqref{Eq_double} to find
\begin{equation}\label{Eq_3phi3-right}
(zq^{1/2-m}t^{n-1};q,t)_{(2m)^n} \,
\qhyp{4}{3}{(n)}\bigg[\genfrac{}{}{0pt}{}
{q^{-m}t^{-1/2},-q^{-m}t^{-1/2},-q^{-m},q^{-m}}
{zq^{1/2-m}t^{n-1},q^{1/2-m}/z,q^{-2m}t^{-1}};q,t;q\bigg].
\end{equation}
Next we equate \eqref{Eq_2phi1-left} and 
\eqref{Eq_3phi3-right}\footnote{The transformation
obtained by equating \eqref{Eq_2phi1-left} and \eqref{Eq_3phi3-right}
generalises Verma's quadratic transformation 
\cite[Equation (2.5)]{Verma80}.}, 
apply the duality \eqref{Eq_duality} to both sides,
and interchange $m$ and $n$ as well as $q$ and $t$. By 
\[
(a;t,q)_{n^m}=(aq^{1-m}t^{n-1};q,t)_{m^n}
\]
(which follows from \eqref{Eq_qtswap_C0} and \eqref{Eq_Gdual_C0} 
for $\la=m^n$) this yields
\begin{multline}\label{Eq_21-43-intermediate}
\qhyp{2}{1}{(n)}\bigg[\genfrac{}{}{0pt}{}
{q^{-m}/t,q^{-m}}{qt^{2n-1}};q,t^2;z^2q^{2m}t\bigg] \\
=(zt^{n-1/2};q,t)_{m^{2n}} \,
\qhyp{4}{3}{(n)}\bigg[\genfrac{}{}{0pt}{}
{q^{1/2}t^n,-q^{1/2}t^n,-t^n,q^{-m}}
{zt^{n-1/2},q^{1-m}t^{n-1/2}/z,qt^{2n}};q,t;q\bigg].
\end{multline}
We finally rewrite the right-hand side using the multiple Sears 
transformation \eqref{Eq_Sears} with
\[
(a,b,c,d,e) \mapsto 
\big({-}t^n,q^{1/2}t^n,-q^{1/2}t^n,
zt^{n-1/2},q^{1-m}t^{n-1/2}/z\big),
\]
and use
\begin{multline*}
(zt^{n-1/2};q,t)_{m^{2n}} \,
\frac{(-zt^{-1/2},-q^{-m}/t;q,t)_{m^n}}
{(zt^{n-1/2},q^{-m}t^{-1-n};q,t)_{m^n}} \\
=\frac{(z^2t^{-1};q^2,t^2)_{m^n} 
(-q^{-m}t^{-1};q,t)_{m^n}}
{(q^{-m}t^{-1-n};q,t)_{m^n}}
\end{multline*}
to clean up the prefactor.
By the substitution $z\mapsto q^{1/2-m}t^{-1/2}/a$ and
application of \eqref{Eq_amn} the claim follows.
\end{proof}

\subsection{The specialisation of Theorem~\ref{THM_BOUNDED5}}
In this section we consider the principal specialisation of the
bounded Littlewood identity \eqref{Eq_bounded5}.

\begin{proposition}\label{Prop_Phinew2}
For $m$ a nonnegative integer,
\begin{multline}\label{Eq_Phinew2}
\qhyp{4}{3}{(n)}\bigg[\genfrac{}{}{0pt}{}
{a,aq,q^{-m},q^{1-m}}{aq^{1-m}/t,aq^{2-m}/t,qt^{2n}};q^2,t^2;q^2\bigg] \\
=\frac{(t^{2n};q^2,t^2)_{m^n}}{(t^{2n},t^{2n-1}/a;q,t^2)_{m^n}} \:
\qhyp{2}{1}{(n)}\bigg[\genfrac{}{}{0pt}{}
{-t^n,q^{-m}}{-q^{1-m}/t};q,t;\frac{q}{a}\bigg].
\end{multline}
\end{proposition}
For $n=1$, and up to the change $t\mapsto b$, this is a transformation 
stated on page 2310 of \cite{BW05}.

Again we derive some related results before giving a proof.
If we let $c,f\to 0$ and $b,d\to\infty$ in the multiple Sears
transformation \eqref{Eq_Sears}, such that $b/d$ and $f/c$ are
fixed as $z/q$ and $azq^{-m}t^{n-1}/e$ respectively, we find
\[
\qhyp{2}{1}{(n)}\bigg[\genfrac{}{}{0pt}{}
{a,q^{-m}}{e};q,t;z\bigg]
=\frac{(e/a;q,t)_{m^n}}{(e;q,t)_{m^n}}\, a^{mn}\:
\qhyp{3}{1}{(n)}\bigg[\genfrac{}{}{0pt}{}
{a,q/z,q^{-m}}{aq^{1-m}t^{n-1}/e};q,t;\frac{z}{e}\bigg].
\]
This can be used to transform the right-hand side of \eqref{Eq_Phinew2},
so that an equivalent form of that identity is given by
\begin{multline*}
\qhyp{4}{3}{(n)}\bigg[\genfrac{}{}{0pt}{}
{a,aq,q^{-m},q^{1-m}}{aq^{1-m}/t,aq^{2-m}/t,qt^{2n}};q^2,t^2;q^2\bigg] \\
=\frac{(q^{1-m}/t;q,t^2)_{m^n}}
{(aq^{1-m}/t;q,t^2)_{m^n}}\, a^{mn} \:
\qhyp{3}{1}{(n)}\bigg[\genfrac{}{}{0pt}{}
{a,-t^n,q^{-m}}{t^{2n}};q,t;-\frac{q^mt}{a}\bigg].
\end{multline*}
This generalises the $c=aq^{1-m}/t$ case of 
Jain's quadratic transformation \cite[Equation (3.6)]{Jain81}:
\[
\qhypc{4}{3}\bigg[\genfrac{}{}{0pt}{}
{a,aq,q^{-m},q^{1-m}}{c,cq,qt^2};q^2,q^2\bigg]
=\frac{(c/a;q)_m}{(c;q)_m}\,a^n \:
\fourIdx{}{4}{}{2}\phi\bigg[\genfrac{}{}{0pt}{}
{a,t,-t,q^{-m}}{t^2,aq^{1-m}/c};q,-\frac{q}{c}\bigg].
\]
Unfortunately, the obvious guess 
\begin{multline*}
\qhyp{4}{3}{(n)}\bigg[\genfrac{}{}{0pt}{}
{a,aq,q^{-m},q^{1-m}}{c,cq,qt^{2n}};q^2,t^2;q^2\bigg] \\
=\frac{(c/a;q,t^2)_{m^n}}{(c;q,t^2)_{m^n}}\, a^{mn} \:
\qhyp{4}{2}{(n)}\bigg[\genfrac{}{}{0pt}{}
{a,t^n,-t^n,q^{-m}}{t^{2n},aq^{1-m}t^{n-1}/c};q,t;-\frac{q}{c}\bigg],
\end{multline*}
is false for all $n\geqslant 2$.

Another rewriting of \eqref{Eq_Phinew2}
arises by expressing the left-hand side as a very well-poised series.
First we note that if we make the simultaneous substitutions
\[
(m,e,q,t)\mapsto \big(\floor{m/2},q^{1-2\ceil{m/2}},q^2,t^2)
\]
in \eqref{Eq_iterated-Watson}, and use
\[
\frac{(aq^2;q^2,t^2)_{\floor{m/2}^n}}
{(aq^{1+2\ceil{m/2}};q^2,t^2)_{\floor{m/2}^n}}=
\frac{(aq;q,t^2)_{m^n}}{(aq;q^2t^2)_{m^n}},
\]
it follows that
\begin{align*}
\Whyp{8}{7}{(n)}&\big(a;b,c,d,q^{-m},q^{1-m};q^2,t^2;
a^2q^{2m+3}t^{2-2n}/bcd\big) \\
&=\frac{(aq/b,aq/c,aq/d;q^2,t^2)_{m^n}(aq;q,t^2)_{m^n}}
{(aq/b,aq/c,aq/d;q,t^2)_{m^n}(aq;q^2,t^2)_{m^n}}\,
q^{-n\binom{m}{2}} \\[1mm] & \qquad \times
\qhyp{4}{3}{(n)}\bigg[\genfrac{}{}{0pt}{}
{q^{1-2m}t^{2n-2}/a,bcdq^{-1-2m}t^{2n-2}/a^2,q^{-m},q^{1-m}}
{bq^{1-2m}t^{2n-2}/a,cq^{1-2m}t^{2n-2}/a,
dq^{1-2m}t^{2n-2}/a};q^2,t^2;q^2\bigg].
\end{align*}
Taking 
\[
(a,b,c,d)\mapsto \big(q^{1-2m}t^{2n-2}/a,qt^{2n}/a,q^{1-m}/t,q^{2-m}/t\big),
\]
this can be applied to transform the left-hand side of \eqref{Eq_Phinew2}.
Then replacing $a\mapsto q/c$ and simplifying the generalised
$q$-shifted factorials using \eqref{Eq_amn}, \eqref{Eq_double} and
\[
(a,aq;q^2,t)_{\la}=(a;q,t)_{2\la},
\]
we obtain the following companion to Corollaries~\ref{Cor_multiple-Cohl}
and \ref{Cor_multiple-GR}.

\begin{corollary}
For $m$ a nonnegative integer,
\begin{multline*}
\qhyp{2}{1}{(n)}\bigg[\genfrac{}{}{0pt}{}
{-t^n,q^{-m}}{-q^{1-m}/t};q,t;c\bigg]
=\frac{(cq^{-m}t^{2n-1};q,t^2)_{m^n}
(cq^{1-2m}t^{2n-2};q^2,t^2)_{m^n}}
{(cq^{1-2m}t^{2n-2};q,t^2)_{m^n}} \\[1mm] 
\times 
\Whyp{8}{7}{(n)}\big(cq^{-2m}t^{2n-2};ct^{2n},q^{1-m}/t,q^{2-m}/t,
q^{-m},q^{1-m};q^2,t^2;c\big).
\end{multline*}
\end{corollary}

Again this permits a nonterminating analogue.

\begin{theorem}\label{Thm_nonterminating3}
For $\abs{q},\abs{c}<1$,
\begin{multline}\label{Eq_nonterminating3}
\qhyp{2}{1}{(n)}\bigg[\genfrac{}{}{0pt}{}
{a,-t^n}{-aq/t};q,t;c\bigg] 
=\frac{(act^{2n-1},acqt^{2n-2};q,t^2)_{\infty^n}}
{(ct^{2n-1};q,t^2)_{\infty^n}
(cqt^{2n-2},a^2cq^2t^{2n-2};q^2,t^2)_{\infty^n}} \\[1mm]
\times \Whyp{8}{7}{(n)}\big(
a^2ct^{2n-2};ct^{2n},a,aq,aq/t,aq^2/t;q^2,t^2;c\big).
\end{multline}
\end{theorem}

This time $c=qt^{-2n}$ provides a consistency check.
Assuming $\abs{q},\abs{q/t^{2n}}<1$ for convergence,
the left-hand side can be summed by \eqref{Eq_Gauss} and the 
right-hand side, which simplifies to a $\Whyp{6}{5}{(n)}$ series,
can be summed by
\[
\Whyp{6}{5}{(n)}\big(a;b,c,d;q,t;aqt^{1-n}/bcd\big)
=\frac{(aq,aq/bc,aq/bd,aq/cd;q,t)_{\infty^n}}
{(aq/b,aq/c,aq/d,aq/bcd;q,t)_{\infty^n}},
\]
see e.g., \cite{CG06,vDS00,W02,Rains05,Rosengren01}.
We also note that for $c=t^{-2n}$ the $\Whyp{8}{7}{(n)}$ series
trivialises to $1$. For such $c$ the left-hand side is, however,
not summable by the multiple Gauss sum \eqref{Eq_Gauss}, 
and we obtain the curious summation
\begin{equation}\label{Eq_Gauss-variant}
\qhyp{2}{1}{(n)}\bigg[\genfrac{}{}{0pt}{}
{a,-t^n}{-aq/t};q,t;t^{-2n}\bigg] 
=\frac{(a/t,aq/t^2;q,t^2)_{\infty^n}}
{(1/t;q,t^2)_{\infty^n}(q/t^2,a^2q^2/t^2;q^2,t^2)_{\infty^n}},
\end{equation}
where $\abs{q},\abs{1/t}<1$. 
For $n=1$ this is a consequence of the $n=1$ case of \eqref{Eq_Gauss},
together with Heine's contiguous relation \cite[page 26]{GR04}
\[
\qhypc{2}{1}\bigg[\genfrac{}{}{0pt}{}{a,b}{cq};q,z\bigg]
=\qhypc{2}{1}\bigg[\genfrac{}{}{0pt}{}{a,b}{c};q,z\bigg]-
cz\frac{(1-a)(1-b)}{(1-c)(1-cq)}\, 
\qhypc{2}{1}\bigg[\genfrac{}{}{0pt}{}{aq,bq}{cq^2};q,z\bigg].
\]
For $n>1$ there does not appear to be a simple analogue
of this relation which, combined with \eqref{Eq_Gauss}, 
would imply \eqref{Eq_Gauss-variant}.

\smallskip

\begin{proof}[Proof of Proposition~\ref{Prop_Phinew2}]
This time we specialise $x=z(1,t^2,\dots,t^{2n-2})$ in \eqref{Eq_bounded5}.
On the left we use \eqref{Eq_Pspec} and
\[
b_{\la;m}^{-}(q,t)=\Big({-}\frac{q}{t}\Big)^{\abs{\la}}
\frac{(q^{-m};q,t)_{\la}}{(-q^{1-m}/t;q,t)_{\la}}
\cdot \frac{C^{-}_{\la}(-t;q,t)}{C^{-}_{\la}(q;q,t)},
\]
(see the proof of Theorem~\ref{THM_BOUNDED5} on 
page~\pageref{proof-Thm_bounded5}), as well as
\[
(a^2;q^2,t^2)_{\la}=(a,-a;q,t)_{\la}\quad\text{and}\quad
C_{\la}(a^2;q^2,t^2)=C_{\la}(a,-a;q,t).
\]
On the right we first write
\[
P^{(\mathrm{B}_n,\mathrm{C}_n)}_{\halfm{n}}(x;q^2,t^2,-t)
=K_{\halfm{n}}(x;q^2,t^2;-t,-qt)
\]
using \eqref{Eq_PBC-K}, then specialise $x$, and finally
apply \eqref{Eq_half-Ksquare}. As a result,
\begin{multline*}
\qhyp{2}{1}{(n)}\bigg[\genfrac{}{}{0pt}{}
{-t^n,q^{-m}}{-q^{1-m}/t};q,t;-\frac{zq}{t}\bigg] 
=(-zq^{1-m}t^{2n-2};q^2,t^2)_{m^n} \\ \times
\qhyp{4}{3}{(n)}\bigg[\genfrac{}{}{0pt}{}
{q^{-m}/t,q^{1-m}/t,q^{-m},q^{1-m}}
{-zq^{1-m}t^{2n-2},-q^{1-m}/z,q^{2-m}/t^2};q^2,t^2;q^2\bigg].
\end{multline*}
Again we apply the Sears transformation \eqref{Eq_Sears}, this time with
\begin{multline*}
(a,b,c,d,e,m,q,t) \\ \mapsto 
(q^{1-2\ceil{m/2}},q^{-m}/t,q^{1-m}/t,
-q^{1-m}/z,-zq^{1-m}t^{2n-2},\floor{m/2},q^2,t^2).
\end{multline*}
Noting that
\begin{align*}
(-zq^{1-m}t^{2n-2};q^2,t^2)_{m^n}
&=\frac{(-zq^{-m+2\ceil{m/2}}t^{2n-2},qt^{2n};
q^2,t^2)_{\floor{m/2}^n}}
{(-zq^{1-m}t^{2n-2},q^{2\ceil{m/2}}t^{2n};
q^2,t^2)_{\floor{m/2}^n}} \\
&=\frac{(t^{2n},-zt^{2n-2};q,t^2)_{m^n}}{(t^{2n};q^2,t^2)_{m^n}}
\end{align*}
and replacing $z\mapsto -t/a$ completes the proof.
\end{proof}


\chapter{Open problems}\label{Ch_Open}
We conclude this paper with a list of open problems.

\section{Missing $q$-analogues}
If we specialise $\{t_2,t_3\}=\{\pm t^{1/2}\}$ or
$\{t_2,t_3\}=\{\pm \iup t^{1/4}\}$ in Theorem~\ref{Thm_bounded6}
then the Rogers--Szeg\H{o} polynomials in the summand factorise
by \eqref{Eq_spec}.

\begin{problem}
Find $q$-analogues of 
\[
\sum t^{\op(\la)/2} 
\bigg( \prod_{i=1}^{2m-1} (t;t^2)_{\ceil{m_i(\la)/2}}\bigg)
P_{\la}(x;t) 
=(x_1\cdots x_n)^m P^{(\mathrm{C}_n)}_{m^n}(x;t,t)
\]
and
\begin{multline*}
\sum t^{\op(\la)/4} 
\bigg(\prod_{\substack{i=1 \\[0.5pt] i \textup{ even}}}^{2m-1}
(t^{1/2};t)_{\ceil{m_i(\la)/2}}
(-t;t)_{\floor{m_i(\la)/2}}\bigg) \\ \times
\bigg(\prod_{\substack{i=1 \\[0.5pt] i \textup{ odd}}}^{2m-1}
(t;t^2)_{m_i(\la)/2} \bigg) P_{\la}(x;t) 
=(x_1\cdots x_n)^m P^{(\mathrm{C}_n)}_{m^n}\big(x;t,t^{1/2}\big).
\end{multline*}
In both cases the sum is over partitions $\la$ such that 
$\la_1\leqslant 2m$ and such that parts of odd size have even multiplicity.
\end{problem}
Similarly, if we specialise $t_2=0$ in Theorem~\ref{THM_BOUNDED7}
then the Rogers--Szeg\H{o} polynomial in the summand trivialises to $1$.
\begin{problem}
Find a $q$-analogue of 
\[
\sum_{\substack{\la \\[1pt] \la_1\leqslant m}} P_{\la}(x;t)
=(x_1\cdots x_n)^{\frac{m}{2}} P_{\halfm{n}}^{(\mathrm{B}_n)}(x;t,0).
\]
\end{problem}
As remarked previously, the above is equivalent to 
\eqref{Eq_Macdonald-HL}, which was key in Macdonald's proof of 
the MacMahon conjecture. This makes finding a $q$-analogue particularly 
desirable.

\section{Littlewood identities for near-rectangular partitions}
The bounded Littlewood identities proven in this paper correspond to
decompositions of $(R,S)$ Macdonald polynomials (or $R$ Hall--Littlewood
polynomials) indexed by rectangular partitions or half-partitions 
of maximal length. 
In the Schur case more general shapes have been considered in the 
literature.
For example, Goulden and Krattenthaler 
\cite{Goulden92,Krattenthaler93,Krattenthaler98} proved the following 
result for the character of the irreducible $\mathrm{Sp}(2n,\Complex)$-module 
of highest weight $\fwc_r+(m-1)\fwc_n$, generalising the
D\'esarm\'enien--Proctor--Stembridge formula \eqref{Eq_DPS}:
\[
\sum_{\substack{\la \\[1.5pt] \la_1\leqslant 2m \\[1.5pt] \op(\la)=r}}
s_{\la}(x) =(x_1\cdots x_n)^m \symp_{2n,m^{n-r} (m-1)^r}(x).
\]
Here $0\leqslant r\leqslant n$ ($\fwc_0:=0$) and $m^{n-r} (m-1)^r$ is shorthand
for the near-rectangular partition
\[
m^{n-r} (m-1)^r=(\underbrace{m,\dots,m}_{n-r \text{ times}},
\underbrace{m-1,\dots,m-1}_{r \text{ times}}).
\]

\begin{problem}
\textup{(}i\textup{)} 
For positive integers $m,n$ and $r$ an integer such that
$0\leqslant r\leqslant n$, prove that\footnote{The identity \eqref{Eq_P3}
has recently been proved by Nguyen \cite{Nguyen17}.}
\begin{equation}\label{Eq_P3}
\sum_{\substack{\la \\[1pt] \op(\la)=r}}
b_{\la;m,r}^{\textup{oa}}(q,t) P_{\la}(x;q,t)
=(x_1\cdots x_n)^m P^{(\mathrm{C}_n,\mathrm{B}_n)}_{m^{n-r}(m-1)^r}(x;q,t,qt),
\end{equation}
where
\[
b_{\la;m,r}^{\textup{oa}}(q,t)=
b_{\la}^{\textup{oa}}(q,t)
\prod_{\substack{s\in\skew{\la}{1^r} 
\\[1pt] a'_{\la}(s) \textup{ even}}}
\frac{1-q^{2m-a'_{\la}(s)}t^{l'_{\la}(s)}}
{1-q^{2m-a'_{\la}(s)-1}t^{l'_{\la}(s)+1}}.
\]
\textup{(}ii\textup{)} 
Prove similar such near-rectangular identities
for other admissible pairs $(R,S)$.
\end{problem}

\section{Littlewood identities of Pfaffian type}
In Chapter~\ref{Ch_Bounded} we obtained bounded 
analogues of most of the known Littlewood identities
in the literature.
There are however a number of extensions of the 
(unbounded) Littlewood identities discussed in that section 
where the product forms on the right are replaced by Pfaffians.
Two characteristic examples are 
\cite[Theorem 4.1]{IOW96}\footnote{There is an unfortunate
error in the right-hand side of \cite[Theorem 4.1]{IOW96}, in that the
factor $(1+stx_ix_j)$ should have been $(1+stvx_ix_j)$. With this
correction, it readily follows that the theorem in question only
depends on two variables instead of four. Scaling 
$x_i\mapsto x_i/(stv)^{1/2}$ and using that 
\[
s^{\sum_{i\geqslant 1} \la_{2i}-\abs{\la}/2}
t^{\sum_{i\geqslant 1} \la_{2i-1}-\abs{\la}/2}=
(t/s)^{\op(\la')/2},
\]
equation \eqref{Eq_Pfaffian-1} follows after replacing
$u/v^{1/2}$ by $a$ and $(t/s)^{1/2}$ by $b$.}
\begin{multline}\label{Eq_Pfaffian-1}
\sum_{\la} a^{\op(\la)} b^{\op(\la')} s_{\la}(x)
=\prod_{1\leqslant i<j\leqslant n} \frac{1}{x_i-x_j} 
\prod_{i=1}^n \frac{1}{1-b^2x_i^2} \\ \times
\Pf_{1\leqslant i,j\leqslant n}
\bigg(\frac{(x_i-x_j)\big( (1+b^2x_ix_j)(1+a^2x_ix_j)
+ab(1+x_ix_j)(x_i+x_j)\big)}{1-x_i^2x_j^2}\bigg)
\end{multline}
and \cite[Corollary 6.26]{Rains14} 
(see also \cite[Conjecture 1]{BWZ14} and \cite[Theorem 5]{WZ16}) 
\begin{multline}\label{Eq_Pfaffian-2}
\sum_{\la' \text{ even}} 
b_{\la}^{\textup{el}}(q,t) P_{\la}(x;q,t)
\prod_{\substack{i=1 \\[1pt] i \text{ even}}}^n 
\big(1-uq^{\la_i}t^{n-i}\big) \\
=\prod_{1\leqslant i<j\leqslant n} \frac{(tx_ix_j;q)_{\infty}}
{(x_i-x_j)(qx_ix_j;q)_{\infty}}
\Pf_{1\leqslant i,j\leqslant n}\bigg(
 \frac{(x_i-x_j)(1-u+(u-t)x_ix_j)}{(1-x_ix_j)(1-tx_ix_j)}\bigg),
\end{multline}
where in both cases $n$ is assumed to be even.
By \cite{LLT89,Stembridge90b} 
\[
\Pf_{1\leqslant i,j\leqslant n}\bigg(
\frac{x_i-x_j}{1-x_ix_j}\bigg)=\prod_{1\leqslant i<j\leqslant n}
\frac{x_i-x_j}{1-x_ix_j}, \qquad \text{$n$ even},
\]
and $\Pf(a_i a_j A_{ij})=\Pf(A_{ij})\prod_i a_i$,
the $b=1$ and $a=1$ specialisations of \eqref{Eq_Pfaffian-1}
simplify to the $q=t$ cases of \eqref{Eq_Pb} and \eqref{Eq_Pc}
respectively.
Similarly, for $u=0$ the identity \eqref{Eq_Pfaffian-2} yields
the $a=0$ case of \eqref{Eq_Pc}.

\begin{problem}
Find bounded analogues of \eqref{Eq_Pfaffian-1} and \eqref{Eq_Pfaffian-2}.
\end{problem}

\section{Elliptic Littlewood identities}
Most of the virtual Koornwinder integrals of Section~\ref{Sec_closed-form}
have elliptic analogues, see \cite{Rains12,Rains14}.

\begin{problem}
Prove elliptic analogues of the bounded Littlewood identities of
Theorems~\ref{THM_BOUNDED1}--\ref{THM_BOUNDED5}.
\end{problem}

For elliptic analogues of \eqref{Eq_Pb}--\eqref{Eq_Kawanaka} we refer 
the reader to \cite{Rains12}.

\section{$q,t$-Littlewood--Richardson coefficients} 
Let $\mathfrak{g}$ be a complex semisimple Lie algebra and
$V(\la)$ an irreducible $\mathfrak{g}$-module of highest weight 
$\la\in P_{+}$. The Littlewood--Richardson coefficient
$c_{\mu\nu}^{\la}$ of type $\mathfrak{g}$ is defined as the multiplicity 
of $V(\la)$ in the tensor product $V(\mu)\otimes V(\mu)$:
\[
c_{\mu\nu}^{\la}=\dim\Hom_{\mathfrak{g}}
\big(V(\la),V(\mu)\otimes V(\nu)\big).
\]
If $\chi_{\la}=\sum_{\mu\in P} K_{\la\mu} \eup^{\mu}$ (with $K_{\la\mu}$
the multiplicity of $\mu$ in $V(\la)$) denotes the formal character 
of $V(\la)$, then
\[
\chi_{\mu}\chi_{\nu}=\sum_{\la\in P_{+}} c_{\mu\nu}^{\la} \chi_{\la}.
\]
In \cite{Stembridge03}, Stembridge gave a classification of all
multiplicity-free tensor products, i.e., of all pairs of dominant 
weight $\mu,\nu$ such that $c_{\mu\nu}^{\la}\leqslant 1$ for all 
$\la\in P_{+}$.
For such a pair it is not generally known for which $\la$ the 
Littlewood--Richardson coefficient is actually non-vanishing,
except for a number of classical Lie algebras and specially chosen
weights (typically, for $\mu$ and $\nu$ multiples of miniscule weights).
For example, using his minor summation formula, Okada proved that 
for $\mathfrak{g}=\mathrm{B}_n$ and 
\begin{equation}\label{Eq_munu-weights}
(\mu,\nu)=(r\fwc_n,s\fwc_n)
\end{equation}
with $r,s$ nonnegative integers,
$c_{\mu\nu}^{\la}=1$ for all dominant weights $\la$ of the form
\eqref{Eq_Pplus-B}, where $(\la_1,\dots,\la_n)$
is a partition/half-partition if $r+s$ is even/odd, 
and $\la_1\leqslant (r+s)/2$, $\la_n\geqslant \abs{r-s}/2$. 
For weights $\la$ not of this form
$c_{\mu\nu}^{\la}=0$.
In terms of odd-orthogonal Schur functions, this result may be
expressed as the identity \cite[Theorem 2.5 (1)]{Okada98}
\[
\so_{2n+1,(\frac{r}{2})^n}(x) \so_{2n+1,(\frac{s}{2})^n}(x)=
\sum \so_{2n+1,(\frac{r+s}{2})^n-\la}(x).
\]
where the sum is over partitions (as opposed to half-partitions)
$\la$ of such that $\la_1\leqslant\min\{r,s\}$.

For $r,s$ nonnegative integers and $\la$ a partition, let
\begin{align*}
b_{\la;m_1,m_2}^{\textup{el}}(q,t):=b_{\la}^{\textup{el}}(q,t) &
\prod_{\substack{s\in\la \\[1pt] l'(s) \textup{ even}}} 
\frac{1-q^{m_1-a'(s)}t^{l'(s)}}{1-q^{m_1-a'(s)-1}t^{l'(s)+1}}
\cdot
\frac{1-q^{m_2-a'(s)}t^{l'(s)}}{1-q^{m_2-a'(s)-1}t^{l'(s)+1}} \\
\times&\prod_{\substack{s\in\la \\[1pt] l(s) \textup{ even}}} 
\frac{1-q^{m_1+m_2-\hat{a}(s)-1}t^{\hat{l}(s)+1}}
{1-q^{m_1+m_2-\hat{a}(s)}t^{\hat{l}(s)}}.
\end{align*}
We note that $b_{\la;m_1,m_2}^{\textup{el}}(q,t)=0$ unless 
$\la_1\leqslant\min\{m_1,m_2\}$, and 
\begin{equation}\label{Eq_m2limit}
\lim_{m_2\to\infty} b_{\la;m,m_2}^{\textup{el}}(q,t)=
b_{\la;m}^{\textup{el}}(q,t),
\end{equation}
where $b_{\la;m}^{\textup{el}}(q,t)$ is defined in \eqref{Eq_blamel}.

\begin{problem}
For $x=(x_1,\dots,x_n)$ and $r,s$ nonnegative integers, prove that
\begin{align}\label{Eq_BnLR}
P_{(\frac{r}{2})^n}^{(\mathrm{B}_n,\mathrm{B}_n)}(x;q,t,t) &
P_{(\frac{s}{2})^n}^{(\mathrm{B}_n,\mathrm{B}_n)}(x;q,t,t) \\[1mm]
&\qquad=\sum_{\la} b_{\la;r,s}^{\textup{el}}(q,t)
P_{(\frac{r+s}{2})^n-\la}^{(\mathrm{B}_n,\mathrm{B}_n)}(x;q,t,t).  \notag
\end{align}
\end{problem}
Dividing both sides by 
\[
(x_1\cdots x_n)^{-\frac{1}{2}r}
P_{(\frac{s}{2})^n}^{(\mathrm{B}_n,\mathrm{B}_n)}(x;q,t,t),
\] 
and then using that
\begin{align*}
(x_1\cdots x_n)^{\frac{1}{2}r} &
\lim_{s\to\infty} 
\frac{P_{(\frac{r+s}{2})^n-\la}^{(\mathrm{B}_n,\mathrm{B}_n)}(x;q,t,t)}
{P_{(\frac{s}{2})^n}^{(\mathrm{B}_n,\mathrm{B}_n)}(x;q,t,t)} \\
&=\frac{\lim_{s\to\infty} (x_1\cdots x_n)^{\frac{1}{2}(r+s)}
P_{(\frac{r+s}{2})^n-\la}^{(\mathrm{B}_n,\mathrm{B}_n)}(x;q,t,t)}
{\lim_{s\to\infty} (x_1\cdots x_n)^{\frac{1}{2}s} 
P_{(\frac{s}{2})^n}^{(\mathrm{B}_n,\mathrm{B}_n)}(x;q,t,t)} \\
&=P_{\la}(x;q,t),
\end{align*}
by \eqref{Eq_PBB-K} and \eqref{Eq_largem-2},
it follows that in the large-$s$ limit we recover the 
bounded Littlewood identity of Theorem~\ref{Thm_qtL}
with $m$ replaced by $r$.

There is an analogous result for $(\mathrm{C}_n,\mathrm{B}_n)$. 
For $m_1,m_2$ nonnegative integers and $\la$ an even partition, let
\begin{align*}
b_{\la;m_1,m_2}^{\textup{oa}}(q,t):=b_{\la}^{\textup{oa}}(q,t) &
\prod_{\substack{s\in\la \\[1pt] a'(s) \textup{ even}}}
\frac{1-q^{2m_1-a'(s)}t^{l'(s)}}{1-q^{2m_1-a'(s)-1}t^{l'(s)+1}}
\cdot
\frac{1-q^{2m_2-a'(s)}t^{l'(s)}}{1-q^{2m_2-a'(s)-1}t^{l'(s)+1}} \\
\times & \prod_{\substack{s\in\la\\[1pt] a(s) \text{ even}}} 
\frac{1-q^{2m_1+2m_2-\hat{a}(s)-1} t^{\hat{l}(s)+1}}
{1-q^{2m_1+2m_2-\hat{a}(s)}t^{\hat{l}(s)}}.
\end{align*}
This time $b_{\la;m_1,m_2}^{\textup{oa}}(q,t)=0$ unless 
$\la_1\leqslant\min\{2m_1,2m_2\}$, and 
\[
\lim_{m_2\to\infty} b_{\la;m,m_2}^{\textup{oa}}(q,t)=
b_{\la;m}^{\textup{oa}}(q,t),
\]
where $b_{\la;m}^{\textup{oa}}(q,t)$ for even partitions $\la$ 
is given by \eqref{Eq_blamel}.

\begin{problem}
For $x=(x_1,\dots,x_n)$ and $r,s$ nonnegative integers, prove that
\begin{align*}
P_{r^n}^{(\mathrm{C}_n,\mathrm{B}_n)}(x;q,t,qt) &
P_{s^n}^{(\mathrm{C}_n,\mathrm{B}_n)}(x;q,t,qt) \\[1mm]
&\qquad=\sum_{\la \textup{ even}} b_{\la;r,s}^{\textup{oa}}(q,t)
P_{(r+s)^n-\la}^{(\mathrm{C}_n,\mathrm{B}_n)}(x;q,t,t).
\end{align*}
Let $x=(x_1,\dots,x_n)$ and $r,s$ integers such that 
$-r\leqslant s\leqslant r$.
\end{problem}

Taking the large-$s$ limit using \eqref{Eq_largem} and \eqref{Eq_PCB-K} 
yields the Littlewood identity \eqref{Eq_bounded1a0}.
Moreover, in the classical limit we recover Okada's 
formula for the $\mathrm{C}_n$ Littlewood--Richardson coefficient
$c_{\mu\nu}^{\la}$ with $\mu$ and $\nu$ given by
\eqref{Eq_munu-weights}, see \cite[Theorem 2.5 (1)]{Okada98}.
For such $\mu,\nu$, $c_{\mu\nu}^{\la}=1$ if $\la$ is a weight
of the form \eqref{Eq_Pplus-C}, 
where $(\la_1,\dots,\la_n)$ is a partition such that 
and $\la_1\leqslant r+s$, $\la_n\geqslant \abs{r-s}$ and
$\la_1+r+s$ is even. For all other $\la$, $c_{\mu\nu}^{\la}=0$.

\section{Dyson--Macdonald-type identities}
Let $\mathfrak{g}$ be an affine Lie algebra 
with simple roots $\{\alpha_0,\dots,\alpha_n\}$.
The map $\eup^{-\alpha_1},\dots,\eup^{-\alpha_n}\mapsto 1$
(so that $\eup^{-\delta}\mapsto \eup^{-a_0\alpha_0}$)
is known as the basic specialisation \cite{Kac90}.
When applied to character formulas for affine Lie algebras,
the basic specialisation results in (generalised) Dyson--Macdonald 
type expansions for powers of the Dedekind eta-function, see e.g., 
\cite{BW13,Dyson72,Kac90,Milas04,Stoyanovsky98,WZ12}.
For example, taking the basic specialisation 
of the $\mathrm{B}_n^{(1)\dagger}$ identity \eqref{Eq_Bndagger-character} 
(and replacing $t$ by $q$) yields
the following generalisation of \cite[p. 135, (6c)]{Macdonald72}:
\begin{multline}\label{Eq_1356c}
\frac{1}{\eta(\tau/2)^{2n}\eta(\tau)^{2n^2-3n}} 
\sum (-1)^{\frac{\abs{v}-\abs{\rho}}{m+2n-1}} 
\chi_{\mathrm{D}}(v/\rho) \hspace{1pt}
q^{\frac{\|v\|^2-\|\rho\|^2}{2(m+2n-1)}+\frac{\|\rho\|^2}{2(2n-1)}} \\[-1mm]
=\sum_{\substack{\la\\[1pt] \la_1\leqslant m}} q^{\abs{\la}/2} 
P'_{\la}(\underbrace{1,\dots,1}_{2n-1 \textup{ times}};q)
\prod_{i=0}^{m-1} (-q^{1/2};q^{1/2})_{m_i(\la)}.
\end{multline}
Here $q=\exp(2\pi\iup\tau)$, $\rho=(n-1,\dots,1,0)$, 
\[
\chi_{\mathrm{D}}(v/w):=\prod_{i<j}(v_i^2-v_j^2)/(w_i^2-w_j^2),
\]
and the sum is over $v\in\Z^n$ such that $v_i\equiv \rho_i \pmod{m+2n-1}$.

Let $C=C_n$ be the Cartan matrix of the Lie algebra $\mathrm{A}_n$, i.e.,
$(C^{-1})_{ab}=\min\{a,b\}-ab/(n+1)$, and for
$\{r_i^{(a)}\}_{1\leqslant a\leqslant n;\,1\leqslant i\leqslant k}$
a set of nonnegative integers, let
\[
R_i^{(a)}:=r_i^{(a)}+\cdots+r_k^{(a)}.
\]
Following \cite{BW13,WZ12} we define
\[
F_{k,n}(q):=
\begin{cases}\displaystyle
(-q^{1/2};q)_{\infty}^{2n-1} & \text{for $k=0$} \\[2mm]\displaystyle
(-q^{1/2};q^{1/2})_{\infty}^{2n-1} \sum_{\{r_i^{(a)}\}}
\frac{q^{\frac{1}{2}\sum_{a,b=1}^n \sum_{i=1}^k C_{ab} R_i^{(a)} R_i^{(b)}}}
{\prod_{a=1}^n \prod_{i=1}^{k-1}\big((q;q)_{r_i^{(a)}}\big)
(q^2;q^2)_{r_k^{(a)}}} & \text{for $k\geqslant 1$}
\end{cases}
\]
and
\[
G_{k,n}(q):=(-q;q)_{\infty}^n \sum_{\{r_i^{(a)}\}}
\frac{q^{\frac{1}{2}\sum_{a,b=1}^n \sum_{i=1}^k C_{ab} R_i^{(a)} R_i^{(b)}}
(-q^{1/2-R_1^{(1)}};q)_{R_1^{(1)}}}
{\prod_{a=1}^n \prod_{i=1}^{k-1}\big((q;q)_{r_i^{(a)}}\big)
(q^2;q^2)_{r_k^{(a)}}}
\]
for $k\geqslant 1$.

\begin{problem}
For $m,n$ positive integers, prove that
\begin{multline*}
\sum_{\substack{\la\\[1pt] \la_1\leqslant m}} q^{\abs{\la}/2} 
P'_{\la}(\underbrace{1,\dots,1}_{2n-1 \textup{ times}};q)
\prod_{i=1}^{m-1} (-q^{1/2};q^{1/2})_{m_i(\la)} \\[-1mm]
=\begin{cases}
F_{k,2n-1}(q) & \text{if $m=2k+1$} \\
G_{k,2n-1}(q) & \text{if $m=2k$}.
\end{cases}
\end{multline*}
\end{problem}
For $n=1$ this follows from a minor modification of \cite[Lemma A.1]{W06},
for $m=1$ it follows from $P'_{(1^r)}(x;t)=e_r(x)$ and
\[
\sum_{r=0}^{\infty} z^r e_r\Big[\frac{n}{1-q}\Big]
=(-z;q)_{\infty}^n,
\]
and for $m=2$ a proof is given in \cite[Theorem 3.7]{BW13}.
By \eqref{Eq_1356c}, the above problem for even $m$ is equivalent to
\cite[Conjecture 2.4; (2.6a)]{WZ12}.

\appendix


\chapter{The Weyl--Kac formula}\label{App_A}

\addtocounter{section}{1}

In this first of two appendices we state some simple 
consequences of the Weyl--Kac formula, needed in the proofs 
of our combinatorial character formulas in Section~\ref{Sec_char}.

Recall the symplectic and odd-orthogonal Schur functions
\eqref{Eq_symplectic} and \eqref{Eq_odd-orthogonal}. 
It will be convenient to also define the normalised functions
\begin{equation}\label{Eq_soandsp}
\widetilde{\so}_{2n+1,\la}(x)=
\Delta_{\mathrm{B}}(x) \so_{2n+1,\la}(x) 
\quad\text{and}\quad
\widetilde{\symp}_{2n,\la}(x)=
\Delta_{\mathrm{C}}(x) \so_{2n+1,\la}(x),
\end{equation}
so that $\widetilde{\so}_{2n+1,0}(x)=\Delta_{\mathrm{B}}(x)$ and 
$\widetilde{\symp}_{2n,0}(x)=\Delta_{\mathrm{C}}(x)$.

Mimicking the proofs of \cite[Lemmas 2.1--2.4]{BW13} yields expressions
for the characters of $\mathrm{B}_n^{(1)}$ and $\mathrm{A}_{2n-1}^{(2)}$
in terms of the symplectic and odd orthogonal Schur functions as follows.

\begin{lemma}[$\mathrm{B}_n^{(1)}$ character formula]\label{Lem_Bn-WK}
Let 
\begin{equation}\label{Eq_qxi}
x_i:=\eup^{-\alpha_i-\cdots-\alpha_n}\quad
(1\leqslant i\leqslant n), \qquad\quad t:=\eup^{-\delta},
\end{equation}
and parametrise $\Lambda\in P_{+}$, as
\[
\Lambda=c_0\fwa_0+(\la_1-\la_2)\fwa_1+\cdots+(\la_{n-1}-\la_n)\fwa_{n-1}
+2\la_n\fwa_n,
\]
where $c_0$ is a nonnegative integer and $\la=(\la_1,\dots,\la_n)$ a 
partition or half-partition. Then
\begin{multline}\label{Eq_char-Bn}
\eup^{-\Lambda} \ch V(\Lambda)=
\frac{1}{(t;t)_{\infty}^n\prod_{i=1}^n \theta(x_i;t)
\prod_{1\leqslant i<j\leqslant n} x_j\theta(x_ix_j^{\pm};t)} \\[1mm]
\times
\sum_{\substack{r\in\Z^n \\[1pt] \abs{r}\equiv 0\:(2)}}
\widetilde{\so}_{2n+1,\la}(xt^r)
\prod_{i=1}^n x_i^{\kappa r_i+\la_i} 
t^{\frac{1}{2}\kappa r_i^2-(n-\frac{1}{2})r_i},
\end{multline}
where
$\kappa=2n-1+c_0+\la_1+\la_2$.
\end{lemma}

\begin{lemma}[$\mathrm{A}_{2n-1}^{(2)}$ character formula]
\label{Lem_A2n12-WK}
Let 
\begin{equation}\label{Eq_qxi-2}
x_i:=\eup^{-\alpha_i-\cdots-\alpha_{n-1}-\alpha_n/2}\quad
(1\leqslant i\leqslant n), \qquad\quad t:=\eup^{-\delta},
\end{equation}
and parametrise $\Lambda\in P_{+}$, as
\[
\Lambda=c_0\fwa_0+(\la_1-\la_2)\fwa_1+\cdots+
(\la_{n-1}-\la_n)\fwa_{n-1}+\la_n\fwa_n,
\]
where $c_0$ is a nonnegative integer and $\la=(\la_1,\dots,\la_n)$
a partition.
Then
\begin{multline}\label{Eq_char-A2n12}
\eup^{-\Lambda} \ch V(\Lambda)=
\frac{1}{(t;t)_{\infty}^{n-1}(t^2;t^2)_{\infty}
\prod_{i=1}^n \theta(x_i^2;t^2)
\prod_{1\leqslant i<j\leqslant n} x_j \theta(x_ix_j^{\pm};t)} \\[1mm]
\times
\sum_{\substack{r\in\Z^n \\[1pt] \abs{r}\equiv 0\:(2)}}
\widetilde{\symp}_{2n,\la}(xt^r)
\prod_{i=1}^n x_i^{\kappa r_i+\la_i} t^{\frac{1}{2}\kappa r_i^2-nr_i},
\end{multline}
where
$\kappa=2n+c_0+\la_1+\la_2$.
\end{lemma}

For a two-parameter subset of weights, the above two formulas may 
be rewritten as a pair of character formulas for $\mathrm{B}_n^{(1)\dagger}$
and $\mathrm{A}_{2n-1}^{(2)\dagger}$ where the sum is over the full
$\mathbb{Z}^n$-lattice. 
For $\kappa,k$ positive integers and $x=(x_1,\dots,x_n)$, define
\[
\mathscr{N}_{\kappa,k}(x;t):=
\sum_{r\in\Z^n} \Delta_{\mathrm{D}}(x^kt^{kr})
\prod_{i=1}^n (-1)^{r_i} x_i^{\kappa r_i-(k-1)(i-1)}
t^{\frac{1}{2}\kappa r_i^2-k(n-1)r_i},
\]
where $\Delta_{\mathrm{D}}(x)$ is the $\mathrm{D}_n$ Vandermonde 
product \eqref{Eq_VdMD} and
$x^k:=(x_1^k,\dots,x_n^k)$. Further set
\[
\mathscr{D}_{\mathfrak{g}}(x;t):=
\begin{cases}\displaystyle
(t;t)_{\infty}^n\prod_{i=1}^n \theta(t^{1/2}x_i;t)
\prod_{1\leqslant i<j\leqslant n} x_j \theta(x_ix_j^{\pm};t) 
& \text{for $\mathfrak{g}=\mathrm{B}_n^{(1)\dagger}$}, \\[3mm]
\displaystyle
(t;t)_{\infty}^{n-1}(t^2;t^2)_{\infty}\prod_{i=1}^n \theta(tx_i^2;t^2)
\prod_{1\leqslant i<j\leqslant n} x_j \theta(x_ix_j^{\pm};t)
& \text{for $\mathfrak{g}=\mathrm{A}_{2n-1}^{(2)\dagger}$}.
\end{cases}
\]

\begin{lemma}[$\mathrm{B}_n^{(1)\dagger}$ and $\mathrm{A}_{2n-1}^{(2)\dagger}$
character formulas]\label{Lem_Bn-WK-2}
Let 
\begin{equation}\label{Eq_qxi-3}
x_i:=\eup^{-\alpha_i-\cdots-\alpha_{n-1}+(\alpha_{n-1}-\alpha_n)/2}
\quad (1\leqslant i\leqslant n), \qquad\quad t:=\eup^{-\delta}
\end{equation}
and, for $k$ a positive integer and $m$ a nonnegative integer,
\begin{equation}\label{Eq_weightkm}
\Lambda=(k-1)\rho+m\fwa_0\in P_{+}.
\end{equation}
Then
\begin{equation}\label{Eq_Bdag-Adag}
\eup^{-\Lambda} \ch V(\Lambda)=
\frac{1}{\mathscr{D}_{\mathfrak{g}}(x;t)} \times
\begin{cases}
\mathscr{N}_{m+k(2n-1),k}(x;t) &
\text{for $\mathfrak{g}=\mathrm{B}_n^{(1)\dagger}$}, \\[3mm]
\mathscr{N}_{m+2kn,k}(x;t) &
\text{for $\mathfrak{g}=\mathrm{A}_{2n-1}^{(2)\dagger}$}.
\end{cases}
\end{equation}
\end{lemma}

For $m=0$ and $k=1$ this gives what may be viewed as
$\mathrm{B}_n^{(1)\dagger}$ and $\mathrm{A}_{2n-1}^{(2)\dagger}$ 
Macdonald identities:
\[
\mathscr{N}_{2n-1,1}(x;t)=\mathscr{D}_{\mathrm{B}_n^{(1)\dagger}}(x;t)
\]
and
\begin{equation}\label{Eq_Macdonald-Adagger}
\mathscr{N}_{2n,1}(x;t)=\mathscr{D}_{\mathrm{A}_{2n-1}^{(2)\dagger}}(x;t).
\end{equation}

\begin{corollary}[$\mathrm{B}_n^{(1)\dagger}$ product formula]
\label{Cor_Bdagger}
Let $k$ be a positive integer and
\begin{equation}\label{Eq_Lambda-k}
\Lambda=(k-1)\rho+k\fwa_0\in P_{+}.
\end{equation}
Then
\[
\eup^{-\Lambda} \ch V(\Lambda)=
\frac{\mathscr{D}_{\mathrm{A}_{2n-1}^{(2)\dagger}}(x;t)}
{\mathscr{D}_{\mathrm{B}_n^{(1)\dagger}}(x;t)}\,
\prod_{i=1}^n x_i^{-(k-1)(i-1)}.
\]
\end{corollary}

\begin{proof}
If we set $m=k$ in the $\mathrm{B}_n^{(1)\dagger}$ case 
of \eqref{Eq_Bdag-Adag} then
\[
\eup^{-\Lambda} \ch V(\Lambda)=\frac{\mathscr{N}_{2kn,k}(x;t)}
{\mathscr{D}_{\mathrm{B}_n^{(1)\dagger}}(x;t)}
\]
with $\Lambda$ as in \eqref{Eq_Lambda-k}. By
\[
\mathscr{N}_{2kn,k}(x;t)=\mathscr{N}_{2n,1}(x^k;t^k)
\prod_{i=1}^n x_i^{-(k-1)(i-1)}
\]
and \eqref{Eq_Macdonald-Adagger} the result follows.
\end{proof}

We still need to prove Lemma~\ref{Lem_Bn-WK-2}. For this we
first prepare a determinantal identity.
For $\sigma=0,1$ and $\kappa,k$ positive integers, let
\begin{multline*}
\mathscr{N}_{\kappa,k;\sigma}(x;t):=(-1)^{\sigma}
\sum_{\substack{r\in\Z^n \\[1pt] \abs{r}\equiv \sigma\:(2)}}
\det_{1\leqslant i,j\leqslant n} 
\Big( x_i^{-\kappa r_i-k(i-j)+i-1} 
t^{\frac{1}{2}\kappa r_i^2+k(n-j)r_i} \\[-1mm] 
- x_i^{-\kappa (r_i+1)+k(2n-i-j)+i-1} 
t^{\frac{1}{2}\kappa (r_i+1)^2-k(n-j)(r_i+1)}\Big).
\end{multline*}

\begin{lemma}\label{Lem_NisN0}
We have
\begin{equation}\label{Eq_fsigma}
\mathscr{N}_{\kappa,k;\sigma}(x;t)=\mathscr{N}_{\kappa,k}(x;t).
\end{equation}
\end{lemma}

\begin{proof}
Let $\varepsilon\in\{-1,1\}$. Then
\begin{align*}
&\mathscr{N}_{\kappa,k;0}(x;t)+
\varepsilon\mathscr{N}_{\kappa,k;1}(x;t) \\[1mm]
&=\sum_{r\in\Z^n} \det_{1\leqslant i,j\leqslant n} 
\Big( (-1)^{\frac{1}{2}(1+\varepsilon)r_i} 
x_i^{-\kappa r_i-k(i-j)+i-1} t^{\frac{1}{2}\kappa r_i^2+k(n-j)r_i} \\
& \qquad \quad
-(-1)^{\frac{1}{2}(1+\varepsilon)r_i} x_i^{-\kappa (r_i+1)+k(2n-i-j)+i-1}
t^{\frac{1}{2}\kappa (r_i+1)^2-k(n-j)(r_i+1)} \Big) \\
&=\det_{1\leqslant i,j\leqslant n} 
\bigg( \sum_{r\in\Z} (-1)^{\frac{1}{2}(1+\varepsilon)r}
x_i^{-\kappa r-k(i-j)+i-1} t^{\frac{1}{2}\kappa r^2+k(n-j)r} 
\\ & \qquad \quad -\sum_{r\in\Z} (-1)^{\frac{1}{2}(1+\varepsilon)r} 
x_i^{-\kappa (r+1)+k(2n-i-j)+i-1}
t^{\frac{1}{2}\kappa (r+1)^2-k(n-j)(r+1)} \bigg) \\
&=\det_{1\leqslant i,j\leqslant n} 
\bigg( \sum_{r\in\Z} (-1)^{\frac{1}{2}(1+\varepsilon)r}
x_i^{\kappa r-k(i-j)+i-1} t^{\frac{1}{2}\kappa r^2-k(n-j)r} \\
& \qquad \qquad \quad 
+\varepsilon\sum_{r\in\Z} (-1)^{\frac{1}{2}(1+\varepsilon)r}
x_i^{\kappa r+k(2n-i-j)+i-1} t^{\frac{1}{2}\kappa r^2+k(n-j)r} \bigg) \\
&=\sum_{r\in\Z^n} 
\prod_{i=1}^n (-1)^{\frac{1}{2}(1+\varepsilon)r_i}
x_i^{\kappa r_i-(k-1)(i-1)} t^{\frac{1}{2}\kappa r_i^2-k(n-1)r_i} \\[-1mm]
& \qquad \qquad \quad \times \det_{1\leqslant i,j\leqslant n} 
\Big((x_it^{r_i})^{k(j-1)}+\varepsilon(x_it^{r_i})^{k(2n-j-1)}\Big).
\end{align*}
Here the second and last equality use multilinearity and the 
third equality follows from a shift of $r\mapsto r-1$ in the second
sum over $r$.
When $j=n$ the final line reads
$(1+\varepsilon)\big(x_it^{r_i}\big)^{k(n-1)}$, so that
\begin{equation}\label{Eq_nul}
\mathscr{N}_{\kappa,k;0}(x;t)-\mathscr{N}_{\kappa,k;1}(x;t)=0.
\end{equation}
Also, by the $\mathrm{D}_n$ Vandermonde determinant \eqref{Eq_VdMD},
\begin{align*}
\mathscr{N}_{\kappa,k;0}(x;t)&+\mathscr{N}_{\kappa,k;1}(x;t) \\
&=2\sum_{r\in\Z^n} \Delta_{\mathrm{D}}(x^k t^{kr})
\prod_{i=1}^n (-1)^{r_i} x_i^{\kappa r_i-(k-1)(i-1)}
t^{\frac{1}{2}\kappa r_i^2-k(n-1)r_i} \\
&=2\mathscr{N}_{\kappa,k}(x;t).
\end{align*}
Together with \eqref{Eq_nul} this implies \eqref{Eq_fsigma}.
\end{proof}

\begin{proof}[Proof of Lemma~\ref{Lem_Bn-WK-2}]
For $\kappa,k$ positive integers, define
\begin{multline*}
\mathscr{M}_{\kappa,k}(x;t):=
\sum_{\substack{r\in\Z^n \\[1pt] \abs{r}\equiv 0\:(2)}}
\det_{1\leqslant i,j\leqslant n} 
\Big(x_i^{\kappa r_i+k(j-1)-(k-1)(i-1)} 
q^{\kappa\binom{r_i}{2}+k(j-1)r_i} \\[-2mm]
-x_i^{\kappa (r_i+1)-k(j-1)-(k-1)(i-1)} 
q^{\kappa \binom{r_i+1}{2}-k(j-1)r_i}\Big).
\end{multline*}

Taking $c_0=k-1$, $\la_i=\tfrac{1}{2}m+(k-1)(n-i+\tfrac{1}{2})$
$(1\leqslant i\leqslant n)$ in \ref{Eq_char-Bn},
and using \eqref{Eq_odd-orthogonal} and \eqref{Eq_soandsp},
we obtain
\begin{multline}\label{Eq_B-case}
\eup^{-(k-1)\rho-m\fwa_n} \ch V\big((k-1)\rho+m\fwa_n\big) \\[1mm]
=\frac{\mathscr{M}_{m+k(2n-1),k}(x;t)}{(t;t)_{\infty}^n\prod_{i=1}^n 
\theta(x_i;t) \prod_{1\leqslant i<j\leqslant n}x_j \theta(x_ix_j^{\pm};t)},
\end{multline}
where $\mathfrak{g}=\mathrm{B}_n^{(1)}$ and
$x_1,\dots,x_n$ are given by \ref{Eq_qxi}.
Similarly, taking $c_0=k-1$, $\la_i=\tfrac{1}{2}m+(k-1)(n-i+1)$
$(1\leqslant i\leqslant n)$ in \ref{Eq_char-A2n12},
and using \eqref{Eq_symplectic} and \eqref{Eq_soandsp}, we get
\begin{multline}\label{Eq_A-case}
\eup^{-(k-1)\rho-m\fwa_n} \ch V\big((k-1)\rho+m\fwa_n\big) \\[1mm]
=\frac{\mathscr{M}_{m+2kn,k}(x)}{(t;t)_{\infty}^{n-1}(t^2;t^2)_{\infty}
\prod_{i=1}^n \theta(x_i^2;t^2)
\prod_{1\leqslant i<j\leqslant n} x_j \theta(x_ix_j^{\pm};t)},
\end{multline}
where $\mathfrak{g}=\mathrm{A}_{2n-1}^{(2)}$ and
$x_1,\dots,x_n$ are given by \ref{Eq_qxi-2}.

Next we replace $\alpha_i\mapsto\alpha_{n-i}$ for $0\leqslant i\leqslant n$.
This maps maps $\fwa_n$ to $\fwa_0$ but leaves the Weyl vector $\rho$ 
unchanged. 
On the right it has the effect of replacing \eqref{Eq_qxi} by
\[
x_{n-i+1}=\eup^{-\alpha_0-\cdots-\alpha_{i-1}}\quad
(1\leqslant i\leqslant n)
\]
in the case of \eqref{Eq_B-case}, and by
\[
x_{n-i+1}=\eup^{-\alpha_0/2-\alpha_1-\cdots-\alpha_{i-1}}\quad
(1\leqslant i\leqslant n).
\]
in the case of $\mathrm{A}_{2n-1}^{(2)}$.
Moreover, in the first case $t$ is now given by
\[
t=\eup^{-2\alpha_0-\cdots-2\alpha_{n-2}-\alpha_{n-1}-\alpha_n}
\]
in accordance with the interpretation of $\delta$ as the null root of
$\mathrm{B}_n^{(1)\dagger}$, and in the second case by
\[
t=\eup^{-\alpha_0-2\alpha_1-\cdots-2\alpha_{n-2}-\alpha_{n-1}-\alpha_n}
\]
in accordance with $\mathrm{A}_{2n-1}^{(2)\dagger}$.
We now replace $x_i\mapsto t^{1/2}/x_{n-i+1}$---so that the transformed 
$x_i$ is given by \eqref{Eq_qxi-3} in both cases---and $r_i\mapsto r_{n-i+1}$ 
for $1\leqslant i\leqslant n$.
Also reversing the order of the rows and columns in the determinant and using 
$\theta(x;t)=\theta(t/x;t)$, we get
\begin{multline*}
\eup^{-(k-1)\rho-m\fwa_0} \ch V\big((k-1)\rho+m\fwa_0\big) \\
=\frac{1}{\mathscr{D}_{\mathfrak{g}}(x;t)}
\times \begin{cases}\mathscr{N}_{m+k(2n-1),k;0}(x)
& \text{for $\mathrm{B}_n^{(1)\dagger}$}, \\[5mm]
\mathscr{N}_{m+2kn,k;0}(x)
& \text{for $\mathrm{A}_{2n-1}^{(2)\dagger}$}.
\end{cases}
\end{multline*}
By Lemma~\ref{Lem_NisN0} the claim follows.
\end{proof}

\chapter{Limits of elliptic hypergeometric integrals}\label{App_B}

\addtocounter{section}{1}


We review some results of van de Bult and the first author 
\cite{vdBR09,vdBR11} regarding limits of elliptic beta integrals.
We then prove the quadratic transformation formulas of
Theorems~\ref{Thm_Cohl}, \ref{Thm_GR} and \ref{Thm_nonterminating3}
by taking limits in three quadratic transformation formulas
for elliptic beta integrals, originally conjectured in \cite{Rains12}
in the context of elliptic Littlewood identities, and subsequently
proved in \cite{vdBult11,Rains14}. 

\medskip

For complex $p,q$ such that $\abs{p},\abs{q}<1$ and $z\in\mathbb{C}^{\ast}$,
let $\Gamma(z;p,q)$ be the elliptic gamma function \cite{Ruijsenaars97}
\[
\Gamma(z;p,q):=
\prod_{i,j=0}^{\infty} \frac{1-z^{-1}p^{i+1} q^{j+1}}{1-zp^iq^j},
\]
which satisfies the functional equation
\begin{equation}\label{Eq_EG-fun}
\Gamma(z;p,q)\Gamma(pq/z;p,q)=1.
\end{equation}
For $z=(z_1,\dots,z_n)\in(\mathbb{C}^{\ast})^n$ and 
$t,t_0,\dots,t_{2m+5}\in\mathbb{C}^{\ast}$ ($m$ a nonnegative integer)
such that
\[
t^{2n-2}t_0\cdots t_{2m+5}=(pq)^{m+1},
\]
define the density
\[
\Delta(z;t_0,\dots,t_{2m+5};t;p,q)
:=\prod_{i=1}^n \frac{\prod_{r=0}^{2m+5} \Gamma(t_r z_i^{\pm};p,q)}
{\Gamma(z_i^{\pm 2};p,q)}
\prod_{1\leq i<j\leq n} \frac{\Gamma(tz_i^{\pm}z_j^{\pm};p,q)}
{\Gamma(z_i^{\pm}z_j^{\pm};p,q)}.
\]
Note that for $0<\alpha<1$,
\[
\lim_{p\to 0}\Delta(z;t_0,t_1,t_2,t_3,p^{\alpha}t_4,
p^{1-\alpha}q/t_0t_1t_2t_3t_4;t;p,q)
=\Delta(z;q,t;t_0,t_1,t_2,t_3),
\]
with on the right the Koornwinder density \eqref{Eq_Kdensity}.
The elliptic density may be used to define the higher-order elliptic 
Selberg integral (also known as a type-II $\mathrm{C}_n$ beta integral) 
as \cite{Rains10,RW17,Spiridonov03}
\begin{multline}\label{Eq_typeII}
\II^{(n)}_m(t_0,\dots,t_{2m+5};t;p,q) \\
:=\kappa_n \int_{C^n}
\Delta(z;t_0,\dots,t_{2m+5};t;p,q)\,
\frac{\dup z_1}{z_1}\cdots\frac{\dup z_n}{z_n},
\end{multline}
where
\[
\kappa_n:=\frac{1}{2^n n!} 
\bigg(\frac{(p;p)_{\infty}(q;q)_{\infty}\Gamma(t;p,q)}
{2\pi\iup}\bigg)^n
\]
and $C$ is positively oriented, star-shaped Jordan curve around 
the origin such that $C=C^{-1}$ and such that the points
$\{t_r p^iq^j\}_{0\leqslant r\leqslant 2m+5;i,j\geqslant 0}$ 
all lie in the interior of $C$.

For $\alpha=(\alpha_0,\dots,\alpha_7)\in\mathbb{R}^7$ 
and $u=(u_0,\dots,u_7)\in(\mathbb{C}^{\ast})^7$ such that
\begin{equation}\label{Eq_u-restriction}
t^{2n-2}u_0\cdots u_7=q^2,
\end{equation}
let
\[
B^{(n)}_{\alpha}(u;t;p,q):=
\II^{(n)}_1(u_0p^{\alpha_0},\dots,u_7p^{\alpha_7};t;p,q).
\]
From \cite[Proposition 4.3]{vdBR09} and \cite[Proposition 6.5]{vdBR11}
we may infer the following $p\to 0$ limit of $B^{(n)}_{\alpha}(u;t;p,q)$.

\begin{proposition}\label{Prop_P43}
Let $q,t,u_0,\dots,u_7\in\mathbb{C}^{\ast}$ such that 
$\abs{q},\abs{u_1u_2}<1$ and such that \eqref{Eq_u-restriction} holds,
and let 
\[
\alpha=(-\gamma,-\beta,\beta,\gamma,\gamma,\delta,1-\delta,1-\gamma)
\]
for $0\leqslant \beta<\gamma<\delta\leqslant 1/2$. Then
\begin{multline*}
\lim_{p\to 0} \,
(u_0u_1p^{-\gamma-\beta}t^{n-1},
u_0u_2p^{-\gamma+\beta}t^{n-1};q,t)_{\infty^n}\,
B^{(n)}_{\alpha}(u;t,p,q) \\
=\frac{(qt^{n-1}u_0/u_7;q,t)_{\infty^n}}
{(t^n,t^{n-1}u_0u_3,t^{n-1}u_0u_4;q,t)_{\infty^n}}\,
\qhyp{2}{1}{(n)}\bigg[\genfrac{}{}{0pt}{}
{t^{n-1}u_0u_3,t^{n-1}u_0u_4}{qt^{n-1}u_0/u_7};q,t;u_1u_2\bigg].
\end{multline*}
\end{proposition}

To shorten some of our subsequent calculations we restate this in a form 
that hides the symmetry in $u_3$ and $u_4$, and which is obtained by 
applying the multiple analogue of Heine transformation 
\cite[Equation (2.2)]{BF99}
\begin{equation}\label{Eq_Multiple-Heine}
\qhyp{2}{1}{(n)}\bigg[\genfrac{}{}{0pt}{}
{a,b}{c};q,t;z\bigg]
=\frac{(b,azt^{n-1};q,t)_{\infty^n}}
{(c,zt^{n-1};q,t)_{\infty^n}} \,
\qhyp{2}{1}{(n)}\bigg[\genfrac{}{}{0pt}{}
{c/b,zt^{n-1}}{azt^{n-1}};q,t;bt^{1-n}\bigg],
\end{equation}
for $\abs{z},\abs{bt^{1-n}}<1$.

\begin{corollary}\label{Cor_vdBR2}
Let $q,t,u_0,\dots,u_7\in\mathbb{C}^{\ast}$ such that 
$\abs{q},\abs{u_0u_4}<1$ and such that \eqref{Eq_u-restriction} holds,
and let 
\[
\alpha=(-\gamma,-\beta,\beta,\gamma,\gamma,\delta,1-\delta,1-\gamma)
\]
for $0\leqslant \beta<\gamma<\delta\leqslant 1/2$. Then
\begin{multline*}
\lim_{p\to 0} \,
(u_0u_1p^{-\gamma-\beta}t^{n-1},
u_0u_2p^{-\gamma+\beta}t^{n-1};q,t)_{\infty^n}\,
B^{(n)}_{\alpha}(u;t,p,q) \\
=\frac{(q^2/u_4u_5u_6u_7;q,t)_{\infty^n}}
{(t^n,t^{n-1}u_1u_2,t^{n-1}u_0u_3;q,t)_{\infty^n}}\,
\qhyp{2}{1}{(n)}\bigg[\genfrac{}{}{0pt}{}
{t^{n-1}u_1u_2,q/u_4u_7}{q^2/u_4u_5u_6u_7};q,t;u_0u_4\bigg].
\end{multline*}
\end{corollary}

The second $p\to 0$ limit of $B^{(n)}_{\alpha}(u;t;p,q)$ we need
is as follows.

\begin{proposition}\label{Prop_vdBR1}
Let $q,t,u_0,\dots,u_7\in\mathbb{C}^{\ast}$ such that 
$\abs{q},\abs{u_1u_5}<1$ and such that \eqref{Eq_u-restriction} holds,
and let 
\begin{equation}\label{Eq_alpha}
\alpha=(-\gamma,-\gamma,\gamma,\gamma,\gamma,\gamma,1-\gamma,1-\gamma)
\end{equation}
for $0<\gamma<1/2$. Then
\begin{multline}\label{Eq_limit-VWP}
\lim_{p\to 0} \,
(u_0u_1p^{-2\gamma}t^{n-1};q,t)_{\infty^n}\, B^{(n)}_{\alpha}(u;t;p,q) \\
=\frac{(qt^{n-1}u_0/u_6,qt^{n-1}u_0/u_7;q,t)_{\infty^n}}
{(t^n,t^{n-1}u_0u_5,q^2t^{n-1}u_0/u_5u_6u_7;q,t)_{\infty^n}} 
\prod_{r=2}^4 \frac{(q^2/u_ru_5u_6u_7;q,t)_{\infty^n}}
{(t^{n-1}u_0u_r,t^{n-1}u_1u_r;q,t)_{\infty^n}} \\[1mm] \times
\Whyp{8}{7}{(n)}\Big(\frac{qt^{n-1}u_0}{u_5u_6u_7};t^{n-1}u_0u_2,
t^{n-1}u_0u_3,t^{n-1}u_0u_4,
\frac{q}{u_5u_6},\frac{q}{u_5u_7};q,t;u_1u_5\Big).
\end{multline}
\end{proposition}

This result follows from a special case 
($m=1$ and $(\alpha_0,\dots,\alpha_7)$ given by the sequence 
\eqref{Eq_alpha}) of \cite[Proposition 6.3]{vdBR11}. 
After some symmetrisation, this expresses
the left-hand side of \eqref{Eq_limit-VWP} as a virtual Koornwinder integral.
By \cite[Theorem 5.15]{Rains05} this integral can be expressed
as the $\Whyp{8}{7}{(n)}$ series on given on the right of \eqref{Eq_limit-VWP}.

\smallskip

\begin{remark}
There is some redundancy in the expression on the right,
and by the substitution 
\[
(u_1,u_2,\dots,u_7)\mapsto 
(u_0u_1,u_2/u_0,u_3/u_0,u_4/u_0,u_5/u_0,u_0u_6,u_0u_7),
\]
so that the constraint \eqref{Eq_u-restriction} simplifies to
$t^{2n-2}u_1\cdots u_7=q^2$, the $u_0$-dependence drops out.
\end{remark}

\smallskip

We now have the tools to prove the nonterminating quadratic transformation 
formulas of Section~\ref{Sec_KM}. Because it is much simpler than
Theorem~\ref{Thm_Cohl}, we first consider Theorem~\ref{Thm_GR}.

\begin{proof}[Proof of Theorem~\ref{Thm_GR}]
Our starting point is the following quadratic transformation formula
for elliptic Selberg integrals.

\begin{theorem}
For $p,q,t,t_0,\dots,t_4\in\mathbb{C}^{\ast}$ such that 
$\abs{p},\abs{q}<1$ and
\begin{equation}\label{Eq_t0t1t2t3-Q1}
t^{n-1}t_0t_1t_2t_3=pq,
\end{equation} 
we have
\begin{multline}\label{Eq_Con-Q1}
\II^{(n)}_1\big(t^{-1/2}t_0^2,t^{-1/2}t_1^2,t^{-1/2}t_2^2,t^{-1/2}t_3^2,
t^{1/2},pt^{1/2},qt^{1/2},pqt^{1/2};t;p^2,q^2\big) \\
=\prod_{i=1}^{2n}\prod_{0\leqslant r<s\leqslant 2}
\frac{\Gamma(-t^{i/2-1}t_rt_s;p,q)}
{\Gamma(t^{i/2-1/2} t_rt_s;p,q)} \\ \times
\II^{(n)}_1\big(t^{\pm 1/4}t_0,t^{\pm 1/4}t_1,t^{\pm 1/4}t_2,
-t^{\pm 1/4}t_3;t;p,q\big).
\end{multline}
\end{theorem}

This theorem is the special case
\[
\boldsymbol{\la}=\boldsymbol{0}\quad\text{and}\quad
(n,t,t_0,t_1,t_2,u_0)\mapsto (2n,t^{1/2},t^{-1/4}t_0,t^{-1/4}t_1,
t^{-1/4}t_2,-t^{-1/4}t_3)
\]
of \cite[Conjecture Q1]{Rains12}, which was proved in 
\cite[Corollary 7.6]{vdBult11} and \cite[Section 8]{Rains14}.

If we make the substitutions 
\[
(t_0,t_1,t_2,t_3)\mapsto 
(p^{-1/4},ap^{1/4}t^{1/2-n},cp^{1/4}t^{1-n},p^{3/4}qt^{n-1/2}/ac),
\]
consistent with the constraint \eqref{Eq_t0t1t2t3-Q1}, the
transformation \eqref{Eq_Con-Q1} takes the form
\begin{align*}
&B^{(n)}_{\alpha}\bigg(t^{-1/2},t^{1/2},qt^{1/2},a^2t^{1/2-2n},c^2t^{3/2-2n},
t^{1/2},qt^{1/2},\frac{q^2t^{2n-3/2}}{a^2c^2};t;p^2,q^2\bigg) \\
&=\prod_{i=1}^{2n}
\frac{\Gamma(-ct^{(1-i)/2},-at^{-i/2},-acp^{1/2}t^{i/2+1/2-2n};p,q)}
{\Gamma(ct^{1-i/2},at^{(1-i)/2},acp^{1/2}t^{i/2+1-2n};p,q)} \\ 
&\times
B^{(n)}_{\alpha'}\bigg(t^{1/4},t^{-1/4},ct^{5/4-n},ct^{3/4-n},
at^{3/4-n},at^{1/4-n},-\frac{qt^{n-1/4}}{ac},
-\frac{qt^{n-3/4}}{ac};t;p,q\bigg),
\end{align*}
where
\begin{equation}\label{Eq_alpha-alphap}
\alpha=\big({-}\tfrac{1}{4},0,0,\tfrac{1}{4},\tfrac{1}{4},
\tfrac{1}{2},\tfrac{1}{2},\tfrac{3}{4}\big)\quad\text{and}\quad
\alpha'=\big({-}\tfrac{1}{4},-\tfrac{1}{4},\tfrac{1}{4},\tfrac{1}{4},
\tfrac{1}{4},\tfrac{1}{4},\tfrac{3}{4},\tfrac{3}{4}\big).
\end{equation}
Multiplying both sides by
\[
(p^{-1/2}t^{n-1};q,t)_{\infty^n}=
(p^{-1/2}t^{n-1},p^{-1/2}qt^{n-1};q^2,t)_{\infty^n},
\]
we can take the $p\to 0$ limit using Corollary~\ref{Cor_vdBR2} and
Proposition~\ref{Prop_vdBR1}.
After some simplifications the claim follows.
\end{proof}

\smallskip

\begin{proof}[Proof of Theorem~\ref{Thm_Cohl}]
This time the proof is more involved, and as a first step we need to
carry out a non-trivial rewriting of the theorem.

Appealing to analytic continuation, it suffices to prove \eqref{Eq_Cohl}
for $\abs{t}<1$. As a function of $a$, both sides of 
\eqref{Eq_Cohl} are analytic for $\abs{a}<\abs{1/cqt^n}$. 
Hence it is enough to prove the identity 
for the sequence of $a$-values $\{t^N\}_{N\geqslant 0}$
with accumulation point $0$.
If we set $a=t^N$ in \eqref{Eq_Cohl} we obtain
\begin{multline}\label{Eq_presplit}
\qhyp{2}{1}{(n)}\bigg[\genfrac{}{}{0pt}{}
{t^N,t^{N-1}}{qt^{2n-1}};q,t^2;c^2q\bigg]
=\frac{(c^2qt^{2N+2n-2};q,t^2)_{\infty^n}}{(c^2qt^{2n-2};q,t^2)_{\infty^n}} 
\cdot \frac{(cqt^{2n-1};q,t)_{\infty^{2n}}}
{(cqt^{N+2n-1};q,t)_{\infty^{2n}}} \\[1mm] 
\times \Whyp{8}{7}{(n)}
\big(ct^{N+2n-1};ct^{N-1},t^N,-t^n,q^{1/2}t^n,-q^{1/2}t^n;q,t;cq\big).
\end{multline}
Now define integers $m_1,m_2$ as
\[
m_1:=\Floor{\tfrac{N}{2}}\quad\text{and}\quad
m_2:=2\Ceil{\tfrac{N}{2}}-1=2(N-m_1)-1.
\]
Since $\{N-1,N\}=\{2m_1,m_2\}$, it follows from
\eqref{Eq_KM-PS} that
\[
\qhyp{r}{s}{(n)}\bigg[\genfrac{}{}{0pt}{}
{t^N,t^{N-1},a_3,\dots,a_r}{b_1,\dots,b_s};q,t^2;z\bigg]
=\qhyp{2}{1}{(m_1)}\bigg[\genfrac{}{}{0pt}{}
{t^{2n},t^{m_2},a_3,\dots,a_r}{b_1,\dots,b_s};q,t^2;z\bigg].
\]
Also, by \eqref{Eq_VWP-BHS},
\[
\Whyp{r+1}{r}{(n)}(a_1;t^N,a_5,\dots,a_{r+1};q,t;z)=
\Whyp{r+1}{r}{(N)}(a_1;t^n,a_5,\dots,a_{r+1};q,t;z).
\]
Applying these two results to \eqref{Eq_presplit}, and using
\begin{multline*}
\frac{(c^2qt^{2N+2n-2},cqt^{2n-1},cqt^{2n-2};q,t^2)_{\infty^n}}
{(c^2qt^{2n-2},cqt^{N+2n-1},cqt^{N+2n-2};q,t^2)_{\infty^n}} \\[1mm]
=\frac{(c^2qt^{2N+2n-2};q,t^2)_{\infty^N}}{(c^2qt^{2N-2};q,t^2)_{\infty^N}}
\cdot \frac{(cqt^{N-1};q,t)_{\infty^N}}{(cqt^{N+2n-1};q,t)_{\infty^N}},
\end{multline*}
yields
\begin{multline*}
\qhyp{2}{1}{(m_1)}\bigg[\genfrac{}{}{0pt}{}
{b^2,t^{m_2}}{b^2q/t};q,t^2;c^2q\bigg]
=\frac{(b^2c^2qt^{2N-2};q,t^2)_{\infty^N}}{(c^2qt^{2N-2};q,t^2)_{\infty^N}}
\cdot \frac{(cqt^{N-1};q,t)_{\infty^N}}{(b^2cqt^{N-1};q,t)_{\infty^N}} \\[1mm]
\times \Whyp{8}{7}{(N)}
\big(b^2ct^{N-1};ct^{N-1},b,-b,bq^{1/2},-bq^{1/2};q,t;cq\big),
\end{multline*}
where $b=t^n$. 
In the following we will prove this transformation formula
for arbitrary $b$. We may then rename $N$ as $n$ so that
the identity to be proved takes the form
\begin{multline}\label{Eq_tobeproved}
\qhyp{2}{1}{(n_1)}\bigg[\genfrac{}{}{0pt}{}
{b^2,t^{n_2}}{b^2q/t};q,t^2;c^2q\bigg]
=\frac{(b^2c^2qt^{2n-2};q,t^2)_{\infty^n}}{(c^2qt^{2n-2};q,t^2)_{\infty^n}}
\cdot \frac{(cqt^{n-1};q,t)_{\infty^n}}{(b^2cqt^{n-1};q,t)_{\infty^n}} \\[1mm]
\times \Whyp{8}{7}{(n)}
\big(b^2ct^{n-1};ct^{n-1},b,-b,bq^{1/2},-bq^{1/2};q,t;cq\big),
\end{multline}
where $\abs{q},\abs{cq},\abs{c^2q}<1$ and 
\begin{equation}\label{Eq_n1n2}
n_1:=\Floor{\tfrac{n}{2}}, \qquad 
n_2:=2\Ceil{\tfrac{n}{2}}-1=2(n-n_1)-1.
\end{equation}
The prerequisite integral transformation is as follows.

\begin{theorem}
Let $p,q,t,t_0,\dots,t_4\in\mathbb{C}^{\ast}$ such that 
$\abs{p},\abs{q}<1$ and such that \eqref{Eq_t0t1t2t3-Q1} holds.
Let $n_1$ and $n_2$ be given by \eqref{Eq_n1n2} and $\tau$ by
\[
\tau=\begin{cases}
1 & \text{if $n$ is even} \\[1mm]
\displaystyle
\frac{1}{\Gamma(t_0^2,t_1^2,t_2^2,t_3^2,p,t,pt;p^2,q)}
& \text{if $n$ is odd}.
\end{cases}
\]
Then
\begin{multline*}
\II^{(n_1)}_1\big(t_0^2,t_1^2,t_2^2,t_3^2,t^{n_2-2n_1+1},
p,t,pt;t^2;p^2,q\big) \\
=\tau\,\prod_{i=1}^n \prod_{0\leqslant r<s\leqslant 2}
\frac{\Gamma(-t^{i-1}t_rt_s;p,q^{1/2})}
{\Gamma(q^{-1/2}t^{i-1}t_rt_s;p,q^{1/2})} \\ 
\times \II^{(n)}_1\big(q^{\pm 1/4}t_0,q^{\pm 1/4}t_1,
q^{\pm 1/4}t_2,-q^{\pm 1/4}t_3;t;p,q\big).
\end{multline*}
\end{theorem}

Noting that $t^{n_2-2n_1+1}$ is $1$ for $n$ even and $t^2$ for $n$ odd,
the above theorem is \cite[Conjecture Q4]{Rains12} with 
\[
\boldsymbol{\la}=\boldsymbol{0}\quad\text{and}\quad
(q,u_0)\mapsto (q^{1/2},-t_3),
\]
which was proved in \cite[Section 8]{Rains14}.

If we make the substitutions
\[
(t_0,t_1,t_2,t_3)=
(p^{-1/4},bp^{1/4}q^{1/2}t^{1-n},cp^{1/4}q^{1/2},p^{3/4}/bc)
\]
then \eqref{Eq_t0t1t2t3-Q1} is automatically satisfied.
Again defining $\alpha$ and $\alpha'$ as in \eqref{Eq_alpha-alphap},
we thus obtain
\begin{align*}
&B^{(n_1)}_{\alpha}\bigg(1,t^{n_2-2n_1+1},t,b^2qt^{2-2n},c^2q,1,t,
\frac{1}{b^2c^2};t^2;p^2,q\bigg) \\
&\quad=\tau'\, \prod_{i=1}^n \frac{\Gamma(-bq^{1/2}t^{1-i},-cq^{1/2}t^{n-i},
-bcp^{1/2}qt^{1-i};p,q^{1/2})}
{\Gamma(bt^{1-i},ct^{n-i},bcp^{1/2}q^{1/2}t^{1-i};p,q^{1/2})} \\ 
&\qquad \times B^{(n)}_{\alpha'}\bigg(
q^{-1/4},q^{1/4},bq^{1/4}t^{1-n},bq^{3/4}t^{1-n},
cq^{1/4},cq^{3/4},-\frac{q^{-1/4}}{bc},-\frac{q^{1/4}}{bc};t;p,q\bigg),
\end{align*}
where $\tau'=1$ if $n$ is even and 
\[
\tau'=
\frac{1}{\Gamma(p^{-1/2},b^2p^{1/2}qt^{2-2n},c^2p^{1/2}q,
p^{3/2}/b^2c^2,p,t,pt;p^2,q)}.
\] 
After multiplying both sides by 
\begin{multline*}
(p^{-1/2}t^{n_2-1},p^{-1/2}t^{2n_1-1};q,t^2)_{\infty^{n_1}} \\
=(p^{-1/2}t^{n-1};q,t)_{\infty^n}
\times 
\begin{cases}
1 & \text{if $n$ is even} \\[1mm] \displaystyle
\frac{1}{(p^{-1/2};q)_{\infty}} & \text{if $n$ is odd},
\end{cases}
\end{multline*}
we can use Corollary~\ref{Cor_vdBR2} and Proposition~\ref{Prop_vdBR1} 
to take the $p\to 0$ limit, resulting in
\eqref{Eq_tobeproved}.
\end{proof}

\smallskip

\begin{proof}[Proof of Theorem~\ref{Thm_nonterminating3}]
Our final proof is very similar to that of Theorem~\ref{Thm_GR}.
We begin with the $\boldsymbol{\la}=0$ case of
\cite[Conjecture Q3]{Rains12}, proved in \cite[Theorem 6.1]{vdBult11} 
and \cite[Section 8]{Rains14}.

\begin{theorem}
For $p,q,t,u,t_0,\dots,t_3\in\mathbb{C}^{\ast}$ such that 
$\abs{p},\abs{q}<1$ and
\begin{equation}\label{Eq_t0t1t2t3-Q2}
t^{2n-2}t_0t_1t_2t_3=pq^2,
\end{equation} 
we have
\begin{align*}
\II^{(n)}_2&\big(t^{-1/2}t_0,t^{-1/2}t_1,t^{-1/2}t_2,
t^{1/2}t_3,t^{-1/2}u,pqt^{-1/2}/u,
\pm t^{1/2},\pm (pt)^{1/2};t;p,q\big) \\
&=\prod_{i=1}^n\prod_{0\leqslant r<s\leqslant 2}
\frac{\Gamma(t^{2i-3}t_rt_s;p,q)}
{\Gamma(t^{2i-2} t_rt_s/q;p,q)} \\ 
& \qquad \times
\II^{(n)}_2\big(q^{\pm 1/2}t_0,q^{\pm 1/2}t_1,q^{\pm 1/2}t_2,
q^{\pm 1/2}tt_3,q^{1/2}u/t,pq^{3/2}/tu;t^2;p,q^2\big).
\end{align*}
\end{theorem}

For $u=t_3$ both sides reduce to a 
$\II^{(n)}_1$ integral by \eqref{Eq_EG-fun}, so that
\begin{multline*}
\II^{(n)}_1\big(t^{-1/2}t_0,t^{-1/2}t_1,t^{-1/2}t_2,t^{-1/2}t_3,
\pm t^{1/2},\pm (pt)^{1/2};t;p,q\big) \\
=\prod_{i=1}^n\prod_{0\leqslant r<s\leqslant 2}
\frac{\Gamma(t^{2i-3}t_rt_s;p,q)}
{\Gamma(t^{2i-2} t_rt_s/q;p,q)}\qquad\qquad  \\ \times
\II^{(n)}_1\big(q^{\pm 1/2}t_0,q^{\pm 1/2}t_1,q^{\pm 1/2}t_2,
(q^{1/2}/t)^{\pm}t_3;t^2;p,q^2\big).
\end{multline*}

After carrying out the substitutions
\[
(t_0,t_1,t_2,t_3)\mapsto 
(p^{-1/4},ap^{1/4}qt^{1-2n},p^{3/4}q/ac,cp^{1/4}t),
\]
consistent with \eqref{Eq_t0t1t2t3-Q2}, we then obtain
\begin{align*}
&B^{(n)}_{\alpha}\bigg(t^{-1/2},t^{1/2},-t^{1/2},
aqt^{1/2-2n},ct^{1/2},t^{1/2},-t^{1/2},\frac{qt^{-1/2}}{ac};t;p,q\bigg) \\
&=\prod_{i=1}^n
\frac{\Gamma(aqt^{-2i},p^{1/2}qt^{2n-2i-1}/ac,pq^2t^{-2i}/c;p,q)}
{\Gamma(at^{1-2i},p^{1/2}t^{2n-2i}/ac,pqt^{1-2i}/c;p,q)} \\
&\times
B^{(n)}_{\alpha'}\bigg(q^{1/2},q^{-1/2},aq^{3/2}t^{1-2n},aq^{1/2}t^{1-2n},
cq^{-1/2}t^2,cq^{1/2},\frac{q^{3/2}}{ac},
\frac{q^{1/2}}{ac};t^2;p,q^2\bigg),
\end{align*}
with $\alpha$ and $\alpha'$ once again given by \eqref{Eq_alpha-alphap}.
Multiplying both sides by
\[
(p^{-1/4}t^{n-1},-p^{1/4}t^{n-1};q,t)_{\infty^n}=
(p^{-1/2}t^{2n-2};q^2,t^2)_{\infty^n}
\]
and then letting $p$ tend to $0$ yields \eqref{Eq_nonterminating3}.
\end{proof}

\backmatter

\bibliographystyle{amsalpha}

\printindex

\end{document}